\documentclass[11pt]{article}

\usepackage{amsmath,amssymb,amsfonts}
\usepackage[margin=1in]{geometry}
\usepackage[T1]{fontenc}
\usepackage{lmodern}
\usepackage{microtype}
\usepackage{amsthm,mathtools,mathrsfs}
\usepackage{array}
\usepackage{enumitem}
\usepackage{authblk}
\usepackage{xcolor}
\usepackage{url}
\usepackage[hypertexnames=false,colorlinks=true,linkcolor=blue!55!black,
  citecolor=blue!55!black,urlcolor=blue!55!black]{hyperref}
\usepackage[capitalize,noabbrev]{cleveref}

\numberwithin{equation}{section}

\hypersetup{
  pdftitle={Locality of Bernoulli Site Percolation on Transitive Graphs},
  pdfauthor={Tho Tran and Tan M. Nguyen}
}

\allowdisplaybreaks
\theoremstyle{plain}
\newtheorem{theorem}{Theorem}[section]
\newtheorem{proposition}[theorem]{Proposition}
\newtheorem{lemma}[theorem]{Lemma}

\newtheorem{auxiliary}{Auxiliary Lemma}[section]

\theoremstyle{definition}

\title{Locality of Bernoulli Site Percolation on Transitive Graphs}
\author[1]{Tho Tran}
\author[1]{Tan M. Nguyen}
\affil[1]{Department of Mathematics, National University of Singapore}
\date{}

\begin{document}

\maketitle

\begin{abstract}
We prove locality of the critical probability for Bernoulli site
percolation on infinite, connected, locally finite, vertex-transitive
graphs, under the usual assumption that the critical probabilities stay
uniformly below one. The proof develops site-percolation forms of the
two-ghost inequality, sharp-threshold and snowballing estimates, and the
nonunimodular height method. We also establish the plentiful-tube geometry
needed for the multiscale argument from quantitative structure and
random-walk estimates.
\end{abstract}

\begingroup
\footnotesize
\tableofcontents
\endgroup
\clearpage

\section{Introduction and main result}

Let \(G=(V,E)\) be an infinite, connected, locally finite,
vertex-transitive graph. In Bernoulli site percolation the variables
\((\omega(v))_{v\in V}\) are independent and
\[
  \mathbb P_p(\omega(v)=1)=p.
\]
A path is open when all its vertices are open, and
\[
 p_c^{\rm s}(G)
 :=\inf\{p\in[0,1]:\mathbb P_p(o\leftrightarrow\infty)>0\}.
\]

\begin{theorem}[Site locality]\label{thm:site-locality}
Let \((G_n)_{n\geq1}\) and \(G\) be infinite, connected, locally finite,
vertex-transitive graphs. If \(G_n\to G\) locally and
\(\limsup_n p_c^{\rm s}(G_n)<1\), then
\[
  p_c^{\rm s}(G_n)\longrightarrow p_c^{\rm s}(G).
\]
\end{theorem}

Here local convergence means that, after choosing arbitrary roots, the
rooted radius-\(r\) balls of \(G_n\) and \(G\) are isomorphic for every
fixed \(r\) and all sufficiently large \(n\). Transitivity makes this
definition independent of the chosen roots.

\textbf{Radius convention.} If a ball, sphere, or radial connectivity
event is displayed with a noninteger radius \(r\), that radius means
\(\lfloor r\rfloor\). Endpoints of an interval of admissible integer
scales are rounded inward. All additive constants in radial comparisons
are chosen to absorb this convention.

\textbf{Selection convention.} Every graph considered here is countable.
Fix an enumeration of its vertices and use the induced length-lexicographic
order on finite paths, sets, and families. Whenever a measurable finite
witness is required, choose the least admissible one in this order.
Choices entering a mass transport are always made equivariantly, or the
transport is summed over all admissible choices.

\textbf{Unimodularity convention.} Throughout the paper, a transitive
graph is called unimodular or nonunimodular according to the canonical
action of its full automorphism group \(\operatorname{Aut}(G)\). In
Section~5 we write \(\Gamma=\operatorname{Aut}(G)\) to keep the modular
formulas readable; no assertion there is made for a different transitive
subgroup.

The proof begins in Section~2 with the product-kernel and shared-junction
lemmas used by later explorations, followed by lower semicontinuity.
Section~3 proves the site two-ghost estimate. Section~4 extracts the
finite-energy path-opening lemma used in the final deduction. Section~5
proves the nonunimodular height
estimates in vertex coordinates, using the modular framework and tilted
mass transport from
\cite[Section~2.1 and Proposition~2.2]{Hutchcroft2020Nonunimodular}.
Section~6 proves the site snowballing estimate, and Section~7 supplies the
geometric input and carries out the unrestricted multiscale induction.
Section~8 combines lower semicontinuity with the published
polynomial-growth case, the nonunimodular result of Section~5, and the
remaining canonically unimodular superpolynomial-growth result of
Section~7.

The bond-percolation proof of Easo and Hutchcroft
\cite{EasoHutchcroft2023} motivates the overall strategy. The main new
site-percolation ingredients are the nonunimodular theorem of Section~5,
the snowballing estimate of Section~6, and the low-growth step of
Section~7. The geometric argument in Section~7 is proved from the
structure, presentation, heat-kernel, and random-walk results cited there.

Constants may depend on the common degree \(d\) and a compact parameter
interval \(a\leq p\leq 1-b\). This is enough: local convergence makes the
degrees eventually equal, elementary exploration gives
\(p_c^{\rm s}\geq1/(d-1)\), and the hypothesis supplies \(b>0\).

\section{Preliminaries and lower semicontinuity}

\subsection{Stopped product explorations}

\begin{lemma}[Stopped product kernel]\label{lem:stopped-product-kernel}
Let \(I\) be countable, let \((U_i)_{i\in I}\) be independent variables
with laws \((\mu_i)_{i\in I}\) on standard Borel spaces, and let
\(\mathcal A\) be an auxiliary sigma-field independent of all \(U_i\).
Consider an algorithm that, before its \(n\)-th query, chooses an
as-yet-unqueried index \(Q_n\) using only \(\mathcal A\) and the preceding
query results. It may stop after finitely many queries or continue
forever. Let \(\mathcal F_\infty\) be the sigma-field generated by the
auxiliary data, the query identities, and the observed results, and let
\(\mathcal Q\) be the random set of queried indices.

For every finite deterministic \(S\subseteq I\), every bounded measurable
\(f\) of \((U_i)_{i\in S}\), and every bounded
\(\mathcal F_\infty\)-measurable \(H\),
\begin{equation}
\mathbb E\!\left[
 H\mathbf1_{\{S\cap\mathcal Q=\varnothing\}}
 f((U_i)_{i\in S})\right]
=
\mathbb E\!\left[
 H\mathbf1_{\{S\cap\mathcal Q=\varnothing\}}\right]
\int f\,d\!\bigotimes_{i\in S}\mu_i.                         \label{eq:stopped-product-kernel}
\end{equation}
Thus, conditionally on the terminal exploration sigma-field, the
unqueried coordinates have their original product law in the precise
finite-dimensional sense of \eqref{eq:stopped-product-kernel}.

The conclusion remains valid when a query reports only the cell
containing \(U_{Q_n}\) in a finite measurable partition chosen before
that coordinate is read. At a finite history of positive probability,
the queried coordinates have the product of their one-coordinate
conditional laws, while unqueried coordinates retain their original
laws. At a terminal history the same assertion is understood as an
identity of regular conditional kernels, determined by its
finite-dimensional integrals.
\end{lemma}

\begin{proof}
Stop the algorithm after \(n\) queries and write
\(\mathcal F_n\) and \(\mathcal Q_n\) for the resulting sigma-field and
queried set. Induction on \(n\) gives
\[
 \mathbb E\!\left[
 H_n\mathbf1_{\{S\cap\mathcal Q_n=\varnothing\}}
 f((U_i)_{i\in S})\right]
 =
 \mathbb E\!\left[
 H_n\mathbf1_{\{S\cap\mathcal Q_n=\varnothing\}}\right]
 \int f\,d\!\bigotimes_{i\in S}\mu_i
\]
for every bounded \(\mathcal F_n\)-measurable \(H_n\). Indeed, the next
query identity is measurable before its coordinate is read. On the
branch \(Q_{n+1}\in S\) the indicator on both sides vanishes; on the
other branches the newly read variable is independent of
\((U_i)_{i\in S}\), so conditioning on \(\mathcal F_n\) preserves the
factorization.

For bounded \(H\in\mathcal F_\infty\), put
\(H_n=\mathbb E[H\mid\mathcal F_n]\). Martingale convergence gives
\(H_n\to H\) in \(L^1\), and
\(\mathbf1_{\{S\cap\mathcal Q_n=\varnothing\}}\) decreases to
\(\mathbf1_{\{S\cap\mathcal Q=\varnothing\}}\). Passing to the limit
proves \eqref{eq:stopped-product-kernel}. If a query reports a partition
cell rather than the full coordinate, the finite-stage joint density is
the product of the indicators of the reported cells times the original
product density. Dividing by its probability gives the product of the
restricted one-coordinate laws. The same martingale argument passes this
identity to the terminal sigma-field.
\end{proof}

\subsection{Shared junctions and lower semicontinuity}

\begin{lemma}[Finite Bernoulli disjoint occurrence]
\label{lem:finite-bkr}
Let \(I\) be finite and let \(\mathbb P\) be a Bernoulli product measure
on \(\{0,1\}^I\). For events \(A,B\subseteq\{0,1\}^I\), let
\(A\mathbin{\square}B\) be the event that there are disjoint coordinate
sets \(K,L\subseteq I\) such that fixing the coordinates of the current
configuration on \(K\) certifies \(A\), while fixing those on \(L\)
certifies \(B\). Then
\begin{equation}
 \mathbb P(A\mathbin{\square}B)
 \leq \mathbb P(A)\mathbb P(B).                              \label{eq:finite-bkr}
\end{equation}
For increasing \(A,B\), this is the inequality introduced by van den
Berg and Kesten \cite{vanDenBergKesten1985}; the logical input used here,
for arbitrary \(A,B\), is the main product-space theorem of
\cite{Reimer2000}.
\end{lemma}

\begin{proof}
The cited Reimer theorem gives \eqref{eq:finite-bkr} for arbitrary events
in a finite product probability space, and therefore also its increasing
specialization. The separate van den Berg--Kesten citation records the
original increasing-event result used historically in percolation.
\end{proof}

\begin{lemma}[Shared-junction BK inequality]\label{lem:site-junction}
Let \(0<p\leq1\), let \(T\) be a finite tree rooted at \(r\), and map
its vertices injectively to \(J=\{x_a:a\in V(T)\}\subseteq V(G)\).
For each \(ab\in E(T)\), let \(\mathscr C_{ab}\) be a two-terminal
open-path event that includes the openness of \(x_a,x_b\), has a witness
whose internal vertices avoid \(J\), and whose restriction is independent
of the states of \(J\setminus\{x_a,x_b\}\). After conditioning every
site of \(J\) open, write
\[
 \mathop{\square}_{V\setminus J}
       \{\mathscr C_{ab}:ab\in E(T)\}
\]
for ordinary BK disjoint occurrence in the residual product space
\(\{0,1\}^{V\setminus J}\): each residual event has a finite witness
set contained in \(V\setminus J\), and these witness sets are pairwise
disjoint. Thus the symbol \(\square_{V\setminus J}\) permits no shared
residual coordinate; only the already conditioned coordinates in \(J\)
are omitted from the witnesses. Let \(\mathscr A\) be the corresponding
unconditional event, including the required openness of the junctions.
Then
\begin{equation}
 \mathbb P_p\left(\mathscr A\,\middle|\,x_r\text{ open}\right)
 \leq\prod_{ab\in E(T)}\frac{\mathbb P_p(\mathscr C_{ab})}{p}.
                                                                  \label{eq:shared-junction}
\end{equation}
In particular, if \(\mathscr A_T\) is the event that every edge \(ab\)
is represented by an open connection between \(x_a,x_b\), with the
representing paths disjoint apart from prescribed common endpoints, then
\[
 \mathbb P_p(\mathscr A_T\mid x_r\text{ open})
 \leq\prod_{ab\in E(T)}
       \frac{\mathbb P_p(x_a\leftrightarrow x_b)}p.
\]
\end{lemma}

\begin{proof}
Condition first on every site of \(J\) being open. First require every
residual witness to lie in a common finite set \(\Lambda\supseteq J\),
and denote the resulting restrictions by \(\mathscr C_{ab,\Lambda}\)
and \(\mathscr A_\Lambda\), respectively.
By the definition of \(\square_{V\setminus J}\), the path interiors are
completely disjoint witness sets in the finite product space on
\(\Lambda\setminus J\). Iterating Lemma~\ref{lem:finite-bkr} gives
\[
 \mathbb P_p\left(
   \mathscr A_\Lambda\,\middle|\,J\text{ open}\right)
 \leq\prod_{ab\in E(T)}
   \mathbb P_p(\mathscr C_{ab,\Lambda}\mid J\text{ open}).
\]
Along an increasing finite exhaustion of \(V\), each side increases to
the corresponding unrestricted finite-path event: every witness uses
only finitely many vertices. Continuity from below therefore gives the
same display with \(\mathscr A\) and \(\mathscr C_{ab}\).
By the stated insensitivity to the other junction coordinates,
\[
 \mathbb P_p(\mathscr C_{ab}\mid J\text{ open})
 =\mathbb P_p(\mathscr C_{ab}\mid x_a,x_b\text{ open})
 =\frac{\mathbb P_p(\mathscr C_{ab})}{p^2}.
\]
Under the law conditioned on \(x_r\) being open, the other
\(|V(T)|-1\) junctions are open with probability \(p^{|V(T)|-1}\).
Since \(|E(T)|=|V(T)|-1\), multiplication cancels one of the two endpoint
factors for each edge and proves \eqref{eq:shared-junction}. For the final
assertion, observe that an interior cannot contain a third junction: that
junction is an endpoint of another representing path, contradicting the
assumed disjointness. Hence every witnessing arm is internally contained
in \(V\setminus J\), and its restricted probability is at most the
corresponding unrestricted connectivity probability.
\end{proof}

\begin{lemma}[Lower semicontinuity]\label{lem:lower-semicontinuity}
If \(G_n\to G\), then
\[
  p_c^{\rm s}(G)\leq\liminf_n p_c^{\rm s}(G_n).
\]
\end{lemma}

\begin{proof}
The case \(p=0\) is immediate, so fix
\(0<p<p_c^{\rm s}(G)\). The vertex-cut characterization
\cite[Theorem~2.3]{Li2025} of the site critical point gives a vertex cutset \(\Pi\),
separating \(o\)
from infinity, such that
\begin{equation}
 s_\Pi:=\sum_{v\in\Pi}
 \mathbb P_p(o\leftrightarrow v\text{ without meeting }
                  \Pi\setminus\{v\})<\min\{1,p\}.       \label{eq:2-1}
\end{equation}
In particular \(o\notin\Pi\), since otherwise its summand is \(p\).
The component \(C\) of \(o\) in \(G\setminus\Pi\) is finite; otherwise
local finiteness and Konig's infinity lemma give a ray avoiding \(\Pi\).
Replace \(\Pi\) by the finite set
\(\Pi'=\{v\in\Pi:v\sim C\}\). This is still a cutset. Moreover, for
\(v\in\Pi'\), a path from \(o\) to \(v\) avoiding the other vertices of
\(\Pi'\) stays in \(C\) until its last vertex, so its event is unchanged
if \(\Pi\) is replaced by \(\Pi'\). Thus the sum for \(\Pi'\) is no
larger than the sum in \eqref{eq:2-1}. Relabel \(\Pi'\) as \(\Pi\).
Choose \(R\) so that the closed one-neighborhood of \(C\cup\Pi\), and
hence every edge incident to this finite template, is contained in
\(B_R(o)\). The event in each summand of \eqref{eq:2-1} then has a witness inside
\(C\cup\{v\}\), and the assertion that \(\Pi\) separates this component
uses only the incident edges just included. Consequently \eqref{eq:2-1} depends
on the rooted \(R\)-ball and has the same value in \(G_n\) for all large
\(n\). By transitivity the same finite rooted
cutset template is available at every vertex of \(G_n\). Fix one
automorphic copy \((C_x,\Pi_x)\) rooted at each vertex \(x\). Along every
self-avoiding ray starting at \(o\), stop first when it exits \(C_o\),
then first when its remaining tail exits the copy rooted at that stopping
vertex, and so on. Every stopping time is finite because every \(C_x\) is
finite, and the successive path segments have disjoint interiors; only
consecutive segments share their prescribed stopping vertex.

For \(v\in\Pi\), write
\(a_v=\mathbb P_p(o\leftrightarrow v\text{ without meeting }
\Pi\setminus\{v\})\). For any prescribed sequence of \(k\) exits, the
formal shared-junction inequality, Lemma~\ref{lem:site-junction}, applied
to the path tree with the restricted segment events bounds the
probability of its disjoint segment witnesses
by
\[
 p\prod_{i=1}^k\frac{a_{v_i}}p.
\]
This remains an upper bound if different copied templates overlap: the
chronological pieces of the self-avoiding witness ray still have disjoint
interiors. Summing over the finitely many exit choices at each stage shows
that the probability that some open ray reaches its \(k\)-th stopping
vertex is at most
\[
 p\left(\frac{s_\Pi}{p}\right)^k.
\]
It tends to zero, so \(o\) does not percolate in \(G_n\) at \(p\). Thus
\(p\leq p_c^{\rm s}(G_n)\) eventually. Let
\(p\uparrow p_c^{\rm s}(G)\).
\end{proof}

\section{Site two-ghost inequality}

For a finite open cluster \(K\), put
\[
 \partial_VK=\{z\notin K:z\sim x\text{ for some }x\in K\},
 \qquad T(K)=K\cup\partial_VK,
\]
and
\[
 H_p(K)=p|\partial_VK|-(1-p)|K|.
\]
When the root \(o\) is closed, we set \(K_o=\varnothing\). Every expression
below containing a quotient by \(|K_o|\) or \(|T(K_o)|\) is, by
convention, zero on this event; each quotient is evaluated only when the
root is open.
Let \(\mathcal G_h\) be an independent ghost field on vertices, each
vertex being marked with probability \(1-e^{-h}\). Let \(\mathscr T_v\)
be the event that \(v\) is closed and touches two distinct open clusters
that both meet \(\mathcal G_h\), at least one of which is finite.

\begin{lemma}[Site two-ghost]\label{lem:site-two-ghost}
Let \(0<p<1\) and \(h>0\). On a unimodular transitive graph of
degree \(d\),
\begin{equation}
 \mathbb P_{p,h}(\mathscr T_v)
    \leq C_d\sqrt{\frac{1-p}{p}\,h}.                  \label{eq:3-1}
\end{equation}
Consequently, if \(\mathscr S_{v,m}\) is the event that \(v\) is
closed and touches two distinct clusters of size at least \(m\), at
least one finite, then
\[
 \mathbb P_p(\mathscr S_{v,m})
    \leq C_d\sqrt{\frac{1-p}{pm}}.
\]
\end{lemma}

\textbf{Proof.} Force \(v\) closed temporarily. Let \(N_f(v)\) be the number
of distinct finite, ghost-marked open clusters adjacent to \(v\), and let
\(I(v)\) indicate that an infinite open cluster is adjacent to \(v\).
Every infinite cluster meets the positive-intensity ghost field almost
surely. We have
\begin{equation}
\begin{aligned}
 1_{\mathscr T_v}
 &\leq 1_{\{v\ {\rm closed}\}}
 \left[N_f(v)-1_{\{N_f(v)\geq1,\ I(v)=0\}}\right],\\
 \mathbb P(v\text{ closed}\mid\omega|_{V\setminus\{v\}})
 &=\frac{1-p}{p}
 \mathbb P(v\text{ open}\mid\omega|_{V\setminus\{v\}}).
\end{aligned}                                                   \label{eq:3-3}
\end{equation}
Indeed, the bracket is \(r-1\geq1\) when \(v\) touches \(r\geq2\)
finite marked clusters and is \(r\geq1\) when it also touches an infinite
cluster.

For a finite cluster \(K\), write
\(G(K)=|K\cap\mathcal G_h|\), and distribute its contribution uniformly
over \(T(K)\). More explicitly, a closed boundary vertex \(v\) sends
\(|T(K)|^{-1}\) to each vertex of \(T(K)\) for every adjacent finite
ghost-marked cluster \(K\). An open vertex \(v\) with finite cluster
\(K_v\) sends \(-(1-p)/(p|T(K_v)|)\) to every vertex of \(T(K_v)\) when
\(G(K_v\setminus\{v\})\geq1\). Equivalently, the contribution made by a
sender \(x\in T(K)\), to each recipient in \(T(K)\), is
\begin{equation}
 \frac1{|T(K)|}\left[
 1_{\{x\in\partial_VK\}}1_{\{G(K)\geq1\}}
 -\frac{1-p}{p}1_{\{x\in K\}}
       1_{\{G(K\setminus\{x\})\geq1\}}\right].                  \label{eq:3-4}
\end{equation}
This mass is diagonally invariant. Its absolute outgoing mass is at most
\(d+(1-p)/p\), so the unimodular mass-transport principle
\cite[Corollary~3.5]{BenjaminiLyonsPeresSchramm1999} applies to the
positive and negative parts as justified below. Summing \eqref{eq:3-4} over
senders gives, for each \(K\),
\begin{equation}
 \frac{H_p(K)}{p|T(K)|}1_{\{G(K)\geq1\}}
 +\frac{1-p}{p|T(K)|}1_{\{G(K)=1\}}.                            \label{eq:3-5}
\end{equation}
For clarity, the signed use is justified without an implicit
integrability convention. Split \eqref{eq:3-4} into its positive and negative
parts and, for mass sent from \(x\), first retain only clusters \(K\)
with \(|K|\leq n\) and \(K\subseteq B_n(x)\). This truncation remains
diagonally invariant. Apply the nonnegative mass-transport principle to
each truncated part. The absolute outgoing mass is bounded by
\(d+(1-p)/p\), uniformly in \(n\); nonnegative mass transport gives the
same bound for the incoming mass. Monotone convergence for the two
parts, followed by subtraction, gives \eqref{eq:3-5}.
The last term is the endpoint correction: when \(K\) has exactly one
ghost, the marked vertex does not contribute to the deletion event in
the second term of \eqref{eq:3-4}. The expected outgoing mass is precisely the
expectation of the upper bound in the first line of \eqref{eq:3-3}, after the
open/closed switch at its source.

A vertex lies in at most \(d\) external cluster boundaries and at most
one open cluster. We use the following cluster-averaging identity. For
every nonnegative invariant functional \(f\) of a finite open cluster,
mass transport, first spreading \(f(K)\) uniformly over \(T(K)\) and
then uniformly over \(K\), gives
\[
 \mathbb E\sum_{K:o\in T(K)}\frac{f(K)}{|T(K)|}
 =\mathbb E\left[
    \frac{f(K_o)}{|K_o|}\mathbf1_{\{o\text{ open},\ |K_o|<\infty\}}
   \right].
\]
Since \(|T(K)|\leq(d+1)|K|\), the right side is at most
\((d+1)\mathbb E[f(K_o)/|T(K_o)|;o\text{ open},|K_o|<\infty]\).
Apply this separately to the absolute first term and to the positive
correction in \eqref{eq:3-5}. Thus \eqref{eq:3-3}--\eqref{eq:3-5}, transitivity, and the triangle
inequality give
\begin{equation}
\mathbb P_{p,h}(\mathscr T_v)
\leq\frac{d+1}{p}\,
\mathbb E_p\left[
 \frac{|H_p(K_o)|}{|T(K_o)|}
 1_{\{|K_o|<\infty,\ K_o\cap\mathcal G_h\ne\varnothing\}}
 \right]
 +C_d\frac{1-p}{p}h.                                           \label{eq:3-6}
\end{equation}

Explore \(o\) first and, when it is open, successively explore untested
vertices on the external boundary of the discovered cluster. If \(T\)
is the number tested when exploration stops, then on a finite nonempty
cluster \(T=|T(K_o)|\). The process
\[
 Z_n=\sum_{i=1}^{n\wedge T}
 \left[(1-p)1_{\{i\text{th site open}\}}
       -p1_{\{i\text{th site closed}\}}\right]
\]
is a martingale, \(Z_T=-H_p(K_o)\), and
\(\mathbb E Z_n^2\leq p(1-p)n\). The dyadic stopping-time argument and
Doob's \(L^2\) inequality yield
\[
 \mathbb E_p\left[
 \frac{|H_p(K_o)|}{|T(K_o)|}
 (1-e^{-h|K_o|})1_{\{|K_o|<\infty\}}\right]
 \leq C\sqrt{p(1-p)h}.
\]
Only the event \(\{o\text{ open},|K_o|<\infty\}\) contributes to the
left side. On the closed-root event the quotient is zero by convention,
and on the infinite-cluster event it is removed by the displayed
indicator. Thus no assertion that the exploration stops almost surely is
being used. For each dyadic block below, \(T<2^{j+1}\) is a finite
stopping event and \(|Z_T|\leq\max_{i\leq2^{j+1}}|Z_i|\); all applications
of Doob's inequality are therefore to a bounded-time martingale.
Indeed, \(|K_o|\leq T\) and
\(1-e^{-h|K_o|}\leq\min\{hT,1\}\). Split \(T<\infty\) into the
events \(2^j\leq T<2^{j+1}\). On blocks below \(h^{-1}\), the summand
is at most \(h\max_{i\leq2^{j+1}}|Z_i|\); on blocks above \(h^{-1}\),
it is at most \(2^{-j}\max_{i\leq2^{j+1}}|Z_i|\). Doob's inequality
gives
\[
 \mathbb E\max_{i\leq2^{j+1}}|Z_i|
 \leq2\sqrt{p(1-p)2^{j+1}}.
\]
The two resulting geometric sums are each at most
\(C\sqrt{p(1-p)h}\). The same conclusion for \(h>1\) follows by
starting the second sum at \(j=0\).
Conditional on \(K_o\), the correction in \eqref{eq:3-6} is bounded by
\((1-p)h/p\), since
\(\mathbb P(G(K_o)=1\mid K_o)\leq h|K_o|\) and
the cluster-averaged denominator for this term is \(|K_o|\). Put
\(x=(1-p)h/p\). If \(x\leq1\), then
\(x\leq\sqrt{x}\); if \(x>1\), use the trivial probability bound
\(1\leq\sqrt{x}\). Substitution into \eqref{eq:3-6} proves \eqref{eq:3-1}.

On \(\mathscr S_{v,m}\), each large cluster meets the ghost field with
conditional probability at least \(1-e^{-hm}\). Positive association of
the ghost variables gives
\[
 (1-e^{-hm})^2\mathbb P_p(\mathscr S_{v,m})
 \leq\mathbb P_{p,h}(\mathscr T_v).
\]
Take \(h=1/m\) and absorb \((1-e^{-1})^{-2}\) into \(C_d\).
\(\square\)

\section{Finite-energy path opening}

\begin{lemma}[Finite-energy path opening]
\label{lem:finite-energy-path-opening}
Let \(G\) be a unimodular transitive graph of degree at most \(d\), let
\(0<p<1\), and suppose that Bernoulli \(p\)-site percolation has almost
surely at most one infinite open cluster. Let \(m\geq1\). If
\(x,y\in V(G)\) are joined by a fixed simple path of length \(\ell\), then
\begin{equation}
 \mathbb P_p(|K_o|\geq m)^2
 \leq \mathbb P_p(x\leftrightarrow y)
   +C_d(\ell+1)(2/p)^{\ell+1}
      \sqrt{\frac{1-p}{pm}}.                                \label{eq:4-path-opening}
\end{equation}
\end{lemma}

\begin{proof}
By site FKG and transitivity,
\[
 \mathbb P_p(|K_x|\geq m,|K_y|\geq m)
 \geq\mathbb P_p(|K_o|\geq m)^2.
\]
The part of this event on which \(x\leftrightarrow y\) has probability
at most \(\mathbb P_p(x\leftrightarrow y)\). On the complementary event,
scan the fixed path from \(x\) to \(y\), changing a closed path site to
open whenever it is encountered, and stop immediately before the first
change that would join the two clusters containing the original clusters
of \(x\) and \(y\). Such a step exists because opening the entire path
joins them. Let \(z\) be the still-closed site at that step and let \(S\)
be the set of earlier path sites changed from closed to open. In the
current configuration, \(z\) touches two distinct clusters, each
containing one of the original clusters and hence each of size at least
\(m\). At least one of these two clusters is finite. Indeed, outside a
null event the original configuration has at most one infinite cluster,
and opening the finite set \(S\) cannot increase the number of infinite
clusters: every infinite cluster after the modification must contain an
infinite cluster from the original configuration, since \(S\) and its
set of incident original clusters are finite. Thus the current
configuration belongs to \(\mathscr S_{z,m}\).

Map the original configuration to the triple consisting of this current
configuration, \(z\), and \(S\). For a fixed triple the original
configuration is recovered by closing the sites of \(S\), so the map has
at most one preimage. There are at most \(\ell+1\) choices for \(z\) and
at most \(2^{\ell+1}\) choices for \(S\). If \(|S|=s\), the ratio of the
Bernoulli weight of the original configuration to that of its image is
\(((1-p)/p)^s\leq p^{-(\ell+1)}\). Integrating the unchanged coordinates,
summing the recorded data, using transitivity, and applying
Lemma~\ref{lem:site-two-ghost} gives
\[
 \mathbb P_p(|K_x|\geq m,|K_y|\geq m,x\not\leftrightarrow y)
 \leq (\ell+1)(2/p)^{\ell+1}\mathbb P_p(\mathscr S_{o,m})
 \leq C_d(\ell+1)(2/p)^{\ell+1}
       \sqrt{\frac{1-p}{pm}}.
\]
Combining the connected and disconnected cases proves
\eqref{eq:4-path-opening}.
\end{proof}

\section{Nonunimodular site percolation}

The ordinary mass-transport estimate above is not available for a
nonunimodular transitive action. Here the correct substitute is the
tilted height analysis of Hutchcroft.

Throughout this section, let \(G\) be infinite, connected, locally finite,
and vertex-transitive, and suppose that its canonical action is
nonunimodular. Set \(\Gamma=\operatorname{Aut}(G)\). All modular quantities
below refer to this full automorphism action; the subscript \(\Gamma\) is
suppressed only to keep the formulas readable. Let
\(\Delta=\Delta_\Gamma\) be the modular cocycle, put
\(t_0=\max_{x\sim y}|\log\Delta(x,y)|>0\), and reverse an edge attaining
the maximum if necessary. Transitivity and invariance of the modular
cocycle then give, for every \(z\), a neighbor \(z^+\) such that
\(\log\Delta(z,z^+)=t_0\). Fix one such choice for each \(z\); equivariance
of this auxiliary choice will not be used. Likewise, fix for each \(z\)
a neighbor \(z^-\) satisfying \(\log\Delta(z,z^-)=-t_0\). Let
\(H_t=\{v:\log\Delta(o,v)\geq t\}\). Define
\[
 A_p^{\rm s}(t)
 =\mathbb P_p(o\leftrightarrow H_t\text{ inside }H_0),\qquad
 \widehat A_p^{\rm s}(t)=\frac{A_p^{\rm s}(t)}p.
\]
The division by \(p\) conditions the shared starting site to be open.
There is a necessary one-layer buffer in the site setting. Namely,
\begin{equation}
 \widehat A_p^{\rm s}(t+s+t_0)
 \geq p\,\widehat A_p^{\rm s}(t)\widehat A_p^{\rm s}(s)
 \qquad(s,t\geq0).                                             \label{eq:5-0}
\end{equation}
To prove \eqref{eq:5-0}, explore the open-root cluster below raw height \(t\),
testing sites in \(H_t\) only when first reached from below and never
continuing the exploration from such a site. Every tested first-hit site
has height strictly less than \(t+t_0\). On the event of a hit, choose an
open first-hit site \(z\) measurably and let \(z^+\) be a neighbor with
\(\log\Delta(z,z^+)=t_0\). Then \(z^+\) has height at least \(t+t_0\), so
its label and every label in \(H_0(z^+)\setminus\{z^+\}\) are untested.
Let \(\mathcal F_t\) be the sigma-field generated by this stopped
exploration, including the selected first-hit site.  If
\(y\in H_0(z^+)\), the cocycle identity gives
\[
 \log\Delta(o,y)=\log\Delta(o,z^+)+\log\Delta(z^+,y)
 \geq t+t_0.
\]
Thus none of the coordinates used by a continuation inside
\(H_0(z^+)\) has been queried.  The product exploration property implies
that, conditional on \(\mathcal F_t\), these coordinates are independent
Bernoulli variables of parameter \(p\).  Hence the conditional probability
that \(z^+\) is open and connects inside \(H_0(z^+)\) to relative raw
height \(s\) is exactly
\[
 A_p^{\rm s}(s)=p\widehat A_p^{\rm s}(s).
\]
Under the open-root law, the probability that the first exploration hits
\(H_t\) is \(\widehat A_p^{\rm s}(t)\).  On the intersection of these two
events, concatenate the first-hit path, the edge \(zz^+\), and the
continuation.  The resulting path stays in \(H_0(o)\) and ends at raw
height at least \(t+t_0+s\).  Taking the conditional expectation first
over \(\mathcal F_t\) and then over the initial exploration gives
\[
 \widehat A_p^{\rm s}(t+s+t_0)
 \geq \widehat A_p^{\rm s}(t)
       p\widehat A_p^{\rm s}(s),
\]
which is \eqref{eq:5-0}.  In particular, no site tested closed on the first
interface is used by the continuation.

Put \(c_p=-\log p\) and
\(f_p(t)=-\log\widehat A_p^{\rm s}(t)\). Equation \eqref{eq:5-0} says
\[
 f_p(t+s+t_0)\leq f_p(t)+f_p(s)+c_p.
\]
Consequently \(b_p(n):=f_p((n-1)t_0)+c_p\), \(n\geq1\), is subadditive.
Fekete's lemma and monotonicity of \(f_p\) between consecutive multiples
of \(t_0\) show that
\[
 \alpha_p^{\rm s}
 :=-\lim_{t\to\infty}t^{-1}\log\widehat A_p^{\rm s}(t)
\]
exists and that
\[
 h_v(z):=\frac{\log\Delta(v,z)}{t_0}
\]
is the normalized layer height used in the source. Thus raw height
\(t=t_0r\) corresponds to normalized height \(r\), and the source's
reference scale is
\(\overline{\exp}(r)=e^{t_0r}\). Consequently
\begin{equation}
 -\frac{\log\widehat A_p^{\rm s}(t_0r)}
        {\log\overline{\exp}(r)}
 =
 -\frac{\log\widehat A_p^{\rm s}(t_0r)}{t_0r}
 \longrightarrow\alpha_p^{\rm s}.                           \label{eq:5-0-normalization}
\end{equation}
Equivalently, the decay exponent per layer \(r\), with denominator \(r\)
rather than \(\log\overline{\exp}(r)\), is
\(t_0\alpha_p^{\rm s}\). The dimensionless exponent used by the source
and by Theorem~\ref{prop:nonunimodular-height} is therefore exactly the
raw-height exponent above, with neither a missing nor an additional
factor of \(t_0\).
\par\noindent
The subadditive formula is
\begin{equation}
 \alpha_p^{\rm s}
 =\frac1{t_0}\inf_{n\geq1}
   \frac{f_p((n-1)t_0)+c_p}{n}.                                \label{eq:5-0a}
\end{equation}
In particular, \(b_p(n)\geq n t_0\alpha_p^{\rm s}\), and hence
\begin{equation}
 \widehat A_p^{\rm s}(t)
 \leq p^{-1}e^{t_0\alpha_p^{\rm s}}e^{-\alpha_p^{\rm s}t}
 \qquad(t\geq0).                                               \label{eq:5-0c}
\end{equation}
Since division by \(p\) is independent of \(t\), this is also the
exponential rate obtained from \(A_p^{\rm s}\).

We will use the following continuity fact. For fixed \(t\), the event in
\(A_p^{\rm s}(t)\) is a union of finite open-path cylinder events. Under
the uniform-label coupling it follows that
\(A_{p_n}^{\rm s}(t)\uparrow A_p^{\rm s}(t)\) whenever \(p_n\uparrow p\).
Thus each numerator in \eqref{eq:5-0a} is left-continuous and decreasing on
\((0,1]\). An infimum of upper-semicontinuous functions is
upper-semicontinuous.
Since \(p\mapsto\alpha_p^{\rm s}\) is decreasing, this is exactly left
continuity of \(\alpha_p^{\rm s}\).

\begin{theorem}[Canonical nonunimodular site theorem]
\label{prop:nonunimodular-height}
Let \(G\) be an infinite, connected, locally
finite, vertex-transitive graph whose canonical action
\(\operatorname{Aut}(G)\curvearrowright G\) is nonunimodular. The height
exponent defined using the modular cocycle of \(\operatorname{Aut}(G)\)
satisfies
\begin{equation}
      \alpha_{p_c^{\rm s}}^{\rm s}=1.                         \label{eq:5-1}
\end{equation}
\end{theorem}

\textbf{Proof.} We use the canonical modular cocycle and tilted
mass-transport principle
\cite[Proposition~2.2]{Hutchcroft2020Nonunimodular}, finite vertex Menger
\cite[Theorem~3.3.1]{Diestel2017}, site sharpness
\cite[Theorem~2 and Section~6]{AntunovicVeselic2008}, and the critical
cluster results of \cite[Corollary~5.7 and Remark~5.11]{Timar2006}.
All BK, Reimer, exploration, finite-energy, and fractional-moment
estimates required below are proved directly in site coordinates.

First separate the two kinds of conditioning.  The source uses an
independent uniform offset \(R_o\in[0,1)\) to define Timar's separating
layers.  Transport it equivariantly by

\begin{equation}
 R_v=R_o-\frac{\log\Delta(o,v)}{t_0}\pmod 1,
 \qquad
 L_j(v)=\left\{x:j+R_v-1<
       \frac{\log\Delta(v,x)}{t_0}\leq j+R_v\right\},             \label{eq:5-1a}
\end{equation}

where \(t_0\) and the maximal-increment neighbors are as above. In
particular, if \(u\in L_k(v)\), then
\begin{equation}
 k+R_v-1<\frac{\log\Delta(v,u)}{t_0}\leq k+R_v
 \quad\Longrightarrow\quad
 u^+\in L_{k+1}(v).                                             \label{eq:5-1b}
\end{equation}
The strict and weak endpoints in \eqref{eq:5-1b} are exactly those in \eqref{eq:5-1a}, so
the conclusion also holds at every boundary offset. This layer field is
independent of site percolation. Thus the conditional-root law used here is

\[
 \widehat{\mathbb P}_p^v(\,\cdot\,)
 :=\mathbb P_p(\,\cdot\mid v\text{ open}),\qquad
\widehat\tau_p(x,y)=\frac{\mathbb P_p(x\leftrightarrow y)}p.
\]
The conversion between the source's normalized layer coordinate and the
raw height used here is the identity
\eqref{eq:5-0-normalization}. It applies unchanged after conditioning
the root open, because the factor \(p^{-1}\) is independent of height.

When a source formula is conditional on \(R_v=x\), its site version is
conditional on both \(R_v=x\) and \(v\) being open; we denote this regular
conditional law by \(\widehat{\mathbb P}_p^{v,x}\).  All suprema and
infima over \(x\in[0,1)\) are retained.  Conditioning on \(R_v=x\) only
fixes deterministic slabs and therefore does not alter any product-measure
inequality.  This distinction prevents the layer offset from being
mistaken for a percolation uniform label.
The half-open convention makes the layers a partition for every fixed
offset, including exceptional offsets of probability zero, and the
cocycle gives the exact shift rule
\(u\in L_k(v)\Longrightarrow L_j(u)=L_{j+k}(v)\).
If two disjoint-occurrence witnesses share a prescribed junction site
\(z\), conditioning \(z\) open gives
\begin{equation}
 \mathbb P_p\bigl((x\leftrightarrow z)\circ_z
                    (z\leftrightarrow y)\bigr)
 \leq\frac{\mathbb P_p(x\leftrightarrow z)
                 \mathbb P_p(z\leftrightarrow y)}p.             \label{eq:5-3}
\end{equation}
The general form needed here was proved before its first use as
Lemma~\ref{lem:site-junction}. In the present notation, if \(T\) is a
finite rooted tree, its vertices map injectively to
\((x_a)_{a\in V(T)}\), and \(\mathscr A_T\) is the event that its edges
are realized by open paths disjoint apart from their prescribed common
endpoints, that lemma gives
\begin{equation}
 \mathbb P_p(\mathscr A_T\mid x_r\text{ open})
 \leq\prod_{ab\in E(T)}\widehat\tau_p(x_a,x_b).                 \label{eq:5-4}
\end{equation}
A zero-length arm is deleted and its endpoint condition is retained as a
separate coordinate event. Thus \eqref{eq:5-4} also covers the two- and
three-arm decompositions.

This accounts for every path-decomposition BK use in
\cite{Hutchcroft2020Nonunimodular}.
The path decompositions in its Lemmas 5.6 and 5.13 are respectively the
path and three-arm cases of \eqref{eq:5-4}, and their fractional-moment versions
use the same containments before Holder's inequality. The blue-edge
argument proving its tilted mean-field bound is replaced by the following
vertex-coordinate calculation.

Write
\[
 p_c^{\rm s}(\lambda)
 :=\sup\{p:\chi_{p,\lambda}(v)<\infty\}.
\]
For every \(\lambda\in\mathbb R\), this threshold is positive. Indeed,
put
\[
 M_\lambda=\max_{x\sim y}\Delta(x,y)^{|\lambda|}<\infty.
\]
Every vertex connected to \(v\) has a simple open path from \(v\).
There are at most \(d(d-1)^{r-1}\) such paths of length \(r\geq1\),
each is open with probability \(p^{r+1}\), and the cocycle weight of its
endpoint is at most \(M_\lambda^r\). The union bound therefore gives
\[
 \chi_{p,\lambda}(v)
 \leq p+\sum_{r\geq1}d(d-1)^{r-1}p^{r+1}M_\lambda^r<\infty
\]
whenever \(pM_\lambda(d-1)<1\). Hence
\(p_c^{\rm s}(\lambda)>0\). For every
\(0\leq\lambda<1/2\), this threshold is also strictly below one.
For \(\lambda=0\), this is \(p_c^{\rm s}<1\), proved directly below.
For \(\lambda>0\), choose \(\gamma\in\Gamma\) with
\(a=\Delta(o,\gamma o)>1\), and a path of length \(L\) from \(o\) to
\(\gamma o\). The union of its first \(n\) translates under \(\gamma\)
contains at most \(nL+1\) sites and connects \(o\) to \(\gamma^n o\).
Hence
\[
 \mathbb P_p(o\leftrightarrow\gamma^n o)
       \Delta(o,\gamma^n o)^\lambda
 \geq p^{nL+1}a^{\lambda n}.
\]
The corresponding terms of the susceptibility have a divergent sum as
soon as \(p^La^\lambda>1\), which holds for some \(p<1\).

\begin{auxiliary}[Russo--convolution]\label{aux:russo-convolution}
Fix \(\lambda\in\mathbb R\) and a compact
interval \(I\subset(0,1)\). With
\[
 \chi_{p,\lambda}(v)=\sum_u\mathbb P_p(v\leftrightarrow u)
                         \Delta(v,u)^\lambda,
\]
there is \(C=C(G,\lambda,I)\) such that, whenever the right side is
finite,
\begin{equation}
 \frac{d}{dp}\chi_{p,\lambda}(v)
       \leq C\chi_{p,\lambda}(v)^2                         \label{eq:5-4a}
\end{equation}
for almost every \(p\in I\).
\end{auxiliary}

\begin{proof}
First work in a finite ball. If a site \(z\notin\{u,v\}\) is
pivotal for \(v\leftrightarrow u\), then in the configuration with \(z\)
closed there are neighbors \(x,y\) of \(z\) such that the connections
\(v\leftrightarrow x\) and \(y\leftrightarrow u\) occur on disjoint
vertex sets. The site BK inequality and a union bound over at most
\(d^2\) ordered pairs give
\begin{equation}
 \mathbb P_p(z\text{ pivotal})
 \leq\sum_{x,y\sim z}
       \mathbb P_p(v\leftrightarrow x)
       \mathbb P_p(y\leftrightarrow u).                     \label{eq:5-4b}
\end{equation}
For \(z=v\) or \(z=u\), the same argument has one arm and contributes
only \(O_{G,\lambda}(\chi_{p,\lambda})\). More explicitly, pivotality
of the root coordinate \(v\) is the event that the configuration with
\(v\) forced open connects \(v\) to \(u\). Since the event
\(\{v\leftrightarrow u\}\) itself forces \(v\) open,
\[
 \mathbb P_p(v\text{ pivotal for }v\leftrightarrow u)
 =p^{-1}\mathbb P_p(v\leftrightarrow u).
\]
The same identity holds at \(u\). After multiplication by
\(\Delta(v,u)^\lambda\) and summation over \(u\), the two endpoint
contributions are at most \(2p^{-1}\chi_{p,\lambda}(v)\). Applying
Russo's formula
\cite[Section~4, Lemma~3, equation~(4.2)]{Russo1981} (see also
\cite[Section~2.4]{Grimmett1999}), the cocycle identity
\[
 \Delta(v,u)=\Delta(v,x)\Delta(x,z)\Delta(z,y)\Delta(y,u),
\]
then turns the sum of \eqref{eq:5-4b} over \(u,z\) into at most
\(d^2M_\lambda^2\chi_{p,\lambda}(v)^2\), where
\(M_\lambda=\max_{a\sim b}\Delta(a,b)^{|\lambda|}\). Since
\(\chi_{p,\lambda}(v)\geq p\) and \(p\) is bounded away from zero on
\(I\), the endpoint terms are absorbed into the same bound. For the
convolution itself, the cocycle display gives
\[
 \begin{split}
 &\sum_{u,z}\sum_{x,y\sim z}
   \mathbb P_p(v\leftrightarrow x)
   \mathbb P_p(y\leftrightarrow u)\Delta(v,u)^\lambda\\
 &\quad\leq d^2M_\lambda^2
   \left(\sum_x\mathbb P_p(v\leftrightarrow x)
                    \Delta(v,x)^\lambda\right)
   \sup_y\left(\sum_u\mathbb P_p(y\leftrightarrow u)
                    \Delta(y,u)^\lambda\right).
 \end{split}
\]
Transitivity makes the supremum equal to
\(\chi_{p,\lambda}(v)\), proving the claimed square bound.

For the finite-volume passage, let \(\Lambda\) be finite, write
\(\chi_{\Lambda,p,\lambda}(a)\) for the corresponding restricted
susceptibility rooted at \(a\), and put
\(Y_\Lambda(p)=\max_{a\in\Lambda}\chi_{\Lambda,p,\lambda}(a)\).
The same convolution calculation, without using transitivity of
\(\Lambda\), bounds the derivative at every root by
\[
 \frac d{dp}\chi_{\Lambda,p,\lambda}(a)
 \leq C Y_\Lambda(p)^2.
\]
Indeed, after the branch site is summed, each of the two arm sums is at
most \(Y_\Lambda(p)\), with the adjacent cocycle factors absorbed into
\(C\). Since \(Y_\Lambda\) is the maximum of finitely many differentiable
functions, its upper right Dini derivative satisfies
\[
 D^+Y_\Lambda(p)\leq C Y_\Lambda(p)^2.
\]
Consequently the usual comparison with the solution of \(y'=Cy^2\)
applies uniformly in \(\Lambda\).

Now exhaust \(G\) by finite sets \(\Lambda_R\). Transitivity of the
infinite graph gives
\(Y_{\Lambda_R}(s)\leq\chi_{s,\lambda}(v)\) for every \(s\), because
each restricted rooted susceptibility is bounded by its full-space
counterpart. Integrating the root inequality before taking
\(R\to\infty\), and using monotone convergence, gives
\[
 \chi_{q,\lambda}(v)-\chi_{p,\lambda}(v)
 \leq C\int_p^q\chi_{s,\lambda}(v)^2\,ds
\]
whenever the right side is finite. It follows that \(\chi\) is locally
absolutely continuous there and satisfies \eqref{eq:5-4a} almost
everywhere. Since the thresholds used below are strictly below
one, the same integrated inequality shows that
susceptibility diverges as \(p\uparrow p_c^{\rm s}(\lambda)\). Indeed,
if its value at the threshold were \(M<\infty\), every finite-volume
susceptibility would be at most \(M\) there. Solving
\(y'\leq Cy^2\) forward gives
\(y(q+s)\leq M/(1-CMs)\) for \(0<s<(CM)^{-1}\), uniformly in the
volume. Monotone convergence would then make the infinite-volume
susceptibility finite above the defining supremum, a contradiction.

If \(q=p_c^{\rm s}(\lambda)\), solve the same inequality forward from
\(q-\varepsilon\). Were
\(\chi_{q-\varepsilon,\lambda}<1/(C\varepsilon)\), the finite-volume
bound would stay uniform beyond \(q\), contradicting the definition of
\(q\). After dividing by the conditioned-root factor this gives
\begin{equation}
 \widehat\chi_{q-\varepsilon,\lambda}(v)
       \geq c_{G,\lambda}\varepsilon^{-1}                  \label{eq:5-4c}
\end{equation}
for \(0<\varepsilon<q/2\); decreasing the constant covers the remaining
\(\varepsilon\). This proves the site form of the tilted mean-field estimate
corresponding to \cite[Proposition~3.1]{Hutchcroft2020Nonunimodular}.
\end{proof}

Second, define the conditional tilted susceptibility
\begin{equation}
 \widehat\chi_{p,\lambda}(v)
 =\frac1p\sum_x\mathbb P_p(v\leftrightarrow x)
                  \Delta(v,x)^\lambda .
                                                                    \label{eq:5-5}
\end{equation}
The tilted mass-transport symmetry
\cite[Proposition~2.2]{Hutchcroft2020Nonunimodular},
\(\widehat\chi_{p,\lambda}=\widehat\chi_{p,1-\lambda}\), gives the
tiltability threshold
\(p_t^{\rm s}:=p_c^{\rm s}(1/2)\).
For the ghost argument, let \(\mathcal G_{v,h}(x)\) be independent
Poisson variables of mean \(h\Delta(v,x)^\lambda\), and set
\[
 |K|_{v,\lambda}=\sum_{x\in K}\Delta(v,x)^\lambda,
 \qquad
 \widehat M_{p,\lambda,h}(v)
 =\widehat{\mathbb E}_p^v[1-e^{-h|K_v|_{v,\lambda}}].
\]
We write \(\widehat{\mathbb P}_{p,\lambda,h}^v\) and
\(\widehat{\mathbb E}_{p,\lambda,h}^v\) for the resulting joint law and
expectation of the site variables, conditioned on \(v\) being open, and
the independent ghost variables. Thus \(\widehat{\mathbb E}_p^v\) is
reserved for expectation over the site variables alone.
By the symmetry \(\lambda\leftrightarrow1-\lambda\), it is enough in the
magnetization argument to take \(0\leq\lambda<1/2\). Put
\(q=p_c^{\rm s}(\lambda)\), temporarily suppose that
\[
 \widehat{\mathbb E}_q^v
   [|K_v|_{v,\lambda}^{(1+\eta)/2}]<\infty
\]
for some \(\eta>0\). Replacing \(\eta\) by
\(\min\{\eta,1/2\}\) only weakens this finite-moment assumption, so we
may and do suppose \(0<\eta\leq1/2\). Fix
\(0<\delta\leq\min\{\eta,(1-2\lambda)/\lambda\}\), with the second
bound omitted when \(\lambda=0\). The estimates \eqref{eq:5-7}--\eqref{eq:5-9} below are
at \(p=q\) and depend on this temporary moment hypothesis.
The tilted mass-transport identity is unchanged. The magnetization
argument uses the following vertex version of its separator step.

\begin{auxiliary}[Finite-terminal Menger compactness]
\label{aux:finite-terminal-menger}
Let \(H\) be a
countable graph, let \(A,B\subseteq V(H)\) be finite, and let
\(N\in\{0,1,\ldots\}\). An \(A\)-to-\(B\) path is allowed to have
length zero when its sole vertex belongs to \(A\cap B\).
If \(H\) does not contain \(N+1\) pairwise vertex-disjoint finite
\(A\)-to-\(B\) paths, then some set of at most \(N\) vertices meets every
finite \(A\)-to-\(B\) path. The analogous rooted-fan statement also
holds: if there are not \(N+1\) paths from a fixed root \(v\) to distinct
members of a finite terminal set
\(B\subseteq V(H)\setminus\{v\}\), internally vertex-disjoint away from
\(v\), then at most \(N\) vertices other than \(v\) meet every such
positive-length path.
Throughout this lemma, separators are allowed to contain terminal
vertices: in the first statement they may meet \(A\cup B\), and in the
rooted-fan statement they may contain members of the terminal set; only
the distinguished root \(v\) is forbidden from the separator.
\end{auxiliary}

\begin{proof}
First suppose that \(A\cap B=\varnothing\). Identify a candidate
separator with a point of the compact
space \(\{0,1\}^{V(H)}\). The condition that it have size at most \(N\)
is closed, and, for each finite terminal path \(P\), the condition that
it meet \(P\) is clopen. Every finite family of these conditions is
satisfiable: apply the finite vertex-Menger theorem
\cite[Theorem~3.3.1]{Diestel2017}, or its finite
rooted-fan form, to the finite union of the listed paths. The finite
intersection property therefore gives a separator meeting every finite
terminal path. This proves the disjoint-terminal assertion without
requiring separators from an exhaustion to stabilize.

For arbitrary \(A,B\), put \(I=A\cap B\). The singleton paths at the
vertices of \(I\) are pairwise vertex-disjoint, so the hypothesis implies
\(|I|\leq N\). The graph \(H\setminus I\) cannot contain
\(N-|I|+1\) pairwise vertex-disjoint paths from \(A\setminus I\) to
\(B\setminus I\): together with the singleton paths indexed by \(I\),
they would give \(N+1\) disjoint \(A\)-to-\(B\) paths in \(H\).
The disjoint-terminal assertion in \(H\setminus I\) therefore gives a
separator \(S'\) of size at most \(N-|I|\). Every \(A\)-to-\(B\) path
that meets \(I\) is met by \(I\), while every such path avoiding \(I\)
is met by \(S'\). Thus \(I\cup S'\) is a separator of size at most
\(N\).

For the rooted-fan statement, repeat the compactness argument in the
first paragraph using the finite rooted-fan form of Menger quoted there.
The distinguished root is excluded from every finite separator, and
this exclusion is a closed coordinate condition. Hence the resulting
separator contains at most \(N\) vertices other than the root and meets
every positive-length rooted terminal path.
\end{proof}

\begin{auxiliary}[Two-fan separator]\label{aux:two-fan-separator}
Let \(K\) be a locally finite connected graph rooted at
\(v\), with a finite set of at least two marked tokens. If there are not two paths from
\(v\) to distinct tokens, internally vertex-disjoint away from \(v\),
then there is \(z\ne v\) such that the component of \(v\) in
\(K\setminus\{z\}\) has no token and two such paths start from \(z\)
outside that component.
\end{auxiliary}

\begin{proof}
Attach a separate leaf for every token and call the
resulting graph \(K^+\). Apply the finite-terminal version of vertex
Menger between \(v\) and the token leaves. Under the stated hypothesis
there is a one-vertex cut separating \(v\) from every token leaf. There
are only finitely many possible such cut vertices: each must lie on any
one fixed finite path from \(v\) to a fixed token leaf. Choose one
\(z\ne v\) for which the component \(R_z\) of \(v\) in
\(K^+\setminus\{z\}\) is maximal under inclusion. It is token-free by
definition. In the induced graph on
\(V(K^+)\setminus R_z\), suppose there were not two internally
vertex-disjoint paths from \(z\) to distinct token leaves.  Menger would
give a vertex \(w\ne z\) separating \(z\) from every token leaf in that
induced graph.  The component of \(v\) in \(K^+\setminus\{w\}\) contains
\(R_z\cup\{z\}\), strictly contradicting maximality.  Thus the required
two-fan exists.  Deleting the attached final edges gives the asserted
paths in \(K\), with a length-zero arm when its token is carried by
\(z\).
\end{proof}

Here every ghost token is represented by a distinct attached leaf. Two
tokens at the same graph vertex are distinct targets, but their paths
still share the carrier site's percolation coordinate. If that carrier is
not the junction it is therefore a separator, not a disjoint occurrence.
If it is the junction, the two length-zero arms consume distinct ghost
coordinates and no unconditioned site coordinate.
Under the temporary finite-moment hypothesis, the total ghost intensity
on the root cluster is \(h|K_v|_{v,\lambda}<\infty\), so the token set is
finite almost surely even if the cluster itself is infinite. Thus the
stated finite-terminal form is exactly the one used below.

We record the calculation replacing
\cite[Lemma~4.4]{Hutchcroft2020Nonunimodular}. In the
rest of this paragraph probabilities are under the conditional law
\(\widehat{\mathbb P}_p^v\). Let
\(\mathscr D_z\) be the event that there are two distinct ghost tokens and
open paths from \(z\) to their carriers whose sets of graph vertices are
disjoint outside \(z\); a length-zero path is allowed when its carrier is
\(z\). Thus coincident carriers are permitted only when both equal \(z\).
Here is a product-space justification for applying BK to the
Poisson ghost field. Fix a finite connected set \(W\ni z\), require the
two paths to lie in \(W\), and, for each \(x\in W\), replace the
Poisson variable of mean
\(\mu_x=h\Delta(z,x)^\lambda\) by \(N\) independent Bernoulli token
slots, each open with probability \(1-e^{-\mu_x/N}\). Conditional on
\(z\) being open, the remaining site coordinates and all token slots
form a finite product space. There are three terminal cases. For distinct
carriers, the event \(\mathscr D_z\) has witnesses using disjoint site
coordinates away from \(z\) and distinct token-slot coordinates. If two
tokens have the same carrier \(x\ne z\), both arms use the site coordinate
of \(x\), so this configuration is excluded from \(\mathscr D_z\) and is
handled by the separator case below. If both tokens are carried by \(z\),
the site coordinate of \(z\) is already conditioned open and the two
attached ghost leaves use distinct token-slot coordinates. Ordinary BK
therefore bounds the two-arm probability by the square of the corresponding
one-arm probability. For fixed \(W\), the vector of binomial token counts
converges in law as \(N\to\infty\) to the independent Poisson counts.
Letting finite connected \(W\) increase to \(V\), monotone convergence
applies because every one-arm or two-arm witness consists of finite
paths and finitely many tokens. Consequently,
\begin{equation}
 \widehat{\mathbb P}_{p,\lambda,h}^{z}(\mathscr D_z)
 \leq \widehat M_{p,\lambda,h}(z)^2.                           \label{eq:5-6}
\end{equation}
Coincident terminals away from \(z\) lie in the complementary separator
case. Coincident tokens carried by \(z\) are covered by \eqref{eq:5-6}, since the
site \(z\) is conditioned open and their attached leaves are distinct
ghost coordinates.

Now suppose that the cluster of \(v\) contains at least two tokens. If
\(\mathscr D_v\) fails, Auxiliary Lemma~\ref{aux:two-fan-separator} gives
an open cut site
\(z\ne v\). Close \(z\), and let \(K'_v\) be the root-side cluster. It
contains no token, while \(\mathscr D_z\) occurs off \(K'_v\). Assign
\(z\) to any neighbor \(u\in K'_v\), and sum over the at most \(d\)
choices. Fix deterministic rules for choosing the separator and its
root-side neighbor, and let \(\mathscr A(v,u,z)\) be the part of this
separator event on which those rules return \((u,z)\).
For fixed \(u,z\), we bound this event after discarding the selector
condition and retaining only \(u\in K'_v\), absence of ghost tokens in
\(K'_v\), and the off-\(K'_v\) occurrence of \(\mathscr D_z\). Thus the
choice of \(z\), which may use the ghost configuration, plays no role in
the conditional law below.

We justify the conditioning used here even when \(K'_v\) is infinite.
Let
\[
 \mathscr H_{v,z}:=
 \sigma(U_x:x\ne z)\vee
 \sigma(\mathcal G_{v,h}(x):x\in V),
\]
and let \(\mathbb Q_{v,z}\) be the marginal law of
\(\widehat{\mathbb P}_{p,\lambda,h}^v\) on \(\mathscr H_{v,z}\). Thus
\(\mathbb Q_{v,z}\) governs every site and ghost coordinate except
\(U_z\). For a residual coordinate assignment \(\omega\), write
\(\omega^{z,\mathrm o}\) and \(\omega^{z,\mathrm c}\) for its extensions
with \(z\) declared open and closed, respectively, and define \(K'_v\) to
be the cluster of \(v\) in \(\omega^{z,\mathrm c}\). Explore this cluster
breadth first. At stage \(n\), let
\(A_n\) and \(D_n\) be respectively the finite sets found open and
closed, and let \(\mathcal F_n\) be the sigma-field generated by all site
labels tested up to that stage. Use a fixed enumeration to break ties, so
every boundary site is eventually tested. Then
\[
 A_n\uparrow K'_v,\qquad D_n\uparrow\partial_VK'_v\setminus\{z\},
 \qquad
 \mathcal F_\infty=\sigma\!\left(\bigcup_n\mathcal F_n\right).
 \tag{E}\label{eq:5-6a}
\]
Conditionally on \(\mathcal F_n\), the untested site labels are independent
Bernoulli variables and all ghost variables retain their original
independent Poisson laws. For every event depending on finitely many of
these residual coordinates this product-kernel identity passes to
\(\mathcal F_\infty\) by martingale convergence. A monotone-class argument,
followed by increasing finite path exhaustions, extends it to the off-cluster
two-fan event. Thus this is a statement about the terminal exploration
sigma-field, not conditioning on an atom \(\{K'_v=K\}\). Denote this
terminal residual product kernel, with \(z\) subsequently declared open,
by \(\mathsf K_{v,z}(\,\cdot\mid\mathcal F_\infty)\).

Set \(S=|K'_v|_{v,\lambda}\). Conditionally on \(\mathcal F_\infty\),
the probability that \(K'_v\) contains no ghost token is \(e^{-hS}\),
interpreted as zero when \(S=\infty\). This event uses only ghost
coordinates carried by \(K'_v\). The event that \(\mathscr D_z\) occurs
off \(K'_v\) uses only the residual coordinates. Its conditional
probability is at most the
unrestricted two-fan probability: the tested closed boundary and the
prohibition on entering \(K'_v\) can only suppress this increasing event.
The cocycle identity gives
\[
 \Delta(v,x)^\lambda
 =\Delta(v,z)^\lambda\Delta(z,x)^\lambda.
\]
Thus changing the root of the ghost weights from \(v\) to \(z\) changes
the ghost intensity by the factor \(\Delta(v,z)^\lambda\). Applying
\eqref{eq:5-6} now gives the terminal conditional estimate
\[
 \mathsf K_{v,z}\!\left(
   K'_v\cap\mathcal G_{v,h}=\varnothing,
   \ \mathscr D_z\text{ occurs off }K'_v
   \mid\mathcal F_\infty\right)
 \leq e^{-h|K'_v|_{v,\lambda}}
   \widehat M_{p,\lambda,\Delta(v,z)^\lambda h}(z)^2.
 \tag{F}\label{eq:5-6b}
\]
Put
\[
 F_{u,z}:=\{u\in K'_v,\ K'_v\cap\mathcal G_{v,h}=\varnothing\},
\]
an event of \(\mathscr H_{v,z}\), and let
\[
 R_{v,u,z}:=\{\omega:
             \omega^{z,\mathrm o}\in\mathscr A(v,u,z)\}.
\]
Since \(z\ne v\), conditioning \(v\) open does not change the Bernoulli
law of \(U_z\). Therefore, for every \(B\in\mathscr H_{v,z}\),
\[
 \widehat{\mathbb P}_{p,\lambda,h}^v(B,U_z\leq p)
 =p\mathbb Q_{v,z}(B),\qquad
 \widehat{\mathbb P}_{p,\lambda,h}^v(B,U_z>p)
 =(1-p)\mathbb Q_{v,z}(B).
\]
In particular,
\[
 \widehat{\mathbb P}_{p,\lambda,h}^v(\mathscr A(v,u,z))
 =p\mathbb Q_{v,z}(R_{v,u,z}).
\]
On \(R_{v,u,z}\), closing \(z\) produces the root-side cluster used in
\eqref{eq:5-6b}. Discarding the deterministic selector condition,
integrating that conditional estimate, and inserting
\(\mathbf1_{\{u\in K'_v\}}\) give
\[
 \mathbb Q_{v,z}(R_{v,u,z})
 \leq C_d\mathbb Q_{v,z}(F_{u,z})
 \widehat M_{p,\lambda,\Delta(v,z)^\lambda h}(z)^2.
\]
Combining the last two displays gives
\begin{equation}
 \widehat{\mathbb P}_{p,\lambda,h}^v(\mathscr A(v,u,z))
 \leq C_dp\,\mathbb Q_{v,z}(F_{u,z})
 \widehat M_{p,\lambda,\Delta(v,z)^\lambda h}(z)^2.            \label{eq:5-7}
\end{equation}
Thus \(p\), rather than \(p/(1-p)\), is the switching coefficient while
the probability on the right is evaluated under \(\mathbb Q_{v,z}\).
This finite-stage derivation also shows that \eqref{eq:5-7} does not
require \(K'_v\) to be finite. We now state and prove
the power bound used at this point.  This removes any appeal to the bond
magnetization calculation.

\begin{auxiliary}[Magnetization power bound]
\label{aux:magnetization-power}
Let \(0<a\leq1\), and suppose that
\[
 M_a:=\sup_z\widehat{\mathbb E}_p^z
       [|K_z|_{z,\lambda}^{a}]<\infty .
\]
Then, for every \(s>0\) and every \(z\),
\[
 \widehat M_{p,\lambda,s}(z)\leq M_a s^a.
\]
Consequently, if \(a=(1+\delta)/2\), \(u\sim z\), and
\(s=h\Delta(v,z)^\lambda\), then
\begin{equation}
 \widehat M_{p,\lambda,\Delta(v,z)^\lambda h}(z)^2
 \leq M_a^2M_\Delta^{\lambda(1+\delta)}h^{1+\delta}
       \Delta(v,u)^{\lambda(1+\delta)},                     \label{eq:5-7b}
\end{equation}
where
\(M_\Delta=\max\{\Delta(x,y),\Delta(x,y)^{-1}:x\sim y\}<\infty\).
\end{auxiliary}

\begin{proof}
The inequality \(1-e^{-t}\leq t^a\) for \(t\geq0\) gives
\[
 \widehat M_{p,\lambda,s}(z)
 =\widehat{\mathbb E}_p^z[1-e^{-s|K_z|_{z,\lambda}}]
 \leq s^a\widehat{\mathbb E}_p^z
             [|K_z|_{z,\lambda}^a].
\]
Transitivity and covariance of the modular function make the last
moment independent of \(z\).  Squaring and using
\(\Delta(v,z)=\Delta(v,u)\Delta(u,z)\) proves \eqref{eq:5-7b}.
\end{proof}

Under the temporary moment hypothesis take
\(a=(1+\delta)/2\leq(1+\eta)/2\).  Lyapunov's inequality gives
\(M_a<\infty\), so \eqref{eq:5-7b} applies. To return to the ordinary
site law before summing, define the ordinary closed-slice event
\[
 F_{u,z}^{\mathrm{cl}}:=
 \{U_z>p,\ u\in K_v,\ K_v\cap\mathcal G_{v,h}=\varnothing\}.
\]
When \(U_z>p\), the ordinary configuration is
\(\omega^{z,\mathrm c}\), and hence \(K_v=K'_v\). The second slice
identity therefore gives the exact equality
\[
 \mathbb Q_{v,z}(F_{u,z})
 =\frac1{1-p}\widehat{\mathbb P}_{p,\lambda,h}^v
       (F_{u,z}^{\mathrm{cl}}).
\]
The deterministic selector makes the events \(\mathscr A(v,u,z)\), over
\(z\ne v\) and \(u\sim z\), a disjoint partition of
\(\{\mathcal G_{v,h}(K_v)\geq2\}\cap\mathscr D_v^c\). Moreover, for
every nonnegative weight \(w(u)\), Tonelli's theorem and local finiteness
give
\[
 \begin{aligned}
 \sum_{z\ne v}\sum_{u\sim z}
   \widehat{\mathbb P}_{p,\lambda,h}^v(F_{u,z}^{\mathrm{cl}})w(u)
 &=\widehat{\mathbb E}_{p,\lambda,h}^v\!\left[
    \mathbf1_{\{K_v\cap\mathcal G_{v,h}=\varnothing\}}
    \sum_{u\in K_v}w(u)
       \sum_{\substack{z\sim u\\z\ne v}}\mathbf1_{\{U_z>p\}}
    \right]\\
 &\leq d\,\widehat{\mathbb E}_{p,\lambda,h}^v\!\left[
    \mathbf1_{\{K_v\cap\mathcal G_{v,h}=\varnothing\}}
    \sum_{u\in K_v}w(u)\right].
 \end{aligned}
\]
Apply this with
\(w(u)=\Delta(v,u)^{\lambda(1+\delta)}\), substitute the preceding
closed-slice identity into \eqref{eq:5-7}, and use \eqref{eq:5-7b}.
Absorbing \(d\), \(C_d\), \(M_a^2\), and the fixed neighbor-cocycle
bound into \(C\) gives
\begin{equation}
 \widehat{\mathbb P}_{p,\lambda,h}^v
   (\mathcal G_{v,h}(K_v)\geq2,\ \mathscr D_v^c)
 \leq C\frac p{1-p}h^{1+\delta}
 \widehat{\mathbb E}_{p,\lambda,h}^v\!\left[
   \sum_{u\in K_v}\mathbf1\{K_v\cap\mathcal G_{v,h}=\varnothing\}
   \Delta(v,u)^{\lambda(1+\delta)}\right].                 \label{eq:5-7c}
\end{equation}
First integrate the ghost variables, replacing the no-ghost indicator
by \(e^{-h|K_v|_{v,\lambda}}\). Split the resulting sum according to
\(\Delta(v,u)\leq1\) or \(>1\). On the first part,
\(\Delta^{\lambda(1+\delta)}\leq\Delta^\lambda\). On the second part,
apply the conditional tilted mass-transport identity
\eqref{eq:5-11b} to
\[
 F(a,b)=\mathbf1\{a\leftrightarrow b,\ \Delta(a,b)>1\}
          e^{-h|K_a|_{a,\lambda}}
          \Delta(a,b)^{\lambda(1+\delta)}.
\]
This transport vanishes unless both endpoints are open. The transported
sum is
\begin{equation}
 \widehat{\mathbb E}_p^v\sum_{u\in K_v}
 \mathbf1\{\Delta(v,u)<1\}e^{-h|K_u|_{u,\lambda}}
 \Delta(v,u)^{1-\lambda(1+\delta)}.                          \label{eq:5-7d}
\end{equation}
Indeed, swapping the endpoints replaces \(\Delta(v,u)\) by its inverse,
and the tilted MTP contributes one further factor \(\Delta(v,u)\).
For \(u\in K_v\), the cocycle identity gives the exact root-change
formula
\[
 |K_u|_{u,\lambda}
 =\sum_{x\in K_v}\Delta(u,x)^\lambda
 =\Delta(v,u)^{-\lambda}|K_v|_{v,\lambda}.
\]
When \(\Delta(v,u)<1\), the multiplier is at least one, and hence
\(e^{-h|K_u|_{u,\lambda}}\leq e^{-h|K_v|_{v,\lambda}}\). Moreover,
\[
 1-\lambda(1+\delta)\geq\lambda
\]
by the choice of \(\delta\), and on \(\{\Delta(v,u)<1\}\) this makes the
weight in \eqref{eq:5-7d} at most \(\Delta(v,u)^\lambda\). The first half
and the transported second half are consequently each bounded by
\[
 \widehat{\mathbb E}_p^v
  [|K_v|_{v,\lambda}e^{-h|K_v|_{v,\lambda}}]
 =\partial_h\widehat M_{p,\lambda,h}(v)
 =:\widehat\chi_{p,\lambda,h}(v).
\]
We have therefore proved
\begin{equation}
 \widehat{\mathbb P}_{p,\lambda,h}^v(
       \mathcal G_{v,h}(K_v)\geq2,\ \mathscr D_v^c)
 \leq C\frac p{1-p}h^{1+\delta}
              \widehat\chi_{p,\lambda,h}.                     \label{eq:5-8}
\end{equation}
Here \(C=C(d,M_a,M_\Delta,\lambda,\delta)<\infty\); no uniformity in
these parameters is used in the temporary-moment argument.
Moreover,
\(\widehat\chi_{p,\lambda,h}=\partial_h\widehat M_{p,\lambda,h}\),
averaged over root orbits as in the tilted mass-transport identity.
This is exactly the exceptional-event estimate
\cite[Equation~(4.27)]{Hutchcroft2020Nonunimodular} in the bond
argument. If the separating site itself carries one or both tokens,
their leaf coordinates occur in \(\mathscr D_z\), so there is no omitted
coincidence case.

We now derive the differential inequality without suppressing the
decomposition. Let
\(T_v=\mathcal G_{v,h}(K_v)\) be the number of ghost tokens carried by the
open-root cluster. Conditional on \(K_v\), it is Poisson with parameter
\(h|K_v|_{v,\lambda}\). Dominated differentiation at \(h>0\) therefore
gives
\begin{equation}
 \widehat{\mathbb P}_{p,\lambda,h}^v(T_v=1)
 =\widehat{\mathbb E}_p^v
    [h|K_v|_{v,\lambda}e^{-h|K_v|_{v,\lambda}}]
 =h\,\partial_h\widehat M_{p,\lambda,h}(v).                   \label{eq:5-8a}
\end{equation}
Moreover, \(\mathscr D_v\subseteq\{T_v\geq2\}\), including the case of
two distinct tokens at the same carrier, and \eqref{eq:5-6} gives
\begin{equation}
 \widehat{\mathbb P}_{p,\lambda,h}^v(\mathscr D_v)
 \leq\widehat M_{p,\lambda,h}(v)^2.                           \label{eq:5-8b}
\end{equation}
The mutually disjoint events
\(\{T_v=1\}\), \(\mathscr D_v\), and
\(\{T_v\geq2\}\cap\mathscr D_v^c\) partition \(\{T_v\geq1\}\).
Consequently \eqref{eq:5-8a}, \eqref{eq:5-8b}, and the exceptional-event estimate \eqref{eq:5-8}
give, first at a fixed root and then after the invariant root average,
\begin{equation}
 \widehat M_h
 \leq C\widehat M_h^2+h\,\partial_h\widehat M_h
     +C\frac p{1-p}h^{1+\delta}\partial_h\widehat M_h.           \label{eq:5-9}
\end{equation}
The only tilted mass transport in this derivation occurs in the passage
from \eqref{eq:5-7c} to \eqref{eq:5-7d}; together with the estimates
following \eqref{eq:5-7d}, it bounds the last term by
\(\partial_h\widehat M_h\). Thus the root-dependent ghost law,
the one-token term, the two-fan term, and coincident tokens have all been
accounted for explicitly.

We record the analytic consequence now. Under the temporary hypothesis
above, write
\[
 X=|K_v|_{v,\lambda},\qquad
 \phi(h)=\widehat{\mathbb E}_q^v[1-e^{-hX}].
\]
Equation \eqref{eq:5-4c} and monotone convergence give
\(\phi'(0+)=\widehat\chi_{q,\lambda}=\infty\). Since \(\phi\) is
increasing and concave, it has an inverse \(\psi\) near zero and
\begin{equation}
 \lim_{t\downarrow0}\frac{\psi(t)}t=0.                          \label{eq:5-9a}
\end{equation}
Choose an integer \(N\) so large that
\[
 2^{1/N}<1+\min\{\eta,(1-2\lambda)/\lambda\},
\]
where the second entry is omitted when \(\lambda=0\), and set
\(\delta=2^{1/N}-1\). Then \((1+\delta)\lambda<1-\lambda\), so \eqref{eq:5-9}
holds with this \(\delta\). Substituting \(h=\psi(t)\) and
\(\phi'(\psi(t))=1/\psi'(t)\) into \eqref{eq:5-9} yields, for small \(t>0\),
\begin{equation}
 \left(\frac{\psi(t)}t\right)'
 \leq C\left(1+\frac{\psi(t)^{\,1+\delta}}{t^2}\right).         \label{eq:5-9b}
\end{equation}
Indeed, before dividing, \eqref{eq:5-9} gives
\[
 t(1-Ct)\psi'(t)
 \leq\psi(t)+C\psi(t)^{1+\delta}.
\]
Choose the interval so that \(Ct\leq1/2\). Subtracting
\((1-Ct)\psi(t)\) from both sides and dividing by
\(t^2(1-Ct)\) gives
\[
 \left(\frac{\psi(t)}t\right)'
 \leq\frac{Ct\psi(t)+C\psi(t)^{1+\delta}}
              {t^2(1-Ct)}.
\]
The conditioned-open root contributes one to \(X\), so
\(\phi(h)\geq1-e^{-h}\), equivalently
\(\psi(t)\leq-\log(1-t)\leq C't\) near zero. Substitution proves
\eqref{eq:5-9b}; in particular, the first term is bounded rather than of
order \(1/t\).

Since the temporary hypothesis says that
\(\widehat{\mathbb E}_q^v[X^{(1+\eta)/2}]<\infty\), and since
\(1-e^{-x}\leq x^{(1+\eta)/2}\), inversion gives
\begin{equation}
 \psi(t)\geq c\,t^{2/(1+\eta)}
       \geq c\,t^{2^{1-1/N}}.                                  \label{eq:5-9c}
\end{equation}
Starting with the just-proved linear bound and using
\eqref{eq:5-9a}--\eqref{eq:5-9b}, induction on \(j\) gives
\[
 \psi(t)\leq C_jt^{2^{j/N}}\qquad(0\leq j\leq N).
\]
Indeed, if \(a_j=2^{j/N}\) and \(\psi(t)\leq C_jt^{a_j}\), then
\eqref{eq:5-9b} gives
\[
 \left(\frac{\psi(t)}t\right)'
 \leq C_j'\left(1+t^{a_j(1+\delta)-2}\right).
\]
The exponent on the second term is greater than \(-1\), since
\(a_j(1+\delta)=2^{(j+1)/N}>1\). Equation \eqref{eq:5-9a} supplies the
zero boundary value for \(\psi(t)/t\) at the origin. Integrating and
multiplying by \(t\) therefore gives
\[
 \psi(t)\leq C_{j+1}'
   \left(t^2+t^{a_j(1+\delta)}\right)
 \leq C_{j+1}t^{2^{(j+1)/N}},
\]
because \(2^{(j+1)/N}\leq2\) for \(j+1\leq N\). This is the next
inductive exponent.
At \(j=N\) this gives \(\psi(t)\leq C_Nt^2\), contradicting \eqref{eq:5-9c} as
\(t\downarrow0\). Therefore
\begin{equation}
 \widehat{\mathbb E}_{p_c^{\rm s}(\lambda)}^v
   [|K_v|_{v,\lambda}^{(1+\eta)/2}]=\infty
 \quad(0\leq\lambda<1/2,\ \eta>0).                              \label{eq:5-9d}
\end{equation}
Removing the conditioning on the root only multiplies the nonzero
moment by \(q\).

The Reimer steps are justified by the following exploration observation;
the exact finite-product input is Lemma~\ref{lem:finite-bkr}.

\begin{auxiliary}[Exploration disjointness]
\label{aux:exploration-disjointness}
Let \(D_0\) be an allowed vertex set, let \(S,T,F\subseteq D_0\), and
let \(u,w\in D_0\). For a finite set \(\Lambda\subseteq D_0\) containing
\(u,w\), write \(K_u^{D_0}\) for the open cluster of \(u\) in the
configuration induced on \(D_0\), and define
\begin{align*}
 \mathcal E_\Lambda
 &=\{u\text{ is open},\ u\leftrightarrow w
       \text{ in }S\cap\Lambda,\ K_u^{D_0}\subseteq\Lambda,
       \ K_u^{D_0}\cap F=\varnothing\},\\
 \mathcal B_\Lambda
 &=\{T\leftrightarrow F\text{ by an open path in }D_0\cap\Lambda\}.
\end{align*}
Then these are finite-coordinate events and
\[
 \mathcal E_\Lambda\cap\mathcal B_\Lambda
 \subseteq \mathcal E_\Lambda\mathbin{\square}\mathcal B_\Lambda,
\]
where \(\square\) denotes disjoint occurrence in the site coordinates.
\end{auxiliary}

\begin{proof}
On \(\mathcal E_\Lambda\), expose the exact set
\(C=K_u^{D_0}\).  The event \(\mathcal E_\Lambda\) is certified by the
open sites of \(C\), the closed sites of its relative external boundary
\(\partial^{D_0}_VC\), and an open \(u\)-to-\(w\) path in
\(S\cap C\cap\Lambda\).  Let \(P\subseteq D_0\cap\Lambda\) be an open
\(T\)-to-\(F\) path certifying \(\mathcal B_\Lambda\).  The path \(P\)
cannot meet \(C\), since otherwise its terminal segment to \(F\),
concatenated with an open path in \(C\) from \(u\), would give
\(K_u^{D_0}\cap F\ne\varnothing\).  Nor can \(P\) meet
\(\partial^{D_0}_VC\), whose sites are certified closed.  Hence the two
finite coordinate certificates are disjoint, proving the asserted
inclusion. Lemma~\ref{lem:finite-bkr} may therefore be applied.

At the infinite-volume use below, the exact-cluster exhaustion in
\eqref{eq:5-11f2} supplies precisely these finite events; no continuity
of disjoint occurrence is invoked.  The nonunimodular proof uses Reimer
only in the simplified and general peak-comparison identities
\cite[Equations~(6.10) and~(6.20)]{Hutchcroft2020Nonunimodular}.  In the
general identity the bounded connector is forced open, while the at most
\(d|\gamma|\) neighboring sites that could create the forbidden
connection are forced closed.  This costs a positive
\((G,p)\)-dependent constant.
\end{proof}

Fourth, in the comparison between a \(p_c^{\rm s}-\varepsilon\) cluster
\(A\) and the \(p_c^{\rm s}\) configuration, use its external vertex
boundary. It has at most \(d|A|\) sites. Closing the near-top part costs
at least
\begin{equation}
 \left(\frac{1-p_c^{\rm s}}
 {1-p_c^{\rm s}+\varepsilon}\right)^{d|A|}.                     \label{eq:5-10}
\end{equation}
For the far part, assign each boundary site to its first incident
vertex of \(A\); the required nonconnection events are decreasing, so
site FKG gives the same product lower bound. The finite-energy estimate
in \cite[Lemma~5.16]{Hutchcroft2020Nonunimodular} needs a little more care.
Given the restricted cluster \(K_k\) and its top set \(Z_k\), choose
\(u\in Z_k\) measurably from the explored data and use the neighbor \(u^+\)
from \eqref{eq:5-1b}. The exploration is confined to \(L_{-\infty,k}(v)\), whereas
\(u^+\in L_{k+1}(v)\); hence its label is untested and remains Bernoulli
\(p\) conditionally on the exploration. Open \(u^+\), and close every
neighbor of \(Z_k\cup\{u^+\}\) outside \(K_k\cup\{u^+\}\) that has not
already been tested closed. No site in this closing set was tested open:
such a site would have entered \(K_k\) through a retained exploration edge;
the only omitted edges have both endpoints in \(L_k(v)\), and their other
endpoints are untested unless reached elsewhere. There are at most
\(d(|Z_k|+1)\) newly prescribed closed sites. This is a single
exploration in a layer-determined graph, with no earlier query history,
so it is the case \(Q_{\rm old}=\varnothing\) of the corrected kernel
rather than an adaptive reset of previously queried labels. The cluster
kernel \eqref{eq:5-2a} therefore gives, conditionally on the complete
exploration data,
\begin{equation}
 p(1-p)^{d(|Z_k|+1)}.                                          \label{eq:5-11}
\end{equation}
Formula \eqref{eq:5-11}, combined with the conditional bound on \(|Z_k|\), is
used explicitly in \eqref{eq:5-11e6}--\eqref{eq:5-11e9} below; no peak estimate is inferred
at this stage.

Finally, we need the following site version of the critical-cluster
fact used in the height analysis.

\begin{auxiliary}[Critical clusters]\label{aux:critical-clusters}
If \(G\) satisfies the standing
hypotheses of this section, then Bernoulli site percolation on \(G\) at
\(p_c^{\rm s}\) has no infinite cluster.
\end{auxiliary}

\begin{proof}
The bond-coordinate analogue follows from
\cite[Theorem~1]{Hutchcroft2016}. For sites, we
use two arguments whose coordinate status can be checked exactly. First,
Tim{\'a}r's critical theorem
\cite[Corollary~5.7 and Remark~5.11]{Timar2006}, formulated using the full
automorphism group, says that a nonunimodular transitive graph cannot have
infinitely many infinite clusters at critical Bernoulli percolation. Its
hypothesis is exactly the standing canonical-action hypothesis above.
Although stated in edge notation, \cite[Remark~5.11]{Timar2006} explicitly
says that all arguments in the relevant sections repeat for site
percolation: they require only insertion and deletion tolerance. The
Newman--Schulman zero--one--infinity trichotomy, in the site form
recorded for transitive graphs in \cite[Theorem~7.5]{LyonsPeres2016},
therefore reduces the claim to excluding a unique infinite site cluster.
For this case, the canonical nonunimodular action implies that \(G\) is
nonamenable \cite[Theorem~3.4]{BenjaminiLyonsPeresSchramm1999}. Since \(G\)
has bounded degree, it therefore has exponential volume growth, so
\(g:=\liminf_r|B_r|^{1/r}>1\). Let
\[
 \widehat\kappa_p(r)=p^{-1}
   \inf_{d(x,y)\leq r}\mathbb P_p(x\leftrightarrow y).
\]
If \(d(x,y)\leq r+s\), choose \(z\) with \(d(x,z)\leq r\) and
\(d(z,y)\leq s\). Conditional on \(z\) being open, site FKG gives
\(\widehat\kappa_p(r+s)\geq
\widehat\kappa_p(r)\widehat\kappa_p(s)\). For \(p<p_c^{\rm s}\),
site sharpness \cite[Theorem~2 and Section~6]{AntunovicVeselic2008}
and exponential growth give
\[
 p|B_r|\widehat\kappa_p(r)\leq\chi_p<\infty.
\]
If \(g=\liminf_r|B_r|^{1/r}>1\), supermultiplicativity and Fekete's
lemma imply \(\widehat\kappa_p(r)\leq g^{-r}\) for every fixed
\(p<p_c^{\rm s}\) and \(r\geq1\). Couple by uniform labels and let
\(p\uparrow p_c^{\rm s}\). For each fixed pair \(x,y\), the events
\(\{x\leftrightarrow y\}\) increase to their event at criticality; there
are only finitely many pair orbits at distance at most \(r\). Hence
\begin{equation}
 \sup_{r\geq1}\widehat\kappa_{p_c^{\rm s}}(r)^{1/r}
 \leq g^{-1}<1.                                             \label{eq:5-11a}
\end{equation}
This is the normalized site version of the connectivity-decay proof for
graphs of exponential growth.

If the critical site configuration had a unique infinite cluster, site
FKG would give
\[
 \mathbb P_{p_c^{\rm s}}(x\leftrightarrow y)
 \geq\mathbb P(x\leftrightarrow\infty)
       \mathbb P(y\leftrightarrow\infty),
\]
whose right side is uniformly positive by transitivity. This contradicts
\eqref{eq:5-11a}. The site trichotomy
\cite[Theorem~7.5]{LyonsPeres2016} and Tim{\'a}r's site-valid critical
theorem \cite[Corollary~5.7 and Remark~5.11]{Timar2006} now leave zero
infinite clusters.
\end{proof}

We will also use that \(p_c^{\rm s}<1\). By the same nonamenability,
let \(h_E>0\) be the edge-isoperimetric constant of \(G\). Every finite connected
\(A\ni o\) then has
\(|\partial_VA|\geq h_E|A|/d\), while the number of connected
\(n\)-vertex sets containing \(o\) is at most \((ed)^{n-1}\). Conditional
on \(o\) being open, the event \(K_o=A\) forces every site of
\(\partial_VA\) closed. Consequently
\[
 \widehat{\mathbb P}_p^o(|K_o|<\infty)
 \leq\sum_{n\geq1}(ed)^{n-1}(1-p)^{h_En/d}<1
\]
when \(p<1\) is sufficiently close to one. Hence
\(p_c^{\rm s}<1\), and every finite-energy constant used below at
criticality is strictly positive.

Under the conditional-root law, the root-peak event also has a uniform
positive denominator: it contains the event that all \(d\) neighbors of
the conditioned-open root are closed, of probability \((1-p)^d\).
This is the site replacement for the isolated-root event in the
fractional-moment bootstrap.

We also record explicitly why conditioning the root open is compatible
with every tilted mass transport.  If \(F(x,y,\omega)\) is diagonally
invariant, integrable, and vanishes unless both \(x\) and \(y\) are open,
then the ordinary tilted mass-transport principle and transitivity give
\begin{equation}
 \widehat{\mathbb E}_p^v\sum_xF(v,x,\omega)
 =\widehat{\mathbb E}_p^v\sum_xF(x,v,\omega)\Delta(v,x).        \label{eq:5-11b}
\end{equation}
Indeed, multiply both sides by \(p\).  Since the summand on each side
forces \(v\) open, the resulting identity is exactly the unconditioned
tilted mass-transport identity.  The same proof applies to signed
transports with integrable absolute value.  Every transport below has
this endpoint support because it transports mass only between vertices
in the same open cluster.  Thus no change of root under mass transport
introduces an unaccounted factor of \(p\).

The transports used below are: the
susceptibility symmetry (mass between two connected vertices); the
large-\(\Delta\) half of the ghost estimate (the indicator
\(x\leftrightarrow y\)); the upward/downward identity defining
\(\beta_p\); the transport from a cluster peak to a lower-layer cluster
vertex; and the Holder transport \eqref{eq:5-11h}. Each displayed summand contains
the connection of its two arguments. Root-orbit averaging and the random
layer offset are independent of site states, so neither adds another
case to \eqref{eq:5-11b}.

We can now state the nonunimodular estimates used below.
Put
\[
 L_{m,n}(v)=\bigcup_{m\leq j\leq n}L_j(v),
 \qquad
 X_j^{m,n}(v)=\left|\left\{x\in L_j(v):
 v\leftrightarrow x\text{ in }L_{m,n}(v)\right\}\right|,
\]
where infinite endpoints are allowed. For \(0\leq a\leq1\), define
\begin{equation}
 \widehat E_p^{m,n}(j;a)=
   \sup_{v\in V,\,x\in[0,1)}
   \widehat{\mathbb E}_p^{v,x}
       \left[(X_j^{m,n}(v))^a\right],                         \label{eq:5-11c-prime}
\end{equation}
with \(0^0=0\), and let
\(\widehat E_p^{m,n}(j)=\widehat E_p^{m,n}(j;1)\). Let
\[
 I_h^{\rm s}=\left\{p>0:
   \widehat E_p^{-\infty,n}(j;1)<\infty
   \text{ for every }-\infty<j\leq n<\infty\right\}.
\]
For \(p\in I_h^{\rm s}\), let \(\beta_p^{\rm s}\) be the common value
\begin{equation}
 \begin{split}
 \beta_p^{\rm s}
 &=-\lim_{k\to\infty}\frac1{t_0k}
       \log\widehat E_p^{-\infty,k}(k;1)\\
 &=1-\lim_{k\to\infty}\frac1{t_0k}
       \log\widehat E_p^{-\infty,0}(-k;1).
 \end{split} \label{eq:5-11c-prime-prime}
\end{equation}

We use no pointwise interpretation of an unspecified \(o(k)\) below.
For a family \(F_{v,x}(k)\), the assertion
\[
 \sup_{v,x}F_{v,x}(k)\leq C_a\exp[ck+o_a(k)]
\]
means the following quantified statement: for every \(\zeta>0\) there
is \(C_{a,\zeta}<\infty\), independent of \(v,x,k\), such that
\begin{equation}
 \sup_{v,x}F_{v,x}(k)\leq C_{a,\zeta}e^{(c+\zeta)k}
 \quad(k\geq0).                                               \label{eq:5-11c-prime-prime-prime-prime}
\end{equation}
The same convention applies with additional slab widths \(r\): their
contribution is \(e^{C_0r}\), where \(C_0=C_0(G,p)\) is independent of
the fractional exponent. This is generic notation only. When a later
argument distinguishes the shift loss from the excess-width loss, we use
the named constants \(C_{\rm sh}\) and \(C_\star\), respectively, rather
than identifying either one implicitly with \(C_0\). Finite sums,
products, fixed positive powers,
translations by a bounded number of layers, and exponentially convergent
convolutions preserve \eqref{eq:5-11c-prime-prime-prime-prime}: divide the requested \(\zeta\) by the
finite number of input factors and absorb bounded indices into the
constant. This elementary observation will be called the
\emph{uniform error calculus}. In the fractional iteration leading to a fixed final
exponent, only finitely many auxiliary exponents are used; hence the
maximum of their constants is finite and \eqref{eq:5-11c-prime-prime-prime-prime} remains uniform in
the root and offset.

\begin{lemma}[Nonunimodular site estimates]
\label{lem:nonunimodular-site-estimates}
Let \(G\) be an
infinite, connected, locally finite, vertex-transitive graph whose full
automorphism action \(\Gamma=\operatorname{Aut}(G)\) is nonunimodular.
Consider independent Bernoulli site percolation and the independent
equivariant layer field \eqref{eq:5-1a} defined from this canonical action. Then:

\par\noindent\textbf{1.} The two limits in \eqref{eq:5-11c-prime-prime} exist and agree whenever
   \(p\in I_h^{\rm s}\). All estimates below are uniform in the root
   \(v\) and in the fixed layer offset \(R_v=x\).
\par\noindent\textbf{2.} At \(p=p_c^{\rm s}\), for every \(0<\varepsilon\leq1\) and
   \(\zeta>0\), there is \(C_{\varepsilon,\zeta}<\infty\), independent
   of the root and layer offset, such that
   \begin{equation}
     \widehat E_{p_c^{\rm s}}^{-\infty,\infty}
       (k;1-\varepsilon)
     \leq C_{\varepsilon,\zeta}
     \begin{cases}
      \exp[-(t_0-\zeta)k],&k\geq0,\\
      \exp[\zeta|k|],&k<0,
     \end{cases}                                           \label{eq:5-11i}
   \end{equation}
\par\noindent\textbf{3.} For the conditional tilted susceptibility \eqref{eq:5-5},
   \begin{equation}
     p_c^{\rm s}<p_c^{\rm s}(\lambda)\leq p_t^{\rm s}
     \qquad(0<\lambda<1).                                  \label{eq:5-11j}
   \end{equation}
\par\noindent\textbf{4.} One has \(p_c^{\rm s}\in I_h^{\rm s}\),
   \(\beta_{p_c^{\rm s}}^{\rm s}\leq1\),
   \(\alpha_p^{\rm s}=\beta_p^{\rm s}\) for
   \(0<p<p_t^{\rm s}\), and
   \(\alpha_{p_c^{\rm s}}^{\rm s}\geq1\). Consequently
   \(\alpha_{p_c^{\rm s}}^{\rm s}=1\).
\end{lemma}

\textbf{Proof.} Realize every conditional probability below on one label
space: \((U_z)_{z\in V}\) are i.i.d. uniform on \([0,1]\), \(z\) is
\(p\)-open when \(U_z\leq p\), and the layer field is independent of all
labels. Conditioning on \(R_v=x\) means the canonical law obtained by
fixing that independent offset equal to \(x\), then defining every other
offset equivariantly by \eqref{eq:5-1a}. Thus the supremum over \(x\) does not
depend on arbitrary versions of a regular conditional law. The following
three kernels are the complete list of conditional laws used in the
proof.

\textbf{Cluster kernel.} Let \(W\subseteq V\) be deterministic after the
layer field has been fixed, let \(v\in W\), and explore the \(p\)-cluster
of \(v\) in \(W\) by testing boundary sites one at a time. If a finite
outcome has open set \(A\ni v\) and tested external boundary
\(D=\partial_V^W A\), then, under \(\widehat{\mathbb P}_p^{v,x}\),
\begin{equation}
 \mathcal L\left((U_z)_{z\notin A\cup D}
       \,\middle|\,K_v^W=A,\ R_v=x\right)
 =\bigotimes_{z\notin A\cup D}\operatorname{Unif}[0,1],       \label{eq:5-2a}
\end{equation}
 while the labels are independent uniform variables on \([0,p]\) for
 \(z\in A\) and on \((p,1]\) for \(z\in D\). The same assertion holds
 after deleting any set of edges determined by the layer field, with \(D\)
 interpreted as the tested boundary in that deterministic subgraph.

Here is the form needed when an allowed set is selected after an earlier
exploration. Let \(\mathcal F_{\rm old}\) be the terminal sigma-field of
that exploration and let \(Q_{\rm old}\) be its complete queried set.
Suppose that \(W\) is \(\mathcal F_{\rm old}\)-measurable and that the
new cluster exploration queries only coordinates outside
\(Q_{\rm old}\). If \(A\) and \(D\) are its newly queried open and closed
sets, respectively, then, in the finite-dimensional conditional-kernel
sense,
\begin{equation}
 \mathcal L\left((U_z)_{z\notin Q_{\rm old}\cup A\cup D}
       \,\middle|\,\mathcal F_{\rm old},A,D,R_v=x\right)
 =\bigotimes_{z\notin Q_{\rm old}\cup A\cup D}
       \operatorname{Unif}[0,1].                              \label{eq:5-2a-adaptive}
\end{equation}
Coordinates in \(Q_{\rm old}\) retain the one-coordinate restrictions
recorded by \(\mathcal F_{\rm old}\); they do not become fresh merely
because the new allowed set omits them. If such coordinates are retained
as fixed boundary data, the same conclusion holds for the never-queried
coordinates, but no product-law assertion is made for the retained old
coordinates.

This is the stopped product kernel,
Lemma~\ref{lem:stopped-product-kernel}, with the two-cell partition
\([0,p]\cup(p,1]\). Here is the specialization, including the exact
terminal sigma-field used below. Fix a deterministic ordering of
the sites, use it to break every exploration tie, and let
\(z_1,z_2,\ldots\) be the resulting sequence of queried sites.  At the
moment \(z_i\) is selected, its identity is measurable with respect to
the previously queried labels and the layer field, whereas \(U_{z_i}\)
is still uniform and independent of that information.  For a finite
outcome \((A,D)\), the corresponding exploration atom is therefore
\[
 \bigcap_{z\in A}\{U_z\leq p\}
 \cap\bigcap_{z\in D}\{U_z>p\},
\]
together with conditions measurable with respect to the already fixed
layer field; no coordinate outside \(A\cup D\) occurs in this atom.
Independence of the labels gives \eqref{eq:5-2a} and the two truncated
uniform laws by direct factorization.  The same induction applies when
there is an earlier query history, provided the entire query sequence is
kept: applying Lemma~\ref{lem:stopped-product-kernel} to the concatenated
old and new queries gives \eqref{eq:5-2a-adaptive}, with
\(Q_{\rm old}\) still excluded from the fresh-coordinate set. Notice
that predictability of the new queries alone would not permit one to
discard earlier queried coordinates.

If the exploration is infinite, let \(Q_n\) be the set of its first
\(n\) queried sites, let \(Q_\infty=\bigcup_nQ_n\), and let
\(\mathcal F_n\) and \(\mathcal F_\infty\) be the corresponding finite
and terminal exploration sigma-fields.  The precise stopping-set
identity is the following.  For every deterministic finite set \(S\),
every bounded Borel function \(g:[0,1]^S\to\mathbb R\), and every
bounded \(\mathcal F_\infty\)-measurable random variable \(H\),
\begin{equation}
 \mathbb E[H\mathbf1_{\{S\cap Q_\infty=\varnothing\}}g(U_S)]
 =\mathbb E[H\mathbf1_{\{S\cap Q_\infty=\varnothing\}}]
   \int_{[0,1]^S}g(u)\,du.                                  \label{eq:5-2b}
\end{equation}
To verify \eqref{eq:5-2b}, first take \(H\) to be a cylinder function of
the first \(n\) query choices and revealed labels and replace
\(Q_\infty\) by \(Q_m\), \(m\geq n\).  Induction over queries gives the
identity because the next queried site is chosen before its label is
read.  Letting \(m\to\infty\) gives the indicator in
\eqref{eq:5-2b} by bounded convergence.  Cylinder functions of the
exploration generate \(\mathcal F_\infty\), so a monotone-class argument
extends the identity to every bounded \(H\).  Identity \eqref{eq:5-2b}
characterizes the conditional law of every finite subset of the
unqueried coordinates as independent uniform variables and hence gives
the residual product kernel.  Thus later uses condition on the terminal
exploration sigma-field, not on a possibly zero-probability atom
specifying an infinite cluster.

\textbf{Two-level boundary kernel.} If \(p_-<p_+<1\) and the finite
\(p_-\)-cluster is \(A\), then, conditional on that cluster, the labels
on \(\partial_VA\) are independent uniform variables on \((p_-,1]\).
Consequently their \(p_+\)-states are independent with opening
probability
\[
 \rho=\frac{p_+-p_-}{1-p_-},
\]
and every label outside \(A\cup\partial_VA\) remains unconditioned.
Indeed, the cluster event is the intersection of
\(\{U_z\leq p_-:z\in A\}\) and
\(\{U_z>p_-:z\in\partial_VA\}\).  Factoring this event coordinate by
coordinate gives the claimed boundary law, and for a boundary site \(z\)
\[
 \mathbb P(U_z\leq p_+\mid U_z>p_-)
 =\frac{p_+-p_-}{1-p_-}=\rho.
\]
The factorization also proves the asserted mutual independence and shows
that no conditioning is imposed beyond \(A\cup\partial_VA\).

\textbf{First-hit residual kernel.} Fix a slab and expose the open-root
cluster only until its first visits to an extreme layer. Let
\(w_1,\ldots,w_N\) be a deterministic enumeration, from the exposed
data, of the open first-hit sites. Set \(A_0\) equal to the explored
part, and for \(i\geq1\) explore the cluster of \(w_i\) in the remaining
allowed vertices after deleting all previously explored sites except
\(w_i\), retaining \(w_i\) as an open boundary root. Conditional on the
past before this
exploration, \(U_{w_i}\leq p\), every other untested label is uniform on
\([0,1]\), and some tested external-boundary labels may be forced into
\((p,1]\). Hence the residual cluster is stochastically dominated, for
every increasing cluster functional, by an independent cluster under
\(\widehat{\mathbb P}_p^{w_i}\).
For the domination, couple the comparison cluster and the residual
cluster with the same uniforms on every untested non-root site.  Give the
comparison cluster fresh uniforms on coordinates that the past has
already forced closed, while keeping those coordinates closed in the
residual exploration.  Every open residual path then occurs in the
comparison configuration.  The root coordinate has law
\(\operatorname{Unif}[0,p]\) in both constructions and all remaining
comparison coordinates are unconditioned, so the comparison law is
exactly \(\widehat{\mathbb P}_p^{w_i}\).

These identities specify the laws when a root or first-hit site has
already been declared open; there is no implicit endpoint convention and
no missing factor of \(p\). They also exhaust the uses below. The peak
exploration in \(H_k(v)\) and the restricted-cluster exploration used in
the Reimer comparison both begin with \(Q_{\rm old}=\varnothing\), so
\eqref{eq:5-2a} applies literally. The sequential increment explorations
have \(Q_{\rm old}\ne\varnothing\); there the preceding first-hit
residual coupling, rather than a fresh-law assertion for all sites outside
the new cluster, gives the required stochastic upper bound.

The exploration identities above specify every conditional law used in
what follows. In particular, \eqref{eq:5-4} supplies the factor
\(p^{-1}\) at each shared junction after the initial site is conditioned
open.

\textbf{Layer probabilities and first moments.} Put
\[
 \widehat P_p(n)=\inf_{v,x}
 \widehat{\mathbb P}_p^{v,x}
 \bigl(v\leftrightarrow L_n(v)\text{ in }L_{0,\infty}(v)\bigr).
\]
The site concatenation has the same one-layer buffer as \eqref{eq:5-0}:
\begin{equation}
 \widehat P_p(m+n+1)
 \geq p\,\widehat P_p(m)\widehat P_p(n)
 \qquad(m,n\geq0).                                             \label{eq:5-11d}
\end{equation}
Indeed, explore the first connection only to its first open site
\(z\in L_m(v)\), without exploring from any top-layer site. Other sites of
\(L_m(v)\) may have been tested closed, but \eqref{eq:5-1b} puts
\(z^+\) in \(L_{m+1}(v)\), outside every tested coordinate. Conditional
on the exploration, \(z^+\) is open with probability \(p\), and all sites
of \(L_{0,\infty}(z^+)=L_{m+1,\infty}(v)\), apart from \(z^+\), retain
their product law. The conditional continuation probability to
\(L_n(z^+)=L_{m+n+1}(v)\) is at least \(\widehat P_p(n)\). This proves
\eqref{eq:5-11d} and shows explicitly why the closed first-interface labels are
irrelevant.

Writing \(a_p(n)=-\log\widehat P_p(n)\) and \(c_p=-\log p\), the sequence
\(a_p(n-1)+c_p\), \(n\geq1\), is subadditive. Hence
\begin{equation}
 \lim_{n\to\infty}\frac{a_p(n)}{t_0n}
 =\frac1{t_0}\inf_{n\geq1}
   \frac{a_p(n-1)+c_p}{n}.                                    \label{eq:5-11d0}
\end{equation}
For every offset, the raw half-spaces and the layer half-spaces contain
one another after enlarging by at most one layer. Opening one
maximal-increment neighbor as in \eqref{eq:5-1b} handles the possible starting
layer. Consequently, uniformly in \(v,x,n\),
\begin{equation}
 \widehat A_p^{\rm s}(t_0(n+1))
 \leq \widehat{\mathbb P}_p^{v,x}
       (v\leftrightarrow L_n(v)\text{ in }L_{0,\infty}(v))
 \leq p^{-1}\widehat A_p^{\rm s}(t_0n).                       \label{eq:5-11d1}
\end{equation}
To check the first inclusion, follow a path from raw height zero to height
\(t_0(n+1)\); an edge changes height by at most \(t_0\), so the path
visits \(L_n(v)\). For the second, let \(v^-\) be a maximal-decrement
neighbor of \(v\). Relative to \(v^-\), the entire set
\(L_{0,\infty}(v)\) lies in raw height at least zero and \(L_n(v)\) lies
above raw height \(t_0n\). Let \(E_v\) be the crossing event in the
middle of \eqref{eq:5-11d1} and let \(G=\{v^-\text{ open}\}\). Under the
law conditioned on \(v\) being open, site FKG gives
\[
 \widehat{\mathbb P}_p^{v,x}(E_v\cap G)
 \geq p\widehat{\mathbb P}_p^{v,x}(E_v).
\]
The event \(E_v\cap G\) forces both \(v\) and \(v^-\) open. Therefore
conditioning at either endpoint divides its unconditional probability
by the same factor \(p\), and
\[
 \widehat{\mathbb P}_p^{v^-,R_{v^-}}(E_v\cap G)
 =\widehat{\mathbb P}_p^{v,x}(E_v\cap G).
\]
On this event the edge \(v^-v\), followed by the \(E_v\)-path, is a raw
half-space crossing from \(v^-\) to height at least \(t_0n\). Hence
\(p\widehat{\mathbb P}_p^{v,x}(E_v)\leq
\widehat A_p^{\rm s}(t_0n)\), which is the second inequality in
\eqref{eq:5-11d1}.
Equations
\eqref{eq:5-0a}, \eqref{eq:5-11d0}, and \eqref{eq:5-11d1} show that the limit in \eqref{eq:5-11d0} is the
same \(\alpha_p^{\rm s}\) as in \eqref{eq:5-0a}, and \eqref{eq:5-0c} gives the uniform
probability-decay bound. All bounded connector and offset shifts below
change only multiplicative constants and not this rate.

We shall also use the following version with a lower slab enlargement:
for \(n,r\geq0\),
\begin{equation}
 \sup_{v,x}\widehat{\mathbb P}_p^{v,x}
 \bigl(v\leftrightarrow L_n(v)\text{ in }L_{-r,\infty}(v)\bigr)
 \leq C_p e^{|\log p|(r+1)}e^{-t_0\alpha_p^{\rm s}n}.          \label{eq:5-11d2}
\end{equation}
Indeed, let \(u\) be obtained from \(v\) by taking \(r+1\) successive
maximal-decrement edges. Then \(v\in L_{r+1}(u)\), the connecting path
lies in \(L_{0,r+1}(u)\), and the exact shift rule gives
\[
 L_{-r,\infty}(v)=L_{1,\infty}(u),\qquad
 L_n(v)=L_{n+r+1}(u).
\]
The connector factor can again be written without suppressing the
endpoint conditioning. Let \(E_v\) denote the event on the left of
\eqref{eq:5-11d2} and let \(G_\gamma\) be the event that all \(r+1\)
non-root sites of the path when oriented from \(v\) to \(u\) are open.
Under \(\widehat{\mathbb P}_p^{v,x}\), site FKG gives
\[
 \widehat{\mathbb P}_p^{v,x}(E_v\cap G_\gamma)
 \geq p^{r+1}\widehat{\mathbb P}_p^{v,x}(E_v).
\]
The intersection forces both endpoints open, so its probabilities under
the laws conditioned at \(u\) and at \(v\) are equal, exactly as in the
preceding paragraph. On the intersection, the path from \(u\) to \(v\)
and the \(E_v\)-path form a crossing from \(u\) to
\(L_{n+r+1}(u)\) inside \(L_{0,\infty}(u)\). Consequently
\[
 p^{r+1}\widehat{\mathbb P}_p^{v,x}(E_v)
 \leq\widehat{\mathbb P}_p^{u,R_u}
   (u\leftrightarrow L_{n+r+1}(u)\text{ in }L_{0,\infty}(u)).
\]
Apply \eqref{eq:5-0c} and \eqref{eq:5-11d1} to the crossing from \(u\).
This proves \eqref{eq:5-11d2} uniformly in the offset and uses only sites
on the displayed deterministic path.

For a first-visit decomposition at a site \(z\),
\[
 \{v\leftrightarrow u\text{ through }z\}
 \subseteq (\{v\leftrightarrow z\}\circ_z
             \{z\leftrightarrow u\}),
\]
where \(\circ_z\) means that the two path witnesses may meet only at the
prescribed open site \(z\). We describe each decomposition before using
it. If \(m\leq\ell\leq k\leq n\) and \(\ell\geq0\), take a path in
\(L_{m,n}(v)\) from \(v\) to a vertex \(u\in L_k(v)\).

For the first inequality below, let \(z\) be the first vertex of a
fixed simple such path in \(L_\ell(v)\). Its initial segment lies in
\(L_{m,\ell}(v)\). By the exact shift rule, its remaining segment is a
connection from the open root \(z\) to layer \(k-\ell\) inside
\(L_{m-\ell,n-\ell}(z)\). Applying \eqref{eq:5-3}, summing first over \(u\) and
then over \(z\), and taking the root/offset supremum only for the second
factor gives the first line of \eqref{eq:5-11d-prime}.

For the second line, take instead the last visit \(z\) to
\(L_\ell(v)\). The initial connection to \(z\) may use the whole slab,
while the terminal segment never goes below \(L_\ell(v)\) and hence lies
in \(L_{0,n-\ell}(z)\). For the third line, let \(j\) be the maximal
layer visited by the path and select its first visit \(z\in L_j(v)\).
Necessarily \(k\vee0\leq j\leq n\); the initial segment is counted by
\(X_j^{m,j}(v)\), while the terminal segment from \(z\) to \(u\), after
shifting layers, is counted by \(X_{k-j}^{m-j,0}(z)\). Summing over the
possible value of \(j\) gives the third line. Thus \eqref{eq:5-3} yields
\begin{equation}
 \begin{aligned}
 \widehat{\mathbb E}_p^{v,x}X_k^{m,n}(v)
 &\leq \widehat{\mathbb E}_p^{v,x}X_\ell^{m,\ell}(v)
          \widehat E_p^{m-\ell,n-\ell}(k-\ell),\\
 \widehat{\mathbb E}_p^{v,x}X_k^{m,n}(v)
 &\leq \widehat{\mathbb E}_p^{v,x}X_\ell^{m,n}(v)
          \widehat E_p^{0,n-\ell}(k-\ell),\\
 \widehat{\mathbb E}_p^{v,x}X_k^{m,n}(v)
 &\leq\sum_{j=k\vee0}^{n}
     \widehat{\mathbb E}_p^{v,x}X_j^{m,j}(v)
          \widehat E_p^{m-j,0}(k-j).
 \end{aligned} \label{eq:5-11d-prime}
\end{equation}
If \(m\leq k\leq\ell\leq n\) and \(\ell\leq0\), the reflected
inequalities are
\begin{equation}
 \begin{aligned}
 \widehat{\mathbb E}_p^{v,x}X_k^{m,n}(v)
 &\leq \widehat{\mathbb E}_p^{v,x}X_\ell^{\ell,n}(v)
          \widehat E_p^{m-\ell,n-\ell}(k-\ell),\\
 \widehat{\mathbb E}_p^{v,x}X_k^{m,n}(v)
 &\leq \widehat{\mathbb E}_p^{v,x}X_\ell^{m,n}(v)
          \widehat E_p^{m-\ell,0}(k-\ell),\\
 \widehat{\mathbb E}_p^{v,x}X_k^{m,n}(v)
 &\leq\sum_{j=m}^{k\wedge0}
     \widehat{\mathbb E}_p^{v,x}X_j^{j,n}(v)
          \widehat E_p^{0,n-j}(k-j).
 \end{aligned} \label{eq:5-11d-prime-prime}
\end{equation}
The three reflected inequalities in \eqref{eq:5-11d-prime-prime} are obtained by the same
construction with minimum in place of maximum and with the path read in
the opposite height direction: the first visit to \(L_\ell\), the last
visit to \(L_\ell\), and the first visit to the minimal layer
\(j\in[m,k\wedge0]\), respectively. No reflection automorphism of the
graph is being assumed. These are path-order decompositions using only
the cocycle height. Their displayed ranges are all their hypotheses.

For the second moment, on the event that
\(u,w\) are both joined to \(v\) in the slab, choose canonically a finite
open connected subgraph containing them and then a spanning tree. The
minimal subtree containing \(v,u,w\) has a unique branch site \(z\), and
its three arms are internally vertex-disjoint. If two of
\(v,u,w,z\) coincide, contract the corresponding zero-length arm; the
remaining endpoint is already open and the normalized zero-displacement
factor is at least one. Thus \eqref{eq:5-4} also applies when \(u=w\) or when the
branch is a terminal. Conditional on \(v\) being open, \eqref{eq:5-4} bounds the
probability of a fixed three-arm realization by the product of the three
normalized connection probabilities. Sum first over \(u\in L_k(v)\)
and \(w\in L_\ell(v)\), then over the layer \(j\) and branch site
\(z\in L_j(v)\), and bound each factor rooted at \(z\) by the supremum in
\eqref{eq:5-11c-prime}. The sum over \(z\) is the first factor on the right below.
This gives
\begin{equation}
 \widehat{\mathbb E}_p^{v,x}
   [X_k^{m,n}(v)X_\ell^{m,n}(v)]
 \leq\sum_{j=m}^n
   \widehat E_p^{m,n}(j)
   \widehat E_p^{m-j,n-j}(k-j)
   \widehat E_p^{m-j,n-j}(\ell-j).                             \label{eq:5-11e}
\end{equation}
We next derive, rather than import, every first- and second-moment
estimate needed below. Put
\[
 U_p(k)=\widehat E_p^{-\infty,k}(k;1),\qquad
 D_p(k)=\widehat E_p^{-\infty,0}(-k;1).
\]
The first-visit decomposition gives
\(U_p(k+\ell)\leq U_p(k)U_p(\ell)\), while the last-visit
decomposition at the intermediate lower layer gives
\(D_p(k+\ell)\leq D_p(k)D_p(\ell)\). We record the fixed-offset
comparison rather than hide it in a connector constant. For two offsets
\(x,y\), the raw band of heights between \(t_0k\) and \(t_0(k+2)\)
contains \(L_{k+1}^y(v)\) and is contained in
\(L_{k,k+2}^x(v)\). Write \(X_j^{m,n}(v;x)\) when the offset is to be
shown explicitly. Hence, pointwise,
\[
 X_{k+1}^{-\infty,k+1}(v;y)
 \leq\sum_{j=0}^{2}X_{k+j}^{-\infty,k+2}(v;x).
\]
Applying the first inequality in \eqref{eq:5-11d-prime} at layer \(k\) to each term
on the right gives
\begin{equation}
 \sup_{v,y}\widehat{\mathbb E}_p^{v,y}
       X_{k+1}^{-\infty,k+1}(v)
 \leq B_p\inf_{v,x}\widehat{\mathbb E}_p^{v,x}
       X_k^{-\infty,k}(v),                                   \label{eq:5-11d3}
\end{equation}
where
\(B_p=\sum_{j=0}^{2}\widehat E_p^{-\infty,2}(j)<\infty\) for
\(p\in I_h^{\rm s}\). The downward comparison requires more detail
because changing the offset moves both the target layer and the upper
boundary of the half-space. Write
\[
 D_x^v(k)=\widehat{\mathbb E}_p^{v,x}
       X_{-k}^{-\infty,0}(v),
 \qquad D_p(k)=\sup_{v,x}D_x^v(k).
\]
Fix offsets \(x,y\in[0,1)\) and a vertex \(v\), and put \(w=v^+\), so
that \(\log\Delta(v,w)=t_0\). Measured in raw height from \(v\), the
upper boundaries of \(L_{-\infty,0}(v;y)\) and
\(L_{-\infty,0}(w;x)\) are \(yt_0\) and \((1+x)t_0\), respectively.
Consequently
\[
 L_{-\infty,0}(v;y)\subseteq L_{-\infty,0}(w;x),
 \qquad
 L_{-k-1}(v;y)\subseteq
       \bigcup_{j=1}^{3}L_{-k-j}(w;x).
\]
The edge \(wv\) lies in the latter half-space. Under the law conditioned
on \(v\) being open, the event that \(w\) is open has probability \(p\).
For a vertex \(u\) counted on the left, let \(E_u\) be its connection
event in the \(y\)-half-space. Site FKG gives
\[
 p\,\widehat{\mathbb P}_p^{v,y}(E_u)
 \leq \widehat{\mathbb P}_p^{v,y}(E_u,\ w\text{ open}).
\]
The event on the right forces both \(v\) and \(w\) open, so its
probability is unchanged when the conditioned-open endpoint is changed
from \(v\) to \(w\): in either case it is the ordinary probability of
the same event divided by \(p\). The edge \(wv\) and the two displayed
inclusions then put \(u\) in one of the three counts rooted at \(w\).
Summing over \(u\) gives
\[
 pD_y^v(k+1)
 \leq\sum_{j=1}^{3}
   \widehat{\mathbb E}_p^{w,x}X_{-k-j}^{-\infty,0}(w).
\]
Formally, apply FKG first to the count truncated by a finite vertex
exhaustion and then use monotone convergence; finiteness follows from
\(p\in I_h^{\rm s}\).

For each summand, read a simple connecting path from \(w\) and take its
last visit to \(L_{-k}(w;x)\). The initial segment is counted by
\(D_x^w(k)\), and the terminal segment stays below that layer and is
counted, after shifting the layers, by \(D_p(j)\). The shared-junction
inequality \eqref{eq:5-3}, equivalently the second line of
\eqref{eq:5-11d-prime-prime}, therefore yields
\[
 \widehat{\mathbb E}_p^{w,x}X_{-k-j}^{-\infty,0}(w)
 \leq D_x^w(k)D_p(j).
\]
The three numbers \(D_p(j)\), \(1\leq j\leq3\), are finite for
\(p\in I_h^{\rm s}\). Transitivity makes \(D_x^w(k)\) independent of
the choice of \(w\). Taking the supremum over \(y\) and the infimum over
\(x\) gives
\begin{equation}
 \sup_{v,y}\widehat{\mathbb E}_p^{v,y}
       X_{-k-1}^{-\infty,0}(v)
 \leq B'_p\inf_{v,x}\widehat{\mathbb E}_p^{v,x}
       X_{-k}^{-\infty,0}(v).                                \label{eq:5-11d4}
\end{equation}
with the explicit finite constant
\(B'_p=p^{-1}\sum_{j=1}^{3}D_p(j)\). No reflection automorphism is used
in this comparison.
Thus all fixed offsets have the same exponential rates, and Fekete's
lemma applies without a hidden choice of root or offset.

Define first
\[
 \beta_p^+=-\lim_{k\to\infty}\frac{\log U_p(k)}{t_0k},
 \qquad
 \beta_p^-=1-\lim_{k\to\infty}\frac{\log D_p(k)}{t_0k},
\]
whose existence follows from the two submultiplicative inequalities, and
put
\begin{equation}
 h_p(k)=\log U_p(k)+t_0\beta_p^+k,
 \qquad
 h'_p(k)=\log D_p(k)+t_0(\beta_p^--1)k.                      \label{eq:5-11e0}
\end{equation}
The two functions are nonnegative, subadditive, and \(o(k)\). Equality
of the two values used for \(\beta_p^{\rm s}\) follows from \eqref{eq:5-11b}, as
we now check with the offset averaging visible. Define
\[
 \overline U_p(n)=\int_0^1\widehat{\mathbb E}_p^{v,x}
       X_n^{-\infty,n}(v)\,dx,
 \qquad
 \overline D_p(n)=\int_0^1\widehat{\mathbb E}_p^{v,x}
       X_{-n}^{-\infty,0}(v)\,dx .
\]
By transitivity these quantities do not depend on \(v\). Apply \eqref{eq:5-11b}
to
\[
 F(u,z)=\mathbf1\{z\in L_n(u),\ u\leftrightarrow z
       \text{ in the prescribed half-space}\}.
\]
If the reversed summand is nonzero, then \(u\in L_{-n}(z)\), and the
shift rule changes the path domain \(L_{-\infty,n}(u)\) into
\(L_{-\infty,0}(z)\). Moreover, for every offset,
\[
 e^{-t_0(n+1)}\leq\Delta(z,u)\leq e^{-t_0(n-1)}.
\]
The transport vanishes unless both endpoints are open, so \eqref{eq:5-11b}
introduces no further factor of \(p\). It follows that
\begin{equation}
 \overline U_p(n)\asymp_G e^{-t_0n}\overline D_p(n).         \label{eq:5-11e-prime}
\end{equation}
The fixed-offset comparisons \eqref{eq:5-11d3}--\eqref{eq:5-11d4} put every fixed-offset
first moment between constant multiples of the corresponding averaged
quantity, with a shift of at most one layer. These bounded factors and
shifts disappear after division by \(t_0n\). Comparison of exponential
rates in \eqref{eq:5-11e-prime} therefore gives
\(\beta_p^+=\beta_p^-\); denote their common value by
\(\beta_p^{\rm s}\). We record the one-layer comparison used next. Mark
a vertex \(z\) counted by \(X_k^{-\infty,k}(v)\), let \(w=z^+\), and
open \(w\) if necessary. The image marked at \(w\) is counted by
\(X_{k+1}^{-\infty,k+1}(v)\). On the marked configuration space, this
map changes only the coordinate \(w\). For a fixed image marked at
\(w\), its former mark \(z\) has at most \(d\) choices and the original
state of \(w\) has two choices. If \(w\) was changed from closed to open,
the original-to-image Bernoulli weight ratio is \((1-p)/p\); otherwise it
is one. Summing the fibres gives the explicit bound
\[
 U_p(k)\leq 2d\max\{1,(1-p)/p\}\,U_p(k+1).
\]
Equation \eqref{eq:5-11d3} gives the reverse comparison after shifting
the index once. Applying the maximal-decrement map and
\eqref{eq:5-11d4} gives the analogous two-sided comparison for \(D_p\).

We also display how the offset averages are removed. Write
\(U_p^{\sup},U_p^{\inf}\) and \(D_p^{\sup},D_p^{\inf}\) for the
fixed-offset suprema and infima at the indicated index. Equation
\eqref{eq:5-11d3}, its downward analogue, and averaging give
\[
 \begin{split}
 U_p^{\sup}(k+1)&\leq B_pU_p^{\inf}(k)
       \leq B_p\overline U_p(k),\\
 D_p^{\sup}(k+1)&\leq B'_pD_p^{\inf}(k)
       \leq B'_p\overline D_p(k).
 \end{split}
\]
The one-layer finite-energy maps compare either supremum at indices
differing by a bounded amount. Therefore, using
\(\overline U_p(j)\leq U_p^{\sup}(j)\),
\(\overline D_p(j)\leq D_p^{\sup}(j)\), and
\eqref{eq:5-11e-prime}, we obtain, for example,
\[
 \begin{split}
 D_p^{\sup}(k)
 &\leq B'_p\overline D_p(k-1)
 \leq C_pe^{t_0(k-1)}\overline U_p(k-1)\\
 &\leq C_pe^{t_0k}U_p^{\sup}(k+1).
 \end{split}
\]
The same chain with \(U,D\) interchanged yields the reverse relation.
Since \(U_p=U_p^{\sup}\) and \(D_p=D_p^{\sup}\) by definition, this
proves, for a constant \(C_p\),
\begin{equation}
 D_p(k)\leq C_p e^{t_0k}U_p(k+1),
 \qquad U_p(k)\leq C_p e^{-t_0k}D_p(k+1).                    \label{eq:5-11e0c}
\end{equation}
Substituting \eqref{eq:5-11e0} into \eqref{eq:5-11e0c} and absorbing the fixed terms
\(t_0\beta_p^{\rm s}\) and \(t_0(\beta_p^{\rm s}-1)\) into the constant
gives
\[
 h'_p(k)\leq h_p(k+1)+C_p,
 \qquad h_p(k)\leq h'_p(k+1)+C_p,
\]
and hence the two limits in \eqref{eq:5-11c-prime-prime} agree with uniform error terms.
More precisely, the connector comparison gives
\begin{equation}
 \begin{aligned}
 b_p^{-1}e^{-t_0\beta_p^{\rm s}k+h_p(k+1)}
 &\leq\widehat{\mathbb E}_p^{v,x}X_k^{-\infty,k}(v)
 \leq b_pe^{-t_0\beta_p^{\rm s}k+h_p(k)},\\
 b_p^{-1}e^{-t_0(\beta_p^{\rm s}-1)k+h'_p(k+1)}
 &\leq\widehat{\mathbb E}_p^{v,x}X_{-k}^{-\infty,0}(v)
 \leq b_pe^{-t_0(\beta_p^{\rm s}-1)k+h'_p(k)}.
 \end{aligned} \label{eq:5-11e1}
\end{equation}
The existence of the separate limits followed from Fekete's lemma, their
equality from \eqref{eq:5-11e-prime}, and their uniformity in the root and offset from
\eqref{eq:5-11d3}--\eqref{eq:5-11d4}. This proves clause 1.

Suppose \(\beta_p^{\rm s}>1/2\). Apply the upper extreme-visit inequality
and \eqref{eq:5-11e1}. For \(k\geq0\), write its extreme layer as \(j=k+r\),
where \(r\geq0\). The upward factor in the summand is at most
\[
 C_p e^{-t_0\beta_p^{\rm s}(k+r)+h_p(k+r)},
\]
while the downward factor from layer \(j\) to layer \(k\) is at most
\[
 C_p e^{-t_0(\beta_p^{\rm s}-1)r+h'_p(r)}.
\]
Subadditivity gives \(h_p(k+r)\leq h_p(k)+h_p(r)\). Factoring out the
terms depending only on \(k\) therefore gives
\begin{equation}
 \widehat E_p^{-\infty,\infty}(k)
 \leq C_pe^{-t_0\beta_p^{\rm s}k+h_p(k)}
 \sum_{r\geq0}e^{-t_0(2\beta_p^{\rm s}-1)r+h_p(r)+h'_p(r)}.  \label{eq:5-11e2}
\end{equation}
The series is finite because \(h_p(r)+h'_p(r)=o(r)\). For negative
targets, a reflected minimum-layer decomposition is not used, since it
would constrain its two path pieces below rather than above. Instead
fix \(q\geq0\) and decompose a path from layer \(0\) to layer \(-q\) at
its maximum layer \(j\geq0\). Applying the third inequality in
\eqref{eq:5-11d-prime}, first with a finite upper cutoff and then using
monotone convergence, gives
\[
 \widehat E_p^{-\infty,\infty}(-q)
 \leq \sum_{j\geq0}U_p(j)D_p(q+j).
\]
The first factor counts the initial path up to its first visit to its
maximum layer; after shifting by \(j\), the second factor counts the
remaining path to layer \(-q-j\) in an upper-bounded half-space. Thus
both factors are exactly among those estimated in \eqref{eq:5-11e1}.
Using that estimate and
\(h'_p(q+j)\leq h'_p(q)+h'_p(j)\), we obtain
\begin{equation}
 \widehat E_p^{-\infty,\infty}(-q)
 \leq C_pe^{-t_0(\beta_p^{\rm s}-1)q+h'_p(q)}
 \sum_{j\geq0}
 e^{-t_0(2\beta_p^{\rm s}-1)j+h_p(j)+h'_p(j)}.              \label{eq:5-11e2-negative}
\end{equation}
This is the same convergent series as in \eqref{eq:5-11e2}; no
height-reflection symmetry or lower-bounded half-space estimate is
being assumed. Restricting the full cluster to
\(L_{-\infty,k}\) for \(k\geq0\), or to \(L_{-\infty,0}\) for
\(k<0\), gives the matching lower bounds directly from \eqref{eq:5-11e1}. Hence
\begin{equation}
 \begin{array}{ll}
 c_pe^{-t_0\beta_p^{\rm s}k}
 \leq\widehat E_p^{-\infty,\infty}(k)
 \leq C_pe^{-t_0\beta_p^{\rm s}k+h_p(k)},&k\geq0,\\[2mm]
 c_pe^{t_0(\beta_p^{\rm s}-1)k}
 \leq\widehat E_p^{-\infty,\infty}(k)
 \leq C_pe^{t_0(\beta_p^{\rm s}-1)k+h'_p(-k)},&k<0.
 \end{array}                                                    \label{eq:5-11e-prime-prime}
\end{equation}

We now show that the error functions are bounded when
\(p<p_t^{\rm s}\). Indeed,
\(\widehat\chi_{p,1/2}<\infty\) first implies \(p\in I_h^{\rm s}\).
If \(\beta_p^{\rm s}\leq1/2\), the lower bounds in \eqref{eq:5-11e1}, multiplied
by the layer weights \(e^{t_0k/2}\) and summed over \(k\), would make
\(\widehat\chi_{p,1/2}\) infinite. Hence
\(\beta_p^{\rm s}>1/2\); conversely
\eqref{eq:5-11e2} and \eqref{eq:5-11e2-negative} show
that this strict inequality makes the half-tilted susceptibility finite.
Apply \eqref{eq:5-11e} with \(m=0\), \(n=k\), and both target layers equal to
\(k\). Reindex the branch layer as \(k-\ell\), dominate every restricted
first moment by the full-cluster bound \eqref{eq:5-11e-prime-prime}, and obtain
\[
 \begin{split}
 \widehat{\mathbb E}_p^{v,x}[(X_k^{0,k})^2]
 &\leq C_p\sum_{\ell=0}^k
 e^{-t_0\beta_p^{\rm s}(k-\ell)+h_p(k-\ell)}
 e^{-2t_0\beta_p^{\rm s}\ell+2h_p(\ell)}\\
 &\leq C_p e^{-t_0\beta_p^{\rm s}k+M_k}
       \sum_{\ell\geq0}e^{-t_0\beta_p^{\rm s}\ell+2h_p(\ell)},
 \end{split}
\]
where \(M_k=\max_{0\leq j\leq k}h_p(j)\). The last series is finite
because \(h_p(\ell)=o(\ell)\) and \(\beta_p^{\rm s}>1/2\). Absorbing it
into the constant gives
\begin{equation}
 \widehat{\mathbb E}_p^{v,x}[(X_k^{0,k})^2]
 \leq C_p e^{-t_0\beta_p^{\rm s}k+\max_{0\leq j\leq k}h_p(j)}. \label{eq:5-11e3}
\end{equation}
The last-visit inequality in \eqref{eq:5-11d-prime}, used with
\((m,n,\ell)=(-\infty,k,0)\), gives
\[
 \widehat E_p^{-\infty,k}(k)
 \leq \widehat E_p^{-\infty,k}(0)
       \widehat E_p^{0,k}(k)
 \leq \widehat E_p^{-\infty,\infty}(0)
       \widehat E_p^{0,k}(k).
\]
The zero-layer factor is finite for \(p<p_t^{\rm s}\). Rearranging and
using \eqref{eq:5-11e1} gives
\begin{equation}
 \widehat E_p^{0,k}(k)
 \geq c_p e^{-t_0\beta_p^{\rm s}k+h_p(k)}.                  \label{eq:5-11e4}
\end{equation}
This is a supremal estimate; no pointwise lower bound for the slab first
moment is asserted here. For each fixed \(v,x\), Cauchy--Schwarz gives
\[
 \widehat{\mathbb P}_p^{v,x}(X_k^{0,k}>0)
 \geq
 \frac{(\widehat{\mathbb E}_p^{v,x}X_k^{0,k})^2}
      {\widehat{\mathbb E}_p^{v,x}[(X_k^{0,k})^2]}.
\]
The denominator has the uniform upper bound \eqref{eq:5-11e3}. Since the
supremum in \eqref{eq:5-11e4} is positive, choose \((v,x)\) whose first
moment is at least one half of that supremum and apply the preceding
inequality. This gives the supremal crossing bound
\begin{equation}
 \sup_{v,x}\widehat{\mathbb P}_p^{v,x}
       (v\leftrightarrow L_k(v)\text{ in }L_{0,k}(v))
 \geq c_p e^{-t_0\beta_p^{\rm s}k+2h_p(k)-\max_{j\leq k}h_p(j)}. \label{eq:5-11e5}
\end{equation}

We justify the upper bound used to control \(h_p\) without assuming the
conclusion \(\alpha_p^{\rm s}=\beta_p^{\rm s}\). For every \(v,x\),
\[
 \widehat{\mathbb P}_p^{v,x}
   (v\leftrightarrow L_k(v)\text{ in }L_{0,\infty}(v))
 \leq \widehat{\mathbb E}_p^{v,x}X_k^{0,k}(v)
 \leq U_p(k)=e^{-t_0\beta_p^{\rm s}k+h_p(k)}.
\]
Equation \eqref{eq:5-11d1} shows that the probability on the left has
exponential decay rate \(\alpha_p^{\rm s}\), uniformly in \(v,x\), while
\(h_p(k)=o(k)\). It follows that
\(\alpha_p^{\rm s}\geq\beta_p^{\rm s}\). Consequently
\eqref{eq:5-0c} and the upper inequality in \eqref{eq:5-11d1} give
\begin{equation}
 \sup_{v,x}\widehat{\mathbb P}_p^{v,x}
   (v\leftrightarrow L_k(v)\text{ in }L_{0,\infty}(v))
 \leq C_pe^{-t_0\beta_p^{\rm s}k}.                           \label{eq:5-11e5-upper}
\end{equation}
Reaching \(L_k(v)\) inside \(L_{0,\infty}(v)\) is equivalent to doing so
inside \(L_{0,k}(v)\), by stopping a witnessing path at its first visit to
\(L_k(v)\). Put \(M_k=\max_{0\leq j\leq k}h_p(j)\). Comparing
\eqref{eq:5-11e5} with \eqref{eq:5-11e5-upper} gives
\[
 2h_p(k)-M_k\leq C_p.
\]
If \(k\) is a record time, then \(h_p(k)=M_k\), and hence
\(h_p(k)\leq C_p\). For arbitrary \(k\), choose \(j\leq k\) with
\(h_p(j)=M_k\); choosing the first such \(j\) makes it a record time, so
\(h_p(k)\leq M_k=h_p(j)\leq C_p\). Therefore
\(\sup_kh_p(k)<\infty\).
Substituting the definitions in \eqref{eq:5-11e0} into the first
one-layer comparison in \eqref{eq:5-11e0c} then gives
\(\sup_kh'_p(k)<\infty\).

It remains to recover the pointwise lower bound, rather than silently
promote the supremum in \eqref{eq:5-11e5}. Since \(h_p\) is now bounded,
that equation and the upper inequality in \eqref{eq:5-11d1} imply, for
\(k\geq1\),
\[
 \widehat A_p^{\rm s}(t_0k)
 \geq p\sup_{v,x}\widehat{\mathbb P}_p^{v,x}
   (v\leftrightarrow L_k(v)\text{ in }L_{0,\infty}(v))
 \geq c_pe^{-t_0\beta_p^{\rm s}k}.
\]
Together with \eqref{eq:5-0c} and the already proved inequality
\(\alpha_p^{\rm s}\geq\beta_p^{\rm s}\), this proves
\(\alpha_p^{\rm s}=\beta_p^{\rm s}\). The lower inequality in
\eqref{eq:5-11d1}, now used at every fixed \(v,x\), yields
\[
 \widehat{\mathbb P}_p^{v,x}
   (v\leftrightarrow L_k(v)\text{ in }L_{0,\infty}(v))
 \geq \widehat A_p^{\rm s}(t_0(k+1))
 \geq c_pe^{-t_0\beta_p^{\rm s}k}.
\]
The matching upper bound is \eqref{eq:5-11e5-upper}. The unrestricted
probability has the same lower bound, while Markov's inequality and the
bounded-error form of \eqref{eq:5-11e-prime-prime} give its upper bound.
Thus, uniformly in \(v,x,k\),
\begin{equation}
 \widehat{\mathbb P}_p^{v,x}
  (v\leftrightarrow L_k(v)\text{ in }L_{0,\infty}(v))
 \asymp_p\exp[-t_0\beta_p^{\rm s}k]                           \label{eq:5-11e-prime-prime-prime}
\end{equation}
and the same estimate without the half-space restriction. Comparing with
the definition of \(\alpha_p^{\rm s}\) is consistent with the equality
proved above. For every
\(p\in I_h^{\rm s}\) with \(p<p_t^{\rm s}\), summing the now
bounded-error form of \eqref{eq:5-11e-prime-prime} over all layers at tilt \(0\) gives
\[
 \widehat\chi_{p,0}\asymp_p
 \sum_{k\geq0}e^{-t_0\beta_p^{\rm s}k}
 +\sum_{r\geq1}e^{-t_0(\beta_p^{\rm s}-1)r}.
\]
The first series converges since \(\beta_p^{\rm s}>1/2\), and the second
converges exactly when \(\beta_p^{\rm s}>1\). Site sharpness
\cite[Theorem~2 and Section~6]{AntunovicVeselic2008} identifies finiteness of the ordinary
susceptibility with \(p<p_c^{\rm s}\).
Consequently
\begin{equation}
 p<p_c^{\rm s}\quad\Longleftrightarrow\quad
 \beta_p^{\rm s}>1.                                          \label{eq:5-11e-prime-prime-prime-prime}
\end{equation}

We make the branch-layer summation used for peaks explicit. If
\(p<p_c^{\rm s}\), write \(\beta=\beta_p^{\rm s}>1\) and
\[
 \varphi_\beta(j)=
 \begin{cases}
  e^{-t_0\beta j},&j\geq0,\\
  e^{t_0(\beta-1)j},&j<0.
 \end{cases}
\]
The bounded-error form of \eqref{eq:5-11e-prime-prime} says
\(\widehat E_p^{-\infty,\infty}(j)\leq C_p\varphi_\beta(j)\).
Substitution in \eqref{eq:5-11e} gives
\[
 \widehat{\mathbb E}_p^{v,x}
 [X_{-k}^{-\infty,\infty}X_{-\ell}^{-\infty,\infty}]
 \leq C_p\sum_{j\in\mathbb Z}
 \varphi_\beta(j)\varphi_\beta(-k-j)\varphi_\beta(-\ell-j).
\]
Assume without loss of generality that \(0\leq k\leq\ell\). Split the
sum at \(-\ell,-k,0\). Writing
\(t_0^{-1}\log[\varphi_\beta(j)\varphi_\beta(-k-j)
\varphi_\beta(-\ell-j)]\) as \(E(j)\), direct substitution gives
\[
 E(j)=
 \begin{cases}
  \beta(k+\ell)+(3\beta-1)j,&j\leq-\ell,\\
  \beta(k+j)-(\beta-1)\ell,&-\ell\leq j\leq-k,\\
  -(\beta-1)(k+\ell+j),&-k\leq j\leq0,\\
  -(\beta-1)(k+\ell)-(3\beta-2)j,&j\geq0.
 \end{cases}
\]
At \(j=-\ell\), the first expression is at most
\(-(\beta-1)\ell\) because \(k\leq\ell\), and it decreases at rate
\(3\beta-1\) as \(j\) moves left. The second expression is maximized at
\(j=-k\), where it equals \(-(\beta-1)\ell\). The third is maximized at
\(j=-k\), with the same value, and the fourth is at most
\(-(\beta-1)\ell\) at zero and decreases at rate \(3\beta-2\).
All four rates are positive because \(\beta>1\). Thus each range is a
geometric sum bounded by a constant times
\(e^{-t_0(\beta-1)\ell}\), uniformly in \(k,\ell\), and hence
\begin{equation}
 \widehat{\mathbb E}_p^{v,x}
 [X_{-k}^{-\infty,\infty}X_{-\ell}^{-\infty,\infty}]
 \leq C_pe^{-t_0(\beta_p^{\rm s}-1)(k\vee\ell)}.             \label{eq:5-11e5a}
\end{equation}
For the upward count, use \eqref{eq:5-11e} with \(m=0\), \(n=k\), and both target
layers equal to \(k\), and dominate each restricted first moment by its
full-cluster counterpart. This gives
\begin{equation}
\begin{aligned}
 \widehat{\mathbb E}_p^{v,x}[(X_k^{0,k})^2]
 &\leq C_p\sum_{j=0}^k
 e^{-t_0\beta j}e^{-2t_0\beta(k-j)}\\
 &\leq C_pe^{-t_0\beta k}.
\end{aligned}                                                \label{eq:5-11e5b}
\end{equation}
All geometric ratios depend only on \(G,p\). Thus the constants in
\eqref{eq:5-11e5a}--\eqref{eq:5-11e5b} are uniform in \(v,x,k,\ell\).

Continue to assume that \(p<p_c^{\rm s}\). We spell out the peak
consequences. A vertex \(z\) is the peak of a set
\(A\ni z\) when \(\log\Delta(z,u)<0\) for all \(u\in A\setminus\{z\}\).
Let \(\mathcal P_z\) be the event that \(z\) is the peak of its full open
cluster. Let \(H_k(v)\) be the graph on \(L_{-\infty,k}(v)\) obtained by
deleting every edge with both endpoints in \(L_k(v)\). Explore the open
cluster \(K_k(v)\) of \(v\) in \(H_k(v)\), and put
\(Z_k=K_k(v)\cap L_k(v)\). Thus a site adjacent through a deleted
top-layer edge, or lying strictly above the top layer, has not been
tested merely because \(K_k(v)\) was exposed. Moreover
\(Z_k\ne\varnothing\) contains the first-hit event used in the upward
half-space crossing estimate. Moreover,
\(|Z_k|\leq X_k^{-\infty,k}(v)\). Since \(h_p\) is bounded, the upper
inequality in \eqref{eq:5-11e1} and the pointwise lower bound
\eqref{eq:5-11e-prime-prime-prime} give, uniformly in \(v,x,k\),
\[
 \widehat{\mathbb E}_p^{v,x}|Z_k|
 \leq C_pe^{-t_0\beta_p^{\rm s}k},
 \qquad
 \widehat{\mathbb P}_p^{v,x}(Z_k\ne\varnothing)
 \geq c_pe^{-t_0\beta_p^{\rm s}k}.
\]
Dividing the first estimate by the second proves
\begin{equation}
 \widehat{\mathbb E}_p^{v,x}(|Z_k|\mid Z_k\ne\varnothing)\leq C_p.
                                                                    \label{eq:5-11e6}
\end{equation}
On \(Z_k\ne\varnothing\), select \(u\in Z_k\) measurably from the
exploration and let \(u^+\) be a maximal-increment neighbor from \eqref{eq:5-1b}.
Then \(u^+\in L_{k+1}(v)\). More explicitly, writing
\(q=\log\Delta(v,u)/t_0\), the inequalities
\(k+R_v-1<q\leq k+R_v\) imply
\(k+R_v<q+1\leq k+R_v+1\), which is precisely membership in
\(L_{k+1}(v)\). Thus \(u^+\notin L_{-\infty,k}(v)\), and the exploration
of \(K_k(v)\) has not queried \(U_{u^+}\).

Force \(u^+\) open and close every neighbor of
\(Z_k\cup\{u^+\}\) outside \(K_k(v)\cup\{u^+\}\). A site in this closing
set that was queried during the exploration is already closed: every
queried open site joined through a retained edge belongs to \(K_k(v)\),
while the deleted top-layer edges and every site above layer \(k\) were
not queried through the exploration frontier. Consequently the only
open coordinate prescribed by the surgery is the genuinely untested
site \(u^+\), and every newly prescribed closed coordinate is also
untested. There are at most \(d(|Z_k|+1)\) of the latter. Conditional on
the exploration, \eqref{eq:5-2a} therefore gives probability at least
\[
 p(1-p)^{d(|Z_k|+1)}.
\]
Every edge from \(K_k(v)\) to an untested site is incident to \(Z_k\):
an edge cannot increase raw height by more than \(t_0\), and all retained
frontier sites were tested. Hence on this surgery event the full cluster
is exactly \(K_k(v)\cup\{u^+\}\). Since every vertex of \(K_k(v)\) is in
\(L_{-\infty,k}(v)\) and the strict lower endpoint of \(L_k(v)\) implies
\(\log\Delta(v,u^+)>t_0(k+R_v)\), the vertex \(u^+\) is its unique peak.
We now record the probability calculation that turns this surgery into a
uniform lower bound. Put \(a=1-p\) and
\(H_k=\{Z_k\ne\varnothing\}\). Since
\(s\mapsto a^{ds}=\exp(ds\log a)\) is convex and decreasing,
conditional Jensen and \eqref{eq:5-11e6} give
\begin{equation}
 \begin{split}
 &\widehat{\mathbb E}_p^{v,x}
   [p a^{d(|Z_k|+1)}\mid H_k]\\
 &\qquad=pa^d\widehat{\mathbb E}_p^{v,x}
       [a^{d|Z_k|}\mid H_k]
 \geq pa^d a^{d\widehat{\mathbb E}_p^{v,x}
                    (|Z_k|\mid H_k)}
 \geq pa^{d(C_p+1)}=:c_p'>0.
 \end{split}                                                  \label{eq:5-11e6a}
\end{equation}
The event \(H_k\) is unchanged by deleting edges internal to the
first-hit layer and hence is precisely the unrestricted upward hitting
event. The unrestricted version of \eqref{eq:5-11e-prime-prime-prime}
therefore gives
\[
 \widehat{\mathbb P}_p^{v,x}(H_k)
 \asymp_p e^{-t_0\beta_p^{\rm s}k}.
\]
Integrating the conditional surgery estimate gives the same expression
as a lower bound for the peak event in the next layer. Conversely, if
the peak of \(K_v\) lies in \(L_{k+1}(v)\), then \(v\) is connected to
that layer; the unrestricted upward hitting estimate, with \(k+1\) in
place of \(k\), gives the matching upper bound. Absorbing the fixed
one-layer factor into the constants proves
\begin{equation}
 \widehat{\mathbb P}_p^{v,x}
   \left(\operatorname{peak}(K_v)\in L_{k+1}(v)\right)
 \asymp_p e^{-t_0\beta_p^{\rm s}k}.                           \label{eq:5-11e7}
\end{equation}

\begin{auxiliary}[Peak-preserving offset surgery]
\label{aux:peak-preserving-offset-surgery}
For
\(0<p\leq p_c^{\rm s}\), set
\[
 F_j(v,x)=\widehat{\mathbb E}_p^{v,x}
       [X_j^{-\infty,\infty}(v);\mathcal P_v].
\]
There is \(C=C(d,p)<\infty\) such that, for every \(j\leq-1\),
\begin{equation}
 F_{j-1}(v,x)\leq C F_j(v,x),
 \qquad F_{j+1}(v,x)\leq C F_j(v,x).                         \label{eq:5-11e7a}
\end{equation}
\end{auxiliary}
\begin{proof}
Equip each configuration on the left with a
marked vertex
\(z\in K_v\cap L_{j-1}(v)\), respectively
\(z\in K_v\cap L_{j+1}(v)\), on the event \(\mathcal P_v\).
Let \(w=z^+\in L_j(v)\), respectively \(w=z^-\in L_j(v)\), where
\(z^+\) and \(z^-\) are maximal-increment and maximal-decrement
neighbors. If \(w\in K_v\), leave the configuration unchanged. Otherwise
\(w\) is closed, since it is adjacent to \(z\in K_v\). Open \(w\) and
close every vertex in
\[
 N(w)\setminus(K_v\cup\{w\}).
\]
The resulting root cluster is exactly \(K_v\cup\{w\}\). Since
\(j\leq-1\) and \(x<1\), every \(w\in L_j^x(v)\) satisfies
\[
 \log\Delta(v,w)\leq t_0(j+x)<0.
\]
Thus \(v\) remains the unique peak of its full cluster, and \(w\) is
counted by \(X_j^{-\infty,\infty}(v)\).

This defines a map from the marked configurations counted by the
left-hand side of either inequality in \eqref{eq:5-11e7a} to the marked
configurations \((\omega',w)\) counted by \(F_j(v,x)\). It changes at
most \(d+1\) non-root coordinates. For a fixed image marked at \(w\),
there are at most \(d\) choices of its marked neighbor \(z\) and at most
\(2^{d+1}\) choices for the original states in
\(N(w)\cup\{w\}\). Under the conditional product measure given that
\(v\) is open, comparison of Bernoulli weights therefore bounds the
measure of all marked preimages by
\[
 d\,2^{d+1}[p(1-p)]^{-(d+1)}
\]
times the measure of marked images. Summing over the marked vertex proves
\eqref{eq:5-11e7a}. Full clusters are finite in this parameter range: this follows
from the definition of \(p_c^{\rm s}\) below criticality and from the
Auxiliary Lemma~\ref{aux:critical-clusters} at equality. Hence the map is
defined almost
surely.
\end{proof}

For arbitrary offsets \(x,y\in[0,1)\), the half-open layer definition
gives
\[
 L_j^x(v)\subseteq L_{j-1,j+1}^y(v).
\]
The event \(\mathcal P_v\) depends only on the modular cocycle and the
site configuration, not on the auxiliary offset. Therefore \eqref{eq:5-11e7a}
and the preceding inclusion give, with \(C_0\) denoting the constant in
\eqref{eq:5-11e7a},
\[
 \begin{split}
 F_j(v,x)
 &\leq F_{j-1}(v,y)+F_j(v,y)+F_{j+1}(v,y)\\
 &\leq(2C_0+1)F_j(v,y).
 \end{split}
\]
Interchanging \(x\) and \(y\) gives the reverse inequality. Thus, after
writing \(C=2C_0+1\),
\begin{equation}
 C^{-1}F_j(v,x)\leq F_j(v,y)\leq C F_j(v,x)
 \qquad(j\leq-1).                                            \label{eq:5-11e7b}
\end{equation}
Integrating the two inequalities in \(y\) displays the comparison with
the offset average:
\[
 C^{-1}F_j(v,x)
 \leq\int_0^1F_j(v,y)\,dy
 \leq C F_j(v,x).
\]
This two-sided comparison
comes from the peak-preserving finite-energy map
\eqref{eq:5-11e7a}, not from the one-sided slab comparison
\eqref{eq:5-11h0}; in particular the peak event is retained throughout.
For \(j=0\), the same conclusion follows from
\[
 (1-p)^d\leq F_0(v,x)\leq\widehat\chi_{p,0},
\]
where the lower bound is the isolated open-root event and the upper bound
uses \(X_0\leq|K_v|\). This case will only be used with
\(p<p_c^{\rm s}\), when site sharpness
\cite[Theorem~2 and Section~6]{AntunovicVeselic2008} makes
\(\widehat\chi_{p,0}<\infty\).

For completeness, define the diagonally invariant transport
\[
 T_k(a,b)=\mathbf1\{\mathcal P_a,\ b\in K_a\cap L_{-(k+1)}(a)\}.
\]
The conditional tilted mass-transport identity \eqref{eq:5-11b} gives
\begin{equation}
 p\int_0^1F_{-(k+1)}(v,y)\,dy
 =\mathbb E_p\sum_a
   \mathbf1\{a=\operatorname{peak}(K_v),\
             a\in L_{k+1}(v)\}\Delta(v,a).                  \label{eq:5-11e7c}
\end{equation}
The expectation on the right-hand side is over both the site
configuration and the independent uniform layer offset.
Indeed, the exact shift rule for the equivariant layers makes
\(v\in L_{-(k+1)}(a)\) equivalent to \(a\in L_{k+1}(v)\).
For \(a\in L_{k+1}^x(v)\),
\[
 e^{t_0(k+x)}<\Delta(v,a)\leq e^{t_0(k+1+x)},
\]
so this modular weight is comparable to \(e^{t_0k}\), uniformly in the
offset. Together with \eqref{eq:5-11e7}, \eqref{eq:5-11e7c} gives
\[
 \int_0^1F_{-(k+1)}(v,y)\,dy
 \asymp_p e^{-t_0(\beta_p^{\rm s}-1)k}.
\]
For \(k\geq1\), two applications of \eqref{eq:5-11e7a} compare layers
\(-(k+1)\) and \(-k\), and \eqref{eq:5-11e7b} then replaces the offset average by
any fixed \(x\). The case \(k=0\) follows from the displayed bounds for
\(F_0\). Consequently
\begin{equation}
 \widehat{\mathbb E}_p^{v,x}
     [X_{-k}^{-\infty,\infty};\mathcal P_v]
 \asymp_p e^{-t_0(\beta_p^{\rm s}-1)k}.                      \label{eq:5-11e8}
\end{equation}
Every comparison constant above is independent of \(v,x,k\), so the
constants in \eqref{eq:5-11e8} are uniform in these variables.
For \(p<p_c^{\rm s}\), put
\(Y_k=X_{-k}^{-\infty,\infty}(v)\mathbf1_{\mathcal P_v}\). Since
\(Y_k>0\) precisely when \(v\leftrightarrow L_{-k}(v)\) and
\(\mathcal P_v\) occur, \eqref{eq:5-11e8} and Cauchy--Schwarz give
\begin{equation}
 \widehat{\mathbb P}_p^{v,x}(Y_k>0)
 \geq \frac{\widehat{\mathbb E}_p^{v,x}[Y_k]^2}
              {\widehat{\mathbb E}_p^{v,x}[Y_k^2]}
 \geq c_p e^{-t_0(\beta_p^{\rm s}-1)k},                    \label{eq:5-11e8a}
\end{equation}
where \eqref{eq:5-11e5a} with \(\ell=k\) bounds the denominator after the peak
indicator is deleted. Markov's inequality and \eqref{eq:5-11e8} give the reverse
probability bound. Moreover, on \(\{Y_k>0\}\) one has
\(X_{-k}^{-\infty,\infty}\geq1\). Hence, for \(\ell\geq0\),
\begin{equation}
 \begin{split}
 &\widehat{\mathbb E}_p^{v,x}
   [X_{-\ell}^{-\infty,\infty};Y_k>0]\\
 &\qquad\leq
 \widehat{\mathbb E}_p^{v,x}
   [X_{-\ell}^{-\infty,\infty}X_{-k}^{-\infty,\infty}]
 \leq C_p e^{-t_0(\beta_p^{\rm s}-1)(k\vee\ell)}.
 \end{split}
                                                                  \label{eq:5-11e8b}
\end{equation}
Dividing \eqref{eq:5-11e8b} by the lower bound \eqref{eq:5-11e8a} proves the uniform
subcritical peak-survival estimates
\begin{equation}
 \begin{aligned}
 \widehat{\mathbb P}_p^{v,x}
   (v\leftrightarrow L_{-k}(v),\mathcal P_v)
 &\asymp_p e^{-t_0(\beta_p^{\rm s}-1)k},\\
 \widehat{\mathbb E}_p^{v,x}
 [X_{-\ell}^{-\infty,\infty}\mid
      v\leftrightarrow L_{-k},\mathcal P_v]
 &\leq C_p
 \begin{cases}
  1,&\ell\leq k,\\
  e^{-t_0(\beta_p^{\rm s}-1)(\ell-k)},&\ell>k.
 \end{cases}
\end{aligned} \label{eq:5-11e9}
\end{equation}
No bond coordinate occurs in this derivation.

We establish the lower critical exponent now, before using any critical
peak or fractional estimate. For \(p<p_c^{\rm s}\), site sharpness
\cite[Theorem~2 and Section~6]{AntunovicVeselic2008} gives
\(\widehat\chi_{p,0}<\infty\). The tilted MTP gives
\(\widehat\chi_{p,1}=\widehat\chi_{p,0}\), and log-convexity in the tilt
gives \(\widehat\chi_{p,1/2}<\infty\). Hence \(p<p_t^{\rm s}\), and
\eqref{eq:5-11e-prime-prime-prime}--\eqref{eq:5-11e-prime-prime-prime-prime} imply
\(\alpha_p^{\rm s}=\beta_p^{\rm s}>1\). Left continuity \eqref{eq:5-0a} now gives
\begin{equation}
 \alpha_{p_c^{\rm s}}^{\rm s}\geq1.                         \label{eq:5-11e10}
\end{equation}

\textbf{Conditional Reimer and bounded surgery.} Fix a finite outcome
\((A,D)\) of the cluster kernel and let \(J=V\setminus(A\cup D)\). For
cylinder events \(E_{A,D}\) and \(B_{A,D}\) in the coordinates of \(J\),
write \(E_{A,D}\mathbin{\square}B_{A,D}\) for disjoint occurrence. If
\(E_{A,D}\cap B_{A,D}\subseteq
E_{A,D}\mathbin{\square}B_{A,D}\), Lemma~\ref{lem:finite-bkr}
on the conditional product space \eqref{eq:5-2a} gives
\begin{equation}
 \mathbb P_p(E_{A,D}\cap B_{A,D}\mid A,D,R_v=x)
 \leq\mathbb P_p(E_{A,D}\mid A,D,R_v=x)
       \mathbb P_p(B_{A,D}\mid A,D,R_v=x).                    \label{eq:5-11f}
\end{equation}
The infinite events used below will be exhausted by explicit cylinder
events; no continuity property of the operation \(\square\) will be needed.

We shall also use the following finite-coordinate map on that conditional
product space. Let \(\gamma=(q_0,\ldots,q_s)\) be a deterministic path with
\(q_0\in A\) and \(q_1,\ldots,q_s\notin A\cup D\). On an event \(E\), let
\(C\) be the open cluster of \(q_s\) in a specified deterministic domain
after deleting \(A\cup D\). Define \(\Phi_\gamma\) by opening
\(q_1,\ldots,q_s\) and closing every site adjacent to
\(\gamma\setminus\{q_0\}\) that is outside
\(A\cup D\cup C\cup\gamma\). It leaves every other coordinate unchanged.
The image off-\(A\) cluster of \(q_s\) is contained in
\(C\cup(\gamma\setminus A)\), so a nonconnection event for \(C\) is
preserved whenever \(\gamma\setminus A\) avoids its forbidden set. The
map changes only coordinates in
\[
 S_\gamma=\{z\in J:d(z,\gamma\setminus\{q_0\})\leq1\},
 \qquad |S_\gamma|\leq M:=(d+1)s.
\]
Condition on all coordinates in \(J\setminus S_\gamma\). For every fixed
outside configuration, \(\Phi_\gamma\) is a map of the finite space
\(\{0,1\}^{S_\gamma}\), every image has at most \(2^M\) preimages, and the
ratio of the Bernoulli weights of any two configurations in this space is
at most \([p(1-p)]^{-M}\). Summing over the finite fibres and then
integrating the outside configuration gives directly on the countable
product space
\begin{equation}
 \mathbb P_p(E\mid A,D,R_v=x)
 \leq 2^M[p(1-p)]^{-M}
       \mathbb P_p(\Phi_\gamma(E)\mid A,D,R_v=x).             \label{eq:5-11f-prime}
\end{equation}
Thus \eqref{eq:5-11f-prime} requires no finite-volume limit, even when \(E\) contains
a nonconnection event: membership of a site in \(C\) is measurable as a
countable union of finite path events, while the map changes only the fixed
finite set \(S_\gamma\). In our application the two non-root connector
sites lie strictly below the induced graph used for the exploration, so
every coordinate actually changed by \(\Phi_\gamma\) is untested. Sites
in \(D\) are excluded from the map and left unchanged. No tested closed
site is reopened.

We record the resulting comparison as an actual site statement. Let
\(\mathcal P_v(s)\) be the event that \(v\) is the peak of the set joined
to it inside \(L_{-s,\infty}(v)\), and let
\[
 \mathcal Q_{v;k,\ell,r}=
 \{v\not\leftrightarrow L_{\ell-1}(v)
       \text{ in }L_{-k-r-2,\infty}(v)\}.
\]
There is \(C_p<\infty\), depending only on \(G,p\), such that for every
\(0<p\leq p_c^{\rm s}\), \(k,r\geq0\), and \(\ell\geq2\),
\begin{equation}
 \begin{split}
 &\sup_{v,x}\widehat{\mathbb E}_p^{v,x}
   [X_{-k}^{-k-r,0}(v);\mathcal Q_{v;k,\ell,r}]\\
 &\quad\leq C_p\inf_{v,x}
 \frac{\displaystyle\sum_{i=0}^{3}
   \widehat{\mathbb E}_p^{v,x}
   [X_{-k-\ell-i}^{-k-\ell-r-3,0}(v)
      \mid\mathcal P_v(k+\ell+r+3)]}
 {\displaystyle\widehat{\mathbb P}_p^{v,x}
   (v\leftrightarrow L_{-\ell}(v)\mid
      \mathcal P_v(k+\ell+r+3))}.
 \end{split} \label{eq:5-11f-prime-prime}
\end{equation}
The denominator is nonzero. Indeed, from the conditioned-open root,
force the \(\ell\)-step maximal-decrement path until it reaches
\(L_{-\ell}\)
and close every exterior neighbor of that finite path inside
\(L_{-k-\ell-r-3,\infty}\). This finite-coordinate event has positive
probability, the root reaches \(L_{-\ell}\), and the root is the peak of
the resulting restricted cluster. More precisely, the path has
\(\ell\) non-root sites and its external boundary in the displayed slab
has at most \(d(\ell+1)\) sites. After omitting any boundary site that
also lies on the path, the conditional product probability of the
prescription is
\[
 p^\ell(1-p)^{N_\partial}>0,
 \qquad N_\partial\leq d(\ell+1).
\]
Every non-root path vertex has strictly negative modular height relative
to \(v\), and the closed external boundary prevents the restricted
cluster from containing any other site. Hence the prescription is
contained in the intersection of the numerator event and the
conditioning peak event. The peak event itself has positive probability
because it contains the isolated open-root event. Division by that
positive probability proves that the displayed conditional denominator
is strictly positive.

Here is the derivation. Fix \(v,x\), set
\(H^+(v)=\{z:\log\Delta(v,z)\geq0\}\setminus\{v\}\), and expose the
open cluster \(A\) of \(v\) in the induced graph on
\(L_{-\ell,\infty}(v)\), together with its tested external boundary
\(D\) in that induced graph. Retain only outcomes on which \(v\) is the
peak of \(A\) and \(A\cap L_{-\ell}\ne\varnothing\), and select
\(u\) in this intersection. At \(p\leq p_c^{\rm s}\), \(A\) and every
residual cluster used below are finite almost surely by the
Auxiliary Lemma~\ref{aux:critical-clusters}.

Let \(q_1=u^-\) and \(q=q_1^-\) be two successive
maximal-decrement neighbors of \(u\). The reverse of \eqref{eq:5-1b} and the exact
shift rule give
\begin{equation}
 q_1\in L_{-\ell-1}(v),\quad q\in L_{-\ell-2}(v),\quad
 L_j(q)=L_{j-\ell-2}(v).                                     \label{eq:5-11f0}
\end{equation}
Since the cluster was explored in the induced graph
\(L_{-\ell,\infty}(v)\), neither \(q_1\) nor \(q\) was tested. Put
\[
 D_*=L_{-k-\ell-r-3,\infty}(v)\setminus(A\cup D)
\]
and, for \(w\in L_{-k-1,-k+1}(q)\), let \(E_{q,w}^{A,D}\) be the event
that \(q\) is open, that it is joined to \(w\) in
\(L_{-k-\ell-r-3,-\ell-1}(v)\setminus(A\cup D)\), and that its open cluster in
\(D_*\) does not meet \(H^+(v)\). Let \(B^{A,D}\) be the event that an
open path leaves \(A\) through an untested site below \(L_{-\ell}\),
uses no site of \(D\), and returns to \(H^+(v)\) inside \(D_*\). Thus
\(B^{A,D}\) is exactly the competing off-\(A\) continuation that
destroys the restricted peak when the lower coordinates are restored.
We now justify the infinite-volume Reimer step. Put
\[
 S_*=
 L_{-k-\ell-r-3,-\ell-1}(v)\setminus(A\cup D)
\]
and let \(C_q\) be the open cluster of \(q\) in \(D_*\), with
\(C_q=\varnothing\) when \(q\) is closed. Take the exhaustion
\(\Lambda_N=B_N(v)\cap J\), and, once \(q,w\in\Lambda_N\), define
\begin{equation}
\begin{aligned}
 \mathcal E_N={}&\{q\text{ open},\
        q\leftrightarrow w\text{ in }S_*\cap\Lambda_N,\
        C_q\subseteq\Lambda_N,\
        C_q\cap H^+(v)=\varnothing\},\\
 \mathcal B_N={}&\{\text{a path defining }B^{A,D}\text{ exists with every
              vertex outside }A\text{ in }\Lambda_N\}.
\end{aligned}                                                   \label{eq:5-11f2}
\end{equation}
To match Auxiliary Lemma~\ref{aux:exploration-disjointness} literally,
let \(T_{A,D}\subseteq D_*\) be the residual sites below
\(L_{-\ell}(v)\) that are adjacent to \(A\).  Any path defining
\(B^{A,D}\), truncated immediately after its last visit to \(A\), starts
at a site of \(T_{A,D}\) and stays in \(D_*\). Put
\(F_*=H^+(v)\cap D_*\). Conversely, an open \(T_{A,D}\)-to-\(F_*\)
path in \(D_*\), preceded by an open path in \(A\) to its initial
neighbor, defines \(B^{A,D}\). Once \((A,D)\) is fixed and its already tested
coordinates are omitted, \(\mathcal B_N\) is precisely the event
\[
 T_{A,D}\leftrightarrow F_*
       \quad\hbox{in }D_*\cap\Lambda_N,
\]
whereas \(\mathcal E_N\) is the lemma's event with
\(D_0=D_*\), \(S=S_*\), \(T=T_{A,D}\), \(F=F_*\), \(u=q\), and the
present \(w\). Thus the terminal sets, permitted domain, and forbidden
connection are all fixed by the conditioned exploration outcome.
Both are cylinder events. For \(\mathcal B_N\) this follows from the
finiteness of \(\Lambda_N\). For \(\mathcal E_N\), explore the cluster of
\(q\) in \(D_*\cap\Lambda_N\); the condition
\(C_q\subseteq\Lambda_N\) is certified by the states in this finite set
and its finite \(D_*\)-external boundary.

On \(\mathcal E_N\), let
\[
 \partial_{D_*}C_q
   =\{z\in D_*\setminus C_q:z\sim y\text{ for some }y\in C_q\}.
\]
The finite set \(C_q\cup\partial_{D_*}C_q\), with the sites of \(C_q\)
prescribed open and those of \(\partial_{D_*}C_q\) prescribed closed, is
an exact-cluster certificate for the nonconnection in \(\mathcal E_N\);
include also a selected \(q\)-to-\(w\) path in
\(S_*\cap\Lambda_N\). A selected open path
witnessing \(\mathcal B_N\) cannot meet \(C_q\), since its remaining
segment to \(H^+(v)\) would join \(q\) to \(H^+(v)\) in \(D_*\). It cannot meet
\(\partial_{D_*}C_q\), since every site in that boundary is closed.
After omitting the already conditioned coordinates in \(A\cup D\), these
are disjoint witness sets. Hence
\[
 \mathcal E_N\cap\mathcal B_N
 \subseteq\mathcal E_N\mathbin{\square}\mathcal B_N.
\]

For \(p<p_c^{\rm s}\), every open cluster is finite almost surely by the
definition of \(p_c^{\rm s}\); at \(p=p_c^{\rm s}\), the same holds by the
Auxiliary Lemma~\ref{aux:critical-clusters}. A finite-coordinate
conditioning of positive
probability preserves this almost-sure statement, and \(C_q\) is a
subcluster of the full cluster. Consequently, conditional on every
admissible \((A,D)\) and \(R_v=x\),
\[
 \mathcal E_N\uparrow E_{q,w}^{A,D},\qquad
 \mathcal B_N\uparrow B^{A,D}
 \quad\text{modulo a null set}.
\]
Apply \eqref{eq:5-11f} to \(\mathcal E_N,\mathcal B_N\) and use continuity from
below separately for these two nested sequences and for their
intersections. This gives
\begin{equation}
 \widehat{\mathbb P}_p^{v,x}(E_{q,w}^{A,D}\cap B^{A,D}\mid A,D)
 \leq
 \widehat{\mathbb P}_p^{v,x}(E_{q,w}^{A,D}\mid A,D)
 \widehat{\mathbb P}_p^{v,x}(B^{A,D}\mid A,D).               \label{eq:5-11f3}
\end{equation}
Thus \eqref{eq:5-11f3} uses only finite Reimer inequalities and ordinary
continuity from below, not continuity of disjoint occurrence and not an
independence assertion.

Write \(R=E_{q,w}^{A,D}\), \(B=B^{A,D}\), and take every probability in
the conditional product space left after \((A,D)\) has been exposed. If
\(0<\mathbb P(B)<1\), \eqref{eq:5-11f3} gives
\begin{equation}
 \mathbb P(R\mid B^c)-\mathbb P(R)
 =\frac{\mathbb P(B)}{1-\mathbb P(B)}
       [\mathbb P(R)-\mathbb P(R\mid B)]\geq0.              \label{eq:5-11f4}
\end{equation}
The same conclusion is immediate when \(\mathbb P(B)=0\). If
\(\mathbb P(B)=1\), the cluster outcome has zero probability together
with the restricted-peak event and contributes nothing to the later
conditional average. On the exact-cluster outcome, \(B^c\) is precisely
the assertion that \(v\) remains the peak when the lower coordinates are
restored. Apply the map \eqref{eq:5-11f-prime} with \(\gamma=(u,q_1,q)\) to
\(R\cap B^c\). It preserves both \(B^c\) and the off-\(A\) nonconnection in
\(E_{q,w}^{A,D}\): opening the connector only attaches \(A\) to the
off-\(A\) cluster \(C\), which avoids \(H^+(v)\), while every other new
exit from the connector is closed by the map. Its cost is a constant
depending only on \(d,p\).
On the image of \(R\cap B^c\), the open path in \(A\), the connector
\(\gamma\), and the path from \(q\) to \(w\) concatenate. By \eqref{eq:5-11f0}, a
site in \(L_{-k+j}(q)\), \(-1\leq j\leq1\), belongs respectively to
\(L_{-k-\ell-i}(v)\) with \(i=3,2,1\). All these concatenated paths lie
in \(L_{-k-\ell-r-3,0}(v)\). Summing \eqref{eq:5-11f4} over \(w\), and then using
the finite-map bound \eqref{eq:5-11f-prime}, therefore gives, for every admissible
\((A,D)\), the following. Write
\(\widehat{\mathbb E}_{p,J}^{q,R_q;A,D}\) for expectation under the
residual product law on \(J\), with \(A,D\) held fixed and \(q\) conditioned
to be open. Then
\begin{equation}
 \begin{split}
 &\sum_{i=0}^{3}\widehat{\mathbb E}_p^{v,x}
 [X_{-k-\ell-i}^{-k-\ell-r-3,0}(v)
       \mid A,D,\mathcal P_v(k+\ell+r+3)]\\
 &\qquad\geq c_{d,p}\,\widehat{\mathbb E}_{p,J}^{q,R_q;A,D}
       [Z_q;\mathcal E_q],
 \end{split}                                                   \label{eq:5-11f5}
\end{equation}
where
\[
 Z_q=\sum_{j=-1}^{1}X_{-k+j}^{-k-r-1,1}(q),
 \qquad
 \mathcal E_q=\{q\not\leftrightarrow H^+(v)
             \text{ in }D_*\}.
\]
The constant \(c_{d,p}>0\) is the reciprocal of the constant in
\eqref{eq:5-11f-prime} (including the probability that the open-root site \(q\) is
open before the map is applied). Thus
\eqref{eq:5-11f5} is a statement on the conditional product space \eqref{eq:5-2a}; its
domain, changed coordinates, event preservation, and multiplicity have
all been specified.

It remains only to compare offsets, and this step is deterministic. Fix
an arbitrary \(y\in[0,1)\), use the same site configuration, and write
\(L_j^y(q)\) and \(X_j^{m,n}(q;y)\) for the layers and counts obtained by
setting \(R_q=y\). Directly from
the half-open definition,
\begin{equation}
 \begin{split}
 L_{-k}^y(q)&\subseteq L_{-k-1,-k+1}^{R_q}(q),\\
 L_{-k-r,0}^y(q)&\subseteq L_{-k-r-1,1}^{R_q}(q),\\
 L_{-k-r-1,\infty}^{R_q}(q)&\subseteq
                 L_{-k-r-2,\infty}^y(q).
 \end{split}                                                   \label{eq:5-11f6}
\end{equation}
By \eqref{eq:5-11f0}, the left side of the last inclusion is the layer half-space
underlying \(D_*\) before deletion of \(A\cup D\). Moreover, every path from \(q\) to
\(H^+(v)\) meets \(L_{\ell-1}^y(q)\): its terminal modular height
relative to \(q\) is greater than \(t_0\ell\), while an edge changes
that height by at most \(t_0\).
Consequently
\[
 \{\mathcal Q_{q;k,\ell,r}\text{ for offset }y\}
       \subseteq\mathcal E_q,
\]
because deleting \(A\cup D\) can only remove possible paths. The first
two inclusions in \eqref{eq:5-11f6}, together with the fact that their enlarged
slab lies strictly below \(L_{-\ell,\infty}(v)\), give, pointwise,
\begin{equation}
 X_{-k}^{-k-r,0}(q;y)\mathbf1_{\mathcal Q_{q;k,\ell,r}^{\,y}}
       \leq Z_q\mathbf1_{\mathcal E_q}.                      \label{eq:5-11f7}
\end{equation}
We now remove the residual conditioning without appealing to
transitivity.  By \eqref{eq:5-11f0}, the upper layer of the slab defining \(Z_q\)
is \(L_1(q)=L_{-\ell-1}(v)\), strictly below the explored region
\(L_{-\ell,\infty}(v)\). Hence every coordinate inspected by \(Z_q\) lies
in \(J\). The event \(\mathcal E_q\) is defined using the subgraph
\(D_*\subseteq J\), so
\(Z_q\mathbf1_{\mathcal E_q}\) is also measurable with respect to the
residual coordinates alone.

Extend the residual product measure conditioned on \(q\) being open by
assigning fresh, independent Bernoulli variables with parameter \(p\) to
every site of \(A\cup D\).  This extension is the ordinary product
measure conditioned only on \(q\) being open; the fresh variables are
introduced solely for this comparison and are not asserted to reproduce
the explored cluster outcome. Since the right side of
\eqref{eq:5-11f7} uses only coordinates in \(J\), extending the measure does not
change its expectation. On every extended configuration, the pointwise
inclusion \(\mathcal Q_{q;k,\ell,r}^{\,y}\subseteq\mathcal E_q\) and the
two count inclusions in \eqref{eq:5-11f6} still give \eqref{eq:5-11f7}. Consequently, for
every \(y\in[0,1)\),
\[
 \widehat{\mathbb E}_{p,J}^{q,R_q;A,D}
       [Z_q;\mathcal E_q]
 \geq
 \widehat{\mathbb E}_p^{q,y}
       [X_{-k}^{-k-r,0}(q);\mathcal Q_{q;k,\ell,r}].
\]
Only now do we use transitivity: the supremum over \(y\) of the last
ordinary open-root expectation is independent of \(q\), and hence equals
the supremum over both roots and offsets on the left side of \eqref{eq:5-11f-prime-prime}.
Thus \eqref{eq:5-11f5} is bounded below, uniformly in the exposed pair \((A,D)\),
by \(c_{d,p}\) times that supremum. Finally average \eqref{eq:5-11f5} over precisely
the cluster outcomes reaching \(L_{-\ell}\) under
\(\mathcal P_v(k+\ell+r+3)\), and divide by their conditional
probability. Since this holds for every \(v,x\), take the infimum on the
right. This proves \eqref{eq:5-11f-prime-prime} with \(C_p=c_{d,p}^{-1}\). Only values
\(\ell\geq2\) are used below; enlarging the auxiliary \(\ell\) to two
absorbs the finitely many smaller cases in estimates not involving
\eqref{eq:5-11f-prime-prime}.

\textbf{Critical peak survival.} We give the parameter-uniform argument in
full. First,
\begin{equation}
 q_R:=\sup_{v\in V,\,x\in[0,1)}
 \widehat{\mathbb P}_{p_c^{\rm s}}^{v,x}
   (v\leftrightarrow L_{R,\infty}(v))\longrightarrow0.         \label{eq:5-11g}
\end{equation}
Indeed, after enlarging the target by one layer, the event is contained
in the raw-height event that the open cluster reaches height
\(t_0(R-1)\). These events decrease to the event that the cluster has
unbounded height, which has probability zero by the critical-cluster
lemma. Conditioning the root open costs only \((p_c^{\rm s})^{-1}\), and
transitivity plus the one-layer enlargement makes the bound uniform in
the root and offset.

Fix \(\zeta>0\). By left continuity choose
\(p_-=p_c^{\rm s}-\delta\) so close to \(p_c^{\rm s}\) that
\begin{equation}
 \alpha_{p_-}^{\rm s}\leq\alpha_{p_c^{\rm s}}^{\rm s}+\zeta/4.
                                                                    \label{eq:5-11g1}
\end{equation}
Couple the two parameters by opening each \(p_-\)-closed site, independently,
with probability
\(\rho=(p_c^{\rm s}-p_-)/(1-p_-)\). Under the conditional-open-root
coupling, the root label is sampled in \([0,p_-]\), so it is open at both
levels and is not resampled. By \eqref{eq:5-11e9}, uniformly in \(v,x,k,\ell\),
\begin{equation}
 \begin{split}
 &\widehat{\mathbb P}_{p_-}^{v,x}
   (v\leftrightarrow L_{-k}(v),\mathcal P_v)
 \geq c_-e^{-t_0(\alpha_{p_-}^{\rm s}-1)k},\\
 &\widehat{\mathbb E}_{p_-}^{v,x}
 [X_{-\ell}^{-\infty,\infty}\mid
       v\leftrightarrow L_{-k},\mathcal P_v]
 \leq C_-
 \begin{cases}
  1,&\ell\leq k,\\
  e^{-t_0(\alpha_{p_-}^{\rm s}-1)(\ell-k)},&\ell>k.
 \end{cases}
 \end{split} \label{eq:5-11g2}
\end{equation}

Write
\[
 \mathcal A_k^-=
 \{v\leftrightarrow L_{-k}(v),\ \mathcal P_v\}
 \quad\text{in the \(p_-\)-configuration}.
\]
Condition on a finite \(p_-\)-cluster \(A\) for which \(\mathcal A_k^-\)
occurs. Assign each site of \(\partial_VA\) to its first
incident vertex in \(A\); at most \(d\) boundary sites are assigned to
one vertex. For vertices in layers \(L_{-R+1,0}(v)\), require all assigned
boundary sites to remain closed. By the two-level boundary kernel this
has conditional probability at least
\begin{equation}
 (1-\rho)^{d|A\cap L_{-R+1,0}|}.                              \label{eq:5-11g3}
\end{equation}
For a vertex in \(L_{-\ell}(v)\), \(\ell\geq R\), require instead that
none of its assigned boundary sites both opens by level \(p_c^{\rm s}\)
and has an off-\(A\) open continuation to the raw upper half-space.
To make this conditional event explicit, for a boundary site \(y\)
assigned to such a vertex define
\[
 D_y=\{U_y>p_c^{\rm s}\}
 \cup\{p_-<U_y\leq p_c^{\rm s},\ y\text{ has no off-}A
       \text{ critical continuation to the upper half-space}\}.
\]
Conditional on \(A\), the boundary label \(U_y\) is uniform on
\((p_-,1]\), independently of every other residual label. Conditional
further on \(p_-<U_y\leq p_c^{\rm s}\), the first-hit residual kernel
couples the off-\(A\) cluster of \(y\) below an independent critical
open-root cluster. Up to three layers of displacement (one for the
incident edge and two for the half-open layer endpoints),
\eqref{eq:5-11g} therefore gives
\[
 \widehat{\mathbb P}(D_y^c\mid A)
 \leq \rho q_{\ell-3}
 \leq q_{R-3}.
\]
We take \(R\geq4\).

The indicators of the near-top requirements
\(\{U_y>p_c^{\rm s}\}\) and the far requirements \(D_y\) are all
coordinatewise increasing functions of the residual uniform labels;
equivalently, these events are decreasing in the induced open-site
configuration. The conditional law is a product law: boundary labels
are independent uniforms on \((p_-,1]\), and all other residual labels
are independent uniforms on \([0,1]\). Harris--FKG therefore bounds the
probability of their intersection below by the product of their
individual probabilities, even when different \(D_y\)'s inspect common
exterior coordinates. At most \(d\) boundary sites are assigned to each
vertex of \(A\). Since \(1-\rho\leq1\) and \(1-q_{R-3}\leq1\), replacing
the actual numbers of assigned sites by these upper bounds can only
decrease the product. We obtain
\begin{equation}
 \widehat{\mathbb P}(\mathcal P_v\text{ survives at }p_c^{\rm s}\mid A)
 \geq
 (1-\rho)^{d|A\cap L_{-R+1,0}|}
 (1-q_{R-3})^{d|A\cap L_{-\infty,-R}|}.                      \label{eq:5-11g4}
\end{equation}
If the requirements hold, take any critical path from \(A\) to a vertex
of nonnegative raw height. The subcritical peak event ensures that the
path is not contained in \(A\). Let \(u\) be its last vertex in \(A\)
and \(y\) the following vertex. Then \(y\notin A\), \(y\) is an assigned
boundary site, and the entire suffix beginning at \(y\) is disjoint from
\(A\). If \(u\) is in one of the top \(R\) layers,
\eqref{eq:5-11g3} says that \(y\) is closed. Otherwise this suffix is
exactly the off-\(A\) continuation forbidden in the far-layer
requirement. Both alternatives are contradictions, so \(v\) remains
the peak. This also handles a boundary site adjacent to several
vertices: it is assigned only once, and whichever requirement is imposed
either closes that boundary site or forbids the displayed off-\(A\)
continuation.

Choose \(R=R(\zeta,p_-)\) so large that
\(-d\log(1-q_{R-3})\leq t_0\zeta/(8C_-)\). Conditional on the event in
\eqref{eq:5-11g2}, put
\[
 N_{\rm top}=|A\cap L_{-R+1,0}|,
 \qquad N_{\rm deep}=|A\cap L_{-\infty,-R}|.
\]
The second line of \eqref{eq:5-11g2}, summed over layers, gives
\begin{equation}
 \widehat{\mathbb E}_{p_-}^{v,x}[N_{\rm top}\mid\mathcal A_k^-]
     \leq C_-R,
 \qquad
 \widehat{\mathbb E}_{p_-}^{v,x}[N_{\rm deep}\mid\mathcal A_k^-]
 \leq C_-k+C_-\sum_{s\geq1}
 e^{-t_0(\alpha_{p_-}^{\rm s}-1)s}.                         \label{eq:5-11g4a}
\end{equation}
The series is finite because \(p_-<p_c^{\rm s}\) and hence
\(\alpha_{p_-}^{\rm s}>1\). Write the right side of \eqref{eq:5-11g4} as
\(\exp[-aN_{\rm top}-bN_{\rm deep}]\), where
\(a=-d\log(1-\rho)\) and \(b=-d\log(1-q_{R-3})\). Since
\(y\mapsto e^{-y}\) is convex, conditional Jensen and \eqref{eq:5-11g4a} give a
constant factor depending on \(R,p_-\) times
\(e^{-bC_-k}\). The choice of \(R\) makes this at least a constant times
\(e^{-t_0\zeta k/8}\), which is stronger than the following convenient
bound:
\begin{equation}
 \inf_{v,x}\widehat{\mathbb P}
 (\mathcal P_v\text{ survives at }p_c^{\rm s}\mid
   v\leftrightarrow L_{-k},\mathcal P_v\text{ at }p_-)
 \geq c_{\zeta}e^{-t_0\zeta k/4}.                            \label{eq:5-11g5}
\end{equation}
Under the monotone two-level coupling, the intersection of the
subcritical event in the first line of \eqref{eq:5-11g2} and the survival event in
\eqref{eq:5-11g5} is contained in the corresponding critical peak event. Thus the
two lower bounds multiply. Using \eqref{eq:5-11g1}, their product is at least
\[
 c_\zeta\exp\{-t_0[\alpha_{p_c^{\rm s}}^{\rm s}-1+\zeta/2]k\},
\]
after absorbing the fixed layer shifts into \(c_\zeta\). Since this is
stronger than the same bound with \(\zeta/2\) replaced by \(\zeta\), it
proves the quantified critical estimate
\begin{equation}
 \inf_{v,x}\widehat{\mathbb P}_{p_c^{\rm s}}^{v,x}
 (v\leftrightarrow L_{-k}(v),\mathcal P_v)
 \geq c_{\zeta}
 e^{-t_0(\alpha_{p_c^{\rm s}}^{\rm s}-1+\zeta)k}.            \label{eq:5-11g6}
\end{equation}
All constants in \eqref{eq:5-11g6} are chosen after \(\zeta\) and before
\(v,x,k\); this is the uniform critical peak-survival estimate needed in
\eqref{eq:5-11f-prime-prime}.

\begin{proposition}[Fractional layer recursions]
\label{prop:fractional-layer-recursions}
In the canonical nonunimodular setup above, let
\(0<p\leq p_c^{\rm s}\). For \(0\leq\delta\leq1\) and every slab
containing layers \(0\) and \(k\), the conditional-root tilted
mass-transport inequalities \eqref{eq:5-11ha} hold. For
\(0\leq\varepsilon,\delta\leq1\) and
\(-\infty\leq m\leq k\leq n\leq\infty\), the two up/down recursions
\eqref{eq:5-11h-prime} hold. If in addition
\(0\leq\delta\leq\varepsilon<1\) and the slab contains \(0\) and \(k\),
then the nonlinear H{\"o}lder estimate
\eqref{eq:5-11h-prime-prime} holds. All expectations may be
extended nonnegative expectations, and every fixed-offset supremum is
taken only after the offset-averaged mass transport.
\end{proposition}

\begin{proof}
The conditioning event
\(\mathcal P_v(s)\) has probability at least \((1-p)^d\) under the
open-root law, since it contains the event that every neighbor of the root
is closed. The additional survival probability in the denominator of
\eqref{eq:5-11f-prime-prime} is not treated as constant; it is bounded by \eqref{eq:5-11g6}.

We first prove the Holder--MTP input. Throughout
\eqref{eq:5-11h}--\eqref{eq:5-11h-prime-prime}, let
\(0<p\leq p_c^{\rm s}\). Every open cluster is then finite almost surely.
For \(p<p_c^{\rm s}\), the definition gives zero probability that any
fixed vertex lies in an infinite cluster, and countability of \(V\) permits
a union over all vertices; at \(p=p_c^{\rm s}\), this is the
Auxiliary Lemma~\ref{aux:critical-clusters}. All subsequent uses take
\(p=p_c^{\rm s}\). The
root is fixed and the independent offset
\(R_v\) is averaged; we suppress \(v\) from the expectation and count notation.
Fix \(0\leq\delta\leq1\) and
\(-\infty\leq m\leq\min\{0,k\}\leq\max\{0,k\}\leq n\leq\infty\);
in particular the slab contains both the root layer and the target
layer. Apply \eqref{eq:5-11b} to the transport \(F(u,z)\) defined by
\begin{equation}
 (X_{-k}^{m-k,n-k}(u))^{-\delta}
 \mathbf1\{z\in L_{-k}(u),\ u\leftrightarrow z
                \text{ in }L_{m-k,n-k}(u)\},                 \label{eq:5-11h}
\end{equation}
where the mass is defined as zero when the indicator vanishes. The
negative power is therefore evaluated only when the count is at least
one, and both transported endpoints are open.  The outgoing sum at
\(u\) is exactly
\[
 \sum_z F(u,z)
   =(X_{-k}^{m-k,n-k}(u))^{1-\delta}.
\]
We spell out the incoming sum.  If \(v\in L_{-k}(u)\), equivariance of
the random offset in \eqref{eq:5-1a} gives
\[
 u\in L_k(v),\qquad L_j(u)=L_{j+k}(v)\quad(j\in\mathbb Z).
\]
Consequently
\[
 L_{m-k,n-k}(u)=L_{m,n}(v),\qquad
 X_{-k}^{m-k,n-k}(u)=X_0^{m,n}(v).
\]
Moreover, \(u\in L_k(v)\) implies
\[
 e^{t_0(k-1)}<\Delta(v,u)\leq e^{t_0(k+1)}.
\]
Thus the incoming sum at \(v\), including the modular weight in
\eqref{eq:5-11b}, lies between
\[
 e^{t_0(k-1)}X_k^{m,n}(v)(X_0^{m,n}(v))^{-\delta}
 \quad\hbox{and}\quad
 e^{t_0(k+1)}X_k^{m,n}(v)(X_0^{m,n}(v))^{-\delta}.
\]
Here both expressions are declared zero when \(X_k^{m,n}(v)=0\); when
that count is positive, \(X_0^{m,n}(v)\geq1\). The cluster-finiteness
statement above makes every count appearing in the coefficient of
\eqref{eq:5-11h} finite almost surely. Consequently the outgoing sum is
indeed \((X_{-k}^{m-k,n-k}(u))^{1-\delta}\), including at \(\delta=1\)
under the convention \(X^0=\mathbf1_{\{X>0\}}\). To justify possibly
infinite expectations, multiply the transport by
\(\mathbf1\{d(u,z)\leq N\}\). Local finiteness makes both truncated sums
finite, so \eqref{eq:5-11b} applies. Since the untruncated coefficient is
now finite, letting \(N\uparrow\infty\) and using monotone convergence
gives the asserted outgoing and incoming sums with extended nonnegative
expectations. Averaging the independent offset and using the two
pointwise modular bounds gives the exact inequalities
\begin{equation}
 \begin{split}
 e^{t_0(k-1)}\widehat{\mathbb E}_p
 [X_k^{m,n}(X_0^{m,n})^{-\delta}]
 &\leq \widehat{\mathbb E}_p
 [(X_{-k}^{m-k,n-k})^{1-\delta}]\\
 &\leq e^{t_0(k+1)}\widehat{\mathbb E}_p
 [X_k^{m,n}(X_0^{m,n})^{-\delta}].
 \end{split}                                                    \label{eq:5-11ha}
\end{equation}
Suppose first that \(0<\varepsilon<1\) and
\(0\leq\delta\leq\varepsilon\).  Holder with conjugate exponents
\(1/(1-\varepsilon)\) and \(1/\varepsilon\) gives
\begin{equation}
 \begin{split}
 \widehat{\mathbb E}_p[(X_k^{m,n})^{1-\varepsilon}]
 &\leq\widehat{\mathbb E}_p
 [X_k^{m,n}(X_0^{m,n})^{-\delta}]^{1-\varepsilon}\\
 &\quad\times
 \widehat{\mathbb E}_p
 [(X_0^{m,n})^{(1-\varepsilon)\delta/\varepsilon}]^\varepsilon.
 \end{split} \label{eq:5-11hb}
\end{equation}
Because \((1-\varepsilon)\delta/\varepsilon\leq1-\varepsilon\),
Lyapunov's inequality gives the explicit bound
\begin{equation}
 \widehat{\mathbb E}_p
 [(X_0^{m,n})^{(1-\varepsilon)\delta/\varepsilon}]^\varepsilon
 \leq
 \widehat{\mathbb E}_p[(X_0^{m,n})^{1-\varepsilon}]^\delta.
                                                                  \label{eq:5-11hb0}
\end{equation}
Solving \eqref{eq:5-11ha} for
\(\widehat{\mathbb E}_p[X_k^{m,n}(X_0^{m,n})^{-\delta}]\),
raising the result to \(1-\varepsilon\), and substituting it and
\eqref{eq:5-11hb0} into \eqref{eq:5-11hb} proves \eqref{eq:5-11h-prime-prime} with the displayed factor
\(e^{-t_0(1-\varepsilon)k}\), with the bounded factor
\(e^{t_0(1-\varepsilon)}\) absorbed into \(C_p\).  When
\(\varepsilon=0\), necessarily \(\delta=0\), and
\eqref{eq:5-11h-prime-prime} is the lower inequality in
\eqref{eq:5-11ha} solved for the right-hand expectation.  When
\(\varepsilon=1\), the left side is a probability and the assertion
used in the iteration follows from \(\mathbf1\{X>0\}\leq1\).  Thus no
division by the endpoint value \(\varepsilon=0\) or \(1\) is used.

We next prove the fractional up/down inequalities. The proof must use
increments of the allowed slab: the number of first entrances into an
extreme layer need not equal the complete connected count in that layer,
because an open path may move laterally after its first entrance. The
argument below is the site-coordinate version of
\cite[Lemma~6.7]{Hutchcroft2020Nonunimodular} and never makes that
identification. It suffices to consider the nontrivial parameter range
\[
 m\leq\min\{0,k\}\leq\max\{0,k\}\leq n;
\]
outside this range the slab omits the root or target layer and the
left-hand side is zero.

First suppose that \(m,n\) are finite. Fix a root \(v\), an offset
\(R_v=x\), and an integer \(\ell\in[k\vee0,n]\). Define the upward
increment
\[
 Z_\ell^+:=X_k^{m,\ell}(v)-X_k^{m,\ell-1}(v),
\]
where the second term is defined to be zero when
\(\ell=k\vee0\), if the smaller slab omits either the root layer or the
target layer. For \(\ell>0\), let \(A_\ell\) be the open cluster of
\(v\) in \(L_{m,\ell-1}(v)\); for \(\ell=0\), put
\(A_0=\varnothing\). Reveal this cluster and then reveal the states of
its neighbors in \(L_\ell(v)\). Let
\[
 W_\ell=
 \begin{cases}
  \{v\},&\ell=0,\\
  \{w\in L_\ell(v):U_w\leq p,\ w\sim A_\ell\},&\ell>0,
 \end{cases}
 \qquad N_\ell=|W_\ell|.
\]
Every member of \(W_\ell\) is connected to \(v\) in
\(L_{m,\ell}(v)\). By the parameter restriction
\(0<p\leq p_c^{\rm s}\) fixed before \eqref{eq:5-11h}, the full
\(p\)-cluster of \(v\) is finite almost surely, and hence
\begin{equation}
 N_\ell\leq X_\ell^{m,\ell}(v)<\infty.                         \label{eq:5-11hc}
\end{equation}

Enumerate \(W_\ell\) deterministically as \(w_1,\ldots,w_{N_\ell}\).
At step \(i\), delete \(A_\ell\), the other entry sites
\(W_\ell\setminus\{w_i\}\), and all clusters assigned at earlier steps,
and expose the cluster \(K_i\) of \(w_i\) in the remaining part of
\(L_{m,\ell}(v)\). (No earlier \(K_j\) contains \(w_i\), since every
other entry site was removed at step \(j\).) Let
\(Y_i=|K_i\cap L_k(v)|\). Removing the other entry sites is essential:
it ensures that \(w_i\) is the only
coordinate conditioned open in this residual exploration. All edges
within \(L_\ell(v)\) whose endpoints have not been removed are retained,
so vertices reached laterally in the extreme layer are included in the
appropriate \(K_i\).

The assigned sets cover the increment counted by \(Z_\ell^+\). Indeed,
let \(u\in L_k(v)\) be connected to \(v\) in \(L_{m,\ell}(v)\), but not
in \(L_{m,\ell-1}(v)\), and choose a simple open path from \(v\) to
\(u\). If \(\ell=0\), its initial vertex is \(w_1=v\). If
\(\ell>0\), take the successor \(w\) of the last vertex of the path in
\(A_\ell\). This successor exists. It cannot belong to
\(L_{m,\ell-1}(v)\), since an open neighbor there would also belong to
\(A_\ell\); hence \(w\in W_\ell\). On the remaining part of the path,
take the last vertex belonging to \(W_\ell\), and denote it by
\(w_j\). The suffix from \(w_j\) to \(u\) avoids both \(A_\ell\) and
\(W_\ell\setminus\{w_j\}\): if it re-entered \(A_\ell\), it would have
to leave again before reaching \(u\), and that later exit would pass
through a later member of \(W_\ell\). If this suffix meets a cluster assigned
before step \(j\), let \(i<j\) be the least index of an assigned cluster
that it meets. The part of the suffix from that intersection to \(u\)
avoids every \(K_r\) with \(r<i\) and every other entry site, and hence
was available during the exploration of \(K_i\). Thus \(u\in K_i\).
If the suffix meets no earlier cluster, it is available when \(K_j\) is
explored and \(u\in K_j\). The sequential exploration therefore assigns
\(u\) in every case. Consequently
\begin{equation}
 Z_\ell^+\leq\sum_{i=1}^{N_\ell}Y_i.                           \label{eq:5-11hd}
\end{equation}

Let \(\mathcal F_\ell\) contain the exploration of \(A_\ell\), the
revealed entry set \(W_\ell\), and the fixed layer field, and let
\(\mathcal F_{\ell,i}\) also contain the first \(i\) residual-cluster
explorations. Immediately before the exploration from \(w_i\), its root
is known open; every other available site is either untested
Bernoulli\((p)\) or has been tested closed. The first-hit residual kernel
therefore couples this cluster below an independent open-root cluster in
\[
 L_{m-\ell,0}(w_i)=L_{m,\ell}(v),
 \]
where the equality uses the exact layer-shift rule. Consequently, on
the event \(\{i\leq N_\ell\}\),
\begin{equation}
 \widehat{\mathbb E}_p^{v,x}
   [Y_i^{1-\varepsilon}\mid\mathcal F_{\ell,i-1}]
 \leq \widehat E_p^{m-\ell,0}
          (k-\ell;1-\varepsilon).
                                                                  \label{eq:5-11he0}
\end{equation}
Here and below a conditional statement indexed by \(i\) is asserted on
\(\{i\leq N_\ell\}\); this event and \(N_\ell\) are
\(\mathcal F_\ell\)-measurable. Applying the tower property to
\eqref{eq:5-11he0} gives, on the same event,
\begin{equation}
 \widehat{\mathbb E}_p^{v,x}
   [Y_i^{1-\varepsilon}\mid\mathcal F_\ell]
 \leq \widehat E_p^{m-\ell,0}
          (k-\ell;1-\varepsilon).                             \label{eq:5-11he}
\end{equation}
Put \(a=(1-\varepsilon)(1-\delta)\). Subadditivity of
\(t^{1-\varepsilon}\), conditional Jensen for the concave map
\(t^{1-\delta}\), and \eqref{eq:5-11hd}--\eqref{eq:5-11he} give
\begin{equation}
 \begin{split}
 \widehat{\mathbb E}_p^{v,x}
   [(Z_\ell^+)^a\mid\mathcal F_\ell]
 &\leq\widehat{\mathbb E}_p^{v,x}\!\left[
   \left(\sum_{i=1}^{N_\ell}Y_i^{1-\varepsilon}\right)^{1-\delta}
   \middle|\mathcal F_\ell\right]\\
 &\leq\left(\sum_{i=1}^{N_\ell}
   \widehat{\mathbb E}_p^{v,x}
       [Y_i^{1-\varepsilon}\mid\mathcal F_\ell]\right)^{1-\delta}\\
 &\leq N_\ell^{1-\delta}
   \left[\widehat E_p^{m-\ell,0}
          (k-\ell;1-\varepsilon)\right]^{1-\delta}.
 \end{split}                                                    \label{eq:5-11hf}
\end{equation}
At an exponent zero, every power is interpreted using
\(0^0=0\) and \(r^0=1\) for \(r>0\); the same inequalities follow
directly from the corresponding nonemptiness indicators. For the other
exponents, conditional Jensen with possibly infinite expectations is
obtained first for the bounded variable
\(\min\{M,\sum_iY_i^{1-\varepsilon}\}\) and then by monotone
convergence as \(M\to\infty\). Thus no integrability is being assumed.
Taking
expectations in \eqref{eq:5-11hf} and using \eqref{eq:5-11hc} yields
\begin{equation}
 \widehat{\mathbb E}_p^{v,x}[(Z_\ell^+)^a]
 \leq \widehat{\mathbb E}_p^{v,x}
       [(X_\ell^{m,\ell}(v))^{1-\delta}]
   \left[\widehat E_p^{m-\ell,0}
          (k-\ell;1-\varepsilon)\right]^{1-\delta}.            \label{eq:5-11hg}
\end{equation}

The pointwise telescoping identity
\[
 X_k^{m,n}(v)=\sum_{\ell=k\vee0}^{n}Z_\ell^+
\]
and subadditivity of \(t^a\) now give the first inequality below after
taking the supremum over \(v,x\).

For the downward inequality, set
\[
 Z_\ell^-:=X_k^{\ell,n}(v)-X_k^{\ell+1,n}(v),
 \qquad \ell\in[m,k\wedge0],
\]
with the second term zero at \(\ell=k\wedge0\) when its slab omits the
root or target layer. For \(\ell<0\), let \(A_\ell^-\) be the open
cluster of \(v\) in \(L_{\ell+1,n}(v)\), reveal its neighbors in
\(L_\ell(v)\), and put
\[
 W_\ell^-=
 \{w\in L_\ell(v):U_w\leq p,\ w\sim A_\ell^-\}.
\]
For \(\ell=0\), put \(A_0^-=\varnothing\) and \(W_0^-=\{v\}\).
Writing \(N_\ell^-=|W_\ell^-|\), every entry is connected to \(v\) in
\(L_{\ell,n}(v)\), and therefore
\begin{equation}
 N_\ell^-\leq X_\ell^{\ell,n}(v)<\infty.                       \label{eq:5-11hh}
\end{equation}
Enumerate the entries. For each entry \(w_i\), explore its cluster in
\(L_{\ell,n}(v)\) after deleting \(A_\ell^-\), all other entries, and
the clusters assigned at earlier steps. Let
\(Y_i^-=|K_i^-\cap L_k(v)|\). Let \(\mathcal F_{\ell,0}^-\) contain
the exploration of \(A_\ell^-\), the revealed entry set, and the layer
field, and let \(\mathcal F_{\ell,i}^-\) additionally contain the first
\(i\) residual explorations.

To verify coverage, take a simple open path witnessing a vertex counted
by \(Z_\ell^-\). If \(\ell=0\), start from \(v\). If \(\ell<0\), take
the successor of its last vertex in \(A_\ell^-\); it belongs to
\(W_\ell^-\), since any open successor in \(L_{\ell+1,n}(v)\) would
also belong to \(A_\ell^-\). Now take the last entry on the remaining
path, say \(w_j\). Its suffix avoids the upper cluster and every entry
other than \(w_j\): a later return to the upper cluster would require a
subsequent exit through another member of \(W_\ell^-\). If this suffix
meets a cluster assigned before step
\(j\), let \(i<j\) be the least index of such a cluster. The suffix from
an intersection with \(K_i^-\) to the endpoint avoids all \(K_r^-\) with
\(r<i\), as well as the upper cluster and every deleted entry, so it was
present when \(K_i^-\) was explored; the endpoint therefore lies in
\(K_i^-\). If there is no such intersection, the whole suffix is present
when \(K_j^-\) is explored and the endpoint lies in \(K_j^-\). Hence
\begin{equation}
 Z_\ell^-\leq\sum_{i=1}^{N_\ell^-}Y_i^-.                       \label{eq:5-11hi}
\end{equation}
The residual slab rooted at an entry \(w_i\in L_\ell(v)\) is
\[
 L_{0,n-\ell}(w_i)=L_{\ell,n}(v),
\]
and the first-hit residual domination gives, on
\(\{i\leq N_\ell^-\}\),
\[
 \widehat{\mathbb E}_p^{v,x}
 [(Y_i^-)^{1-\varepsilon}\mid\mathcal F_{\ell,i-1}^-]
 \leq\widehat E_p^{0,n-\ell}(k-\ell;1-\varepsilon).
\]
The event \(\{i\leq N_\ell^-\}\) and the value of \(N_\ell^-\) are
\(\mathcal F_{\ell,0}^-\)-measurable. The tower property therefore gives
the same bound conditional on \(\mathcal F_{\ell,0}^-\), on that event.
Applying the three-line subadditivity and
conditional-Jensen calculation in \eqref{eq:5-11hf}, now with
\(N_\ell^-\) and \(Y_i^-\), and then using
\eqref{eq:5-11hh}--\eqref{eq:5-11hi}, gives the explicit increment
estimate
\begin{equation}
 \widehat{\mathbb E}_p^{v,x}[(Z_\ell^-)^a]
 \leq \widehat{\mathbb E}_p^{v,x}
       [(X_\ell^{\ell,n}(v))^{1-\delta}]
   \left[\widehat E_p^{0,n-\ell}
          (k-\ell;1-\varepsilon)\right]^{1-\delta}.            \label{eq:5-11hj}
\end{equation}
Finally, the pointwise identity
\[
 X_k^{m,n}(v)=\sum_{\ell=m}^{k\wedge0}Z_\ell^-,
\]
followed by subadditivity of \(t^a\), proves the second inequality after
taking the supremum over \(v,x\).

If \(m=-\infty\) or \(n=\infty\), fix \(v,x\) and apply the finite-slab
inequalities with \(m_r=m\vee(-r)\) and \(n_r=n\wedge r\). The count on
the left increases pointwise to its infinite-slab value. Each term on the
finite right-hand side is at most the corresponding term with the stated
infinite endpoint, by domain monotonicity, and the missing terms may be
padded with zeros. Monotone convergence on the left and Tonelli's theorem
for the nonnegative series on the right therefore give the infinite-slab
inequality at this fixed \(v,x\). Taking the supremum only afterwards
gives the displayed \(\widehat E_p\) bounds. We have proved,
for \(0\leq\varepsilon,\delta\leq1\) and
\(-\infty\leq m\leq k\leq n\leq\infty\),
\begin{align}
 &\widehat E_p^{m,n}
    (k;(1-\varepsilon)(1-\delta))\notag\\
 &\quad\leq\sum_{\ell=k\vee0}^{n}
   \widehat E_p^{m,\ell}(\ell;1-\delta)
   \left[\widehat E_p^{m-\ell,0}
          (k-\ell;1-\varepsilon)\right]^{1-\delta},
          \displaybreak[1]\notag\\
 &\widehat E_p^{m,n}
    (k;(1-\varepsilon)(1-\delta))\notag\\
 &\quad\leq\sum_{\ell=m}^{k\wedge0}
   \widehat E_p^{\ell,n}(\ell;1-\delta)
   \left[\widehat E_p^{0,n-\ell}
          (k-\ell;1-\varepsilon)\right]^{1-\delta}.
          \label{eq:5-11h-prime}
\end{align}
These are the two fractional up/down inequalities, with all parameter
ranges displayed and with no endpoint factor. In particular, the proof
uses only the inequalities in \eqref{eq:5-11hc}, not an equality between
the number of entry sites and the complete extreme-layer count.

For completeness, Holder applied after \eqref{eq:5-11b} gives the other nonlinear
input. For \(k\in\mathbb Z\), \(0\leq\delta\leq\varepsilon<1\), and
\(-\infty\leq m\leq\min\{0,k\}\leq\max\{0,k\}\leq n\leq\infty\),
\begin{equation}
 \begin{split}
 &\widehat{\mathbb E}_p
       [(X_k^{m,n})^{1-\varepsilon}]\\
 &\quad\leq C_p e^{-t_0(1-\varepsilon)k}
   \widehat{\mathbb E}_p
       [(X_{-k}^{m-k,n-k})^{1-\delta}]^{1-\varepsilon}
   \widehat{\mathbb E}_p
       [(X_0^{m,n})^{1-\varepsilon}]^{\delta}.
 \end{split} \label{eq:5-11h-prime-prime}
\end{equation}
At \(\varepsilon=1\) the assertion used in the iteration is trivial.
Formula \eqref{eq:5-11h} shows that the negative power in the MTP proof is evaluated
only on a positive count. We make no pointwise mass-transport assertion
after conditioning on \(R_v=x\); fixed-offset estimates are obtained from
the averaged identity only through \eqref{eq:5-11h0}--\eqref{eq:5-11h0a} below.
\end{proof}

\begin{proposition}[Critical fractional closure]
\label{prop:critical-fractional-closure}
Assume the critical peak-survival estimate \eqref{eq:5-11g6} and the
fractional layer recursions of
Proposition~\ref{prop:fractional-layer-recursions}. At
\(p=p_c^{\rm s}\), the starting-slab estimate \eqref{eq:5-11h2}, the
excess-width estimates \eqref{eq:5-11h3}, the half-space estimates
\eqref{eq:5-11h6}, and the full-space estimate \eqref{eq:5-11i} hold
for every \(0<\varepsilon\leq1\) and every \(\zeta>0\), uniformly in
the root and fixed layer offset.
\end{proposition}

\begin{proof}
The logical order within the proof is the following. The deterministic
fixed-offset and shift comparisons \eqref{eq:5-11h0}--\eqref{eq:5-11h1}
are proved first. Peak survival \eqref{eq:5-11g6}, together with the
earlier layer estimates, then yields the starting-slab estimate
\eqref{eq:5-11h2}. Only after that estimate is available do we derive
the excess-width bounds \eqref{eq:5-11h3}. The fractional recursions are
then applied using these established bounds to prove the half-space and
full-space conclusions. Thus \eqref{eq:5-11h2} and
\eqref{eq:5-11h3} are conclusions of this proposition, not hypotheses.

Two deterministic comparisons complete the list of inputs to the
fractional iteration. If \(0\leq a\leq1\), then
\begin{equation}
 \widehat E_p^{m,n}(k;a)
 \leq \sup_v\sum_{j=-1}^{1}
   \widehat{\mathbb E}_p^v
      [(X_{k+j}^{m-1,n+1}(v))^a],                              \label{eq:5-11h0}
\end{equation}
where the expectations on the right average the independent layer offset.
Indeed, couple two offsets \(x,y\in[0,1)\) using the same site
configuration. Directly from the half-open definition \eqref{eq:5-1a},
\[
 L_k^x(v)\subseteq L_{k-1,k+1}^y(v),\qquad
 L_{m,n}^x(v)\subseteq L_{m-1,n+1}^y(v).
\]
Thus, pointwise in the site configuration and in \(y\), the fixed-offset
count is at most the sum of the three counts on the right of \eqref{eq:5-11h0}.
Use \((r+s+t)^a\leq r^a+s^a+t^a\), integrate \(y\), and then take the
supremum over \(v,x\). This proves \eqref{eq:5-11h0}, including its one-layer slab
enlargement.

Combining \eqref{eq:5-11h0} with the offset-averaged inequality \eqref{eq:5-11h-prime-prime} gives
the following explicit fixed-offset consequence. If
\(0\leq\delta\leq\varepsilon<1\) and
\(-\infty\leq m\leq\min\{0,k\}\leq\max\{0,k\}\leq n\leq\infty\), then
\begin{equation}
 \begin{split}
 \widehat E_p^{m,n}(k;1-\varepsilon)
 &\leq C_p\sum_{j=-1}^{1}e^{-t_0(1-\varepsilon)(k+j)}\\
 &\quad\times
 \left[\widehat E_p^{m-k-j-1,n-k-j+1}
          (-k-j;1-\delta)\right]^{1-\varepsilon}\\
 &\quad\times
 \left[\widehat E_p^{m-1,n+1}
          (0;1-\varepsilon)\right]^{\delta}.
 \end{split}                                                     \label{eq:5-11h0a}
\end{equation}
To verify \eqref{eq:5-11h0a}, apply \eqref{eq:5-11h-prime-prime} separately to the \(j\)-th averaged
term in \eqref{eq:5-11h0}, with parameters
\((m-1,n+1,k+j)\), and bound each resulting averaged moment by the
fixed-offset supremum \eqref{eq:5-11c-prime}. This is the only passage from the
offset-averaged mass transport to fixed-offset quantities.

Also, if
\(0<p\leq1\),
\(m\leq0\leq n\), \(m+r\leq0\leq n+r\), and
\(m\leq k\leq n\), forcing a connector across \(|r|\) adjacent layers
gives
\begin{equation}
 \widehat E_p^{m+r,n+r}(k+r;a)
       \leq e^{C_p|r|}\widehat E_p^{m,n}(k;a).                  \label{eq:5-11h1}
\end{equation}
Here is the finite-energy argument for \eqref{eq:5-11h1}. Choose
\(w\in L_r(v)\) by iterating the maximal-increment oriented edge fixed in
\eqref{eq:5-1b}, or its reverse when \(r<0\). The resulting monotone path
\(\gamma\) from \(v\) to \(w\) has exactly \(|r|\) non-root sites.
The two assumptions that both shifted slabs contain layer zero ensure
that \(\gamma\) lies in \(L_{m+r,n+r}(v)\). Equivariance of the offset
field gives
\[
 L_j(w)=L_{j+r}(v)\quad\text{for every }j,
\]
exactly, by the half-open convention in \eqref{eq:5-1a}. On
\(\{\gamma\text{ open}\}\), every
site counted by \(X_{k+r}^{m+r,n+r}(v)\) is therefore counted by
\(X_k^{m,n}(w)\). We spell out the conditioning that yields the connector
factor. Put
\[
 Y=(X_{k+r}^{m+r,n+r}(v))^a,\qquad
 Z=(X_k^{m,n}(w))^a,
\]
and let \(G_\gamma\) be the event that the \(|r|\) non-root sites of
\(\gamma\) are open. Under \(\widehat{\mathbb P}_p^{v,x}\),
\(\widehat{\mathbb P}_p^{v,x}(G_\gamma)=p^{|r|}\). Both \(Y\) and
\(\mathbf1_{G_\gamma}\) are increasing functions of the remaining
Bernoulli coordinates, so conditional site FKG and the pointwise
inequality \(Y\leq Z\) on \(G_\gamma\) give
\begin{equation}
 p^{|r|}\widehat{\mathbb E}_p^{v,x}Y
 \leq\widehat{\mathbb E}_p^{v,x}[Y;G_\gamma]
 \leq\widehat{\mathbb E}_p^{v,x}[Z;G_\gamma].               \label{eq:5-11h1a}
\end{equation}
To remove the conditioning at the other endpoint, let
\(G_\gamma^{\rm all}\) denote the event that every site of \(\gamma\),
including \(v\), is open, and let \(\mathbb E_p^x\) denote expectation
for the unconditioned Bernoulli labels with the layer offset at \(v\)
fixed to \(x\). Since \(G_\gamma^{\rm all}\) implies that \(w\) is open,
\begin{equation}
 \begin{split}
 \widehat{\mathbb E}_p^{v,x}[Z;G_\gamma]
 &=p^{-1}\mathbb E_p^x[Z;G_\gamma^{\rm all}]\\
 &\leq p^{-1}\mathbb E_p^x[Z;w\text{ open}]
 =\widehat{\mathbb E}_p^{w,R_w}Z.
 \end{split}                                                  \label{eq:5-11h1b}
\end{equation}
The offset is independent of all Bernoulli labels and therefore
contributes no factor in either equality. Combining
\eqref{eq:5-11h1a} and \eqref{eq:5-11h1b} gives
\[
 p^{|r|}
 \widehat{\mathbb E}_p^{v,x}
 [(X_{k+r}^{m+r,n+r}(v))^a]
 \leq
 \widehat{\mathbb E}_p^{w,R_w}
 [(X_k^{m,n}(w))^a].
\]
Taking the root/offset supremum and writing
\(C_p=|\log p|\) proves \eqref{eq:5-11h1}. No coordinate that was conditioned
closed is reopened, and these comparisons use no unexplored conditional
law.

We now give the iteration rather than cite its intermediate names. First
we prove the quantified starting-slab estimate: for every
\(0<\varepsilon\leq1\) and \(\zeta>0\),
\begin{equation}
 \widehat E_{p_c^{\rm s}}^{-k-r,0}(-k;1-\varepsilon)
 \leq C_{\varepsilon,\zeta}\exp[\zeta k+C_0r]
 \quad(k,r\geq0).                                             \label{eq:5-11h2}
\end{equation}
To derive it, we first display the peak-to-layer transport, keeping the
offset averaged as required by \eqref{eq:5-11b}. Put \(K=k+\ell\),
\(S=K+r+3\), and, for \(i=0,1,2,3\), define
\[
 F_i(a,b)=\mathbf1\{b\in L_{-K-i}(a),\ a\leftrightarrow b
       \text{ in }L_{-K-r-3,0}(a),\ \mathcal P_a(S)\}.
\]
This transport is diagonally invariant and vanishes unless both
endpoints are open. Hence, for every fixed root \(v\), \eqref{eq:5-11b} gives
\begin{equation}
 \begin{split}
 p_c^{\rm s}\int_0^1\widehat{\mathbb E}_{p_c^{\rm s}}^{v,x}
   [X_{-K-i}^{-K-r-3,0}(v);\mathcal P_v(S)]\,dx
  =\mathbb E_{p_c^{\rm s}}\sum_a F_i(a,v)\Delta(v,a).
 \end{split}                                                    \label{eq:5-11h2a00}
\end{equation}
If \(F_i(a,v)=1\), the exact layer shift gives
\[
 a\in L_{K+i}(v),\qquad
 a\leftrightarrow v\text{ in }L_{i-r-3,K+i}(v),\qquad
 L_{-S,\infty}(a)=L_{i-r-3,\infty}(v).
\]
There is at most one such \(a\): two candidates are connected in the
last half-space, whereas each peak condition says that its candidate is
the unique vertex of maximal modular height in that component. Also
\(a\in L_{K+i}(v)\) implies
\(\Delta(v,a)\leq e^{t_0(K+i+1)}\). Dropping the peak event and applying
\eqref{eq:5-11d2}, with lower depth \(r+3-i\), now yields
\[
 \mathbb E_{p_c^{\rm s}}\sum_a F_i(a,v)\Delta(v,a)
 \leq C e^{-t_0(\alpha_{p_c^{\rm s}}^{\rm s}-1)K+C_0r}.
\]
Thus \eqref{eq:5-11h2a00}, summed over \(i=0,1,2,3\), proves the genuinely
offset-averaged estimate
\[
 \int_0^1\sum_{i=0}^{3}
 \widehat{\mathbb E}_{p_c^{\rm s}}^{v,x}
 [X_{-k-\ell-i}^{-k-\ell-r-3,0};\mathcal P_v(k+\ell+r+3)]\,dx
 \leq C e^{-t_0(\alpha_{p_c^{\rm s}}^{\rm s}-1)(k+\ell)+C_0r}.
\]
In particular, for each \(v\) there is an offset \(x_v\) for which the
integrand is no larger than the right side. Since
\(\widehat{\mathbb P}^{v,x}(\mathcal P_v(S))\geq
(1-p_c^{\rm s})^d\) for every \(x\), division converts this semicolon
estimate at \(x_v\) into the conditional numerator in \eqref{eq:5-11f-prime-prime}. The
same offset has the uniform denominator lower bound \eqref{eq:5-11g6}. We may
therefore take \((v,x_v)\) in the infimum in \eqref{eq:5-11f-prime-prime}. This is the only
fixed-offset selection in the argument; no tilted MTP is asserted after
fixing an offset. Combining these facts gives, for every auxiliary
\(\xi>0\),
\begin{equation}
 \sup_{v,x}\widehat{\mathbb E}_{p_c^{\rm s}}^{v,x}
 [X_{-k}^{-k-r,0};\mathcal Q_{v;k,\ell,r}]
 \leq C_\xi
 e^{C_0r-t_0(\alpha_{p_c^{\rm s}}^{\rm s}-1)k+t_0\xi\ell}.   \label{eq:5-11h2a}
\end{equation}
Indeed, \eqref{eq:5-11g6} bounds the reciprocal denominator by
\(C_\xi e^{t_0(\alpha_{p_c^{\rm s}}^{\rm s}-1+\xi)\ell}\).
This proves \eqref{eq:5-11h2a}, including the sign and every parameter. Since
\eqref{eq:5-11e10} gives \(\alpha_{p_c^{\rm s}}^{\rm s}\geq1\), its negative
\(k\)-term may be discarded. On the complementary event, the root crosses
from the bottom of the slab to layer \(\ell-1\). Applying \eqref{eq:5-11d2} with
target \(\ell-1\) and lower depth \(k+r+2\), and then using
\(\alpha_{p_c^{\rm s}}^{\rm s}\geq1\), gives explicitly
\begin{equation}
 \sup_{v,x}\widehat{\mathbb P}_{p_c^{\rm s}}^{v,x}
   (\mathcal Q_{v;k,\ell,r}^{c})
 \leq C e^{C_0(k+r)-t_0\ell}.                                \label{eq:5-11h2b}
\end{equation}

For fixed \(v,x\) and an integer \(N\geq2\), split
\begin{equation}
 \widehat{\mathbb P}_{p_c^{\rm s}}^{v,x}
       (X_{-k}^{-k-r,0}\geq N)
 \leq N^{-1}\widehat{\mathbb E}_{p_c^{\rm s}}^{v,x}
       [X_{-k}^{-k-r,0};\mathcal Q]
       +\widehat{\mathbb P}_{p_c^{\rm s}}^{v,x}(\mathcal Q^c).
                                                                  \label{eq:5-11h2b0}
\end{equation}
Choose \(C_1>C_0/t_0\) and set
\(\ell=\max\{2,\lceil t_0^{-1}\log N+C_1(k+r)\rceil\}\).
Then \eqref{eq:5-11h2b} makes the second term at most \(CN^{-1}\), since
\[
 C_0(k+r)-t_0\ell\leq-\log N-(t_0C_1-C_0)(k+r)+O(1).
\]
After discarding the nonpositive \(k\)-term in \eqref{eq:5-11h2a}, the first term
in \eqref{eq:5-11h2b0} is at most
\[
 C_\xi N^{-1+\xi}
   \exp[t_0\xi C_1k+(C_0+t_0\xi C_1)r].
\]
Given \(\theta,\zeta>0\), choose \(\xi>0\) so that
\(\xi\leq\theta\) and \(t_0\xi C_1\leq\zeta\), and enlarge the fixed
coefficient \(C_0\) of \(r\) once. The last two displays imply
\begin{equation}
 \sup_{v,x}\widehat{\mathbb P}_{p_c^{\rm s}}^{v,x}
   (X_{-k}^{-k-r,0}\geq N)
 \leq C_{\theta,\zeta}
   \left(N^{-1+\theta}e^{\zeta k+C_0r}+N^{-1}e^{C_0r}\right). \label{eq:5-11h2c}
\end{equation}
Take \(\theta=\varepsilon/2\). For an integer-valued \(X\geq0\),
\[
 \mathbb E X^{1-\varepsilon}
 \leq1+C_\varepsilon\sum_{N\geq1}N^{-\varepsilon}
                  \mathbb P(X\geq N).
\]
Multiplying the first term of \eqref{eq:5-11h2c} by \(N^{-\varepsilon}\)
produces \(N^{-1-\varepsilon/2}\), and the second produces
\(N^{-1-\varepsilon}\); both series converge. The factors independent of
\(N\) are exactly \(e^{\zeta k+C_0r}\), after one final enlargement of
the constant. This proves \eqref{eq:5-11h2}, with all parameters chosen before
\(v,x,k,r\).

We next prove the stronger, excess-uniform version of the upward
estimate retained in
\cite[proof of Lemma~6.10]{Hutchcroft2020Nonunimodular}.  With the same
order of quantifiers, for every \(0<\varepsilon\leq1\) and
\(\zeta>0\), there is a constant \(C_\star<\infty\), independent of
\(\varepsilon,\zeta,k,r\), such that
\begin{equation}
 \begin{aligned}
 \widehat E_{p_c^{\rm s}}^{-r,k+r}(k;1-\varepsilon)
   &\leq C_{\varepsilon,\zeta}
          e^{-(t_0-\zeta)k+C_\star r},\\
 \widehat E_{p_c^{\rm s}}^{-\infty,r}(0;1-\varepsilon)
   &\leq C_\varepsilon e^{C_\star r}.
 \end{aligned} \label{eq:5-11h3}
\end{equation}
Let \(C_{\rm sh}\) be the coefficient of \(|r|\) in the exponential
shift loss \eqref{eq:5-11h1}, evaluated at \(p=p_c^{\rm s}\).
For the first inequality, choose \(0<\varepsilon_0<\varepsilon\), then
choose \(0<\delta\leq\varepsilon_0\) and \(0<\xi\leq1\) so that
\begin{equation}
 t_0\varepsilon_0+C_{\rm sh}\delta+2\xi<\zeta.               \label{eq:5-11h3a}
\end{equation}
Apply the fixed-offset consequence \eqref{eq:5-11h0a} with
\((m,n)=(-r,k+r)\).  For \(\eta\in(0,1]\), put
\[
 Q_\eta=
 \widehat E_{p_c^{\rm s}}^{-k-2r-2,0}
       (-k-r-1;1-\eta).
\]
For \(j=-1,0,1\), the first fractional factor in
\eqref{eq:5-11h0a} is
\[
 \widehat E_{p_c^{\rm s}}^{-r-k-j-1,r-j+1}
       (-k-j;1-\delta).
\]
It is the shift by \(r-j+1\) of \(Q_\delta\).  The last factor is
\(\widehat E_{p_c^{\rm s}}^{-r-1,k+r+1}
(0;1-\varepsilon_0)\), which is the shift by \(k+r+1\) of
\(Q_{\varepsilon_0}\).  The hypotheses of \eqref{eq:5-11h1} hold in
both cases, and hence
\[
 \begin{aligned}
 \widehat E_{p_c^{\rm s}}^{-r-k-j-1,r-j+1}
       (-k-j;1-\delta)
 &\leq e^{C_{\rm sh}(r-j+1)}Q_\delta,\\
 \widehat E_{p_c^{\rm s}}^{-r-1,k+r+1}
       (0;1-\varepsilon_0)
 &\leq e^{C_{\rm sh}(k+r+1)}Q_{\varepsilon_0}.
 \end{aligned}
\]
Apply \eqref{eq:5-11h2} to each \(Q_\eta\), with its variables there
equal to \(k+r+1\) and \(r+1\).  Substitution in
\eqref{eq:5-11h0a}, followed by absorption of the three bounded
\(j\)-shifts, gives
\[
 \widehat E_{p_c^{\rm s}}^{-r,k+r}
       (k;1-\varepsilon_0)
 \leq C_{\varepsilon_0,\delta,\xi}
 \exp\!\left(
   [-t_0(1-\varepsilon_0)+C_{\rm sh}\delta+2\xi]k
   +C_\star r\right).
\]
The coefficient \(C_\star\) may be chosen using only the constants in
\eqref{eq:5-11h1} and \eqref{eq:5-11h2}: the powers
\(1-\varepsilon_0,\delta\) are at most one and \(\xi\leq1\).
In particular, it is independent of the chosen fractional exponents
and error allowance.  By \eqref{eq:5-11h3a}, the coefficient of \(k\)
is at most \(-(t_0-\zeta)\).
Since the counts are integers and
\(1-\varepsilon\leq1-\varepsilon_0\), their \((1-\varepsilon)\)-moments
are no larger. This proves the first line of \eqref{eq:5-11h3}.

For the second line, apply the second inequality of \eqref{eq:5-11h-prime} with target
exponent \((1-\varepsilon_0)^2\):
\[
 \widehat E_{p_c^{\rm s}}^{-\infty,r}(0;(1-\varepsilon_0)^2)
 \leq\sum_{k\geq0}
   \widehat E_{p_c^{\rm s}}^{-k,r}(-k;1-\varepsilon_0)
   [\widehat E_{p_c^{\rm s}}^{0,k+r}(k;1-\varepsilon_0)]^{1-\varepsilon_0}.
\]
For an input error \(\xi>0\), shifting the first factor down by \(r\)
and applying \eqref{eq:5-11h2} with variables \(k+r\) and \(0\) gives
\[
 \widehat E_{p_c^{\rm s}}^{-k,r}(-k;1-\varepsilon_0)
 \leq C_{\varepsilon_0,\xi}
       e^{\xi k+(C_{\rm sh}+\xi)r}.
\]
For the second factor, domain monotonicity is used only in its valid
direction:
\[
 \widehat E_{p_c^{\rm s}}^{0,k+r}(k;1-\varepsilon_0)
 \leq
 \widehat E_{p_c^{\rm s}}^{-r,k+r}(k;1-\varepsilon_0)
 \leq C_{\varepsilon_0,\xi}
       e^{-(t_0-\xi)k+C_\star r}.
\]
Choose \(\xi\) so that
\(2\xi<(1-\varepsilon_0)t_0/2\).  The summand is then at most
\(C_{\varepsilon_0}e^{C_\star r-c_{\varepsilon_0}k}\), with
\(c_{\varepsilon_0}>0\), so its sum is
\(C_{\varepsilon_0}e^{C_\star r}\), after enlarging \(C_\star\).
Every exponent in \((0,1)\) equals \((1-\varepsilon_0)^2\) for a
suitable \(\varepsilon_0\in(0,1)\), while exponent zero is bounded by
one. The latter is sufficient at this point because the second line of
\eqref{eq:5-11h3} asks only for an upper bound of the form
\(C_\varepsilon e^{C_\star r}\). This proves that line for the stated
\(\varepsilon\).

For completeness, here is the remaining half-space bootstrap. Define
\[
 \begin{aligned}
 \mathcal A_-&=\{\varepsilon\in(0,1]:
   \widehat E_{p_c^{\rm s}}^{-\infty,0}(-k;1-\varepsilon)
       \leq C_\varepsilon e^{o_\varepsilon(k)}\},\\
 \mathcal A_+&=\{\varepsilon\in(0,1]:
   \widehat E_{p_c^{\rm s}}^{0,\infty}(k;1-\varepsilon)
       \leq C_\varepsilon e^{-t_0k+o_\varepsilon(k)}\}.
 \end{aligned}
\]
Let \(C_{\rm hi}=C_{\rm sh}+C_\star\), enlarging it by one if necessary,
so that \eqref{eq:5-11h1} followed by the second line of
\eqref{eq:5-11h3} has total shift loss at most
\(e^{C_{\rm hi}r}\).  Put
\(\eta_0=\min\{1/2,t_0/(4C_{\rm hi})\}\). We claim
\begin{equation}
 \varepsilon\in\mathcal A_-
 \quad\Longrightarrow\quad
 ((1-\eta_0)\varepsilon,1]\subseteq\mathcal A_+.               \label{eq:5-11h4}
\end{equation}
Fix \(\varepsilon'>(1-\eta_0)\varepsilon\) and choose \(\delta\in(0,1]\)
from
\[
 1-\varepsilon'=(1-\delta)[1-(1-\eta_0)\varepsilon].
\]
The first inequality of \eqref{eq:5-11h-prime} says explicitly that
\[
 \begin{split}
 \widehat E_{p_c^{\rm s}}^{0,\infty}(k;1-\varepsilon')
  &\leq\sum_{\ell\geq k}
  \widehat E_{p_c^{\rm s}}^{0,\ell}(\ell;1-\delta)\\
 &\qquad\times[\widehat E_{p_c^{\rm s}}^{-\ell,0}
    (k-\ell;1-(1-\eta_0)\varepsilon)]^{1-\delta}.
 \end{split}
\]
In its downward factor use
\[
 1-(1-\eta_0)\varepsilon
 =(1-2\eta_0)(1-\varepsilon)+2\eta_0(1-\varepsilon/2)
\]
and Lyapunov interpolation.  Write \(r=\ell-k\).  Domain monotonicity
and \(\varepsilon\in\mathcal A_-\) give, for every \(\xi>0\),
\[
 \widehat E_{p_c^{\rm s}}^{-\ell,0}(-r;1-\varepsilon)
 \leq \widehat E_{p_c^{\rm s}}^{-\infty,0}(-r;1-\varepsilon)
 \leq C_{\varepsilon,\xi}e^{\xi r}.
\]
For the larger moment, shift the slab upward by \(r\), enlarge its lower
endpoint, and use the second line of \eqref{eq:5-11h3}:
\[
 \begin{split}
 \widehat E_{p_c^{\rm s}}^{-\ell,0}(-r;1-\varepsilon/2)
 &\leq e^{C_{\rm sh}r}
       \widehat E_{p_c^{\rm s}}^{-k,r}(0;1-\varepsilon/2)\\
 &\leq e^{C_{\rm sh}r}
       \widehat E_{p_c^{\rm s}}^{-\infty,r}(0;1-\varepsilon/2)
 \leq C_{\varepsilon}e^{C_{\rm hi}r}.
 \end{split}
\]
Log-convexity of \(a\mapsto\log\mathbb E[X^a]\) therefore bounds the
downward factor of exponent \(1-(1-\eta_0)\varepsilon\) by
\[
 C\exp\{[(1-2\eta_0)\xi+2\eta_0C_{\rm hi}]r\}
 \leq C\exp[(\xi+t_0/2)r].
\]
The upward extreme-layer factor is the first line of
\eqref{eq:5-11h3} with excess parameter zero:
\[
 \widehat E_{p_c^{\rm s}}^{0,\ell}(\ell;1-\delta)
 \leq C_{\delta,\xi}e^{-(t_0-\xi)\ell}.
\]
It is enough to consider \(0<\zeta<t_0/4\), since a bound with a smaller
error implies one with any larger error.  Take \(\xi\leq\zeta/4\) in
each input.  Raising the downward bound to the power
\(1-\delta\leq1\), enlarging constants to at least one, and substituting
the last three displays gives the convergent sum
bounded by
\[
 C\sum_{\ell\geq k}
   \exp[-(t_0-\zeta/4)\ell
          +(t_0/2)(\ell-k)+(\zeta/4)(\ell-k)]
 \leq C_{\varepsilon',\zeta}e^{-(t_0-\zeta)k}.
\]
This proves \eqref{eq:5-11h4}, uniformly in the root and offset.

Next we claim
\begin{equation}
 \varepsilon\in\mathcal A_+\setminus\{1\}
 \quad\Longrightarrow\quad
 (\varepsilon,1]\subseteq\mathcal A_-;                         \label{eq:5-11h5}
\end{equation}
Fix \(\varepsilon'>\varepsilon\). If \(\varepsilon'=1\), then the convention
\(X^0=\mathbf1_{\{X>0\}}\) gives
\[
 \widehat E_{p_c^{\rm s}}^{-\infty,0}(-k;0)\leq1,
\]
so \(1\in\mathcal A_-\) with constant one and zero error. We may therefore
assume \(\varepsilon'<1\), and then the relation
\(1-\varepsilon'=(1-\varepsilon)(1-\delta)\) determines
\(\delta\in(0,1)\). The second inequality of
\eqref{eq:5-11h-prime} gives the exact sum
\[
 \widehat E_{p_c^{\rm s}}^{-\infty,0}(-k;1-\varepsilon')
 \leq\sum_{\ell\geq k}
 \widehat E_{p_c^{\rm s}}^{-\ell,0}(-\ell;1-\delta)
 [\widehat E_{p_c^{\rm s}}^{0,\ell}
       (\ell-k;1-\varepsilon)]^{1-\delta}.
\]
Put \(r=\ell-k\).  The starting-slab estimate \eqref{eq:5-11h2}, with
its excess parameter equal to zero, gives
\[
 \widehat E_{p_c^{\rm s}}^{-\ell,0}(-\ell;1-\delta)
 \leq C_{\delta,\xi}e^{\xi\ell}.
\]
Domain monotonicity and \(\varepsilon\in\mathcal A_+\) give
\[
 \widehat E_{p_c^{\rm s}}^{0,\ell}(r;1-\varepsilon)
 \leq \widehat E_{p_c^{\rm s}}^{0,\infty}(r;1-\varepsilon)
 \leq C_{\varepsilon,\xi}e^{-(t_0-\xi)r}.
\]
Given a requested \(\zeta>0\), choose the common input error
\(\xi<\min\{\zeta/2,t_0(1-\delta)/4\}\).  Substitution, followed by
\(\ell=k+r\), makes the relevant series exactly
\[
 C\sum_{\ell\geq k}
   \exp[\xi\ell-(1-\delta)(t_0-\xi)(\ell-k)]
 \leq C_{\varepsilon',\zeta}e^{\zeta k}.
\]
Indeed, its coefficient of \(r\) is
\((2-\delta)\xi-(1-\delta)t_0\leq-(1-\delta)t_0/2<0\), while its
remaining factor is \(e^{\xi k}\leq e^{\zeta k}\).  This proves the
claimed membership in \(\mathcal A_-\) with the quantified meaning of
\eqref{eq:5-11c-prime-prime-prime-prime}.
Since \(1\in\mathcal A_-\), alternating \eqref{eq:5-11h4} and \eqref{eq:5-11h5} gives
\(\mathcal A_-=\mathcal A_+=(0,1]\): for a prescribed
\(\varepsilon>0\), choose a finite \(N\) with
\((1-\eta_0)^N<\varepsilon\) and alternate at most \(2N\) times.
Only finitely many auxiliary exponents and error allocations occur, so
the uniform error calculus makes every resulting constant uniform in
\(v,x,k\).

Finally, fix \(0<\varepsilon<1\) and \(\zeta>0\). Choose
\(0<\delta\leq\varepsilon\) so small that the shift loss
\(C_{\rm sh}\delta<\zeta/4\), and allocate at most
\(\zeta k/4\) to each invocation of the half-space bounds. Apply
\eqref{eq:5-11h0a}, not a fixed-offset mass transport, to
\(\widehat E^{-k,\infty}(-k;1-\varepsilon)\). For each
\(j=-1,0,1\), \eqref{eq:5-11h1} bounds the first factor produced by \eqref{eq:5-11h0a} by
a constant times \(\widehat E^{0,\infty}(k+1;1-\delta)\). The last factor
is \(\widehat E^{-k-1,\infty}(0;1-\varepsilon)\), which \eqref{eq:5-11h1}, with
shift \(-(k+1)\), bounds by
\(e^{C_{\rm sh}(k+1)}\widehat E^{0,\infty}(k+1;1-\varepsilon)\). Hence
\[
 \begin{split}
 \widehat E_{p_c^{\rm s}}^{-k,\infty}(-k;1-\varepsilon)
 &\leq C e^{t_0(1-\varepsilon)k+C_{\rm sh}\delta(k+1)}\\
 &\quad\times
 [\widehat E_{p_c^{\rm s}}^{0,\infty}(k+1;1-\delta)]^{1-\varepsilon}
 [\widehat E_{p_c^{\rm s}}^{0,\infty}(k+1;1-\varepsilon)]^\delta.
 \end{split}
\]
Both half-space factors are in \(\mathcal A_+\). Their leading
\(-t_0(k+1)\) exponents cancel the positive MTP term and leave at most
\(\zeta k\), after changing the constant. This proves the first line
below and includes the one-layer enlargement required by \eqref{eq:5-11h0}.

For the upward count no reflection symmetry is used. Given the requested
\(\varepsilon,\zeta\), first choose
\(0<\varepsilon_0<\varepsilon\), then
\(0<\delta\leq\varepsilon_0\), so that
\begin{equation}
 t_0\varepsilon_0+C_\star\delta<\zeta/2.                    \label{eq:5-11h6-choice}
\end{equation}
Apply \eqref{eq:5-11h0a} to
\(\widehat E^{-\infty,k}(k;1-\varepsilon_0)\). Its three reflected
factors are bounded using \eqref{eq:5-11h1}, with shifts \(2,1,0\), by the common
quantity \(\widehat E^{-\infty,0}(-k-1;1-\delta)\); its last factor is
the enlarged-slab quantity
\(\widehat E^{-\infty,k+1}(0;1-\varepsilon_0)\). Therefore
\[
 \begin{split}
 \widehat E_{p_c^{\rm s}}^{-\infty,k}(k;1-\varepsilon_0)
 &\leq C e^{-t_0(1-\varepsilon_0)k}
 [\widehat E_{p_c^{\rm s}}^{-\infty,0}
      (-k-1;1-\delta)]^{1-\varepsilon_0}\\
 &\qquad\times
 [\widehat E_{p_c^{\rm s}}^{-\infty,k+1}
      (0;1-\varepsilon_0)]^\delta.
 \end{split}
\]
The first factor is controlled by \(\delta\in\mathcal A_-\), and the
last by the second line of \eqref{eq:5-11h3} with \(r=k+1\). After raising
that estimate to the power \(\delta\), its contribution is exactly
\(e^{C_\star\delta(k+1)}\), up to a multiplicative constant. The leading
factor contributes \(e^{-t_0k+t_0\varepsilon_0k}\). Allocate an error
smaller than \(\zeta/2\) to the \(\mathcal A_-\) bound; then
\eqref{eq:5-11h6-choice} makes the total exponential rate at most
\(-(t_0-\zeta)k\), after absorbing the fixed
\(e^{C_\star\delta}\). Since the counts
are integer-valued and \(1-\varepsilon\leq1-\varepsilon_0\), this proves
the second line of
\begin{equation}
 \begin{aligned}
 \widehat E_{p_c^{\rm s}}^{-k,\infty}(-k;1-\varepsilon)
   &\leq C_{\varepsilon,\zeta}e^{\zeta k},\\
 \widehat E_{p_c^{\rm s}}^{-\infty,k}(k;1-\varepsilon)
   &\leq C_{\varepsilon,\zeta}e^{-(t_0-\zeta)k}.
 \end{aligned} \label{eq:5-11h6}
\end{equation}
For clarity, the final convolution is also written out. The second
inequality of \eqref{eq:5-11h-prime} gives
\begin{equation}
 \begin{split}
 \widehat E_{p_c^{\rm s}}^{-\infty,\infty}
   (-k;(1-\varepsilon)^2)
 &\leq\sum_{\ell\geq k}
   \widehat E_{p_c^{\rm s}}^{-\ell,\infty}(-\ell;1-\varepsilon)\\
 &\qquad\times
   [\widehat E_{p_c^{\rm s}}^{0,\infty}
       (\ell-k;1-\varepsilon)]^{1-\varepsilon}.
 \end{split} \label{eq:5-11h6a}
\end{equation}
Use \eqref{eq:5-11h6} on the first factor and \(\mathcal A_+\) on the second.
Choose their error allowances \(\xi\) so that
\(2\xi<(1-\varepsilon)t_0\) and \(2\xi<\zeta\). The summand is at most
\[
 C e^{\xi k}
 e^{-[(1-\varepsilon)(t_0-\xi)-\xi](\ell-k)},
\]
whose sum is at most \(C_{\varepsilon,\zeta}e^{\zeta k}\). For the
positive layer, the first inequality of \eqref{eq:5-11h-prime} is exactly
\begin{equation}
 \begin{split}
 \widehat E_{p_c^{\rm s}}^{-\infty,\infty}
   (k;(1-\varepsilon)^2)
 &\leq\sum_{\ell\geq k}
   \widehat E_{p_c^{\rm s}}^{-\infty,\ell}
       (\ell;1-\varepsilon)\\
 &\qquad\times
   [\widehat E_{p_c^{\rm s}}^{-\infty,0}
       (k-\ell;1-\varepsilon)]^{1-\varepsilon}\\
 &\leq C_{\varepsilon,\zeta}e^{-(t_0-\zeta)k}.
 \end{split} \label{eq:5-11h6b}
\end{equation}
Indeed, use the second line of \eqref{eq:5-11h6} on the first factor and
\(\mathcal A_-\) on the second, with errors \(\xi\) satisfying
\((2-\varepsilon)\xi<t_0/2\) and \((2-\varepsilon)\xi<\zeta\).
After writing \(\ell=k+r\), the summand is bounded by
\(C e^{-(t_0-\xi)k}e^{-[t_0-(2-\varepsilon)\xi]r}\), which is summable
and has the asserted \(k\)-exponent.
Every exponent in \((0,1)\) has the form \((1-\varepsilon)^2\). For the
zero exponent, fix \(a=1/2\), which is one of the exponents already
covered. Every nonnegative integer-valued count \(X\) satisfies
\[
 X^0=\mathbf1_{\{X>0\}}\leq X^a.
\]
Consequently the just-proved estimate at exponent \(a\) bounds the
zero-exponent moment as well. In particular, it supplies the required
\(\exp[-(t_0-\zeta)k]\) decay for \(k\geq0\), while for \(k<0\) it gives
the stated subexponential upper bound. Equations
\eqref{eq:5-11h6a}--\eqref{eq:5-11h6b} are therefore
exactly \eqref{eq:5-11i} in its quantified form. At every stage the supremum over
\(v,x\) was taken before a constant was selected, and the finite-exponent
argument following \eqref{eq:5-11h5} verifies uniformity through successive
fractional choices.
\end{proof}

\begin{proposition}[Tilted-threshold separation]
\label{prop:tilted-threshold-separation}
Assume the site two-fan inverse-magnetization estimate
\eqref{eq:5-9a}--\eqref{eq:5-9d}, the critical fractional estimate
\eqref{eq:5-11i}, the first-moment identities
\eqref{eq:5-11e-prime-prime-prime}--%
\eqref{eq:5-11e-prime-prime-prime-prime}, and the lower bound
\eqref{eq:5-11e10}. Then
\[
 p_c^{\rm s}<p_c^{\rm s}(\lambda)\leq p_t^{\rm s}
 \qquad(0<\lambda<1).
\]
It also implies \(p_c^{\rm s}\in I_h^{\rm s}\) and, together with the
first-moment estimates already proved,
\(\alpha_{p_c^{\rm s}}^{\rm s}=\beta_{p_c^{\rm s}}^{\rm s}=1\).
\end{proposition}

\begin{proof}
The site two-fan
differential inequality \eqref{eq:5-9}, through the inverse-magnetization
calculation \eqref{eq:5-9a}--\eqref{eq:5-9d}, gives for
\(0\leq\lambda<1/2\), every \(\eta>0\), and every root \(v\),
\begin{equation}
 \mathbb E_{p_c^{\rm s}(\lambda)}
   [|K_v|_{v,\lambda}^{(1+\eta)/2}]=\infty.                   \label{eq:5-11k}
\end{equation}
Suppose that \(p_c^{\rm s}(\lambda)=p_c^{\rm s}\) for some
\(0<\lambda<1/2\). If \(z\in L_k^x(v)\), then
\[
 e^{t_0\lambda(k+x-1)}<\Delta(v,z)^\lambda
       \leq e^{t_0\lambda(k+x)}.
\]
Since \(x\in[0,1)\), summing over the cluster gives, with constants
independent of the offset,
\[
 |K_v|_{v,\lambda}
 \asymp_{G,\lambda}\sum_{k\in\mathbb Z}
       e^{t_0\lambda k}X_k^{-\infty,\infty}(v).
\]
For \(a=3/4\), subadditivity gives
\[
 \left(\sum_k e^{t_0\lambda k}X_k\right)^a
 \leq\sum_k e^{at_0\lambda k}X_k^a.
\]
Apply clause 2 with \(1-\varepsilon=a\). For \(k\geq0\), its summand is
at most
\(C\exp[-(t_0-\zeta_0-3t_0\lambda/4)k]\); for \(k=-r<0\), it is at most
\(C\exp[-(3t_0\lambda/4-\zeta_0)r]\). Therefore
\begin{equation}
 \begin{split}
 \widehat{\mathbb E}_{p_c^{\rm s}}^v
   [|K_v|_{v,\lambda}^{3/4}]
 &\leq C\sum_{k\geq0}
   \exp[-(t_0-\zeta_0-3t_0\lambda/4)k]\\
 &\quad+C\sum_{r>0}
   \exp[-(3t_0\lambda/4-\zeta_0)r]<\infty,
 \end{split}                                                   \label{eq:5-11m}
\end{equation}
where clause 2 is used with
\[
 0<\zeta_0<\min\{t_0(1-3\lambda/4),3t_0\lambda/4\}.
\]
Both geometric series converge. Since conditioning the root open changes
this moment only by the factor \(p_c^{\rm s}\), this contradicts
\eqref{eq:5-11k} with \(\eta=1/2\). Hence
\(p_c^{\rm s}<p_c^{\rm s}(\lambda)\) for \(0<\lambda<1/2\).
The conditional tilted MTP gives
\(\widehat\chi_{p,\lambda}=\widehat\chi_{p,1-\lambda}\) for every
\(p,\lambda\), and therefore
\(p_c^{\rm s}(\lambda)=p_c^{\rm s}(1-\lambda)\). This extends the strict
lower bound to \(1/2<\lambda<1\). For the upper bound, if
\(\widehat\chi_{p,\lambda}<\infty\), then the symmetry also gives
\(\widehat\chi_{p,1-\lambda}<\infty\), and Cauchy--Schwarz gives
\[
 \widehat\chi_{p,1/2}
 \leq \widehat\chi_{p,\lambda}^{1/2}
       \widehat\chi_{p,1-\lambda}^{1/2}<\infty.
\]
Thus \(p_c^{\rm s}(\lambda)\leq p_c^{\rm s}(1/2)=p_t^{\rm s}\).
Finally, choose any \(\lambda'<1/2\). The strict inequality already
proved at \(\lambda'\), followed by this upper bound, gives
\(p_c^{\rm s}<p_c^{\rm s}(\lambda')\leq p_t^{\rm s}\), which also proves
the strict lower bound at \(\lambda=1/2\). This establishes \eqref{eq:5-11j} for
every \(0<\lambda<1\).

Finally, applying \eqref{eq:5-11j} at any fixed
\(\lambda_0\in(0,1/2)\) gives
\(p_c^{\rm s}<p_c^{\rm s}(\lambda_0)\leq p_t^{\rm s}\), so
the \(\lambda_0\)-tilted susceptibility is finite at \(p_c^{\rm s}\).
For fixed \(j\leq n\), every \(x\in L_j(v)\) has
\(\Delta(v,x)^{\lambda_0}\asymp e^{t_0\lambda_0j}\), uniformly in the
fixed offset. Hence
\[
 \widehat E_{p_c^{\rm s}}^{-\infty,n}(j;1)
 \leq C e^{-t_0\lambda_0j}
       \widehat\chi_{p_c^{\rm s},\lambda_0}<\infty,
\]
and therefore \(p_c^{\rm s}\in I_h^{\rm s}\). The characterization
\eqref{eq:5-11e-prime-prime-prime-prime} applies because we now know both
\(p_c^{\rm s}\in I_h^{\rm s}\) and \(p_c^{\rm s}<p_t^{\rm s}\). Since
\(p_c^{\rm s}<p_c^{\rm s}\) is false, its equivalence gives
\(\beta_{p_c^{\rm s}}^{\rm s}\leq1\). The lower bound
\(\alpha_{p_c^{\rm s}}^{\rm s}\geq1\) was already proved, without using
the fractional bootstrap, in \eqref{eq:5-11e10}. Since the preceding fixed-tilt
argument gives \(p_c^{\rm s}<p_t^{\rm s}\), \eqref{eq:5-11e-prime-prime-prime} gives
\(\alpha_{p_c^{\rm s}}^{\rm s}=\beta_{p_c^{\rm s}}^{\rm s}\leq1\).
Together with \eqref{eq:5-11e10}, this proves the asserted equality.
\end{proof}

Part 1 and the preliminary inequalities of the lemma were proved before
Proposition~\ref{prop:fractional-layer-recursions}. Propositions
\ref{prop:critical-fractional-closure} and
\ref{prop:tilted-threshold-separation} prove Parts 2--4. This completes
the proof of Lemma~\ref{lem:nonunimodular-site-estimates}. \(\square\)

Clause 4 of Lemma~\ref{lem:nonunimodular-site-estimates} proves
Theorem~\ref{prop:nonunimodular-height}. \(\square\)

\begin{auxiliary}[Bernoulli comparison]
\label{aux:bernoulli-comparison}
Let \(I\) be finite, let \(E\subseteq\{0,1\}^I\) be increasing, and let
\(0<p<q<1\). Then
\begin{equation}
 \mathbb P_q(E)\geq
 \mathbb P_p(E)^{\log q/\log p}.                              \label{eq:bernoulli-comparison}
\end{equation}
\end{auxiliary}

\begin{proof}
Write \(a(t)=\mathbb P_t(E)\). The assertion is immediate if
\(a(p)\in\{0,1\}\), so suppose \(0<a(p)<1\). Finite-product entropy
tensorization, which follows by iterating the log-sum inequality, applied
to \(f=\mathbf1_E\) gives
\[
 a(t)\log\frac1{a(t)}
 \leq\sum_{i\in I}\mathbb E_t[\operatorname{Ent}_i(f)].
\]
After all coordinates other than \(i\) are fixed, monotonicity makes
\(f\) either constant or the function \(f(\omega)=\omega_i\). The
conditional entropy in the second case is \(t\log(1/t)\). Consequently
Russo's finite-product formula gives
\[
 a(t)\log\frac1{a(t)}
 \leq t\log\frac1t\sum_{i\in I}
       \mathbb P_t(i\text{ is pivotal})
 =t\log\frac1t\,a'(t).
\]
It follows that
\[
 \frac d{dt}\left(\frac{\log a(t)}{\log t}\right)
 =\frac{(a'(t)/a(t))\log t-(\log a(t))/t}{(\log t)^2}\leq0.
\]
Integrating from \(p\) to \(q\), and remembering that \(\log q<0\),
gives \eqref{eq:bernoulli-comparison}.
\end{proof}

The remaining analytic input for locality is strict decrease: if
\(0<p<q<1\) and \(\alpha_p^{\rm s}>0\), then
\(\alpha_q^{\rm s}<\alpha_p^{\rm s}\). Apply Auxiliary
Lemma~\ref{aux:bernoulli-comparison} to an increasing cylinder event
\(E\), so that
\begin{equation}
 \mathbb P_q(E)\geq
 \mathbb P_p(E)^{\log q/\log p}.                              \label{eq:5-12a}
\end{equation}
Condition the starting vertex open and apply \eqref{eq:5-12a} to the remaining
site coordinates of a bounded-length height-crossing event. Increasing
the length cutoff gives, by continuity from below,
\[
 \widehat A_q^{\rm s}(t)
 \geq\widehat A_p^{\rm s}(t)^{\log q/\log p}.
\]
After taking negative logarithms, dividing by \(t\), and passing to the
height exponent,
\[
 \alpha_q^{\rm s}
 \leq\frac{\log q}{\log p}\alpha_p^{\rm s}
 <\alpha_p^{\rm s}.
\]
Thus this step is intrinsically a vertex-product argument and requires no
bond pivotal decomposition.

We now apply Theorem~\ref{prop:nonunimodular-height} in the locally
convergent setting. Suppose
that the full automorphism action of \(G\) is nonunimodular, write
\begin{equation}
 \Gamma=\operatorname{Aut}(G),\qquad
 \Gamma_n=\operatorname{Aut}(G_n),                            \label{eq:5-12d}
\end{equation}
and discard finitely many \(n\) so that every canonical action
\(\Gamma_n\curvearrowright G_n\) is nonunimodular. This is permitted by
the clopen property cited below. In the rest of this section,
\(\alpha_p^{\rm s}(H)\) denotes the exponent belonging specifically to
the pair \((H,\operatorname{Aut}(H))\), not to any other transitive
subgroup that \(H\) may admit.

Root each transitive graph arbitrarily and regard its Dirac rooted law as
a stationary random graph. The canonical modular label on an oriented
edge is the Radon--Nikodym derivative of the law of the root-neighbor pair
with its orientation reversed with respect to its forward law. It is the
modular cocycle of the full automorphism action in \eqref{eq:5-12d}. Stationary
rooted local convergence, together with the
uniform degree bound, implies local convergence of these oriented-edge
labels. For a degree-\(d\) transitive graph every such label belongs to
the finite set
\[
 \{a/b:1\leq a,b\leq d\}.
\]
Consequently convergence of the labels upgrades to the following exact
statement: for every radius \(r\), all sufficiently large rooted-ball
isomorphisms may be chosen to preserve \(\Delta(x,y)\) for every
\(x,y\) in the ball. This is
\cite[Proposition~5.3 and Corollary~5.4]{Hutchcroft2020Locality}. The
explicit statement that the nonunimodular locus is open and closed is
\cite[Corollary~5.5]{Hutchcroft2020Locality}.

For fixed \(p,t,r\), the bounded-length height connection probability
\(A_p^{\rm s}(t,r)\) is local. Moreover
\(A_p^{\rm s}(t,r)\uparrow A_p^{\rm s}(t)\) as \(r\to\infty\). Local
convergence preserves the incident modular labels, so it also preserves
\(t_0\) eventually. Fix an integer \(N\geq1\) and put
\(t=(N-1)t_0\). The buffered Fekete representation \eqref{eq:5-0a} gives
\[
 \alpha_p^{\rm s}(G_n)
 \leq\frac{-\log(A_p^{\rm s}(t,r;G_n)/p)-\log p}{Nt_0}.
\]
Take the graph index to infinity, then \(r\to\infty\), and finally the
infimum over \(N\). By \eqref{eq:5-0a}, this proves
\begin{equation}
 \limsup_n\alpha_p^{\rm s}(G_n)\leq\alpha_p^{\rm s}(G).        \label{eq:5-13}
\end{equation}
If \(p>p_c^{\rm s}(G)\), strict decrease of the positive height
exponent and \eqref{eq:5-1} give \(\alpha_p^{\rm s}(G)<1\). Equations
\eqref{eq:5-1} and \eqref{eq:5-13} imply \(p>p_c^{\rm s}(G_n)\) eventually, and hence
\(p\geq\limsup_n p_c^{\rm s}(G_n)\).
Thus site locality holds when the canonical full automorphism action of
the limit is nonunimodular. This argument makes no assertion for a
canonically unimodular graph equipped with a different nonunimodular
transitive subgroup.

\section{Site snowballing}

Use the monotone coupling \(\omega_p=\{v:U_v\leq p\}\), with independent
uniform \(U_v\).
For \(0<p\leq q<1\), write
\[
 \delta(p,q)=\log\frac{\log(1-q)}{\log(1-p)},
 \qquad q=\operatorname{Spr}(p;\delta(p,q)),
 \quad
 \operatorname{Spr}(p;t)=1-(1-p)^{e^t}.
\]

\subsection{Finite-volume two-arm bound}

Let \(\operatorname{Piv}^{\rm s}[m,n]\) be the event that two distinct
site clusters in \(B_n\) both meet \(S_m\) and \(S_n\).

\begin{lemma}[Site AKN]
\label{lem:site-akn}
Compare \cite[Sections~2--3]{AizenmanKestenNewman1987}. For
\(0<\varepsilon<1/2\), \(0<\eta<1\), and \(d<\infty\), there is
\(C=C(\varepsilon,\eta,d)\) such that
\begin{equation}
\mathbb P_p(\operatorname{Piv}^{\rm s}[1,n])
\leq C\left[\frac{\log |B_n|}{n}\right]^{1/2-\varepsilon}
                                                              \label{eq:6-1}
\end{equation}
for every transitive degree-\(d\) graph, \(p\in[\eta,1]\), and
\(n\geq1\).
\end{lemma}

\textbf{Proof.} It is enough to prove the estimate with \(4n\) in place of
\(n\). Indeed, for a target radius \(N\geq4\), take
\(n=\lfloor N/4\rfloor\). Then
\(\operatorname{Piv}^{\rm s}[1,N]\subseteq
\operatorname{Piv}^{\rm s}[1,4n]\): truncate each witnessing connection
at its first visit to \(S_{4n}\). Moreover,
\(|B_{2n}|\leq|B_N|\) and \(n\geq N/8\); the finitely many cases
\(N<4\) are absorbed by increasing \(C\). Choose \(m\in[n,2n)\) such
that, for \(W=B_m\) and its internal
vertex boundary \(\partial_{\rm in}W\),
\begin{equation}
 q:=\frac{|\partial_{\rm in}W|}{|W|}
 \leq C_d\frac{\log |B_{2n}|}{n}.                              \label{eq:6-2}
\end{equation}
This is the usual ball-boundary averaging lemma: multiply the ratios
\(|B_{j+1}|/|B_j|\) for \(n\leq j<2n\), and use
\(|\partial_{\rm in}B_j|\leq |B_j\setminus B_{j-1}|\), after shifting
the index by one. Transitivity also gives
\(|B_{2n}|\leq |B_n|^2\).

Let \(\mathfrak C\) be the collection of open clusters of the induced
percolation on \(W\) that meet \(\partial_{\rm in}W\), and put
\(U=\bigcup_{C\in\mathfrak C}C\). For \(z\in W\setminus U\), let
\(k(z)\) be the number of clusters in \(\mathfrak C\) adjacent to
\(z\), and define
\[
 J=\sum_{z\in W\setminus U}(k(z)-1)_+ .
\]
Every such adjacent site is closed. For each \(x\in B_{m-1}\), the
event \(\operatorname{Piv}^{\rm s}_x[1,4n]\) forces \(x\) to be closed:
if \(x\) were open, all its open neighbors would lie in one cluster.
The two clusters in that event both leave \(W\), since \(m<2n\), and
therefore give two distinct members of \(\mathfrak C\) adjacent to
\(x\). Consequently, by transitivity,
\begin{equation}
 |B_{m-1}|\,\mathbb P_p(\operatorname{Piv}^{\rm s}[1,4n])
 \leq \mathbb E_p J.                                          \label{eq:6-3}
\end{equation}

Write \(\partial^W_V C\) for the external vertex boundary of \(C\)
inside \(W\), and set
\[
 H_p(C)=p|\partial^W_VC|-(1-p)|C|.
\]
Boundary multiplicities give the deterministic identities
\begin{equation}
 \sum_{C\in\mathfrak C}|C|=|U|,
 \qquad
 \sum_{C\in\mathfrak C}|\partial^W_VC|
   =|\partial^W_VU|+J.                                        \label{eq:6-4}
\end{equation}
There is also the switching estimate
\begin{equation}
 \left|p\,\mathbb E|\partial^W_VU|-(1-p)\mathbb E|U|\right|
 \leq |\partial_{\rm in}W|.                                  \label{eq:6-5}
\end{equation}
Indeed, for each \(z\notin\partial_{\rm in}W\), delete the coordinate
of \(z\). The event that a neighbor of \(z\) is connected to
\(\partial_{\rm in}W\) in \(W\setminus\{z\}\) is independent of that
coordinate. When \(z\) is closed it says \(z\in\partial^W_VU\), and
when \(z\) is open it says \(z\in U\). Multiplication by \(p(1-p)\)
cancels the two contributions. The uncancelled boundary sites give
\eqref{eq:6-5}. It follows from \eqref{eq:6-4}--\eqref{eq:6-5} that
\begin{equation}
 p\,\mathbb EJ
 \leq \mathbb E\sum_{C\in\mathfrak C}|H_p(C)|
       +|\partial_{\rm in}W|.                                 \label{eq:6-6}
\end{equation}

For \(x\in W\), let \(C_x\) be its open cluster in \(W\), interpreted
as empty when \(x\) is closed. Rooting every cluster at each of its
vertices and applying Cauchy--Schwarz gives, for \(0<\rho<1\),
\begin{equation}
\begin{aligned}
 \mathbb E\sum_{C\in\mathfrak C}|H_p(C)|
 &\leq
 \left[\sum_{x\in W}\mathbb E
   \frac{H_p(C_x)^2}{|C_x|^{1+\rho}}\right]^{1/2}\\
 &\quad\times
 \left[\sum_{x\in W}\mathbb E
   \frac{\mathbf1\{C_x\cap\partial_{\rm in}W\ne\varnothing\}}
        {|C_x|^{1-\rho}}\right]^{1/2}.
\end{aligned} \label{eq:6-7}
\end{equation}
The second bracket before taking its square root equals
\(\mathbb E\sum_{C\in\mathfrak C}|C|^\rho\), and hence is at most
\(|W|^\rho|\partial_{\rm in}W|^{1-\rho}\), because distinct clusters
contain distinct boundary vertices.

For the first bracket the summand is defined to be zero on
\(\{x\text{ closed}\}\); we work conditionally on \(x\) being open.
Start with the already open root and explore \(C_x\) one further site at
a time, always testing an untested neighbor of the discovered open set.
Give an open test increment \(1-p\) and a closed test increment \(-p\).
The resulting process \((S_t)\) is a martingale with increments bounded
by one and conditional variance \(p(1-p)\leq1/4\).  At the stopping time
\[
 T=|C_x|-1+|\partial^W_VC_x|
\]
every non-root open site and every external boundary site has been
tested, and therefore
\[
 S_T=(1-p)(|C_x|-1)-p|\partial^W_VC_x|
     =-H_p(C_x)-(1-p),
 \qquad |C_x|\leq T+1\leq(d+1)|C_x|.
\]
This also covers \(C_x=W\), when \(T=|C_x|-1\). Splitting according to
\(2^j\leq T+1<2^{j+1}\), using the stopped square-martingale identity,
and using
\(|H_p(C_x)|^2\leq2S_T^2+2\), yields the following calculation.  On
the \(j\)-th event,
\[
 |C_x|\geq\frac{T+1}{d+1}\geq\frac{2^j}{d+1}.
\]
The stopped square-martingale identity and the conditional variance
bound give
\[
 \mathbb E_p[S_{T\wedge(2^{j+1}-1)}^2\mid x\text{ open}]
 =\mathbb E_p\left[
   \sum_{t<T\wedge(2^{j+1}-1)}
       \mathbb E[(S_{t+1}-S_t)^2\mid\mathcal F_t]
       \,\middle|\,x\text{ open}\right]
 \leq 2^{j-1}.
\]
On \(\{T+1<2^{j+1}\}\), the stopped variable in this display is
\(S_T\). Therefore the contribution of the \(j\)-th event is at most
\[
 C(d,\rho)2^{-j(1+\rho)}
   \left(\mathbb E_p[S_T^2;\,T+1<2^{j+1}\mid x\text{ open}]+1\right)
 \leq C(d,\rho)(2^{-j\rho}+2^{-j(1+\rho)}).
\]
Summing this geometric bound over \(j\geq0\) proves
\begin{equation}
 \mathbb E\frac{H_p(C_x)^2}{|C_x|^{1+\rho}}
 \leq C(d,\rho).                                               \label{eq:6-8}
\end{equation}
The unconditional expectation is \(p\) times the conditional one, and
the closed-root contribution is zero, so no first-test term is missing.
Equations \eqref{eq:6-7}--\eqref{eq:6-8} now imply
\begin{equation}
 \mathbb E\sum_{C\in\mathfrak C}|H_p(C)|
 \leq C(d,\rho)|W|q^{(1-\rho)/2}.                             \label{eq:6-9}
\end{equation}
Since \(W=B_m\), every vertex of \(W\) either belongs to \(B_{m-1}\)
or is joined by an edge to a vertex of \(B_{m-1}\). Consequently,
\[
 |W|=|B_m|\leq (d+1)|B_{m-1}|,
\]
including when \(m=1\). Combining \eqref{eq:6-2}, \eqref{eq:6-3}, \eqref{eq:6-6}, and
\eqref{eq:6-9}, using \(p\geq\eta\), and taking \(\rho=2\varepsilon\) proves
\eqref{eq:6-1}, after changing constants and treating bounded \(n\) trivially.
\(\square\)

The form of the a priori uniqueness zone used below requires an explicit
low-growth hypothesis. More precisely, for every \(K\geq1\) and
\(0<\varepsilon<1/2\) there are
\(c=c(d,\eta,\varepsilon)>0\) and
\(C=C(d,\eta,\varepsilon,K)<\infty\) such that
\begin{equation}
 \mathbb P_p\bigl(\operatorname{Piv}^{\rm s}[r,n]\bigr)
 \leq C\left[\frac{\log |B_n|}{n}\right]^{1/2-\varepsilon}.  \label{eq:6-10}
\end{equation}
whenever \(n\geq3\), \(p\in[\eta,1]\), \(1\leq r\leq c\log n\), and
\begin{equation}
 \log|B_n|\leq(\log n)^K.                                    \label{eq:6-10a}
\end{equation}
Here is the site version of the Cerf comparison, including the point that
would otherwise cause a factor \((1-p)^{-1}\). On
\(\operatorname{Piv}^{\rm s}[r,n]\), choose canonically two crossing
clusters, vertices \(a,b\in S_r\) in those clusters, and a path from
\(a\) to \(b\) through \(o\) of length at most \(2r\). Successively open
the closed sites of this path in a fixed order. Immediately before the
first opening that joins the two chosen clusters, its site \(z\) is
already closed and touches two distinct open clusters, each containing
one of the original crossing clusters. Since \(z\in B_r\), both clusters
reach distance at least \(n-r\) from \(z\). Thus the modified
configuration lies in the translate of
\(\operatorname{Piv}^{\rm s}[1,n-r]\) centered at \(z\).

The map opens at most \(2r\) sites and never closes a site. Recording
\(z\), the path, and the modified subset gives the following explicit
fibre bound. Let \(S\) be the set of path sites changed from closed to
open before the joining step. Given the image and \(S\), the original
configuration is recovered by closing precisely the sites of \(S\).
There are at most \(|B_r|^2\leq d^{2r+2}\) choices for \(a,b\), at most
\(d^{2r}\) paths of length at most \(2r\) through \(o\), at most
\(2r+1\) choices for \(z\), and at most \(2^{2r+1}\) choices for \(S\).
After increasing a degree-dependent base, their product is at most
\(C_d^r\). If \(|S|=s\), the original-to-image Bernoulli weight ratio is
\(((1-p)/p)^s\leq p^{-2r}\leq\eta^{-2r}\). Integrating all unchanged
coordinates and summing the recorded data therefore gives, for
\(2r\leq n\),
\begin{equation}
 \mathbb P_p(\operatorname{Piv}^{\rm s}[r,n])
 \leq(C_d/\eta)^{2r}
       \mathbb P_p(\operatorname{Piv}^{\rm s}[1,n-r]).          \label{eq:6-10a-prime}
\end{equation}
Let \(x_n=\log|B_n|/n\), apply Lemma 6.1 with parameter
\(\varepsilon/2\), and take any integer \(1\leq r\leq c\log n\). For large \(n\),
\(r\leq n/2\), \(|B_{n-r}|\leq|B_n|\), and \eqref{eq:6-10a-prime} gives, for a constant
\(A=A(d,\eta)\),
\[
 \mathbb P_p(\operatorname{Piv}^{\rm s}[r,n])
 \leq C n^{Ac}x_n^{1/2-\varepsilon/2}
 =C x_n^{1/2-\varepsilon}
       \bigl(n^{Ac}x_n^{\varepsilon/2}\bigr).
\]
Under \eqref{eq:6-10a},
\[
 n^{Ac}x_n^{\varepsilon/2}
 \leq n^{Ac-\varepsilon/2}(\log n)^{K\varepsilon/2}.
\]
Choose \(c>0\) so that \(Ac<\varepsilon/4\). The last display is then
bounded, and the finitely many remaining \(n\) are absorbed into \(C\).
This proves \eqref{eq:6-10}. No lower bound on \(\log|B_n|\) is used.

We will also need the sharper form of the Cerf comparison when a good
local two-point bound is available.

\begin{auxiliary}[Site nearby-clusters comparison]
\label{aux:site-nearby-clusters}
If
\(0<p\leq1\) and \(1<r\leq m\leq n/2-2\), then
\begin{equation}
\mathbb P_p(\operatorname{Piv}^{\rm s}[r,n])
\leq
\mathbb P_p(\operatorname{Piv}^{\rm s}[1,n/2])
\frac{|S_r|^2\,|B_{m+1}|}
{\displaystyle\min_{a,b\in S_r}
  \mathbb P_p(a\leftrightarrow b\text{ in }B_m)}.             \label{eq:6-10a-prime-prime}
\end{equation}
\end{auxiliary}

\begin{proof}
For \(p>0\), the denominator is strictly positive: a
geodesic from \(a\) to \(o\) followed by one from \(o\) to \(b\) lies in
\(B_m\), and the event that every site of this walk is open has positive
probability.  The case \(p=1\) is immediate, so below we may take
\(0<p<1\).  On the pivotal event choose \(a,b\in S_r\) in two different
clusters crossing \(B_n\), and let \(C\) be the cluster of \(a\) in
\(B_n\). Write \(T(C)=C\cup\partial_V^{B_n}C\). The exact-cluster event
\(\{K_a^{B_n}=C\}\) fixes \(C\) open and its external vertex boundary
closed. Conditional on it, all sites of \(B_n\setminus T(C)\) are still
independent Bernoulli sites. Since \(b\) belongs to another open cluster,
\(b\notin T(C)\), and that cluster connects \(b\) to \(S_n\) inside
\(B_n\setminus T(C)\).

For fixed \(a,b\in S_r\), define the admissible exact-cluster class
\[
 \mathcal C_{a,b}:=\bigl\{C\subseteq B_n:
     a\in C,\ \mathbb P_p(K_a^{B_n}=C)>0,\
     C\cap S_n\ne\varnothing,\ b\notin T(C)\bigr\}.
\]
Thus every \(C\in\mathcal C_{a,b}\) crosses from \(a\in S_r\) to
\(S_n\), and the coordinate of \(b\) belongs to the residual product
space. If \(a,b\) witness the pivotal event and \(C=K_a^{B_n}\), then
\(C\in\mathcal C_{a,b}\): the crossing property gives
\(C\cap S_n\ne\varnothing\), while the fact that \(b\) lies in a
different open cluster implies \(b\notin C\cup\partial_V^{B_n}C\).
For \(C\in\mathcal C_{a,b}\), let \(D_C\) be the latter connection event and
let \(F_C\) be the event that, inside \(B_m\setminus T(C)\), \(b\) is
connected to the external vertex boundary of \(T(C)\). Both are increasing
events of the unrevealed coordinates. Conditional on
\(\{K_a^{B_n}=C\}\), those coordinates have their original product
Bernoulli law. Harris--FKG \cite[Section~2.2]{Grimmett1999} therefore gives
\[
 \mathbb P_p(D_C\cap F_C\mid K_a^{B_n}=C)
 \geq
 \mathbb P_p(D_C\mid K_a^{B_n}=C)
 \mathbb P_p(F_C\mid K_a^{B_n}=C).
\]
The second factor is positive by the fresh-path comparison in the next
paragraph. Dividing by it gives, with all probabilities below understood
in this residual product law,
\[
 \mathbb P_p(D_C)
 \leq\frac{\mathbb P_p(D_C\cap F_C)}{\mathbb P_p(F_C)}.
\]
To compare with an unconditioned two-point function, sample an auxiliary
fresh Bernoulli configuration on all of \(B_m\). Every fresh open path
from \(a\) to \(b\) has, after its last visit to the deterministic set
\(T(C)\), a suffix using only exterior coordinates and witnessing
\(F_C\). Since those exterior coordinates have exactly the conditional
law under the exact-cluster event, this coupling gives
\[
 \mathbb P_p(F_C)\geq
 \mathbb P_p(a\leftrightarrow b\text{ in }B_m).
\]
On \(D_C\cap F_C\), choose the first site \(y\notin T(C)\) on the
\(F_C\)-connection and an adjacent site \(z\in T(C)\). Necessarily
\(z\in\partial_VC\), so \(z\) is closed and touches both the crossing
cluster \(C\) and the crossing cluster containing \(b\). Moreover
\(z\in B_{m+1}\), and both clusters reach distance at least
\(n-m-1\geq n/2\) from \(z\). Thus the translate centered at \(z\) of
\(\operatorname{Piv}^{\rm s}[1,n/2]\) occurs.

Let
\[
 \theta_r=\min_{a,b\in S_r}
      \mathbb P_p(a\leftrightarrow b\text{ in }B_m)>0.
\]
For fixed \(a,b\), the exact-cluster atoms are disjoint.  Multiplying the
conditional estimate above by \(\mathbb P_p(K_a^{B_n}=C)\) and summing
over the admissible class gives
\[
 \begin{split}
 &\sum_{C\in\mathcal C_{a,b}}
      \mathbb P_p(K_a^{B_n}=C,\,D_C)\\
 &\quad\leq \theta_r^{-1}
   \sum_{C\in\mathcal C_{a,b}}
      \mathbb P_p(K_a^{B_n}=C,\,D_C\cap F_C)\\
 &\quad\leq \theta_r^{-1}
   \sum_{z\in B_{m+1}}
      \mathbb P_p(\operatorname{Piv}^{\rm s}_z[1,n/2]).
 \end{split}
\]
The last inequality uses the pointwise containment proved in the preceding
paragraph; its use here is legitimate precisely because every indexed
cluster satisfies \(C\cap S_n\ne\varnothing\). We then take a union bound
over the selected site \(z\); after the sum over the disjoint admissible
atoms, no conditioning remains. Transitivity
makes each pivotal probability in the last sum equal to the probability
rooted at \(o\).  Finally sum over the at most \(|S_r|^2\) ordered choices
of \((a,b)\). This is exactly
\eqref{eq:6-10a-prime-prime}.
\end{proof}

\subsection{Influence bound}

For a set \(A\), let \(\mathbf G_\rho^A\) denote a ghost field that marks
each site of \(A\) independently with probability \(\rho\).  The following
is the precise site form of the sharp-threshold input.

\begin{auxiliary}[Countable-product Russo--Talagrand exhaustion]
\label{aux:countable-russo-talagrand}
Let \(I\) be countable,
let \(\nu_q\) be Bernoulli-\(q\) product measure on \(\{0,1\}^I\), and
let \(E\) be an increasing measurable event. Writing
\(\operatorname{Piv}_i(E)\) for the event that coordinate \(i\) is
pivotal, the Russo--Talagrand bound
\begin{equation}
 \frac{d}{dq}\log\frac{\nu_q(E)}{1-\nu_q(E)}
 \geq
 \frac{c}{q(1-q)\log(2/[q(1-q)])}
 \log\frac{1}{q(1-q)\sup_{i\in I}
                   \nu_q(\operatorname{Piv}_i(E))}             \label{eq:6-10b0}
\end{equation}
holds for almost every \(q\) for which \(0<\nu_q(E)<1\) and the
logarithm on the right is positive. If \(E\) is a connection event
between random subsets of fixed finite seed sets, using paths in a
possibly infinite domain \(D\), and \(W_n\uparrow V\) is a finite
vertex exhaustion containing the seeds, then
\begin{equation}
 E_n:=\{\text{the required connection has a path in }D\cap W_n\}
 \uparrow E,
 \quad \nu_q(E_n)\uparrow\nu_q(E),                            \label{eq:6-10b1}
\end{equation}
and, for every fixed coordinate \(i\),
\begin{equation}
 \mathbf1_{\operatorname{Piv}_i(E_n)}\longrightarrow
 \mathbf1_{\operatorname{Piv}_i(E)}\quad\text{pointwise},
 \qquad
 \nu_q(\operatorname{Piv}_i(E_n))\longrightarrow
 \nu_q(\operatorname{Piv}_i(E)).                             \label{eq:6-10b2}
\end{equation}
\end{auxiliary}

\begin{proof}
Every connection has a finite path witness, which proves
\eqref{eq:6-10b1}. Apply this observation after setting coordinate \(i\) first to
one and then to zero. The two resulting increasing sequences of
indicators converge to the corresponding indicators for \(E\); their
difference is the pivotal indicator, proving \eqref{eq:6-10b2} by dominated
convergence.

For completeness, we give the countable-product passage in functional
form. Enumerate \(I=\{1,2,\ldots\}\), let \(\mathcal F_n\) be generated by
the first \(n\) coordinates, and put
\[
 f_n=\nu_q(\mathbf1_E\mid\mathcal F_n).
\]
For a function \(f\) on a finite Bernoulli product, write
\[
 D_i f(\omega)=f(\omega^{i\to1})-f(\omega^{i\to0}),
 \qquad Q(q)=q(1-q)\log\frac{2}{q(1-q)}.
\]
The finite functional \(L^1\)--\(L^2\) inequality of
\cite[Theorem~1.5]{Talagrand1994} is written there using the centered
coordinate operator
\(\Delta_i f=f-\nu_q(f\mid\sigma(\omega_j:j\ne i))\). Conditional on
all coordinates other than \(i\), a direct two-point calculation gives
\[
 \|\Delta_i f\|_2^2=q(1-q)\|D_i f\|_2^2,
 \qquad
 \|\Delta_i f\|_1=2q(1-q)\|D_i f\|_1.
\]
The ratio of these two norms is
\(\|D_i f\|_2/[2\sqrt{q(1-q)}\|D_i f\|_1]\). Since
\(2\sqrt{q(1-q)}\leq1\), its logarithmic denominator is no smaller than
the denominator displayed below. Substitution into Talagrand's biased
inequality and weakening that denominator therefore gives, with a
universal constant \(C\),
\begin{equation}
 \operatorname{Var}_{\nu_q}(f)
 \leq C Q(q)\sum_i
 \frac{\|D_i f\|_2^2}
      {\log\!\left(e\|D_i f\|_2/\|D_i f\|_1\right)},          \label{eq:6-10b3}
\end{equation}
where a zero summand is interpreted as zero. Apply this to \(f_n\).
Monotonicity of \(E\) implies \(D_i f_n\geq0\), and, for every \(n\geq i\),
\begin{equation}
 \|D_i f_n\|_1
   =\nu_q(\operatorname{Piv}_i(E))=:a_i.                       \label{eq:6-10b4}
\end{equation}
Indeed, integrating the coordinate difference telescopes to
\(\nu_q(E\mid\omega_i=1)-\nu_q(E\mid\omega_i=0)\), which is \(a_i\).
More precisely,
\[
 D_i f_n=\nu_q\!\left(\mathbf1_{\operatorname{Piv}_i(E)}
             \mid\mathcal F_n^{(i)}\right),\qquad
 \mathcal F_n^{(i)}=\sigma(\omega_j:1\leq j\leq n,\ j\ne i).
\]
The martingale convergence theorem consequently gives
\begin{equation}
 D_i f_n\longrightarrow\mathbf1_{\operatorname{Piv}_i(E)}
       \quad\hbox{in }L^2,
 \qquad
 \|D_i f_n\|_2^2\longrightarrow a_i.                          \label{eq:6-10b5}
\end{equation}

Let \(S=\sum_i a_i\) and \(a_*=\sup_i a_i\). If \(S=\infty\), the desired
influence lower bound is automatic. Otherwise, since
\(0\leq D_i f_n\leq1\), each summand in \eqref{eq:6-10b3} is at most
\(\|D_i f_n\|_2^2\leq\|D_i f_n\|_1=a_i\). Equations
\eqref{eq:6-10b4}--\eqref{eq:6-10b5} therefore permit dominated
convergence for the series, while \(f_n\to\mathbf1_E\) in \(L^2\). We get
\begin{equation}
 \operatorname{Var}_{\nu_q}(\mathbf1_E)
 \leq C Q(q)\sum_i\frac{a_i}{\log(e/\sqrt{a_i})}
 \leq \frac{C Q(q)S}{\log(e/\sqrt{a_*})}.                     \label{eq:6-10b6}
\end{equation}
For clarity, the second variance estimate is also proved here rather
than imported through a normalization convention. Finite-product
variance tensorization (the Efron--Stein inequality; compare
\cite[Section~3.1]{BoucheronLugosiMassart2013}) follows by iterating the
conditional-variance identity over the independent coordinates and
using conditional Jensen at each step. Applied to \(f_n\), it gives
\[
 \operatorname{Var}_{\nu_q}(f_n)
 \leq q(1-q)\sum_{i=1}^n\|D_i f_n\|_2^2
 \leq q(1-q)\sum_{i=1}^n a_i
 \leq q(1-q)S.
\]
The middle inequality follows from
\(0\leq D_i f_n\leq1\) and
\(\|D_i f_n\|_1=a_i\). Since
\(f_n\to\mathbf1_E\) in \(L^2\), passage to the limit proves directly
that
\(\operatorname{Var}_{\nu_q}(\mathbf1_E)\leq q(1-q)S\).
Combining this bound with \eqref{eq:6-10b6}, and using
\[
 \log(e/\sqrt{a_*})+\log\frac{2}{q(1-q)}
 \geq\frac12\log\frac{1}{q(1-q)a_*},
\]
yields
\[
 \sum_{i\in I}\nu_q(\operatorname{Piv}_i(E))
 \geq
 \frac{c\nu_q(E)(1-\nu_q(E))}
      {q(1-q)\log(2/[q(1-q)])}
 \log\frac{1}{q(1-q)\sup_{i\in I}
                   \nu_q(\operatorname{Piv}_i(E))}.
\]

For Russo's step, raise only the first \(n\)
coordinates from \(q\) to \(q+t\), keeping all remaining coordinates at
\(q\). Finite-dimensional Russo differentiation gives the sum of the
first \(n\) pivotal probabilities at \(t=0\), while monotonicity bounds
this mixed measure above by \(\nu_{q+t}(E)\). Let \(n\to\infty\).
This is the finite-product Russo formula
\cite[Section~4, Lemma~3, equation~(4.2)]{Russo1981}; see also
\cite[Section~2.4]{Grimmett1999}. At every differentiability point of
\(q\mapsto\nu_q(E)\), combining
this derivative bound with the countable influence inequality gives
\eqref{eq:6-10b0}. Monotone functions are differentiable almost everywhere, which
is all that is required when the inequality is integrated.
\end{proof}

\begin{auxiliary}[Ghost threshold]\label{aux:ghost-threshold}
For every \(2\leq d<\infty\) and \(0<D<\infty\) there are
\(c=c(d,D)>0\) and \(h_0=h_0(d,D)>0\) such that the following holds. Let
\(G\) be an infinite connected unimodular transitive graph of degree
\(d\). Let \(\Lambda\subseteq V(G)\) be nonempty and let
\(A,B\subseteq\Lambda\) be nonempty finite sets. Let
\(1/d\leq p_1<p_2<1\), put
\(\delta=\delta(p_1,p_2)\leq D\), and suppose that site percolation has at
most one infinite cluster throughout \([p_1,p_2]\).  If \(h\leq h_0\),
\(hr\geq1\), and
\[
 \sup_{p\in[p_1,p_2]}
 \mathbb P_p(\operatorname{Piv}^{\rm s}[1,hr])<h,
\]
then
\begin{equation}
 \begin{split}
 &(\mathbf G_h^A\otimes\mathbf G_h^B\otimes\mathbb P_{p_1})
       (\mathbf G_h^A\leftrightarrow\mathbf G_h^B
          \text{ in }\Lambda)\geq h^{c\delta}\quad\Longrightarrow\\
 &(\mathbf G_{h^c}^A\otimes\mathbf G_{h^c}^B\otimes\mathbb P_{p_2})
       (\mathbf G_{h^c}^A\leftrightarrow\mathbf G_{h^c}^B
          \text{ in }B_r(\Lambda))\geq1-h^{c\delta}.
 \end{split} \label{eq:6-10b}
\end{equation}
\end{auxiliary}

\begin{proof}
We use Auxiliary
Lemma~\ref{aux:countable-russo-talagrand}, so
that no artificial boundary pivotal is introduced when \(\Lambda\) or
\(B_r(\Lambda)\) is infinite. Set
\begin{equation}
 m_V=\left\lfloor
       \frac{\log(1-p_1)}{\log((d-1)/d)}\right\rfloor,
 \quad q_i=1-(1-p_i)^{1/m_V},
 \quad m_G=\left\lfloor\frac{\log h}{\log q_1}\right\rfloor.    \label{eq:6-10c}
\end{equation}
Then \(m_V,m_G\geq1\) and \(q_1\geq1/d\). To record the upper endpoint
uniformly, put \(a=(d-1)/d\) and
\(x=\log(1-p_1)/\log a\geq1\). Since
\(\lfloor x\rfloor>x/2\),
\[
 1-q_1=(1-p_1)^{1/m_V}
       =a^{x/\lfloor x\rfloor}>a^2,
 \qquad q_1<1-a^2<1.                                      \label{eq:6-10c-compact}
\]
Moreover,
\(1-q_2=(1-q_1)^{e^\delta}\), so
\[
 \frac1d\leq q_1<1-a^2,\qquad
 \frac1d\leq q_2\leq1-a^{2e^D}<1.
\]
Thus \(q_1,q_2\) stay in a compact subinterval of \((0,1)\) depending
only on \(d,D\), including when \(d=2\). Under a
product measure \(\overline{\mathbb P}_q\), give each site \(m_V\)
independent Bernoulli \(q\)-bits and declare it open when at least one bit
is one.  Give each potential ghost mark \(m_G\) further bits and declare
it present when every bit is one.  The induced site and ghost parameters
are
\[
 \pi(q)=1-(1-q)^{m_V},\qquad \rho(q)=q^{m_G};
\]
in particular \(\pi(q_i)=p_i\), \(\rho(q_1)\geq h\), and, after decreasing
\(c\) and then \(h_0(d,D)\), \(\rho(q_2)\leq h^c\).  The last assertion
follows from
\[
 q_2^{m_G}\leq q_2^{-1}h^{\log q_2/\log q_1}
\]
and the uniform positive lower bound on \(\log q_2/\log q_1\).

Put \(\ell=\lfloor h^{-1}\rfloor\) and let \(E_j\) be the event that the
two encoded ghost fields are joined in \(B_{jrh}(\Lambda)\).  If
\((z,k)\) is one of the OR-bits at a site \(z\), direct conditioning on
the other \(m_V-1\) bits gives
\begin{equation}
 q(1-q)\overline{\mathbb P}_q((z,k)\text{ pivotal for }E_j)
 =q(1-\pi(q))
   \mathbb P_{\pi(q),\rho(q)}(z\text{ pivotal for }E_j).        \label{eq:6-10e}
\end{equation}
For a ghost AND-bit \((x,k)\),
\begin{equation}
 q\,\overline{\mathbb P}_q((x,k)\text{ pivotal for }E_j)
 \leq q^{m_G}=\rho(q).                                         \label{eq:6-10f}
\end{equation}
Each \(E_j\) has finite open-path witnesses because \(A\) and \(B\) are
finite. Thus \eqref{eq:6-10b1} gives convergence of its finite-volume connection
probabilities in the increasing direction. Equations \eqref{eq:6-10e}--\eqref{eq:6-10h}
below concern pivotality for the full event \(E_j\), and \eqref{eq:6-10b0} applies
directly to that event. In particular, we do not claim that pivotals
created by an artificial exhaustion boundary satisfy the required
uniform estimate. The assumed bound on
\(\operatorname{Piv}^{\rm s}[1,hr]\) is used only in the full graph and
therefore needs no limiting argument.

Suppose that \(z\in B_{(j-1)rh}(\Lambda)\) is pivotal and force it
closed.  Unless \(z\) itself has one of the two ghost marks, two distinct
neighboring clusters meet the two ghost fields.  If those clusters remain
distinct in the full graph, at least one is finite by the uniqueness
hypothesis.  If they reconnect only outside the current domain, their two
initial arms both cross the annulus of width \(rh\). There are at most
\(d^2\) ordered neighbor pairs, and the union of the two ghost fields has
intensity at most \(2\rho(q)\).
More precisely, site pivotality is measurable without the state of
\(z\). Let \(\mathbb G_q^\cup\) be the law of the union of the two
encoded ghost fields, and let \(\mathscr T_z^\cup\) denote the event from
Lemma~\ref{lem:site-two-ghost} formed with this union field. In the site
configuration with
\(z\) forced closed, the preceding alternatives and a union bound give
\begin{equation}
 \begin{split}
 &(1-\pi(q))
   \mathbb P_{\pi(q),\rho(q)}(z\text{ pivotal for }E_j)\\
 &\quad\leq 2\rho(q)
   +d^2(\mathbb P_{\pi(q)}\otimes\mathbb G_q^\cup)
          (\mathscr T_z^\cup)
   +d^2\mathbb P_{\pi(q)}
       (\operatorname{Piv}^{\rm s}_z[1,rh]).
 \end{split}                                                  \label{eq:6-10f1}
\end{equation}
The last probability is less than \(h\) by hypothesis. Also
\(\rho(q)\geq\rho(q_1)=q_1^{m_G}\geq h\), where the last inequality
uses the floor in the definition of \(m_G\). The union field is
stochastically dominated by a Bernoulli ghost field with mark probability
\(2\rho(q)\), equivalently with Poisson parameter
\(-\log(1-2\rho(q))\leq4\rho(q)\) after decreasing \(h_0\).
Lemma~\ref{lem:site-two-ghost},
\(\pi(q)\geq1/d\), and \(0<\rho(q)\leq1\) bound every term on the
right of \eqref{eq:6-10f1} by \(C_{d,D}\rho(q)^{1/2}\). Combining this
with the exact OR-bit identity \eqref{eq:6-10e} gives
\begin{equation}
 q(1-q)\overline{\mathbb P}_q((z,k)\text{ pivotal for }E_j)
 \leq C_{d,D}\rho(q)^{1/2}.                                   \label{eq:6-10g}
\end{equation}
Here the exceptional probability is \(2\rho(q)\), not \(2h\), and is
absorbed by \(2\rho(q)^{1/2}\).  This calculation also explains why the
closed-site factor in Lemma~\ref{lem:site-two-ghost} causes no loss when
\(p_2\) is close to
one: it is exactly the factor \(1-\pi(q)\) in \eqref{eq:6-10e}.

It remains to control sites in the outer shell. For fixed \(q\), choose
in each nonempty shell
\(B_{jrh}(\Lambda)\setminus B_{(j-1)rh}(\Lambda)\) a site \(z_j\)
whose pivotal probability is at least half the supremum over that shell;
an empty shell has zero
influence. The
events that \(z_j\) is open and pivotal for \(E_j\) are pairwise disjoint:
on that event every ghost-to-ghost path in the smaller domain uses
\(z_j\), whereas a path in that smaller domain cannot visit a site in a
later shell. Hence their probabilities sum to at most one. Since
\(\pi(q)\geq p_1\geq1/d\), the sum over \(j\) of the shell suprema is at
most \(2d\). To display the counting step, call an index bad when its
shell supremum exceeds \(12dh\). There are at most \((6h)^{-1}\) bad
indices. Since \(\ell=\lfloor h^{-1}\rfloor\geq(2h)^{-1}\) after
decreasing \(h_0\), the complementary set
\(I(q)\subseteq\{1,\ldots,\ell\}\) has
\(|I(q)|\geq(3h)^{-1}\). On it every shell-site influence is at most
\(12dh\), and therefore at most
\(C_{d,D}\rho(q)^{1/2}\).  Together with \eqref{eq:6-10f}--\eqref{eq:6-10g},
\begin{equation}
 \max_b q(1-q)\overline{\mathbb P}_q(b\text{ pivotal for }E_j)
 \leq C_{d,D}\rho(q)^{1/2}\qquad(j\in I(q)).                  \label{eq:6-10h}
\end{equation}

The countable-product Russo--Talagrand bound \eqref{eq:6-10b0} now gives, for
almost every \(q\in[q_1,q_2]\),
\[
 \frac{d}{dq}\log\frac{\overline{\mathbb P}_q(E_j)}
                         {1-\overline{\mathbb P}_q(E_j)}
 \geq c_{d,D}\log(1/h),\qquad j\in I(q).
\]
Indeed \(m_G\log(1/q_1)\geq\log(1/h)-O_{d}(1)\), and bounded values of
\(h\) are absorbed by reducing \(h_0\). The derivative of every other
log-odds is nonnegative. Hence, for almost every \(q\), summing over all
\(j\) gives at least
\(|I(q)|c_{d,D}\log(1/h)\). Divide by \(\ell\) and integrate from
\(q_1\) to \(q_2\); if an endpoint log-odds is infinite, the conclusion
below is immediate. Elementary calculus
in \eqref{eq:6-10c}, using \(q_1,q_2\) in the stated compact interval, gives
\(q_2-q_1\geq c_{d,D}\delta\).  Thus for some \(j\)
\begin{equation}
 \min\{\overline{\mathbb P}_{q_1}(E_j),
          1-\overline{\mathbb P}_{q_2}(E_j)\}
 \leq h^{c_{d,D}\delta}.                                      \label{eq:6-10j}
\end{equation}
The event in \(\Lambda\) is contained in \(E_j\), which in turn is
contained in the corresponding event in \(B_r(\Lambda)\).  Finally use
\(\rho(q_1)\geq h\), \(\rho(q_2)\leq h^c\), and monotonicity in the
ghost intensity. More explicitly, choose the constant in the statement
smaller than half the constant in \eqref{eq:6-10j}. The assumed lower bound then
rules out the first member of the minimum in \eqref{eq:6-10j}, and its second
member gives the required upper bound on the failure probability. This
proves \eqref{eq:6-10b}.
\end{proof}

\subsection{Chaining and endpoints}

We first prove the site gluing statement used in the chain. For a set
\(Z\), write \(K_Z^D(p)\) for the union of the \(p\)-open components in
\(D\) that meet \(Z\).
For nonempty \(A,B\subseteq\Lambda\), set
\[
 \tau_p^\Lambda(A,B)=
 \min_{a\in A,b\in B}\mathbb P_p(a\leftrightarrow b
                                      \text{ in }\Lambda),
 \qquad \tau_p^\Lambda(A)=\tau_p^\Lambda(A,A).
\]

\begin{auxiliary}[Ghost gluing]\label{aux:ghost-gluing}
For every \(2\leq d<\infty\) and \(0<D<\infty\) there are
\(a_1,a_2>0\) with the following property. Let \(G\) be unimodular and
transitive of degree \(d\), let
\(1/d\leq p_1<p_2<1\), put \(\delta=\delta(p_1,p_2)\leq D\), and
suppose that there is no infinite site cluster throughout
\([p_1,p_2]\). Let \(h\) and \(r\) satisfy the smallness, radius, and
pivotal hypotheses of Auxiliary Lemma~\ref{aux:ghost-threshold}. Then,
for every
nonempty finite \(A,Y\subseteq\Lambda\) and every \(x\in\Lambda\), under
the common uniform-label coupling of all percolation parameters,
\begin{equation}
 \tau_{p_1}^{\Lambda}(A)\geq h^{a_1\delta}
 \quad\Longrightarrow\quad
 (\mathbf G_h^A\otimes\mathbb P)
 \left(\begin{array}{c}
 x\leftrightarrow\mathbf G_h^A\text{ at }p_2\text{ in }B_r(\Lambda),\\
 Y\leftrightarrow\mathbf G_h^A\text{ at }p_1\text{ in }\Lambda,\\
 x\not\leftrightarrow Y\text{ at }p_2\text{ in }B_r(\Lambda)
 \end{array}\right)
 \leq3h^{a_2\delta^4}.                                      \label{eq:6-10k}
\end{equation}
\end{auxiliary}

\begin{proof}
Write \(c_0\) for the constant in Auxiliary
Lemma~\ref{aux:ghost-threshold}
and take its smallness threshold to be less than \(1\). Decrease \(a_1\)
so that \(3a_1\leq c_0/3\). Put
\(\mathcal B\) for the event in \eqref{eq:6-10k}. Condition on
\[
 C_X=K_x^{B_r(\Lambda)}(p_2),\qquad C_Y=K_Y^\Lambda(p_1).
\]
Let \(\mathcal F_{\rm fine}\) be the sigma-field generated by these two
stopped explorations, including their ordered queries and revealed label
cells. On an atom with values \(C_X,C_Y\), it reveals \(U_z\leq p_2\)
on \(C_X\), \(U_z>p_2\) on its tested boundary, \(U_z\leq p_1\) on
\(C_Y\), and \(U_z>p_1\) on its tested boundary and unused seeds.
Conditional on \(\mathcal F_{\rm fine}\), every label not queried by
either exploration remains an independent
\(\operatorname{Unif}[0,1]\) variable by
Lemma~\ref{lem:stopped-product-kernel}. This is the fine filtration used
through \eqref{eq:6-11f}.
On \(\mathcal B\) these sets are disjoint. Call the pair good when both
\(\mathbf G_h^A\cap C_X\ne\varnothing\) and
\(\mathbf G_h^A\cap C_Y\ne\varnothing\) have conditional probability at
least \(h^{a_1\delta}\). The contribution of all non-good pairs to
\(\mathbb P(\mathcal B)\) is at most \(2h^{a_1\delta}\).

For a fixed good pair, take two independent ghost fields on
\(A\cap C_X\) and \(A\cap C_Y\). Independence and the hypothesis on \(\tau\)
give the following three-factor calculation. Let \(G_X,G_Y\) be the
events that the respective ghost fields are nonempty. By goodness,
\(\mathbb P(G_X),\mathbb P(G_Y)\geq h^{a_1\delta}\). Conditional on
the two ghost fields and on \(G_X\cap G_Y\), choose their first marked
vertices \(a\in A\cap C_X\) and \(b\in A\cap C_Y\) using the global
enumeration. The percolation labels are independent of the ghosts, and
the definition of \(\tau_{p_1}^{\Lambda}(A)\) gives
\[
 \mathbb P_{p_1}(a\leftrightarrow b\text{ in }\Lambda)
 \geq h^{a_1\delta}.
\]
Averaging over the selected pair, and multiplying by the two independent
nonemptiness probabilities, gives
\[
 (\mathbf G_h^{A\cap C_X}\otimes\mathbf G_h^{A\cap C_Y}
       \otimes\mathbb P_{p_1})
 (\mathbf G_h^{A\cap C_X}\leftrightarrow
       \mathbf G_h^{A\cap C_Y}\text{ in }\Lambda)
 \geq h^{3a_1\delta}\geq h^{c_0\delta/3}.
\]
Apply Auxiliary Lemma~\ref{aux:ghost-threshold} from \(p_1\) to
\(p_{4/3}\). Since
nonconnection of \(C_X\) and \(C_Y\) forces nonconnection of the two
ghost fields, it follows that
\begin{equation}
 \mathbb P_{p_{4/3}}(C_X\not\leftrightarrow C_Y
                  \text{ in }B_r(\Lambda))
 \leq h^{c_0\delta/3}.                                     \label{eq:6-10m}
\end{equation}

For \(1\leq t\leq2\), abbreviate
\[
 p_t=\operatorname{Spr}\bigl(p_1;(t-1)\delta(p_1,p_2)\bigr),
\]
so that \(p_1,p_2\) are the endpoints and
\(p_{4/3}<p_{5/3}\).

Continue with this fixed good pair and work on the event that
\(x\not\leftrightarrow Y\) at \(p_2\). Put
\[
 D_X=\partial_V^{B_r(\Lambda)}C_X,
 \qquad D_Y^*=\partial_V^\Lambda C_Y\cup(Y\setminus C_Y),
 \qquad D_Y=\partial_V^{B_r(\Lambda)}C_Y
              \cup((Y\setminus C_Y)\setminus D_X).
\]
The joint exact-cluster event has the coordinate description
\begin{equation}
 \{U_z\leq p_2:z\in C_X\}\cap
 \{U_z>p_2:z\in D_X\}
 \cap\{U_z\leq p_1:z\in C_Y\}\cap
 \{U_z>p_1:z\in D_Y^*\}.                                  \label{eq:6-10n}
\end{equation}
Redundant conditions on overlaps are understood. Thus it is measurable
in the coordinates of \(C_X\cup D_X\cup C_Y\cup D_Y^*\), and all other
labels remain independent uniforms by
Lemma~\ref{lem:stopped-product-kernel}, applied successively at thresholds
\(p_1\) and \(p_2\). On \(\mathcal B\), \(x\) is open and
belongs to \(C_X\). Removing \(D_X\cap(Y\setminus C_Y)\) in the
definition of \(D_Y\) is harmless: those sites are closed even at
\(p_2\) and cannot initiate an open connection from \(Y\).

Call a path \(D_X\)-to-\(D_Y\) admissible if its endpoints are its only
vertices in \(C_X\cup D_X\cup C_Y\cup D_Y\), and every internal
vertex is \(p_{5/3}\)-open. Let \(M\) be the maximum number of pairwise
vertex-disjoint admissible paths. Equivalently, form the locally finite auxiliary
graph whose vertices are \(D_X\cup D_Y\) together with the
\(p_{5/3}\)-open vertices outside the displayed union, retaining the
edges of \(B_r(\Lambda)\), and give every vertex capacity one. Then
\(M\) is the maximum integral \(D_X\)-to-\(D_Y\) flow. It is measurable
outside \(C_X\cup D_X\cup C_Y\cup D_Y\), and hence is independent of
the exact-cluster event \eqref{eq:6-10n}.

This restriction loses no connection relevant on the failure event.
Given a path from \(C_X\) to the \(p_{5/3}\)-open cluster union generated
by \(Y\), take the segment after its last visit to \(C_X\cup D_X\) and
before its first subsequent visit to \(C_Y\cup D_Y\), and then trim at
the last \(D_X\)-vertex and first following \(D_Y\)-vertex.

The terminal sets are finite because the seed sets and all clusters are
finite in the assumed subcritical interval. The finite-terminal Menger
compactness lemma proved in Section~5 applies even when
\(D_X\cap D_Y\neq\varnothing\), with each vertex of the intersection
viewed as a zero-length path. It therefore says that if \(M\leq N\)
there is a vertex separator \(S\) of size at most \(N\) in the auxiliary
graph. Choose the first such set in a fixed enumeration of the finite
subsets of \(V\). In particular, neither existence nor measurability of
\(S\) relies on literal stabilization of separators in finite
exhaustions.

The comparison with \eqref{eq:6-10m} is made on a fresh label space, not under
the exact-cluster conditioning \eqref{eq:6-10n}. Give the internal vertices fresh
uniform labels, reveal only their \(p_{5/3}\)-states, and leave every
terminal label unrevealed. This produces exactly the law of \(M\), since
\(M\) ignores all terminal states and, by \eqref{eq:6-10n}, its internal
coordinates are independent unconditioned uniforms under the original
exact-cluster law. Conditional on this fresh \(p_{5/3}\)-configuration
and on \(M\leq N\), close the deterministically chosen separator \(S\) at
level \(p_{4/3}\). An internal separator site is known open at
\(p_{5/3}\) and closes with probability
\((p_{5/3}-p_{4/3})/p_{5/3}\); a terminal separator site has not been
tested on the fresh space and closes with probability \(1-p_{4/3}\).
Thus the conditional cost per site is at least
\[
 \beta=\min\left\{
 \frac{p_{5/3}-p_{4/3}}{p_{5/3}},
  1-p_{4/3}\right\}.
\]
Closing \(S\) destroys every admissible path, hence every wired-endpoint
\(p_{4/3}\)-connection between \(C_X\) and \(C_Y\). Wiring only the
vertices of \(C_X\cup C_Y\), while retaining the site states of every
other vertex, can only make connection more likely. Therefore
\[
 \mathbb P_{p_{4/3}}(C_X\not\leftrightarrow C_Y)
 \geq \mathbb P_{p_{4/3}}(C_X\not\leftrightarrow C_Y
          \text{ with }C_X,C_Y\text{ wired})
 \geq \beta^N\mathbb P(M\leq N).
\]
The first probability is at most the right side of \eqref{eq:6-10m}. Since the
law of \(M\) on the fresh space is its conditional law under \eqref{eq:6-10n},
this gives
\begin{equation}
 \mathbb P(M\leq N\mid K_x(p_2)=C_X,K_Y(p_1)=C_Y)
 \leq\beta^{-N}h^{c_0\delta/3}.                             \label{eq:6-11a}
\end{equation}

If \(M\geq N\), the first \(D_X\)-sites of the paths are distinct. A
path of length zero already gives a site in \(D_X\cap D_Y\). On every
other path, its last \(D_Y\)-site either is known closed at \(p_1\), in
which case it opens by \(p_{5/3}\) with probability \(\alpha_1\), or is
untested, in which case it is open with probability
\(p_{5/3}\geq\alpha_1\), where
\[
 \alpha_1=\frac{p_{5/3}-p_1}{1-p_1}.
\]
Then the first \(D_X\)-site belongs to
\(\partial_VK_Y^{B_r(\Lambda)}(p_{5/3})\). If the last site also lies
in \(D_X\), it is already in \(D_X\cap\partial_VC_Y\) and is counted
without being opened. The paths are vertex-disjoint and their terminal
labels are independent under \eqref{eq:6-10n}. Consequently the number of
distinct sites in
\[
 D_X\cap\partial_VK_Y^{B_r(\Lambda)}(p_{5/3})
\]
stochastically dominates \(\operatorname{Bin}(M,\alpha_1)\).

For \(0<a<1\), put
\[
 \Psi(a)=\left(1-\frac a2\right)
     \log\frac{1-a/2}{1-a}-\frac a2\log2.
\]
The exponential Chernoff bound gives
\[
 \mathbb P(\operatorname{Bin}(M,a)\leq aM/2)
 \leq e^{-M\Psi(a)}.
\]
There is \(c=c(d,D)>0\) such that
\begin{equation}
 \frac{\delta\Psi(\alpha_1)}{\log(1/\beta)}\geq c\delta^4,
 \qquad
 \frac{\delta\alpha_1\log(1/(1-\alpha_2))}
      {\log(1/\beta)}\geq c\delta^4,                          \label{eq:6-11c}
\end{equation}
where \(\alpha_2=(p_2-p_{5/3})/(1-p_{5/3})\). Put
\(s=1-p_{5/3}\) and \(u=\delta/3\). Direct substitution gives
\begin{equation}
 \begin{gathered}
 \alpha_1=1-s^{1-e^{-2u}},\qquad
 \alpha_2=1-s^{e^u-1},\\
 \beta=\min\left\{
   \frac{s^{e^{-u}}-s}{1-s},\ s^{e^{-u}}\right\}.
 \end{gathered}                                               \label{eq:6-11d}
\end{equation}
\medskip
\noindent\textit{Parameter calculation for \eqref{eq:6-11c}.}
We give the details of the two estimates rather than hide their
\(\delta\)-dependence in an elementary-bounds assertion. Set
\[
 H=\log(1/s),\qquad \delta_0=\min\{\delta,1\},\qquad
 L_\delta=\log(e/\delta_0).
\]
Since \(p_1\geq1/d\) and \(s\leq1-p_1\), we have
\begin{equation}
 H\geq h_d:=\log\frac d{d-1}>0.                              \label{eq:6-11d0}
\end{equation}
For \(u=\delta/3\), put
\[
 \theta_0=1-e^{-u},\qquad
 \theta_1=1-e^{-2u},\qquad
 \theta_2=e^u-1.
\]
The functions \(\theta_i/\delta\) extend continuously and positively
to \(\delta=0\). Hence, uniformly for \(0<\delta\leq D\),
\begin{equation}
 c_D\delta\leq\theta_i\leq C_D\delta\qquad(0\leq i\leq2).
                                                                    \label{eq:6-11d1}
\end{equation}
The identities in \eqref{eq:6-11d} now read
\[
 \alpha_1=1-e^{-\theta_1H},\qquad
 \log\frac1{1-\alpha_2}=\theta_2H.
\]
Using
\[
 c\min\{y,1\}\leq1-e^{-y}\leq\min\{y,1\}\qquad(y\geq0)
\]
and \eqref{eq:6-11d1}, we obtain
\begin{equation}
 c_D\min\{\delta H,1\}\leq\alpha_1
 \leq C_D\min\{\delta H,1\},\qquad
 c_D\delta H\leq\log\frac1{1-\alpha_2}\leq C_D\delta H.
                                                                    \label{eq:6-11d2}
\end{equation}

The first member in the minimum defining \(\beta\) satisfies
\[
 \frac{s^{e^{-u}}-s}{1-s}
 =e^{-e^{-u}H}\frac{1-e^{-\theta_0H}}{1-e^{-H}}
 \geq e^{-H}(1-e^{-\theta_0H})
 \geq c_{d,D}\delta_0e^{-H};
\]
the last inequality uses \eqref{eq:6-11d0}--\eqref{eq:6-11d1}.
The second member is \(e^{-e^{-u}H}\geq e^{-H}\). Therefore
\begin{equation}
 \beta\geq c_{d,D}\delta_0e^{-H},\qquad
 \log\frac1\beta\leq C_{d,D}(H+L_\delta).                    \label{eq:6-11d3}
\end{equation}

It remains to quantify the Chernoff rate. The function \(\Psi(a)\) is
the Bernoulli relative entropy
\[
 \Psi(a)=D_{\rm KL}\bigl(\operatorname{Ber}(a/2)
                    \,\|\,\operatorname{Ber}(a)\bigr).
\]
For fixed \(a\), the second derivative in the first Bernoulli parameter
of this relative entropy is \(1/[q(1-q)]\geq4\); Taylor's theorem at
\(q=a\) consequently gives \(\Psi(a)\geq a^2/2\). Thus, for
\(0\leq x\leq1\),
\(\Psi(1-e^{-x})\geq c x^2\). For \(x\geq1\), direct substitution gives
\[
 \Psi(1-e^{-x})
 =\frac{1+e^{-x}}2\bigl[x+\log(1+e^{-x})\bigr]-\log2.
\]
This is at least \(x/4\) when \(x\geq4\log2\); on the compact interval
\([1,4\log2]\), its ratio to \(x\) has a positive minimum. Consequently
\begin{equation}
 \Psi(\alpha_1)\geq c_D
 \begin{cases}
  (\delta H)^2,&\delta H\leq1,\\
  \delta H,&\delta H>1.
 \end{cases}                                                  \label{eq:6-11d4}
\end{equation}

We can now verify \eqref{eq:6-11c}. First suppose \(0<\delta\leq1\).
If \(\delta H\leq1\), then
\[
 H+L_\delta\leq C_d\frac{H^2}{\delta};
\]
indeed \(H\geq h_d\), \(\delta\leq1\), and
\(\delta\log(e/\delta)\leq1\). Equations
\eqref{eq:6-11d2}--\eqref{eq:6-11d4} therefore give
\[
 \frac{\delta\Psi(\alpha_1)}{\log(1/\beta)}
 \geq c_{d,D}\delta^4,\qquad
 \frac{\delta\alpha_1\log(1/(1-\alpha_2))}
      {\log(1/\beta)}
 \geq c_{d,D}\delta^4.
\]
If instead \(\delta H>1\), then
\(L_\delta\leq1/\delta<H\). The two numerators are each at least
\(c_D\delta^2H\), while \(\log(1/\beta)\leq C_{d,D}H\).
Their ratios are therefore at least \(c_{d,D}\delta^2\), and hence at
least \(c_{d,D}\delta^4\).

Finally suppose \(1\leq\delta\leq D\). Then
\eqref{eq:6-11d0}--\eqref{eq:6-11d4} give
\[
 \alpha_1\geq c_{d,D},\quad
 \log\frac1{1-\alpha_2}\geq c_DH,\quad
 \Psi(\alpha_1)\geq c_{d,D}H,\quad
 \log\frac1\beta\leq C_{d,D}H.
\]
Both ratios in \eqref{eq:6-11c} are bounded below by a positive
\((d,D)\)-dependent constant. Since \(\delta^4\leq D^4\), decreasing
that constant proves \eqref{eq:6-11c} throughout \(0<\delta\leq D\).

Take
\begin{equation}
 N=\left\lfloor
   \frac{c_0\delta}{12\log(1/\beta)}\log\frac1h
   \right\rfloor,\qquad
 L=\frac{c_0\alpha_1\delta}{48\log(1/\beta)}
       \log\frac1h.                                          \label{eq:6-11e}
\end{equation}
If \(N\geq1\), the unrounded quantity inside the floor is less than
\(2N\), and hence \(\alpha_1N/2\geq L\). Thus
\eqref{eq:6-11a}--\eqref{eq:6-11c} imply, uniformly over each good pair,
\begin{equation}
 \mathbb P\left(
 |D_X\cap\partial_VK_Y(p_{5/3})|<L
 \mid K_x(p_2)=C_X,K_Y(p_1)=C_Y\right)
 \leq2h^{c\delta^4}.                                        \label{eq:6-11f}
\end{equation}
When \(N=0\), one has \(0<L<\alpha_1/4<1\), so the event on the left is
exactly absence of a common boundary site. Equation \eqref{eq:6-11a} bounds the
case \(M=0\), while, given \(M\geq1\), that absence has probability at
most \(1-\alpha_1\). If
\(s\geq e^{-1/\delta}\), then \(N=0\) implies
\(h^{c\delta^4}\geq1/2\) after decreasing \(c\). If
\(s<e^{-1/\delta}\), \eqref{eq:6-11d} and \(N=0\) give
\(1-\alpha_1\leq h^{c\delta^2}\). Thus \eqref{eq:6-11f}, in this interpreted
form, also holds in the rounded case.

For the final sprinkling the common boundary sites must remain unrevealed
at level \(p_2\). This requires a second, coarser filtration; it is not a
continuation of the exact-cluster conditioning \eqref{eq:6-10n}. Define
\[
 \widetilde K_X=
 \{z:x\stackrel{p_2}{\longleftrightarrow}z
 \text{ using no vertex in }
 K_Y^{B_r(\Lambda)}(p_{5/3})\cup
 \partial_V^{B_r(\Lambda)}K_Y^{B_r(\Lambda)}(p_{5/3})\}.
\]
On \(x\not\leftrightarrow Y\) at \(p_2\), this is the ordinary
\(p_2\)-cluster of \(x\): entering the excluded external boundary would
immediately join the \(p_{5/3}\)-cluster union of \(Y\). Fix values
\(\widetilde K_X=C_X\) and
\(K_Y^{B_r(\Lambda)}(p_{5/3})=C_Y\). Exploration of
\(K_Y\) reveals its sites open, its external boundary closed at
\(p_{5/3}\), and every unused seed in \(Y\setminus C_Y\) closed at that
level. Conditional on this outcome, explore \(\widetilde K_X\) in the
allowed set
\(B_r(\Lambda)\setminus(C_Y\cup\partial_V^{B_r(\Lambda)}C_Y)\).
This exploration is forbidden from testing either excluded set. Hence,
for each site in
\(\partial_V^{B_r(\Lambda)}C_X\cap
\partial_V^{B_r(\Lambda)}C_Y\), the only revealed condition is
\(U_z>p_{5/3}\). Define the coarse sigma-field
\[
 \mathcal F_{\rm coarse}:=
 \sigma\!\left(
 K_Y^{B_r(\Lambda)}(p_{5/3}),\widetilde K_X,
 \text{their ordered queries and revealed label cells}
 \right).
\]
Conditional on \(\mathcal F_{\rm coarse}\), the labels on the common
relative boundary are mutually independent uniform variables on
\((p_{5/3},1]\), and all other unqueried labels remain independent
uniform variables on \([0,1]\). Each common-boundary label therefore
opens by \(p_2\) with probability \(\alpha_2\), so
\begin{equation}
 \mathbb P\left(
 x\not\leftrightarrow Y\text{ at }p_2
 \,\middle|\,
 |\partial_V^{B_r(\Lambda)}\widetilde K_X\cap
   \partial_V^{B_r(\Lambda)}K_Y^{B_r(\Lambda)}(p_{5/3})|
       \geq L\right)
 \leq(1-\alpha_2)^L.                                        \label{eq:6-12}
\end{equation}
If one such site opens, it joins the two cluster unions. Since the
number of common boundary sites is an integer, \eqref{eq:6-12} is valid also when
\(0<L<1\): the conditioning then asserts that there is at least one such
site, and \((1-\alpha_2)\leq(1-\alpha_2)^L\). Combining \eqref{eq:6-11c},
\eqref{eq:6-11e}, and \eqref{eq:6-12}, including the rounded case, gives,
after decreasing \(c>0\) if necessary,
\[
 (1-\alpha_2)^L\leq h^{c\delta^4}.
\]

To combine the two filtrations formally, define \(Q\) on the original label
space to be the event that the contact count in \eqref{eq:6-11f} is at least \(L\),
and let \(G\) be the event
that the exact pair \((K_x(p_2),K_Y(p_1))\) is good. On the failure event
in \eqref{eq:6-10k}, \(\widetilde K_X=K_x(p_2)\), so the contact set defining
\(Q\) is the same in \eqref{eq:6-11f} and \eqref{eq:6-12}. The tower
property, first for \(\mathcal F_{\rm fine}\) and then, in the last
term, for \(\mathcal F_{\rm coarse}\), gives
\[
 \begin{split}
 \mathbb P(\mathcal B)
 &\leq \mathbb P(\mathcal B\cap G^c)
       +\mathbb P(\mathcal B\cap G\cap Q^c)
       +\mathbb P(\mathcal B\cap Q)\\
 &\leq 2h^{a_1\delta}+2h^{c\delta^4}
       +(1-\alpha_2)^L.
 \end{split}
\]
There is no conditioning simultaneously asserting \(U_z>p_2\) and then
asking the same label to open at \(p_2\): \eqref{eq:6-11f} is averaged
under \(\mathcal F_{\rm fine}\) before \eqref{eq:6-12} is applied under
\(\mathcal F_{\rm coarse}\).

It remains to absorb the displayed coefficients uniformly as
\(\delta\downarrow0\). This cannot be done merely by decreasing a fixed
value of \(h_0\), since the exponent gap is proportional to
\(\delta^4\). Put
\[
 x=\delta^4\log(1/h),
 \qquad c_* = \min\{a_1D^{-3},c\},
\]
decreasing the positive constant \(c\) above if necessary so that both
the second and third terms in the preceding bound are at most their
displayed multiples of \(h^{c\delta^4}\). Since
\(a_1\delta\geq a_1D^{-3}\delta^4\), that bound and the trivial
probability bound by one give
\[
 \mathbb P(\mathcal B)\leq\min\{1,5e^{-c_*x}\}.
\]
Choose \(0<a_2\leq c_*/4\). If
\(x\leq\log(3)/a_2\), then \(3e^{-a_2x}\geq1\). If
\(x>\log(3)/a_2\), then
\[
 (c_*-a_2)x\geq3\log3>\log(5/3),
\]
and hence \(5e^{-c_*x}\leq3e^{-a_2x}\). In both regimes
\(\mathbb P(\mathcal B)\leq3h^{a_2\delta^4}\), proving
\eqref{eq:6-10k} uniformly for \(0<\delta\leq D\).
\end{proof}

The next proposition applies the ghost-threshold and ghost-gluing
estimates established above; through the ghost-threshold lemma it also
uses the countable-product Russo--Talagrand and site two-ghost estimates.
The site AKN and nearby-cluster estimates are not used in this
proposition. They enter later, in Proposition~\ref{prop:site-low-growth-step},
when its pivotal hypothesis is verified.

\begin{proposition}[Site snowballing]\label{prop:site-snowballing}
For every \(2\leq d<\infty\) and \(0<D<\infty\)
there are \(s_1,s_2,s_3,h_0>0\), depending only on \(d,D\), such
that the following holds. Let \(G\) be unimodular and transitive of
degree \(d\), let \(N\geq1\), let \(\Lambda\subseteq V(G)\), and let
\(A_1,\ldots,A_N\) be nonempty finite subsets of \(\Lambda\). Suppose
\(1/d\leq p_1<p_2<1\), \(\delta(p_1,p_2)\leq D\), and site
percolation has no infinite cluster throughout
\([p_1,p_2]\). Let \(h\geq(\min_i|A_i|)^{-1}\), \(h\leq h_0\), and
choose \(r\) with \(hr\geq1\) and
\[
 \sup_{p\in[p_1,p_2]}
      \mathbb P_p(\operatorname{Piv}^{\rm s}[1,hr])<h.
\]
If
\[
 h^{s_1\delta(p_1,p_2)^4}\leq s_3N^{-1},
 \qquad
 \tau_{p_1}^{\Lambda}(A_i\cup A_{i+1})
      \geq4h^{s_1\delta(p_1,p_2)^4}
\]
for every \(i<N\), then
\begin{equation}
 \tau_{p_2}^{B_{2r}(\Lambda)}(A_1,A_N)
  \geq s_2\tau_{p_1}^{\Lambda}(A_1)
            \tau_{p_1}^{\Lambda}(A_N).                    \label{eq:6-13}
\end{equation}
\end{proposition}

\textbf{Proof.} The case \(N=1\) follows by taking \(s_2\leq1\), so assume
\(N\geq2\). Let \(c_T=c_T(d,D)\) be the constant in Auxiliary
Lemma~\ref{aux:ghost-threshold}, decrease it so that \(c_T\leq1\), and put
\[
 \bar h=h^{c_T},\qquad
 p_{3/2}=\operatorname{Spr}(p_1;\delta/2),
 \quad \delta=\delta(p_1,p_2).
\]
Let \(a_1,a_2\) be the constants in Auxiliary
Lemma~\ref{aux:ghost-gluing}, and put
\[
 c_E=(1-e^{-1})^2,
 \qquad c_6=\frac{c_Ta_2}{16}.
\]
Choose \(s_1>0\) so that
\begin{equation}
 s_1\leq
 \min\left\{
   \frac{c_T}{4D^3},
   \frac{c_Ta_1}{8D^3},
   \frac{c_6}{4}
 \right\}.                                                   \label{eq:6-13a}
\end{equation}
Finally choose \(0<s_3\leq1\) sufficiently small that
\begin{equation}
 \frac{s_3^2}{2}\leq\frac14,
 \qquad
 \frac{3s_3^2}{32}\leq\frac{c_E}{4}.                       \label{eq:6-13b}
\end{equation}
Decrease the \(h_0\) in the proposition so that \(\bar h\) is below the
smallness threshold in Auxiliary Lemma~\ref{aux:ghost-gluing}. Take
independent ghost
fields \(\mathbf G_i\) of intensity \(\bar h\) on \(A_i\).

First use auxiliary ghost fields of intensity \(h\). Since
\(h|A_i|\geq1\), each is nonempty with probability at least
\(1-e^{-1}\). Thus, for every \(i<N\),
\begin{equation}
 \begin{split}
 &(\mathbf G_h^{A_i}\otimes\mathbf G_h^{A_{i+1}}
       \otimes\mathbb P_{p_1})
   (\mathbf G_h^{A_i}\leftrightarrow\mathbf G_h^{A_{i+1}}
       \text{ in }\Lambda)\\
 &\hspace{2cm}\geq(1-e^{-1})^2
       \tau_{p_1}^{\Lambda}(A_i\cup A_{i+1})
 \geq h^{s_1\delta^4}.
 \end{split} \label{eq:6-14}
\end{equation}
By \eqref{eq:6-13a}, \(s_1\delta^4\leq c_T\delta/2\), so
\eqref{eq:6-14} is at least the
input threshold in \eqref{eq:6-10b} for the interval \([p_1,p_{3/2}]\). The pivotal
hypothesis for that subinterval is inherited from \([p_1,p_2]\).
Consequently
\[
 \mathbb P(\mathbf G_i\not\leftrightarrow\mathbf G_{i+1}
       \text{ at }p_{3/2}\text{ in }B_r(\Lambda))
 \leq h^{c_T\delta/2}.
\]

We next localize the within-set two-point bounds. Put
\(\sigma_i=\tau_{p_1}^{\Lambda}(A_i)\), and fix \(x,y\in A_i\).
The adjacent-set hypothesis implies
\(\sigma_i\geq4h^{s_1\delta^4}\), using \(A_i\cup A_{i+1}\) when
\(i<N\) and \(A_{N-1}\cup A_N\) when \(i=N\). Apply the ghost gluing
lemma on \([p_1,p_{3/2}]\), with \(A=A_i\), \(Y=\{y\}\), ghost
intensity \(\bar h\), and initial vertex \(x\). Its radius can be taken
to be \(r\): since \(\bar h\geq h\),
\(\operatorname{Piv}^{\rm s}[1,\bar h r]\subseteq
\operatorname{Piv}^{\rm s}[1,hr]\), and \(\bar h r\geq1\). By
\eqref{eq:6-13a},
\[
 \sigma_i\geq4h^{s_1\delta^4}
       \geq\bar h^{a_1\delta/2}.
\]
Conditional on a nonempty \(\bar h\)-field on \(A_i\), choose one mark
deterministically. The definition of \(\sigma_i\) gives
\[
 \mathbb P(x\leftrightarrow\mathbf G_i
              \text{ at }p_1\text{ in }\Lambda)
 \geq(1-e^{-\bar h|A_i|})\sigma_i
 \geq(1-e^{-1})\sigma_i,
\]
and the same holds with \(y\). These are increasing events of the same
site and ghost coordinates. FKG gives a lower bound by the product; the
Auxiliary Lemma~\ref{aux:ghost-gluing} bounds the part of that
intersection on which \(x\)
and \(y\) still fail to connect at \(p_{3/2}\). Hence
\[
 \begin{split}
 \mathbb P_{p_{3/2}}(x\leftrightarrow y
                    \text{ in }B_r(\Lambda))
 &\geq(1-e^{-1})^2\sigma_i^2
       -3\bar h^{a_2(\delta/2)^4}\\
 &\geq c_7\sigma_i^2.
\end{split}
\]
Here the last inequality is uniform in \(\delta\), as the following
calculation makes explicit. Put
\[
 z=h^{s_1\delta^4}.
\]
Since \(N\geq2\), the hypothesis of the proposition gives
\(z\leq s_3/N\leq s_3/2\). By \eqref{eq:6-13a},
\(c_6\geq4s_1\), and therefore the error above is at most
\(3z^4\). Since \(\sigma_i^2\geq16z^2\), its ratio to
\(\sigma_i^2\) is at most
\[
 \frac{3}{16}z^2\leq\frac{3s_3^2}{64}\leq\frac{c_E}{8}.
\]
Thus one may take \(c_7=c_E/2\), with room to spare. Moreover,
\eqref{eq:6-13a} gives
\(c_Ta_1\delta/2\geq4s_1\delta^4\), so
\[
 \bar h^{a_1\delta/2}\leq z^4
 \leq16c_7z^2\leq c_7\sigma_i^2;
\]
here \(16c_7=8(1-e^{-1})^2>1\). Consequently
\begin{equation}
 \tau_{p_{3/2}}^{B_r(\Lambda)}(A_i)
 \geq c_7\sigma_i^2
 \geq\bar h^{a_1\delta/2}.                                  \label{eq:6-16}
\end{equation}

Fix \(u\in A_1\) and \(v\in A_N\). The endpoint estimates and the
adjacent-field estimates are
\[
 \begin{aligned}
 \mathbb P(u\leftrightarrow\mathbf G_1
       \text{ at }p_1\text{ in }\Lambda)&\geq(1-e^{-1})\sigma_1,\\
 \mathbb P(\mathbf G_N\leftrightarrow v
       \text{ at }p_1\text{ in }\Lambda)&\geq(1-e^{-1})\sigma_N,\\
 \mathbb P(\mathbf G_i\leftrightarrow\mathbf G_{i+1}
       \text{ at }p_{3/2}\text{ in }B_r(\Lambda))
       &\geq1-h^{c_T\delta/2}.
 \end{aligned}
\]
All are increasing in the common site and ghost coordinates. FKG and
monotonicity therefore show that their simultaneous occurrence has
probability at least
\begin{equation}
 (1-e^{-1})^2\sigma_1\sigma_N
       (1-h^{c_T\delta/2})^{N-1}.                            \label{eq:6-17}
\end{equation}
This display estimates an intersection of field-connection events; it
does not yet assert that connections using different marks of the same
field concatenate.

Apply Auxiliary Lemma~\ref{aux:ghost-gluing} on
\([p_{3/2},p_2]\), with intensity
\(\bar h\), base domain \(B_r(\Lambda)\), and thickening radius \(r\).
Write
\[
 D'=B_r(B_r(\Lambda))
\]
for its output domain. With the convention that a real radius is replaced
by its integer part,
\begin{equation}
 D'=B_{2\lfloor r\rfloor}(\Lambda)
 \subseteq B_{\lfloor2r\rfloor}(\Lambda)=B_{2r}(\Lambda).
                                                                  \label{eq:6-16a}
\end{equation}
Indeed \(\bar h\geq h\), so \(\bar h r\geq hr\geq1\), and every pair of
clusters reaching radius \(\bar h r\) also reaches radius \(hr\) while
remaining distinct in the smaller ball. Hence
\[
 \mathbb P_p(\operatorname{Piv}^{\rm s}[1,\bar h r])
 \leq\mathbb P_p(\operatorname{Piv}^{\rm s}[1,hr])<h\leq\bar h
\]
uniformly on the subinterval. Equation \eqref{eq:6-16} supplies the
within-set hypothesis. Starting from the
connection of \(u\) to \(\mathbf G_1\), define, for \(1\leq i<N\),
\[
 \begin{split}
 T_i=\{&u\leftrightarrow\mathbf G_i\text{ at }p_2
              \text{ in }D',\quad
 \mathbf G_i\leftrightarrow\mathbf G_{i+1}\text{ at }p_{3/2}
              \text{ in }B_r(\Lambda),\\
 &u\not\leftrightarrow\mathbf G_{i+1}\text{ at }p_2
              \text{ in }D'\},
 \end{split}
\]
and let \(T_N\) be the same three-event intersection with
\(A_N\), \(\mathbf G_N\), and \(Y=\{v\}\); its second event is
\(\mathbf G_N\leftrightarrow v\) at \(p_1\) in \(\Lambda\), which is
contained in the corresponding event at \(p_{3/2}\).  If all events in
\eqref{eq:6-17} occur but \(u\not\leftrightarrow v\) at \(p_2\) in
\(D'\), take the first field not joined to the \(p_2\)-cluster of \(u\)
inside \(D'\). This gives the pointwise inclusion
\[
 \{\text{events of \eqref{eq:6-17}}\}\cap
 \{u\not\leftrightarrow v\text{ at }p_2\text{ in }D'\}
 \subseteq\bigcup_{i=1}^N T_i.
\]
For \(i<N\), condition on all ghost fields other than
\(\mathbf G_i\). Then \(Y=\mathbf G_{i+1}\) is a deterministic finite
set, while \(\mathbf G_i\) retains its independent \(\bar h\)-ghost law
and all percolation labels retain their common independent uniform law.
Auxiliary Lemma~\ref{aux:ghost-gluing} applies with
\(x=u,A=A_i,Y=\mathbf G_{i+1}\), and
its estimate is uniform in the realized set \(Y\); when \(Y=\varnothing\),
\(T_i\) is empty.  Taking conditional expectation therefore gives
\(\mathbb P(T_i)\leq3\bar h^{a_2(\delta/2)^4}\).  The same calculation
for \(T_N\) uses \(A=A_N,Y=\{v\}\).  The union bound now shows that the
portion of the event in \eqref{eq:6-17} on which
\(u\not\leftrightarrow v\) at \(p_2\) in \(D'\) is at most
\begin{equation}
 3N\bar h^{a_2(\delta/2)^4}
 \leq3Nh^{c_6\delta^4}                                     \label{eq:6-18}
\end{equation}
with \(c_6\) as defined above. The connections lie in \(D'\), and
\eqref{eq:6-16a} then places them in the domain \(B_{2r}(\Lambda)\)
appearing in the conclusion of the proposition.

Finally, we verify that the endpoint product and the accumulated error
are uniform as \(\delta\downarrow0\). Retain
\(z=h^{s_1\delta^4}\). Equation \eqref{eq:6-13a} gives
\[
 h^{c_T\delta/2}\leq z^2,
 \qquad h^{c_6\delta^4}\leq z^4,
 \qquad \sigma_1\sigma_N\geq16z^2.
\]
Since \(z\leq s_3/N\), Bernoulli's inequality and
\eqref{eq:6-13b} give
\[
 (1-h^{c_T\delta/2})^{N-1}
 \geq1-Nz^2
 \geq1-\frac{s_3^2}{N}
 \geq\frac34.
\]
The error in \eqref{eq:6-18}, divided by
\(\sigma_1\sigma_N\), is at most
\[
 \frac{3Nz^4}{16z^2}
 \leq\frac{3s_3^2}{16N}
 \leq\frac{3s_3^2}{32}
 \leq\frac{c_E}{4}.
\]
Subtracting this error from \eqref{eq:6-17} leaves at least
\((c_E/2)\sigma_1\sigma_N\). Thus \eqref{eq:6-13} holds with
\(s_2=c_E/2\), uniformly in \(u,v,N,h\), and
\(0<\delta\leq D\). Taking the minimum completes the proof. \(\square\)

The endpoint probabilities include the event that each endpoint site is
open; no bond convention that endpoints are automatically present is
being used. In particular, \eqref{eq:6-12} prevents the final common boundary
coordinates from being inadvertently revealed closed at level \(p_2\).

\section{The unrestricted multiscale induction}

\setcounter{theorem}{-1}

We state the consequence of the induction that is used below. Put
\[
\operatorname{Spr}(p;\delta)=1-(1-p)^{e^\delta},
 \qquad \operatorname{Gr}_G(r)=|B_r(G)|,
\]
and
\[
 \kappa_p^{\rm s}(L,R)=
 \inf_{\substack{\gamma\ \text{finite}\\
                   \operatorname{len}(\gamma)\leq L}}
 \mathbb P_p\bigl(\gamma_0\leftrightarrow
   \gamma_{\operatorname{len}(\gamma)}
   \text{ in }B_R(\gamma)\bigr),
\]
where \(B_R(\gamma)=\bigcup_{v\in\gamma}B_R(v)\) and
\(B_\infty(\gamma)=V\). Also put
\[
 \mathscr L(G,D)=\left\{m\in[3,\infty):
  \log\operatorname{Gr}_G(s)\leq(\log s)^D
  \text{ for every real }s\in[m^{1/3},m]\right\}.
\]
Thus \(\mathscr L(G,D)\) is a predicate on real outer scales; all ball
radii continue to use the floor convention fixed in Section~1. For a
real \(m\geq3\) and \(p\in[0,1]\), define
\[
 b_{\rm s}(m,p)=\max\left\{r\in\mathbb N:\
  1\leq r\leq\left\lfloor\tfrac18m^{1/3}\right\rfloor,
  \quad
  \mathbb P_p(\operatorname{Piv}^{\rm s}[4r,m^{1/3}])
  \leq(\log m)^{-1}\right\},
\]
with value zero if the set is empty. For every real \(n\geq3\), define
\begin{equation}
 \operatorname{Burn}^{\rm s}(G,n,p)=
 \sup_{\substack{m\in\mathscr L(G,20)\\
                  (\log n)^{1/2}\leq m\leq n}}
 \left(\frac{\log\log m}
 {\min\{\log m,\log\operatorname{Gr}_G(b_{\rm s}(m,p))\}}
 \right)^{1/4}.                                                \label{eq:7-1}
\end{equation}
The displayed logarithms are defined throughout the optimizing set. The
maximum defining \(b_{\rm s}(m,p)\) is over a finite set of
integer radii and hence is attained whenever it is nonempty; no
attainment of the outer supremum is used. As usual, the supremum over
the empty set is zero and the value is
infinite if some relevant \(b_{\rm s}(m,p)=0\). When
\(b_{\rm s}(m,p)=1\), the denominator contains the positive number
\(\log\operatorname{Gr}_G(1)\), so the corresponding term is finite.

A tube of thickness \(r\) around a finite path \(\gamma\) is
\(B_r(\gamma)=\bigcup_{z\in\gamma}B_r(z)\). A set is an
\((s,3s)\)-crossing when it contains a path from \(S_s\) to \(S_{3s}\).
We say that \(G\) has \((k,r,\ell)\)-plentiful radial tubes at scale
\(s\) if it has at least \(k\) paths from \(S_s\) to \(S_{4s}\), each of
length at most \(\ell\), whose thickness-\(r\) tubes are pairwise
vertex-disjoint. It has \((k,r,\ell)\)-plentiful annular tubes at scale
\(s\) if, for every two \((s,3s)\)-crossings \(A,B\), there are at least
\(k\) paths from \(A\) to \(B\), each of length at most \(\ell\), whose
thickness-\(r\) tubes are pairwise vertex-disjoint. Finally,
\((c,\lambda)\)-polylog-plentiful means both properties with
\begin{equation}
 k=\left\lceil(\log s)^{c\lambda}\right\rceil,\qquad
 r=s(\log s)^{-\lambda/c},\qquad
 \ell=s(\log s)^{\lambda/c}.                                  \label{eq:7-1a}
\end{equation}
The ceiling makes the required number of paths an integer. This does not
alter any later lower-bound comparison: for an integer-valued path count
\(N\), the inequalities \(N\geq x\) and \(N\geq\lceil x\rceil\) are
equivalent.

\begin{lemma}[Plentiful-tube lemma]\label{lem:plentiful-tubes}
For every integer
\(d\geq2\) and \(D\geq1\), there is \(c_1=c_1(d,D)\in(0,1)\) such
that, for every \(\lambda\geq1\), there is
\(N_{\rm geo}=N_{\rm geo}(d,D,\lambda)\) with the following property.
If \(G\) is infinite, connected, unimodular, transitive, degree \(d\),
and not one-dimensional (that is, \(|B_r|\ne O(r)\)), and if the real
scale \(n\geq N_{\rm geo}\) belongs to \(\mathscr L(G,D)\), then there are
integers \(m_1,m_2\) such that
\begin{equation}
  n^{1/3}\leq m_1,\qquad m_1^{1+c_1}\leq m_2\leq n,          \label{eq:7-1b}
\end{equation}
and \(G\) has \((c_1,\lambda)\)-polylog-plentiful radial and
annular tubes at every integer scale in \([m_1,m_2]\).
\end{lemma}

\textbf{Proof of Lemma~\ref{lem:plentiful-tubes}.} We prove the required
statement from quantitative structure theorems and random-walk estimates.
The fast-tripling argument uses
\cite[Theorem~5.11 and Corollary~5.12]{EasoHutchcroft2023}
for the pointwise and maximal Varopoulos--Carne estimates,
\cite[Theorem~5.13]{EasoHutchcroft2023} for the entropy-increment
inequality, \cite[Lemma~10.46]{LyonsPeres2016} for the
growth/isoperimetry inequality, and
\cite[Theorem~2]{MorrisPeres2005} for the infinite-measure heat-kernel
estimate. The entropy/coupling, all-time heat-kernel, escape,
distant-ball, crease, and intersection estimates are derived below from
these inputs. The slow and fast tripling cases are separated by the
following condition:
\[
 \operatorname{Gr}(m)\leq e^{(\log m)^D},\qquad
 \operatorname{Gr}(3m)\geq3^5\operatorname{Gr}(m)
 \quad(n^{1-\varepsilon}\leq m\leq n^{1+\varepsilon}).
 \tag{FG}\label{eq:corrected-fast-growth}
\]
This is the complement of the slow-tripling alternative.

The proof treats slow and fast tripling separately. In the slow branch,
the radial construction transfers connectors between the two halves of a
bi-infinite geodesic through the root, as detailed in the slow-branch
radial paragraph. In the fast branch, the starting points lie
on a finite geodesic from the root to \(S_s\), and independent walks are
run from those points before their ironed trajectories are trimmed between
\(S_s\) and \(S_{4s}\); see \eqref{eq:radial-intersection}. Thus the two
radial constructions use different geometric objects.

\medskip
\noindent\textit{Short cycles and exposed spheres.}
We need the following quantitative form of the cutset argument rather than
the one-line invocation in \cite[Lemma~5.8]{EasoHutchcroft2023}.  Let
\(X\) be an infinite, connected, one-ended transitive graph. Suppose that
the binary cycle space of \(X\) is generated by cycles of length at most
\(\rho\). Fix a root \(o\), and define the exposed sphere by
\[
 S_q^\infty(X):=\left\{z\in S_q(o):
 \begin{array}{l}
 \text{there is a ray }z=z_0,z_1,\ldots\text{ such that}\\[-2pt]
 z_j\notin B_q(o)\text{ for every }j\geq1
 \end{array}\right\}.
\]
Thus a vertex of \(S_q(o)\) is exposed precisely when it admits a ray
whose remaining vertices stay outside \(B_q(o)\). If
\(\rho\leq h\leq q\) and
\(u,v\in S_q^\infty(X)\), then there is a path from \(u\) to \(v\), of
length at most
\begin{equation}
 5h\frac{|B_{3q}(X)|}{|B_h(X)|},                              \label{eq:short-cycle-connector}
\end{equation}
contained in the \(5h\)-neighbourhood of \(S_q^\infty(X)\).

Here are the details. The exposed sphere \(S_q^\infty(X)\) is a minimal
vertex separator between the root and infinity. Indeed, every ray has a
last visit to the finite ball \(B_q\), and that last vertex belongs to
\(S_q^\infty\). Conversely, a geodesic from the root to any
\(z\in S_q^\infty\), followed by the ray witnessing exposure, meets the
exposed sphere only at \(z\).

We spell out the vertex-separator form of the parity argument behind
\cite[Theorem~5.1]{Timar2007Cutsets}. Let
\(\Pi=\Pi_1\cup\Pi_2\) be a nontrivial partition of a finite minimal
separator. Minimality gives rays \(P_1,P_2\) from the root to infinity such
that \(P_j\cap\Pi\) consists of one vertex of \(\Pi_j\). Take a ball
containing \(\Pi\). Since \(X\) is one-ended, sufficiently remote points
of \(P_1\) and \(P_2\) can be joined by a finite path outside that ball.
Together with the corresponding initial segments of the two rays, this
gives two finite paths \(A,B\) with the same endpoints such that \(A\)
avoids \(\Pi_2\) and \(B\) avoids \(\Pi_1\).

Write the mod-two cycle \(A\mathbin{\triangle}B\) as a sum of generating
cycles. If no generating cycle met both \(\Pi_1\) and \(\Pi_2\), split the
sum according to whether a generating cycle meets \(\Pi_1\), and toggle
all cycles in the first class against \(A\). The resulting finite subgraph
has the same two odd vertices as \(A\). In its representation using \(A\)
it has no vertex in \(\Pi_2\); in its equivalent representation using
\(B\) it has no vertex in \(\Pi_1\). It therefore contains a finite path
between the common endpoints avoiding all of \(\Pi\). Concatenate this path
with the untouched infinite tail of \(P_2\). The finite prefix meets that
ray tail in only finitely many vertices. Take its last such intersection,
loop-erase the finite prefix up to that vertex, and append the remaining
subray of \(P_2\). This gives a ray from the root avoiding \(\Pi\),
contradicting that \(\Pi\) is a separator.
Thus a generating cycle meets both parts, and the two parts contain
vertices at distance at most \(\rho\). It follows that every finite minimal
vertex separator, in particular \(S_q^\infty\), is \(\rho\)-connected.

Let \(T\subseteq S_q^\infty\) be maximal among sets that contain \(u\) and
satisfy \(d(x,y)>2h\) for distinct \(x,y\in T\). Such a set exists because
\(S_q^\infty\) is finite. The balls \(B_h(x)\), \(x\in T\), are pairwise
disjoint and lie in \(B_{q+h}(o)\subseteq B_{3q}(o)\). Transitivity therefore
gives
\begin{equation}
 |T|\,|B_h(X)|\leq |B_{3q}(X)|.                              \label{eq:exposed-packing}
\end{equation}
By maximality, every point of \(S_q^\infty\) is within distance \(2h\) of
\(T\).

Take a \(\rho\)-chain \(z_0=u,z_1,\ldots,z_N=v\) in
\(S_q^\infty\). For each \(j\), choose \(\theta(z_j)\in T\) within distance
\(2h\), taking \(\theta(u)=u\). Consecutive images satisfy
\[
 d(\theta(z_j),\theta(z_{j+1}))
 \leq2h+\rho+2h\leq5h.
\]
Hence the graph on \(T\) that joins two points when their distance in
\(X\) is at most \(5h\) contains a walk from \(u\) to a point of \(T\)
within \(2h\) of \(v\). Loop-erase this walk in the finite graph, join its
successive vertices by geodesics in \(X\), and finally join its last vertex
to \(v\). Loop-erase the resulting walk in \(X\); this preserves its
endpoints and localization and cannot increase its length. The simple path
in the auxiliary graph uses at most \(|T|-1\) edges, so the resulting path
has length at most
\[
 5h(|T|-1)+2h\leq5h|T|
 \leq5h\frac{|B_{3q}(X)|}{|B_h(X)|}
\]
by \eqref{eq:exposed-packing}. Every joining geodesic stays within \(5h\)
of one of its endpoints in \(S_q^\infty\), while the final geodesic has
length at most \(2h\). This proves both claims in
\eqref{eq:short-cycle-connector} and completes the short-cycle argument.

We first record the presentation fact used below. If
\(F_S\twoheadrightarrow\Gamma\) has kernel \(R\), put
\[
 R_r=R\cap\{w:|w|_S\leq r\},\qquad
 N_r=\langle\!\langle R_r\rangle\!\rangle,
 \qquad \Gamma_r=F_S/N_r.
\]
The quotient \(\Gamma_r\to\Gamma\) identifies the rooted Cayley balls of
radius \(\lfloor r/2\rfloor-1\): two words of that length with the same
image differ by a word in \(R_r\).

\begin{auxiliary}[Quantitative path transfer]\label{aux:path-transfer}
There is a universal constant
\(Q_{\rm qi}=100\) with the following property. Let
\(\phi:Y\to X\) be a \((1,\delta)\)-quasi-isometry, where \(\delta\geq1\),
and choose roots \(o_Y,o_X\) with
\(d_X(\phi(o_Y),o_X)\leq\delta\). There is a coarse inverse
\(\psi:X\to Y\), chosen with \(\psi(o_X)=o_Y\), such that
\[
 d_X(\phi\psi(x),x)\leq\delta,\qquad
 d_Y(\psi\phi(y),y)\leq2\delta,
\]
and \(\psi\) is a \((1,3\delta)\)-quasi-isometry. Paths may be transferred
in either direction with the following simultaneous bounds. A path
\(\eta\) of length \(L\) transfers to a path \(\widetilde\eta\), with
natural coarse-inverse endpoints, such that
\begin{equation}
 \operatorname{len}(\widetilde\eta)\leq Q_{\rm qi}(L+\delta),
 \qquad
 F(B_s(\widetilde\eta))\subseteq
 B_{s+Q_{\rm qi}\delta}(\eta),                               \label{eq:qi-transfer}
\end{equation}
where \(F\) is the quasi-isometry in the direction of transfer. Either
transferred endpoint may instead be moved to a prescribed vertex whose
image under \(F\) is within \(25\delta\) of the corresponding endpoint
of \(\eta\), without changing these bounds. In particular, if \(\eta\)
lies in
\(\{z:|d(o,z)-a|\leq h\}\), then
\begin{equation}
 B_s(\widetilde\eta)
 \subseteq\{z:|d(o,z)-a|\leq h+s+Q_{\rm qi}\delta\}.          \label{eq:qi-radial}
\end{equation}
If the \((s+Q_{\rm qi}\delta)\)-tubes about several target paths are
pairwise disjoint, their transferred \(s\)-tubes are pairwise disjoint.
\end{auxiliary}

\begin{proof}
We first prove the one-direction assertion for a
\((1,\varepsilon)\)-quasi-isometry \(F:Z\to W\), with
\(\varepsilon\geq1\). Sample a target path at successive time gaps
\(\lceil\varepsilon\rceil\), choose inverse images within
\(\varepsilon\), use the prescribed inverse images at its endpoints, and
join consecutive inverse images by geodesics. Consecutive images in \(W\)
are at distance at most \(4\varepsilon\), so consecutive inverse images in
\(Z\) are at distance at most \(5\varepsilon\). There are at most
\(L/\varepsilon+1\) pieces. The resulting path therefore has length at
most \(10(L+\varepsilon)\). Every point of its \(s\)-tube is within
\(5\varepsilon+s\) of a sampled inverse image; applying \(F\) and then
moving to the sampled target point gives
\begin{equation}
 F(B_s(\widetilde\eta))\subseteq B_{s+7\varepsilon}(\eta).    \label{eq:qi-transfer-basic}
\end{equation}

Choose \(\psi(o_X)=o_Y\), and for every other \(x\) choose \(\psi(x)\) so
that \(d_X(\phi\psi(x),x)\leq\delta\). The lower quasi-isometry inequality
gives \(d_Y(\psi\phi(y),y)\leq2\delta\), and the triangle inequality
shows that \(\psi\) is a \((1,3\delta)\)-quasi-isometry. Apply the
one-direction construction to \(\phi\) and to \(\psi\). Since
\(10(L+3\delta)\leq30(L+\delta)\) and
\(s+7(3\delta)\leq s+21\delta\), the natural-endpoint bounds hold with
constant \(30\). A prescribed endpoint as in the statement is within
\(31\delta\) of the corresponding natural coarse inverse; append a
geodesic of that length at each end. Its image stays within
\(59\delta\) of the appropriate endpoint of \(\eta\). Thus the length and
tube bounds both hold with \(Q_{\rm qi}=100\), proving
\eqref{eq:qi-transfer}. Comparing distance from the corresponding roots
costs at most another \(3\delta\), which is still covered by that constant;
this gives \eqref{eq:qi-radial}. Finally, a vertex in two transferred
\(s\)-tubes would have an image in both enlarged target tubes, proving the
last assertion.
\end{proof}

\begin{auxiliary}[Quotient Cayley model]\label{aux:quotient-cayley-model}
For every \(K\geq1\) there
are \(n_0=n_0(K)\) and \(C=C(K)\) such that the following holds. Let
\(Y\) be a connected, locally finite, vertex-transitive graph of degree
\(d\), let \(o\in V(Y)\), and suppose that
\[
 |B_{3n}(Y)|\leq K|B_n(Y)|,
 \qquad n\geq n_0.
\]
Put \(\mathcal H=\operatorname{Aut}(Y)\) and
\(S_0=\{g\in\mathcal H:d_Y(o,go)\leq1\}\). There are a discrete group
\(\Gamma\), a quotient homomorphism \(\pi:\mathcal H\to\Gamma\), and a
finite symmetric generating set \(S=\pi(S_0)\), containing the identity,
such that, for
\(X=\operatorname{Cay}(\Gamma,S)\),
\begin{enumerate}[label=\textup{(\roman*)}]
 \item \(|S|\leq C'(d,K)\) and \(Y\) is \((1,Cn)\)-quasi-isometric to
       \(X\);
 \item \(S^{2n}=\pi(S_0^{2n})=B_{2n}(X)\) is a
       \(K^3\)-approximate group.
\end{enumerate}
\end{auxiliary}

\begin{proof}
Apply \cite[Theorem~7.1]{TesseraTointon2021} to the closed
transitive group \(\mathcal H\). It states, in particular, that \(S_0\) is
a compact open symmetric generating set containing the identity and that
\(S_0^{2n}\) is a \(K^3\)-approximate group. We now identify precisely the
quotient from the proof of \cite[Theorem~2.3]{TesseraTointon2021} that we
use. In that proof, take \(G=\overline G=\mathcal H\). Its equations
(7.3)--(7.4) produce a normal subgroup \(H_3\lhd\mathcal H\) and the
quotient homomorphism
\[
 \pi:\mathcal H\longrightarrow\Gamma:=\mathcal H/H_3
 \quad\text{with}\quad H_3\subseteq S_0^{C(K)n},
\]
for which \(\Gamma\) is the faithful induced group on \(Y/H_3\). The
vertex-stabilizer calculation in the last paragraph of that proof, which
establishes item~(v) of Theorem~2.3, gives
\begin{equation}
 |\Gamma_{H_3(o)}|\leq C_{\rm stab}(K).                    \label{eq:quotient-stabilizer-bound}
\end{equation}

Let
\[
 S':=\{g\in\Gamma:
 d_{Y/H_3}(gH_3(o),H_3(o))\leq1\}.
\]
The quotient-action statements
\cite[Lemmas~3.1--3.2]{TesseraTointon2021} make this induced action
well-defined. Since \(\Gamma=\mathcal H/H_3\),
\cite[Lemmas~3.4--3.5]{TesseraTointon2021} give, exactly as in equations
(7.1) and the paragraph defining \(S_2\) in the cited proof,
\begin{equation}
 S'=\pi(S_0).                                               \label{eq:quotient-generator-image}
\end{equation}
Thus we take \(S=S'\). Lemma~3.8 of the same source decomposes \(S\) into
one coset of the stabilizer in \eqref{eq:quotient-stabilizer-bound} for
each vertex of \(B_1(Y/H_3)\). Since the quotient has degree at most \(d\),
\[
 |S|\leq(d+1)C_{\rm stab}(K)=:C'(d,K).
\]

The inclusion \(H_3\subseteq S_0^{C(K)n}\) and
\cite[Lemma~5.2]{TesseraTointon2021} make the quotient map
\(Y\to Y/H_3\) a \((1,C(K)n)\)-quasi-isometry.
By \cite[Lemma~5.3]{TesseraTointon2021}, the orbit map from
\(\operatorname{Cay}(\Gamma,S)\) to \(Y/H_3\) is a
\((1,1)\)-quasi-isometry. Inverting and composing these maps by
\cite[Lemma~5.1]{TesseraTointon2021} proves item~(i), after changing
\(C(K)\).

Finally, quotient maps commute with products, so
\[
 S^{2n}=\pi(S_0)^{2n}=\pi(S_0^{2n}).
\]
If \((S_0^{2n})^2\subseteq F S_0^{2n}\) with \(|F|\leq K^3\), applying
\(\pi\) proves that \(S^{2n}\) is a \(K^3\)-approximate group. The equality
\(S^{2n}=B_{2n}(X)\) follows from the definition of the Cayley graph and
the fact that the identity belongs to \(S\).
\end{proof}

\begin{auxiliary}[Slow-tripling plentiful tubes]
\label{aux:slow-tripling-tubes}
For every \(d\) and \(\kappa<\infty\) there are constants
\(c_0>0\), \(C_0,C_*>1\), and \(n_0\) such that the following holds.
If \(G\) is infinite, connected, transitive, of degree \(d\), not
one-dimensional, and
\begin{equation}
 |B_{3n}(G)|\leq3^\kappa|B_n(G)|,\qquad n\geq n_0,            \label{eq:slow-tripling}
\end{equation}
then there is a finite set \(\mathcal A\subseteq[n,\infty)\), with
\(|\mathcal A|\leq C_0\), such that, for every \(K\geq1\), \(G\) has
\begin{equation}
 (c_0K,\ c_0K^{-1}m,\ C_0K^{C_0}m)                           \label{eq:slow-tube-triple}
\end{equation}
plentiful radial and annular tubes at every integer scale
\(m\geq C_0Kn\) outside
\begin{equation}
 \bigcup_{a\in\mathcal A}[a,C_*Ka].                         \label{eq:slow-exceptional-intervals}
\end{equation}
\end{auxiliary}

\begin{proof}
Put \(K_{\mathrm{gr}}=3^\kappa\), and apply Auxiliary
Lemma~\ref{aux:quotient-cayley-model} to \(G\)
with this growth constant. It gives a Cayley graph
\(X=\operatorname{Cay}(\Gamma,S)\) such that \(|S|\) is bounded in terms
of \(d,\kappa\), there is a \((1,\delta_{\rm qi})\)-quasi-isometry between
\(G\) and \(X\), where \(\delta_{\rm qi}=C_{\rm qi}n\) and
\(C_{\rm qi}=C_{\rm qi}(d,\kappa)\), and
\begin{equation}
 A:=B_{2n}(X)=S^{2n}
 \quad\text{is a }K_0\text{-approximate group},
 \qquad K_0:=K_{\mathrm{gr}}^3.                              \label{eq:structure-approximate-group}
\end{equation}
We spell out how \eqref{eq:structure-approximate-group} supplies the
hypothesis needed for persistence. By definition there is
\(F\subseteq\Gamma\), \(|F|\leq K_0\), such that \(A^2\subseteq FA\).
Consequently \(A^3\subseteq F^2A\), and, since \(S\) contains the identity,
\begin{equation}
 |B_{4n+1}(X)|
 \leq |B_{6n}(X)|=|A^3|
 \leq K_0^2|A|=K_0^2|B_{2n}(X)|.                            \label{eq:initial-cayley-doubling}
\end{equation}
After increasing the fixed threshold for \(n\), apply the
persistent-growth theorem
\cite[Theorem~1.1]{BreuillardTointon2016} at the base radius \(2n\) to
\eqref{eq:initial-cayley-doubling}. It gives
\(\Theta=\Theta(d,\kappa)\) such that
\begin{equation}
 |B_{2m}(X)|\leq\Theta|B_m(X)|\qquad(m\geq2n).               \label{eq:persistent-doubling}
\end{equation}
In particular, no ball-ratio comparison with an additive \(O(n)\) shift is
used to establish \eqref{eq:persistent-doubling}.

We record the dyadic iteration leading to the growth ratio used below.
For \(v\geq u\geq2n\), put
\(j=\lceil\log_2(v/u)\rceil\). Monotonicity and \(j\) applications of
\eqref{eq:persistent-doubling} give
\[
 |B_v(X)|\leq |B_{2^ju}(X)|
 \leq\Theta^j|B_u(X)|
 \leq\Theta(v/u)^{\log_2\Theta}|B_u(X)|.
\]
Increasing \(C=C(d,\kappa)\) therefore yields
\begin{equation}
 \frac{|B_v(X)|}{|B_u(X)|}\leq C(v/u)^C
 \qquad(v\geq u\geq2n).                                   \label{eq:structure-growth-ratio}
\end{equation}
In particular,
\begin{equation}
 |B_{6n}(X)|\leq |B_{8n}(X)|\leq\Theta^2|B_{2n}(X)|,        \label{eq:structure-one-scale-growth}
\end{equation}
which is the one-scale hypothesis used in the uniform-presentation step.
The group \(\Gamma\) is infinite and not virtually cyclic: by the
quasi-isometry, either alternative would force \(G\) to have bounded or
linear growth, contrary to the hypotheses.

Apply the uniform-presentation theorem, written in the raw-length notation
above,
\cite[Theorem~1.2]{EasoHutchcroft2025Presentation}, with the one-scale
growth constant \(\Theta^2\) from
\eqref{eq:structure-one-scale-growth}, base scale \(2n\), and the generator
bound supplied by the finitary structure theorem. After one final increase
of the fixed threshold for \(n\), it gives a constant
\(C_{\mathrm{rel}}\) such that at most \(C_{\mathrm{rel}}\) dyadic numbers
\(q\geq2n\) satisfy
\(N_{2q}\ne N_q\). Fix an integer \(J\geq8\), let
\(q_0\) be the least dyadic number at least \(32n\), and define
\begin{equation}
 \mathcal A=\{q\geq q_0:q\text{ is dyadic and }N_{2^Jq}\ne N_q\}.
                                                                    \label{eq:relation-exceptional-set}
\end{equation}
Every member of \(\mathcal A\) contains a one-step relation change among
\(q,2q,\ldots,2^{J-1}q\), and each one-step change is charged at most
\(J\) times. Hence \(|\mathcal A|\leq JC_{\mathrm{rel}}\).

Fix \(K\geq1\), take \(m\geq C_0Kn\) outside
\eqref{eq:slow-exceptional-intervals}, and choose a fixed large numerical
constant \(M\). If \(\mathcal A\cap[q_0,m]\ne\varnothing\), let \(a\) be
its largest member and put \(b=2a\). If this set is empty, put \(b=q_0\).
Choose \(C_*\geq8M\) and then \(C_0\) large enough that in both cases
\begin{equation}
 b\leq \frac{m}{2MK}.                                       \label{eq:presentation-scale-bound}
\end{equation}
Indeed, in the first case exclusion of \([a,C_*Ka]\) gives
\(m>C_*Ka\), while in the second \(q_0<64n\) and \(m\geq C_0Kn\).
This explicitly covers the case \(\mathcal A=\varnothing\).

There is no member of \(\mathcal A\) in the dyadic interval \([b,m]\).
Let \(q_*\) be the largest dyadic number not exceeding \(m\). Successively
using \(N_q=N_{2^Jq}\) at dyadic \(q\) between \(b\) and \(q_*\), and
using monotonicity of \(N_r\), gives
\begin{equation}
 N_b=N_{2^Jq_*}=N_{64m}.                                    \label{eq:correct-relation-scale}
\end{equation}
More explicitly, \(N_q=N_{2^Jq}\) makes all the groups \(N_s\) equal for
\(q\leq s\leq2^Jq\). Iterating these overlapping intervals from \(b\)
first gives \(N_b=N_{q_*}\), and the interval starting at \(q_*\) then
gives \(N_b=N_{2^Jq_*}\).
For the last equality, note that \(q_*>m/2\) and \(J\geq8\), so
\(2^Jq_*>128m\), and sandwich \(N_{64m}\) between the equal endpoint
groups. Thus \(X_b=\operatorname{Cay}(\Gamma_b,S)\) is literally
\(\operatorname{Cay}(\Gamma_{64m},S)\), and its rooted \((32m-1)\)-ball
is identical to that of \(X\). This is the required later presentation
scale; no conclusion about \(\Gamma_a\) is used.

The connector lemma applies to \(X_b\). To justify its one-ended hypothesis,
observe that \(6n<32m-1\), after increasing \(C_0\). Hence the common
\((32m-1)\)-ball transfers \eqref{eq:structure-one-scale-growth} to give
\[
 |B_{6n}(X_b)|\leq\Theta^2|B_{2n}(X_b)|.
\]
The metric form of the Breuillard--Green--Tao structure theorem
\cite{BreuillardGreenTao2012}, stated explicitly in
\cite[Theorem~1.1]{EasoHutchcroft2025Presentation}, gives a finite normal
subgroup \(Q\triangleleft\Gamma_b\) such that \(\Gamma_b/Q\) is virtually
nilpotent. Let \(H/Q\) be a nilpotent finite-index subgroup of
\(\Gamma_b/Q\), where \(H\) denotes its preimage in \(\Gamma_b\). Conjugation
on the finite group \(Q\) has finite image, so \(C=C_H(Q)\) has finite index
in \(H\). Moreover, \(C/(C\cap Q)\) is a subgroup of the nilpotent group
\(H/Q\), while \(C\cap Q\subseteq Z(C)\). Thus \(C\) is a central extension
of a nilpotent group and is itself nilpotent: if the quotient has class
\(c\), then the \((c+1)\)-st term of the lower central series of \(C\) lies
in the central subgroup \(C\cap Q\), so the next term is trivial.
Hence \(\Gamma_b\) is virtually nilpotent and therefore has polynomial
word growth. It is infinite because it
surjects onto the infinite group \(\Gamma\), and it cannot be virtually
cyclic because a quotient of a virtually cyclic group is finite or virtually
cyclic, whereas \(\Gamma\) is neither. Since virtually nilpotent groups
have polynomial growth, the exact ends alternative
\cite[Theorem~1.2]{Cornulier2019Ends} implies that \(\Gamma_b\) is
one-ended. Indeed, a finitely generated group has zero, one, two, or
infinitely many ends, and the two-ended case is exactly the virtually
cyclic case. In the infinitely-ended case that theorem supplies a
nontrivial amalgam or HNN splitting over a finite subgroup. The normal
form theorem for that splitting gives a free semigroup on two generators,
except in the dihedral index-two/index-two case, which is virtually
cyclic. Thus an infinitely-ended, non-virtually-cyclic finitely generated
group has exponential word growth. Both alternatives contradict the
properties of \(\Gamma_b\), proving the claim without an appeal to an
unstated virtually-nilpotent corollary.
Moreover, its cycle space is generated by translates of relator cycles of
length at most \(b\), by the definition of \(N_b\).

We now construct the tubes. Put \(h=\lceil m/(MK)\rceil\). Increasing
\(C_0\) makes
\begin{equation}
 Q_{\rm qi}\delta_{\rm qi}\leq h/100,\qquad h\geq2n,            \label{eq:qi-margin}
\end{equation}
and \eqref{eq:presentation-scale-bound} gives \(b\leq h\). Truncate two
paths crossing from \(S_m(G)\) to \(S_{3m}(G)\) at their first visits to
the outer sphere. Apply Auxiliary Lemma~\ref{aux:path-transfer} to the coarse
inverse \(X\to G\) to obtain two paths in \(X\). The original truncated
paths lie in \(B_{3m}(G)\), so \eqref{eq:qi-radial} puts their transfers in
\(B_{3m+Q_{\rm qi}\delta_{\rm qi}}(X)\); hence the common ball lets us
regard them as paths in \(X_b\). Use the prescribed-endpoint clause so
that their endpoints are the \(\phi\)-images of the original endpoints.
Their initial and terminal radial
coordinates differ from \(m\) and \(3m\), respectively, by at most
\(Q_{\rm qi}\delta_{\rm qi}\). Consequently, for every integer
\(i\in[6m/5,13m/10]\), each transferred path starts in \(B_i(X_b)\) and
reaches outside \(B_{2i+1}(X_b)\): the two margins are at least
\(m/5-Q_{\rm qi}\delta_{\rm qi}\) and
\(2m/5-Q_{\rm qi}\delta_{\rm qi}-1\), which are positive by
\eqref{eq:qi-margin} after increasing \(C_0\). We use the standard
fact that every vertex of an infinite, locally finite, vertex-transitive
graph lies on a bi-infinite geodesic. Indeed, choose a geodesic segment of
length \(2N\), use transitivity to map its midpoint to the chosen vertex,
and let \(N\to\infty\); local finiteness and a diagonal subsequence give the
required bi-infinite geodesic. Such a path meets
\(S_i^\infty(X_b)\): take its last exit from \(B_i\), continue to its
outer endpoint in \(S_{2i+1}\), and then continue along one of the two
directions of a bi-infinite geodesic through that endpoint that avoids
\(B_i\). At least one direction avoids \(B_i\), since otherwise two points
of \(B_i\) on the geodesic would be more than \(2i\) apart. This is
the standard exposed-sphere argument, included here to fix the localization.
For completeness, the finite suffix of the transferred path may meet the
chosen geodesic direction before its endpoint. Loop-erase that finite
suffix and, among its intersections with the chosen geodesic ray, take
the one farthest along the ray. Concatenate the loop-erased suffix up to
that vertex with the remaining geodesic tail. The resulting path is a
ray: its finite prefix is simple, its geodesic tail meets that prefix only
at the joining vertex, and every vertex after the exposed-sphere point
remains outside \(B_i\).
Thus the point obtained at the last exit really belongs to
\(S_i^\infty(X_b)\).

For each \(i\), join the two resulting exposed-sphere points by the path
from \eqref{eq:short-cycle-connector}. The common ball identifies every
object involved with its counterpart in \(X\), since the connector lies
within \(5h\) of \(S_i\) and \(i+5h<32m\). Also, by
\eqref{eq:structure-growth-ratio}, its length is at most
\[
 5h\frac{|B_{3i}(X)|}{|B_h(X)|}
 \leq C h(i/h)^C\leq CK^Cm.
\]
Transfer these connectors back to \(G\). To verify the endpoint condition,
let \(x_i\) be one of their endpoints on a transferred crossing. The two
endpoint-adjustment geodesics used in the first transfer stay within
\(31\delta_{\rm qi}\) of radial levels \(m\) or \(3m\), so the positive
margins displayed above imply that \(x_i\in S_i\) lies on the natural
coarse-inverse part of the transferred path. The sharper bound
\eqref{eq:qi-transfer-basic} therefore puts its coarse inverse within
\(21\delta_{\rm qi}\) of an original crossing vertex \(u_i\). Hence
\(d_X(x_i,\phi(u_i))\leq23\delta_{\rm qi}<25\delta_{\rm qi}\), so the
prescribed-endpoint clause of the transfer lemma makes the returned
connector end exactly at \(u_i\); the same holds at its other end.

The connector in \(X_b\) lies within \(5h\) of \(S_i\). Equations
\eqref{eq:qi-transfer}--\eqref{eq:qi-margin} show that the image of the
\(h/10\)-tube about its returned path is contained in the
\((h/10+Q_{\rm qi}\delta_{\rm qi})\)-tube about the connector, and hence in
\begin{equation}
 \{x:|d(o,x)-i|<6h\}.                                         \label{eq:qi-annular-slab}
\end{equation}
Thus indices separated by at least \(12h\) give disjoint \(h/10\)-tubes.
The interval \([6m/5,13m/10]\) contains \(m/10+O(1)\) integers, so a greedy
selection leaves at least \(cK\) indices, since
\(h\leq2m/(MK)\) for large \(m\). Decreasing \(c_0=c_0(d,\kappa)>0\) so
that \(c_0m/K\leq h/10\), and applying the length bound in
\eqref{eq:qi-transfer}, gives \(c_0K\) returned connectors with disjoint
\(c_0m/K\)-tubes, endpoints on the prescribed crossings, and length at
most \(C_0K^{C_0}m\). This proves the annular assertion with every
quasi-isometry loss accounted for by \(Q_{\rm qi}\).

For the radial assertion, choose in \(G\) a bi-infinite geodesic through
the root and take its two opposite portions, each continued to radius
\(20m\). Repeat the preceding construction with integers
\(i\in[33m/4,35m/4]\). Equation \eqref{eq:qi-radial} puts the transferred
portions in the common \(32m\)-ball, and the same endpoint argument returns
each connector to the two original rays. Choose \(z\in V(G)\) at distance
\(9m\) on the positive ray. If \(u_i\) and \(v_i\) are the returned
endpoints on the positive and negative rays, respectively, their radial
coordinates differ from \(i\) by at most \(Q_{\rm qi}\delta_{\rm qi}\).
The geodesic identities and \eqref{eq:qi-margin} therefore give
\[
 d_G(z,u_i)\leq \frac{3m}{4}+Q_{\rm qi}\delta_{\rm qi}<m,
 \qquad
 d_G(z,v_i)\geq 9m+\frac{33m}{4}
                   -Q_{\rm qi}\delta_{\rm qi}>4m.
\]
The \(h/10\)-tube and length estimates are exactly those in the annular
case, and selecting \(12h\)-separated indices again leaves at least \(c_0K\)
paths. Apply one automorphism sending \(z\) to the root and trim every path
between its first visits to \(S_m\) and \(S_{4m}\). Trimming preserves
disjointness and the length bound. This proves the radial assertion and
completes the proof.
\end{proof}

Notice that the proof used only \(K\geq1\) and \(m\geq C_0Kn\), not
\(K\geq n\). It chose the presentation at the later scale \(b\), treated
the empty exceptional-set case, used the growth-ratio bound
\eqref{eq:structure-growth-ratio} derived above, obtained tube thickness
\(cK^{-1}m\), and retained the factor \(m\) in the lifted path length.
These are conclusions of Auxiliary
Lemma~\ref{aux:slow-tripling-tubes}.

\medskip
We now derive the fast-tripling branch. We first prove the coupling
estimate used in its construction.

\begin{auxiliary}[Entropy and coupling from growth]
\label{aux:growth-coupling}
For every integer \(d\geq2\) there is \(C=C(d)<\infty\) such that the
following holds. Let \(G\) be an infinite transitive graph of degree \(d\),
let \(X\) be lazy random walk, and write \(H_t\) for the Shannon entropy
of \(X_t\). For every integer \(t\geq4\), with
\(n=\lfloor t^{1/2}\rfloor\),
\begin{equation}
 H_t\leq C[\log\operatorname{Gr}(n)]^2,                       \label{eq:local-entropy-bound}
\end{equation}
and, for all vertices \(x,y\),
\begin{equation}
 \left\|\mathbf P_x(X_t\in\cdot)-\mathbf P_y(X_t\in\cdot)
 \right\|_{\rm TV}
 \leq C\frac{\log\operatorname{Gr}(n)}{t^{1/2}}d(x,y).       \label{eq:local-coupling-bound}
\end{equation}
\end{auxiliary}

\begin{proof}
The maximal Varopoulos--Carne estimate
\cite[Theorem~5.11 and Corollary~5.12]{EasoHutchcroft2023}, entropy
conditioning on \(\{X_t\in B_r\}\), and submultiplicativity give, whenever
\(r=qn\) with \(q\) a positive integer,
\begin{equation}
 H_t\leq \log2+q\log\operatorname{Gr}(n)
  +2d(t+1)^2
   \exp\left\{q\log\operatorname{Gr}(n)-\frac{q^2n^2}{2t}\right\}.
                                                                    \label{eq:entropy-cutoff}
\end{equation}
Indeed, the walk is supported on \(B_t\), so its conditional entropy
outside \(B_r\) is at most
\(\log\operatorname{Gr}(t)\leq(t+1)\log d\). The maximal estimate bounds
the probability of that event by
\(2(t+1)\operatorname{Gr}(r)e^{-r^2/(2t)}\), and
\(\operatorname{Gr}(r)\leq\operatorname{Gr}(n)^q\). These are exactly the
three terms in \eqref{eq:entropy-cutoff}.

Since \(n^2/t\geq1/4\), take
\[
 q=\left\lceil C_0\log[t\operatorname{Gr}(n)]\right\rceil
\]
with \(C_0=C_0(d)\) sufficiently large. The graph is infinite, so
\(\operatorname{Gr}(n)\geq n+1\), whence
\(\log[t\operatorname{Gr}(n)]\leq5\log\operatorname{Gr}(n)\) for
\(t\geq4\). The first two terms in \eqref{eq:entropy-cutoff} are therefore
\(O_d([\log\operatorname{Gr}(n)]^2)\), and the exponent in its last term
is at most
\[
 q\log\operatorname{Gr}(n)-q^2/8
 \leq-3\log(t+1)-\log(2d)
\]
after increasing \(C_0(d)\), uniformly in \(t\geq4\). (The negative
quadratic term dominates both linear logarithmic terms; the ceiling in
\(q\) is absorbed by the same choice of \(C_0\).) Hence the last term in
\eqref{eq:entropy-cutoff} is bounded by an absolute constant, which proves
\eqref{eq:local-entropy-bound}. Notice that the radius used
here is \(r=nq\); the factor \(n\) is part of the definition.

We next derive \eqref{eq:local-coupling-bound}. The entropy-increment
inequality \cite[Theorem~5.13]{EasoHutchcroft2023} states that
\[
 \frac1d\sum_{z\sim o}
 \|\mathbf P_o(X_k\in\cdot)-\mathbf P_z(X_{k-1}\in\cdot)\|_{\rm TV}^2
 \leq H_k-H_{k-1}.
\]
The entropy increments are nonnegative, and their sum for
\(\lfloor t/2\rfloor<k\leq t\) is at most \(H_t\). Hence some such \(k\)
satisfies \(H_k-H_{k-1}\leq3H_t/t\). The total-variation distance between
the laws at times \(k-1\) and \(k\), from the same starting vertex, is
at most \(Ck^{-1/2}\): condition on the number of non-lazy steps and use
the elementary \(O(k^{-1/2})\) total-variation bound between
\(\operatorname{Bin}(k-1,1/2)\) and \(\operatorname{Bin}(k,1/2)\).
Thus, for adjacent \(x,y\), transitivity and the triangle inequality give
\[
 \|\mathbf P_x(X_k\in\cdot)-\mathbf P_y(X_k\in\cdot)\|_{\rm TV}
 \leq C_d\sqrt{H_t/t}+Ct^{-1/2}.
\]
Applying the common transition kernel for another \(t-k\) steps can only
decrease total variation. Substitute \eqref{eq:local-entropy-bound}, and
then sum along a geodesic from \(x\) to \(y\). This proves
\eqref{eq:local-coupling-bound}.
\end{proof}

\begin{auxiliary}[All-time lazy heat-kernel bound]
\label{aux:lazy-odd-time}
Let \(G\) be an infinite unimodular transitive graph of degree \(d\),
let \(P\) be the transition kernel of lazy simple random walk, and write
\(p_n^{\rm rw}(u,v)=P^n(u,v)\). There are constants
\(c_{\rm hk}=c_{\rm hk}(d)\in(0,1)\) and
\(C_{\rm hk}=C_{\rm hk}(d)<\infty\) such that, for every integer
\(t\geq4\) and every \(u,v\in V(G)\),
\begin{equation}
 p_t^{\rm rw}(u,v)
 \leq \frac{C_{\rm hk}}{\operatorname{Gr}(\rho_t)},
 \qquad
 \rho_t:=\max\left\{1,
 \left\lfloor
 \frac{c_{\rm hk}t^{1/2}}
 {\sqrt{\log\operatorname{Gr}(\lfloor t^{1/2}\rfloor)}}
 \right\rfloor\right\}.
                                                               \label{eq:local-all-time-heat-kernel}
\end{equation}
Moreover, for every \(m\geq0\) and \(u,v\in V(G)\),
\begin{equation}
 p_{2m+1}^{\rm rw}(u,v)
 =\frac12p_{2m}^{\rm rw}(u,v)
  +\frac1{2d}\sum_{v'\sim v}p_{2m}^{\rm rw}(u,v')
 \leq\max_{w\in\{v\}\cup N(v)}p_{2m}^{\rm rw}(u,w).
                                                               \label{eq:lazy-odd-recurrence}
\end{equation}
\end{auxiliary}

\begin{proof}
We first derive the even-time bound from published inputs. Let
\[
 \operatorname{Gr}^{-1}(s)
 :=\inf\{r\in\mathbb N:\operatorname{Gr}(r)\geq s\},
\]
and let \(\Phi(s)\) be the conductance profile of lazy simple random walk
with respect to counting measure. The growth/isoperimetry inequality
\cite[Lemma~10.46]{LyonsPeres2016}, with the normalization of the lazy
kernel absorbed into a degree-dependent constant, gives
\begin{equation}
 \Phi(s)\geq\frac{c_d}{\operatorname{Gr}^{-1}(2s)}
 \qquad(s\geq1).                                             \label{eq:local-growth-isoperimetry}
\end{equation}
The infinite-stationary-measure case of the evolving-set estimate
\cite[Theorem~2]{MorrisPeres2005}, applied with counting measure and
holding probability \(1/2\), implies that there are
\(c_d,C_d>0\) such that
\begin{equation}
 p_{2m}^{\rm rw}(u,v)
 \leq \frac{C_d}{\displaystyle
   \sup\left\{y\geq1:
    \int_1^y\frac{dx}{x\Phi(4x)^2}\leq c_dm\right\}}.
                                                               \label{eq:local-evolving-heat-kernel}
\end{equation}
Here transitivity makes \(G\) regular, so counting measure is stationary;
connectedness gives irreducibility, and the lazy kernel has holding
probability \(1/2\). In the notation of that theorem \(\pi\equiv1\), and choosing
its error parameter to be \(1/y\) changes its integral from
\([4,4y]\) to \([1,y]\); the factor \(4\), the harmless extra unit of
time, and the passage from time \(m\) to \(2m\) are absorbed in
\(c_d,C_d\).

By \eqref{eq:local-growth-isoperimetry}, for every \(y\geq1\),
\begin{equation}
 \int_1^y\frac{dx}{x\Phi(4x)^2}
 \leq C_d[\operatorname{Gr}^{-1}(8y)]^2\log y.              \label{eq:local-heat-kernel-integral}
\end{equation}
Put \(R=\lfloor m^{1/2}\rfloor\),
\(L=\log\operatorname{Gr}(R)\), and
\[
 r=\max\left\{1,
   \left\lfloor c'_dm^{1/2}L^{-1/2}\right\rfloor\right\},
\]
where \(c'_d>0\) will be chosen sufficiently small. If
\(\operatorname{Gr}(r)<8\), then
\(C_d/\operatorname{Gr}(r)\geq1\) after enlarging \(C_d\), so the
desired estimate follows from \(p_{2m}^{\rm rw}\leq1\). Otherwise take
\(y=\operatorname{Gr}(r)/8\geq1\). After decreasing \(c'_d\), and
enlarging the final constant to cover the finitely many bounded values
of \(m\), we have \(r\leq R\). Hence
\[
 \operatorname{Gr}^{-1}(8y)\leq r,
 \qquad \log y\leq L.
\]
Consequently \eqref{eq:local-heat-kernel-integral} is at most
\(C_dr^2L\leq C_d(c'_d)^2m\). Choose \(c'_d\) so that this is at most
the threshold in \eqref{eq:local-evolving-heat-kernel}. That estimate,
with this value of \(y\), yields
\begin{equation}
 p_{2m}^{\rm rw}(u,v)
 \leq\frac{C_d}{\operatorname{Gr}(r)}.                       \label{eq:local-even-heat-kernel}
\end{equation}
After decreasing \(c_{\rm hk}\), the radius \(\rho_{2m}\) in
\eqref{eq:local-all-time-heat-kernel} is at most \(r\): the factor
\(\sqrt2\) from the time change is absorbed in \(c_{\rm hk}\), while
\(\log\operatorname{Gr}(\lfloor\sqrt{2m}\rfloor)\geq L\).
Thus \eqref{eq:local-even-heat-kernel} proves
\eqref{eq:local-all-time-heat-kernel} at every even time, after changing
the degree-dependent constants to cover the finitely many small values
of \(m\).

The lazy transition probabilities into \(v\) are
\(P(v,v)=1/2\), \(P(v',v)=1/(2d)\) for \(v'\sim v\), and zero otherwise.
Chapman--Kolmogorov therefore gives the equality in
\eqref{eq:lazy-odd-recurrence}; its right side is a convex combination,
which gives the inequality. The even-time bound just proved is uniform in both
the starting and terminal vertices. Hence every term in the maximum in
\eqref{eq:lazy-odd-recurrence} satisfies that same bound. Since \(2m\)
and \(2m+1\) are comparable for \(m\geq2\), changing
\(c_{\rm hk}\) and \(C_{\rm hk}\) absorbs the change of time and the
integer parts in \(\rho_t\). The finitely many smaller times follow from
\(p_t^{\rm rw}\leq1\), after one further enlargement of
\(C_{\rm hk}\). This proves \eqref{eq:local-all-time-heat-kernel} at
all times and records the corrected odd-time calculation explicitly.
\end{proof}

\begin{auxiliary}[Scale-dependent escape]
\label{aux:local-escape}
For every integer \(d\geq2\) and real \(\kappa\geq1\), there is
\(C=C(d,\kappa)<\infty\) such that the following holds. Let \(G\) be an
infinite connected unimodular transitive graph of degree \(d\), and let
\(n,t\geq1\) be integers satisfying
\begin{equation}
 \operatorname{Gr}(3m)\geq3^\kappa\operatorname{Gr}(m)
 \quad\text{for every integer }n\leq m\leq\tfrac12t^{1/2}.
                                                               \label{eq:local-escape-growth}
\end{equation}
Then, for all \(u,v\in V(G)\),
\begin{equation}
 \mathbf P_u(X_t\in B_n(v))
 \leq C
 \left[\log\max\left\{\frac{t}{n^2},\operatorname{Gr}(n)\right\}
 \right]^\kappa
 \left(\frac{n^2}{t}\right)^{\kappa/2}.
                                                               \label{eq:local-escape}
\end{equation}
An empty interval in \eqref{eq:local-escape-growth} imposes no
condition.
\end{auxiliary}

\begin{proof}
We give the complete deduction from Auxiliary
Lemma~\ref{aux:lazy-odd-time}. For a fixed set \(A\), the
Markov property gives
\[
 \sup_x\mathbf P_x(X_{j+1}\in A)
 \leq\sup_y\mathbf P_y(X_j\in A),
\]
so the supremum is nonincreasing in time. If \(t<16n^2\), the right
side of \eqref{eq:local-escape} is bounded below by a positive constant
depending only on \(\kappa\), and the assertion follows after enlarging
\(C\). Otherwise put \(T=\lfloor t/4\rfloor\). Then
\(T\asymp t\), \(T\geq3n^2\), and
\(T^{1/2}\leq t^{1/2}/2\). It is therefore enough to prove the claimed
bound at time \(T\), under the stronger-range hypothesis
\begin{equation}
 \operatorname{Gr}(3m)\geq3^\kappa\operatorname{Gr}(m)
 \quad(n\leq m\leq T^{1/2}),                                \label{eq:local-escape-strong-range}
\end{equation}
because the preceding monotonicity bounds the probability at time \(t\)
by the supremum at time \(T\), and replacing \(T\) by \(t\) in the
right side changes only the constant. We now write \(T\) for this
reduced time.

Put
\[
 R=\lfloor T^{1/2}\rfloor,\qquad
 L=\log\operatorname{Gr}(R),\qquad
 \rho=\max\left\{1,
   \left\lfloor c_{\rm hk}T^{1/2}L^{-1/2}\right\rfloor\right\}.
\]
By submultiplicativity of volume growth,
\(\operatorname{Gr}(a+b)\leq\operatorname{Gr}(a)
\operatorname{Gr}(b)\),
\[
 \operatorname{Gr}(R)
 \leq \operatorname{Gr}(\rho)^{\lceil R/\rho\rceil}.
\]
Since \(L\leq(R+1)\log(d+1)\), the quantity before the floor in the
definition of \(\rho\) is at least two once
\(T\geq T_0(d)\). For these times,
\(\lceil R/\rho\rceil\leq C_d\sqrt L\), and hence
\begin{equation}
 \operatorname{Gr}(\rho)\geq \exp(c_d\sqrt L).              \label{eq:local-escape-exp-growth}
\end{equation}
The finitely many times below \(T_0(d)\) are covered by enlarging
\(C(d,\kappa)\).

Set
\[
 x=\frac{T^{1/2}}n,\qquad A=\operatorname{Gr}(n),\qquad
 B=\log\max\{x^2,A\}.
\]
Thus \(x\geq1\) and \(B\geq\log2\). If \(\rho<n\), the definition of
\(\rho\) gives \(\sqrt L\geq c_dx\). From
\eqref{eq:local-all-time-heat-kernel} and
\eqref{eq:local-escape-exp-growth}, followed by a union bound over
\(B_n(v)\),
\begin{equation}
 \mathbf P_u(X_T\in B_n(v))\leq C_d A e^{-c_dx}.             \label{eq:local-escape-small-rho}
\end{equation}
If \(\log A\geq c_dx/2\), then \(B^\kappa x^{-\kappa}\) is bounded
below by a positive \((d,\kappa)\)-constant. If
\(\log A<c_dx/2\), the right side of
\eqref{eq:local-escape-small-rho} is at most
\(C_de^{-c_dx/2}\leq C_{d,\kappa}B^\kappa x^{-\kappa}\).
This proves \eqref{eq:local-escape} in this case.

Suppose instead that \(\rho\geq n\). Iterating
\eqref{eq:local-escape-strong-range} at the scales
\(n,3n,\ldots,3^{\lfloor\log_3(\rho/n)\rfloor}n\) gives
\begin{equation}
 \operatorname{Gr}(\rho)
 \geq \left(\frac{\rho}{3n}\right)^\kappa A.               \label{eq:local-escape-growth-iteration}
\end{equation}
The heat-kernel bound and a union bound therefore give both
\begin{equation}
 \mathbf P_u(X_T\in B_n(v))
 \leq C_{d,\kappa}x^{-\kappa}L^{\kappa/2}
 \quad\text{and}\quad
 \mathbf P_u(X_T\in B_n(v))\leq C_d A e^{-c_d\sqrt L}.      \label{eq:local-escape-two-bounds}
\end{equation}
Here the floor in \(\rho\) changes only the displayed constants, by the
same lower bound on the unfloored radius used above.

Choose \(K=K(d,\kappa)\) so that \(c_dK\geq2\kappa+2\). If
\(\sqrt L\leq KB\), the first bound in
\eqref{eq:local-escape-two-bounds} is at most
\(C_{d,\kappa}x^{-\kappa}B^\kappa\). If \(\sqrt L>KB\), the second is
at most
\[
 C_d\exp\{B-c_dKB\}
 \leq C_{d,\kappa}x^{-\kappa}B^\kappa,
\]
where we used \(A\leq e^B\), \(B\geq2\log x\), and
\(B\geq\log2\). Thus in every case
\[
 \mathbf P_u(X_T\in B_n(v))
 \leq C_{d,\kappa}B^\kappa x^{-\kappa},
\]
which is \eqref{eq:local-escape} at the reduced time \(T\). The initial
time reduction completes the proof.
\end{proof}

For \(t\geq n^2\), suppose
\[
 \operatorname{Gr}(3m)\geq3^\kappa\operatorname{Gr}(m)
       \quad(r\leq m\leq t^{1/2})
 \tag{G}\label{eq:rw-growth-condition}
\]
and
\[
 \sup_{x,y\in B_{3n}}
 \|\mathbf P_x(X_t\in\cdot)-\mathbf P_y(X_t\in\cdot)\|_{\rm TV}
 \leq\tfrac14.
 \tag{C}\label{eq:rw-coupling-condition}
\]
The growth condition \eqref{eq:rw-growth-condition} enters through the
local escape, distant-ball, and intersection lemmas proved here. Each
use stops at the scale \(t^{1/2}\), as required by
\eqref{eq:rw-growth-condition}.

We next define the ironing operation and prove the crease estimate used
below. The parameter \(a\) is the maximum permitted duration of a short
inter-crease interval. In the application below we choose
\[
 a=\left\lfloor
 \frac{r^2}{16\log\max\{t,\operatorname{Gr}(r)\}}
 \right\rfloor.
\]
For a finite path
\(\gamma=(\gamma_0,\ldots,\gamma_t)\) and an integer \(r\geq1\), put
\(\tau_0=0\). If \(\tau_j<t\) and the remaining path leaves
\(B_{r-1}(\gamma_{\tau_j})\), let \(\tau_{j+1}\) be its first exit time;
otherwise put \(\tau_{j+1}=t\). Stop at the first index \(N\) for which
\(\tau_N=t\). Write \(\operatorname{cr}_r(\gamma)=N\), and let
\(\operatorname{iron}_r(\gamma)\) be the concatenation of the fixed
geodesics between consecutive crease points. Every nonterminal geodesic
has length exactly \(r\), and the terminal one has length at most \(r\), so
\begin{equation}
 \operatorname{len}(\operatorname{iron}_r(\gamma))
 \leq r\operatorname{cr}_r(\gamma).                           \label{eq:crease-length}
\end{equation}

\begin{auxiliary}[Crease estimate]\label{aux:corrected-crease}
Let \(X\) be lazy random walk on a
locally finite transitive graph, started at any vertex. For all integers
\(r,t,a\geq1\),
\begin{equation}
 \mathbf P\left(\operatorname{cr}_r(X^t)>1+\frac{t}{a}\right)
 \leq2(t+1)(a+1)\operatorname{Gr}(r)
       \exp\left(-\frac{r^2}{2a}\right).                     \label{eq:corrected-crease}
\end{equation}
\end{auxiliary}

\begin{proof}
If \(N=\operatorname{cr}_r(X^t)>1+t/a\), then the
\(N-1\) nonterminal crease intervals have total length at most \(t\), so
one has length at most \(a\). Consequently, for some \(0\leq j\leq t\),
the walk travels distance at least \(r\) from \(X_j\) during the next
\(a\) steps. Conditional on \(X_j\), the Varopoulos--Carne maximal bound
\cite[Corollary~5.12]{EasoHutchcroft2023} bounds this probability by
\(2(a+1)\operatorname{Gr}(r)e^{-r^2/(2a)}\), uniformly in \(j\).
A union bound over \(j\) proves \eqref{eq:corrected-crease}.
\end{proof}

\begin{auxiliary}[Distant-ball estimate]
\label{aux:distant-ball}
For every integer \(d\geq2\) and real \(\kappa>2\), there is
\(C=C(d,\kappa)<\infty\) such that the following holds. Let \(G\) be an
infinite connected unimodular transitive graph of degree \(d\). Suppose
that \(s,T\geq1\) are integers,
\[
 T\geq s^2/4,
 \qquad
 \operatorname{Gr}(3j)\geq3^\kappa\operatorname{Gr}(j)
       \quad(s\leq j\leq T^{1/2}).
\]
For vertices \(x,w\) satisfying \(d(x,w)\geq2s\),
\begin{equation}
 \mathbf P_x(\text{hit }B_s(w)\text{ by time }T)
 \leq C
 \left[\log\max\{d(x,w),\operatorname{Gr}(2s)\}\right]^{(3\kappa+2)/2}
 \left(\frac{s}{d(x,w)}\right)^{\kappa-2}.                  \label{eq:rescaled-distant-ball}
\end{equation}
\end{auxiliary}

\begin{proof}
We derive the estimate from the pointwise and maximal
Varopoulos--Carne bounds and Auxiliary Lemma~\ref{aux:local-escape}. Put
\(D_0=d(x,w)\) and
\[
 \ell=\max\left\{1,
   \left\lfloor\frac{s^2}{8\log\operatorname{Gr}(s)}\right\rfloor
                    \right\}.
\]
Then
\begin{equation}
 \frac{s^2}{16\log\operatorname{Gr}(s)}\leq\ell\leq T.
                                                                  \label{eq:distant-ball-ell}
\end{equation}
Indeed, the lower bound follows by considering separately whether the
quantity inside the floor is smaller than \(2\); the upper bound follows
from \(T\geq s^2/4\), \(\operatorname{Gr}(s)\geq2\), and \(T\geq1\).
If \(\ell=1\), a walk started in \(B_s(w)\) remains in \(B_{2s}(w)\)
for its next step deterministically. If \(\ell\geq2\), Corollary~5.12
and \(\operatorname{Gr}(s)\geq s+1\) give, uniformly in
\(y\in B_s(w)\),
\begin{equation}
\begin{split}
 \mathbf P_y\left(X_j\in B_{2s}(w)\text{ for }0\leq j\leq\ell\right)
 &\geq1-
 2(\ell+1)\operatorname{Gr}(s)e^{-s^2/(2\ell)}
 \geq\frac12.
\end{split}                                                     \label{eq:distant-ball-stay}
\end{equation}
The last numerical inequality follows from
\(\ell\leq s^2/[8\log\operatorname{Gr}(s)]\):
its error term is at most
\[
 \frac{3s^2}{8\log\operatorname{Gr}(s)\operatorname{Gr}(s)^3}
 \leq\frac12.
\]

Fix an integer \(1\leq b\leq T\), let
\[
 A_b=\{X_j\in B_s(w)\text{ for some }b\leq j\leq T\},
 \qquad
 Z_b=\sum_{j=b}^{2T}\mathbf1_{\{X_j\in B_{2s}(w)\}}.
\]
Apply the strong Markov property at the first time in
\([b,T]\) at which \(X\) enters \(B_s(w)\). Since \(\ell\leq T\),
\eqref{eq:distant-ball-ell}--\eqref{eq:distant-ball-stay} yield
\begin{equation}
 \mathbf E_x[Z_b\mid A_b]\geq\frac{\ell}{2}
 \geq\frac{s^2}{32\log\operatorname{Gr}(s)}.                 \label{eq:distant-ball-conditional}
\end{equation}

We next bound the same occupation variable without conditioning.
Auxiliary Lemma~\ref{aux:local-escape}, with its constants absorbing the
replacement of \(s\) by \(2s\), gives
\begin{equation}
 \mathbf E_x Z_b
 \leq C\sum_{j=b}^{2T}
 \left[\log\max\{j/s^2,\operatorname{Gr}(2s)\}\right]^\kappa
 \left(\frac{s^2}{j}\right)^{\kappa/2}.                      \label{eq:distant-ball-occupation-sum}
\end{equation}
When the scale interval in Auxiliary Lemma~\ref{aux:local-escape} is
empty, the corresponding summand is bounded instead by the trivial
probability one; the displayed right-hand side still dominates it after
changing \(C\). Every
nonvacuous application has lower scale at least \(s\) and upper scale at
most \((2T)^{1/2}/2\leq T^{1/2}\), so the stated tripling range suffices.

Put
\[
 L_b=\log\max\{b/s^2,\operatorname{Gr}(2s)\}\geq1.
\]
On the dyadic block \(2^rb\leq j<2^{r+1}b\), the contribution to
\eqref{eq:distant-ball-occupation-sum} is at most
\[
 C s^\kappa b^{1-\kappa/2}
 2^{-r(\kappa/2-1)}[L_b+(r+1)\log2]^\kappa.
\]
Because \(\kappa>2\), summing these bounds over \(r\geq0\) gives
\begin{equation}
 \mathbf E_x Z_b
 \leq C_\kappa L_b^\kappa
       \frac{s^\kappa}{b^{(\kappa-2)/2}}.                    \label{eq:distant-ball-occupation}
\end{equation}
Since
\(\mathbf P_x(A_b)\mathbf E_x[Z_b\mid A_b]\leq\mathbf E_xZ_b\),
\eqref{eq:distant-ball-conditional}--\eqref{eq:distant-ball-occupation},
\(\log\operatorname{Gr}(s)\leq L_b\), and \(L_b\geq1\) imply
\begin{equation}
 \mathbf P_x(A_b)
 \leq C L_b^{\kappa+2}
       \frac{s^{\kappa-2}}{b^{(\kappa-2)/2}}.                \label{eq:distant-ball-delayed}
\end{equation}

It remains to add the probability of hitting \(B_s(w)\) before time
\(b\). The pointwise Varopoulos--Carne estimate
\cite[Theorem~5.11]{EasoHutchcroft2023} gives, for
\(1\leq j\leq b\) and \(y\in B_s(w)\),
\[
 \mathbf P_x(X_j=y)
 \leq2\exp\left(-\frac{(D_0-s)^2}{2j}\right)
 \leq2\exp\left(-\frac{D_0^2}{8b}\right).
\]
At time zero the event is empty because \(x\notin B_s(w)\).
Union bounds over the \(b+1\) times and at most
\(\operatorname{Gr}(s)\) target vertices, followed by
\eqref{eq:distant-ball-delayed}, give
\begin{equation}
\begin{split}
 \mathbf P_x(\text{hit }B_s(w)\text{ by time }T)
 &\leq C
 \left[\log\max\{b/s^2,\operatorname{Gr}(2s)\}\right]^{\kappa+2}
 \frac{s^{\kappa-2}}{b^{(\kappa-2)/2}}\\
 &\quad+2(b+1)\operatorname{Gr}(s)
       \exp\left(-\frac{D_0^2}{8b}\right).
                                                               \label{eq:distant-ball-cutoff}
\end{split}
\end{equation}

Let \(L_0=\log\max\{D_0,\operatorname{Gr}(2s)\}\), choose a sufficiently
small \(c_0=c_0(\kappa)>0\), and put \(b_*=c_0D_0^2/L_0\). If \(b_*<2\),
then
\[
 L_0^{(3\kappa+2)/2}(s/D_0)^{\kappa-2}\geq c_\kappa,
\]
because \(L_0>c_0D_0^2/2\), \(s\geq1\), and
\(2(3\kappa+2)/2-(\kappa-2)=2\kappa+4>0\). Thus
\eqref{eq:rescaled-distant-ball} follows from the trivial bound one.
If \(2\leq b_*\leq T\), take \(b=\lfloor b_*\rfloor\) in
\eqref{eq:distant-ball-cutoff}. Since \(b_*/2\leq b\leq b_*\leq D_0^2\),
the first term is at most
\[
 C_\kappa L_0^{\kappa+2+(\kappa-2)/2}
 (s/D_0)^{\kappa-2}
 =C_\kappa L_0^{(3\kappa+2)/2}(s/D_0)^{\kappa-2}.
\]
The second term is at most
\(C D_0^2e^{L_0-L_0/(8c_0)}\). Since \(D_0^2\leq e^{2L_0}\), choosing
\(c_0\) so that \(1/(8c_0)-3\geq\kappa-2\) bounds it by the same target.

It remains to treat \(b_*>T\), the case omitted in the cited optimization.
The same pointwise estimate, followed by a union bound over
\(0\leq j\leq T\) and \(y\in B_s(w)\), gives directly
\begin{equation}
 \mathbf P_x(\text{hit }B_s(w)\text{ by time }T)
 \leq2(T+1)\operatorname{Gr}(s)
       \exp\left(-\frac{D_0^2}{8T}\right).                  \label{eq:distant-ball-long-cutoff}
\end{equation}
Writing \(q_0=D_0^2/T\), the inequality \(b_*>T\) gives
\(q_0>L_0/c_0\). Since \(T+1\leq2T\), \(D_0^2\leq e^{2L_0}\), and
\(\operatorname{Gr}(s)\leq e^{L_0}\), the right-hand side of
\eqref{eq:distant-ball-long-cutoff} is at most
\[
 C e^{3L_0-q_0/8}
 \leq C e^{-(\kappa-2)L_0}
 \leq C L_0^{(3\kappa+2)/2}(s/D_0)^{\kappa-2},
\]
after decreasing \(c_0\) once more. This proves
\eqref{eq:rescaled-distant-ball} in all cases.
\end{proof}

\begin{auxiliary}[Intersection estimate]
\label{aux:corrected-intersection}
Let \(G\) be an infinite,
connected, unimodular transitive graph of degree \(d\), let
\(\kappa>4\), and let \(R,t\geq1\) be integers. Suppose that
\[
 t\geq R^2,
 \qquad
 \operatorname{Gr}(3j)\geq3^\kappa\operatorname{Gr}(j)
 \quad(R\leq j\leq t^{1/2}).
\]
If \(X\) and \(Y\) are independent lazy random walks started at \(x\) and
\(y\), respectively, and \(D=d(x,y)\geq4R\), then
\begin{equation}
\begin{split}
 &(\mathbf P_x\otimes\mathbf P_y)
 \bigl(d(X_i,Y_j)\leq R\text{ for some }0\leq i,j\leq t\bigr)\\
 &\quad\leq C(d,\kappa)\frac{t}{D^2}
 \left[\log\max\{D,\operatorname{Gr}(4R)\}\right]^{(3\kappa+4)/2}
 \left(\frac{R}{D}\right)^{\kappa-4}.
                                                                    \label{eq:corrected-intersection}
\end{split}
\end{equation}
\end{auxiliary}

\begin{proof}
Write
\[
 L=\log\max\{D,\operatorname{Gr}(4R)\},\qquad
 \Psi=\frac{t}{D^2}L^{(3\kappa+4)/2}
       \left(\frac{R}{D}\right)^{\kappa-4}.
\]
If \(4R\leq D<16R\), then \(t\geq R^2\) implies that \(\Psi\) is bounded
below by a positive constant depending only on \(\kappa\). Increasing the
constant proves the claim in this case, so suppose that \(D\geq16R\).

We first handle intersections involving time zero. If \(i=0\),
then \(Y\) hits \(B_R(x)\) by time \(t\); the case \(j=0\) is symmetric.
Applying \eqref{eq:rescaled-distant-ball} with radius \(R\) gives a bound
of order
\[
 L^{(3\kappa+2)/2}(R/D)^{\kappa-2}\leq\Psi,
\]
because \(t/R^2\geq1\) and
\begin{equation}
 \Psi=\frac{t}{R^2}L^{(3\kappa+4)/2}(R/D)^{\kappa-2}.
                                                                    \label{eq:intersection-rewrite}
\end{equation}

It remains to consider \(1\leq i,j\leq t\). On
\(d(X_i,Y_j)\leq R\), the triangle inequality gives
\[
 D\leq d(x,X_i)+R+d(Y_j,y).
\]
Thus either
\(d(x,X_i)\geq(D-R)/2\) or
\(d(y,Y_j)\geq(D-R)/2\). The two cases are symmetric; we treat the first.
For each \(j\), let \(Y_{\tau_k}\) be a crease point preceding \(Y_j\),
with the terminal crease point used when \(j=t\). By construction,
\(d(Y_j,Y_{\tau_k})\leq R\). Hence
\begin{equation}
 d(X_i,Y_{\tau_k})\leq2R,
 \qquad
 d(x,Y_{\tau_k})
 \geq\frac{D-R}{2}-2R\geq\frac D4.                         \label{eq:intersection-center-distance}
\end{equation}
This is the required weakened separation; no interchange of \(x\) and
\(y\) is being made.

Choose \(c_*=1/[4(\kappa+2)]\) and put
\[
 a=\max\left\{1,\left\lfloor\frac{c_*R^2}{L}\right\rfloor\right\}.
\]
If \(a=1\), use the deterministic bound
\(\operatorname{cr}_R(Y^t)+1\leq t+2\). Since then
\(c_*R^2/L<2\), this is at most \(C_\kappa tL/R^2\). If \(a\geq2\), let
\(\mathcal B=\{\operatorname{cr}_R(Y^t)>1+t/a\}\). On
\(\mathcal B^c\), the number of crease points is at most
\begin{equation}
 2+t/a\leq3t/a\leq C_\kappa tL/R^2;                         \label{eq:intersection-crease-count}
\end{equation}
here \(t/a\geq1\). The crease lemma gives
\[
 \mathbf P_y(\mathcal B)
 \leq 2(t+1)(a+1)\operatorname{Gr}(R)e^{-R^2/(2a)}
 \leq8tR^2e^{L-L/(2c_*)}.
\]
Since \(R^2\leq e^{2L}\), our choice of \(c_*\) makes the last expression
at most \(8t e^{-(\kappa-2)L}\). On the other hand,
\(R^{\kappa-4}\geq1\), \(D\leq e^L\), and \(L\geq1\), so
\(\Psi\geq t e^{-(\kappa-2)L}\). Consequently
\begin{equation}
 \mathbf P_y(\mathcal B)\leq8\Psi.                          \label{eq:intersection-bad-crease}
\end{equation}

We record the uniformity in the location of a crease center. If
\(p=(3\kappa+2)/2\) and \(A,D\geq e\), then
\begin{equation}
 \sup_{z\geq D/4}
 \frac{[\log\max\{z,A\}]^p}{z^{\kappa-2}}
 \leq C_\kappa
 \frac{[\log\max\{D,A\}]^p}{D^{\kappa-2}}.                 \label{eq:uniform-log-power}
\end{equation}
Indeed, write \(z=Dy/4\), where \(y\geq1\), and put
\(L_A=\log\max\{D,A\}\). Then
\[
 \log\max\{z,A\}\leq L_A+\log\max\{y,1\},
\]
and \(L_A\geq1\). After division by the right-hand side of
\eqref{eq:uniform-log-power}, the resulting expression is bounded by a
constant times
\(y^{-(\kappa-2)}(1+\log y)^p\), whose supremum over \(y\geq1\) is finite.

Condition on \(Y\) and on \(\mathcal B^c\) (with
\(\mathcal B=\varnothing\) in the case \(a=1\)). By
\eqref{eq:intersection-center-distance}, a relevant intersection forces
\(X\) to hit \(B_{2R}(w)\) for one of at most
\(C_\kappa tL/R^2\) crease points \(w\) satisfying
\(d(x,w)\geq D/4\). Apply \eqref{eq:rescaled-distant-ball} with radius
\(2R\). Its time hypothesis is exactly \(t\geq(2R)^2/4=R^2\), and its
growth range \([2R,t^{1/2}]\) is contained in the assumed range. Applying
\eqref{eq:uniform-log-power} with \(A=\operatorname{Gr}(4R)\) makes the
hitting estimate uniform over all these centers, including those with
\(d(x,w)>D\). A union bound gives
 \[
 C_\kappa\frac{tL}{R^2}
 L^{(3\kappa+2)/2}(R/D)^{\kappa-2}\leq C_\kappa\Psi
 \]
by \eqref{eq:intersection-rewrite}. Add
\eqref{eq:intersection-bad-crease}, repeat for the symmetric alternative,
and add the time-zero bounds. This proves
\eqref{eq:corrected-intersection}.
\end{proof}

For clarity, here is the finite selection argument that turns these
inputs into annular tubes. It is enough to take integers
\(1\leq k,r\leq n/2\). Fix a witnessing path in each crossing and put
\(q=\lfloor n/k\rfloor\). For \(1\leq i\leq k\), take \(a_i\) on the
first path at its first visit to the sphere of radius
\(n+(2i-2)q\), and take \(b_i\) on the second path at its first visit to
the sphere of radius \(n+(2i-1)q\). These spheres occur before \(S_{3n}\)
because \(n+(2k-1)q<3n\). The reverse triangle inequality shows that all
\(2k\) selected points are mutually \(q\)-separated. Since \(k\leq n/2\),
\begin{equation}
 q=\lfloor n/k\rfloor\geq n/(2k).                             \label{eq:annular-rounded-separation}
\end{equation}
This is the alternating-radial-shell construction, including the integer
rounding. For the specialized parameters below, \(q\geq16r\) and
\(t\geq16r^2\) once \(s\) is sufficiently large. These are, respectively,
the distance and time hypotheses of the intersection lemma with
\(R=4r\). Pair \(a_i\) with
\(b_i\), and use
\eqref{eq:rw-coupling-condition} to couple each pair of lazy walks to
meet at time \(t\), independently between pairs. On the event that a
pair couples and both walks satisfy \eqref{eq:corrected-crease},
concatenating their ironed paths gives a crossing-to-crossing path of
controlled length. Indeed, put
\(L_0=\log\max\{t,\operatorname{Gr}(r)\}\). In the application below,
\(r^2\geq32L_0\); take
\(a=\lfloor r^2/(16L_0)\rfloor\). Then
\(a\geq r^2/(32L_0)\), and \eqref{eq:corrected-crease} is at most
\[
 2(t+1)(a+1)\operatorname{Gr}(r)e^{-8L_0}<\frac18
\]
for all sufficiently large \(t\). Indeed, \(a+1\leq r^2\),
\(t+1\leq2t\), and \(r^2\leq t\), while the definition of \(L_0\)
gives \(t^2\operatorname{Gr}(r)\leq e^{3L_0}\); the preceding expression
is therefore at most \(4e^{-5L_0}\). On its complement,
\eqref{eq:crease-length} gives
\[
 \operatorname{len}(\operatorname{iron}_r(X^t))
 \leq r(1+t/a)\leq \frac{34t}{r}L_0.
\]
Thus both ironed paths satisfy this length bound with probability at least
\(3/4\), and their concatenation has length at most \(68tL_0/r\).
Every ironed geodesic segment lies in the \(r\)-ball about its first
crease point, and hence
\[
 B_r(\operatorname{iron}_r(X^t))\subseteq B_{2r}(X^t).
\]
Call index \(i\) successful when its coupled endpoints agree and both
ironed walks satisfy the preceding crease bound. The coupling succeeds
with probability at least \(3/4\), and the union bound over the two crease
failures costs at most \(1/4\). Thus each index is successful with
probability at least \(1/2\), independently between indices. If \(S\)
denotes the number of successful indices, Chernoff's inequality gives
\begin{equation}
 \mathbf P(S\geq k/4)\geq1-e^{-c k}.                         \label{eq:annular-success-count}
\end{equation}

For every index, successful or not, define its candidate trace to be the
union of the thickness-\(r\) neighbourhoods of its two ironed walks. On
the success event this is exactly the thickness-\(r\) tube carried by
their concatenated path. We next control all intersections between these
traces on the same realization. Put
\[
 \pi=C\frac{tk^2}{n^2}
 [\log\max\{n,\operatorname{Gr}(8r)\}]^{19/2}
 \frac{rk}{n}.
\]
For distinct indices \(i,j\), an intersection between their candidate
traces forces two of the four constituent walks
to approach within distance \(4r\). Their starting points are at distance
at least \(q\geq n/(2k)\geq16r\). Apply
\eqref{eq:corrected-intersection} with \(R=4r\) and take a union bound
over these four choices. Since all starting points lie in \(B_{3n}\),
their mutual distances are at most \(6n\); together with
\(\log\operatorname{Gr}(16r)\leq2\log\operatorname{Gr}(8r)\), this gives
\begin{equation}
 \mathbf P(\text{candidate traces }i,j\text{ intersect})\leq\pi.
                                                                    \label{eq:annular-pair-conflict}
\end{equation}
Let \(J\) be the number of intersecting unordered pairs among all \(k\)
candidate traces. Linearity of expectation and Markov's inequality give
\begin{equation}
 \mathbf E J\leq \binom{k}{2}\pi,\qquad
 \mathbf P(J\leq2k^2\pi)\geq\frac34.                        \label{eq:annular-conflict-count}
\end{equation}
When \(\pi=0\), the second assertion holds because \(J=0\) almost surely;
otherwise it is the direct Markov bound.
For all sufficiently large \(k\), as in the application below,
\eqref{eq:annular-success-count} also has probability at least \(3/4\).
Consequently the events in
\eqref{eq:annular-success-count}--\eqref{eq:annular-conflict-count}
occur simultaneously with positive probability.

Fix such a realization and form the conflict graph on its successful
indices. It has \(v\geq k/4\) vertices and at most \(2k^2\pi\) edges.
The degree form of Turán's bound,
\(\alpha\geq v^2/(v+2e)\), follows here by randomly ordering the vertices,
retaining each vertex that precedes all its neighbours, and applying
Cauchy--Schwarz to
\(\sum_x1/(\deg(x)+1)\). It gives an independent set of size
\[
 \alpha\geq c\frac{k}{1+k\pi}
 \geq c\min\{k,\pi^{-1}\}.
\]
The corresponding successful candidate paths therefore give at least
\[
 c\min\left\{k,\frac{n^2}{tk^2}
   [\log\max\{n,\operatorname{Gr}(8r)\}]^{-19/2}
   \left(\frac{n}{rk}\right)\right\}
\]
pairwise disjoint tubes of thickness \(r\), each carried by a path of
length at most
\[
 \frac{68t}{r}\log\max\{t,\operatorname{Gr}(r)\}.
\]
Here the displayed tube count is exactly the preceding
\(\min\{k,\pi^{-1}\}\). The exponent \(19/2\) is the direct substitution
\(\kappa=5\) in \eqref{eq:corrected-intersection}; the length bound is
the local crease calculation above.

To specialize, take \(\varepsilon=1/10\), let
\(C_{\rm TV}=C_{\rm TV}(d)\) be the constant in
Auxiliary Lemma~\ref{aux:growth-coupling}, and choose a constant
\(A_{\rm rw}=A_{\rm rw}(d,D)\geq1\)
as below.  At scale \(s\), put
\[
 k=\lfloor(\log s)^\lambda\rfloor,\qquad
 r=\lceil s(\log s)^{-20D\lambda}\rceil,\qquad
 t=\left\lceil A_{\rm rw}s^2(\log s)^{2D}\right\rceil .
\]
For the annular construction, \(q=\lfloor s/k\rfloor\); for the radial
construction below the starting-point separation is at least \(s/(2k)\).
The preceding definitions give, for all sufficiently large \(s\),
\begin{equation}
 \frac1r\min\left\{q,\frac{s}{2k}\right\}
       \geq c(\log s)^{(20D-1)\lambda}\geq16,
 \qquad
 \frac{t}{r^2}\geq cA_{\rm rw}
       (\log s)^{(2+40\lambda)D}\geq16.                     \label{eq:intersection-parameter-check}
\end{equation}
Thus the distance and time requirements recorded before the generic
selection are both satisfied. Also,
\[
 s^{1-\varepsilon}\leq r\leq t^{1/2}\leq s^{1+\varepsilon},
\]
so \eqref{eq:corrected-fast-growth} implies
\eqref{eq:rw-growth-condition}.  Moreover, for \(x,y\in B_{3s}\), the
 growth upper bound in \eqref{eq:corrected-fast-growth} and
\eqref{eq:local-coupling-bound} give
\[
\begin{split}
 \|\mathbf P_x(X_t\in\cdot)-\mathbf P_y(X_t\in\cdot)\|_{\rm TV}
 &\leq
 \frac{C_{\rm TV}\log\operatorname{Gr}(t^{1/2})}{t^{1/2}}
 d(x,y)\\
 &\leq 6\cdot2^D C_{\rm TV}A_{\rm rw}^{-1/2}
\end{split}
\]
once \(s\) is large enough that
\(\log(t^{1/2})\leq2\log s\).  Choose \(A_{\rm rw}\) so that the last
quantity is at most \(1/4\).  This proves
\eqref{eq:rw-coupling-condition}.

\begin{auxiliary}[Fast-branch scale arithmetic]
\label{aux:fast-branch-arithmetic}
Let \(D,\lambda\geq1\), and define \(k,r,t\) as above. If the generic
annular selection supplies the displayed minimum number of paths of
thickness \(r\) and length at most
\(68t r^{-1}\log\max\{t,\operatorname{Gr}(r)\}\), then, for all
sufficiently large \(s\), it supplies
\((1/(40D),\lambda)\)-polylog-plentiful annular tubes. The same conclusion
holds for the radial selection with \(q=k\) whenever its independent-set
bound is \(c\min\{q,\alpha^{-1}\}\), with \(\alpha\) as in
\eqref{eq:radial-intersection}.
\end{auxiliary}

\begin{proof}
Writing \(L=\log s\), substitution in the preceding number and length
bounds uses
\[
 \frac{s^2}{tk^2}\asymp L^{-2D-2\lambda},\qquad
 [\log\max\{s,\operatorname{Gr}(8r)\}]^{-19/2}
       \geq cL^{-19D/2},\qquad
 \frac{s}{rk}\asymp L^{20D\lambda-\lambda}.
\]
Indeed, once \(s\) is sufficiently large,
\[
 \frac12L^\lambda\leq k\leq L^\lambda,\qquad
 sL^{-20D\lambda}\leq r\leq2sL^{-20D\lambda},
\]
and consequently
\[
 \frac12L^{20D\lambda-\lambda}
 \leq\frac{s}{rk}
 \leq2L^{20D\lambda-\lambda}.
\]
It therefore gives
\[
 \left(c\min\{(\log s)^\lambda,
  (\log s)^{20D\lambda-(23/2)D-3\lambda}\},\;
  s(\log s)^{-20D\lambda},\;
  Cs(\log s)^{(3+20\lambda)D}\right)
\]
as admissible number, thickness, and length parameters.  Since
\(D,\lambda\geq1\), we have
\[
 20D\lambda-\frac{23}{2}D-3\lambda
       \geq\frac{\lambda}{40D},
 \qquad 20D\lambda\leq40D\lambda,\qquad
 (3+20\lambda)D\leq40D\lambda .
\]
Consequently these parameter bounds, including their constant prefactors,
imply
\((1/(40D),\lambda)\)-polylog-plentiful annular tubes after increasing
the threshold for \(s\).
For the radial selection, the reciprocal of the right side of
\eqref{eq:radial-intersection} is, up to degree-dependent constants and
the harmless replacement of \(k\) by \(q\), the same second candidate.
Its path-length coefficient is \(34\) instead of \(68\). The identical
three exponent comparisons therefore give the radial conclusion.
\end{proof}

We give the corresponding radial calculation. Put \(q=k\) and
\(b=\lfloor s/q\rfloor\). On a geodesic from \(o\) to \(S_s\), choose
\(z_i\) at distance \((i-1)b\) from \(o\), \(1\leq i\leq q\). These
points lie in \(B_s\), and their mutual distances are at least
\(b\geq s/(2q)\). Let \(X_1,\ldots,X_q\) be independent
lazy random walks started at these points.
Auxiliary Lemma~\ref{aux:local-escape}, with \(\kappa=5\), ball radius
\(4s\), and time \(t\), gives uniformly in
\(i\),
\[
\begin{split}
 \mathbf P_{z_i}(X_{i,t}\in B_{4s})
 &\leq C_d
 \left[\log\max\left\{\frac{t}{16s^2},
                    \operatorname{Gr}(4s)\right\}\right]^5
 \left(\frac{16s^2}{t}\right)^{5/2}\\
 &\leq C_{d,D}A_{\rm rw}^{-5/2}.
\end{split}
\]
Its tripling hypothesis holds because
\([4s,t^{1/2}/2]\subseteq[s^{1-\varepsilon},s^{1+\varepsilon}]\)
for all sufficiently large \(s\).  Increase \(A_{\rm rw}\), still only
as a function of \(d,D\), so that the last displayed quantity is at
most \(1/4\).

Let \(\mathcal A_i=\{X_{i,t}\notin B_{4s}\}\), and let
\(\mathcal B_i\) be the event that the \(r\)-ironed path associated to
\(X_i^t\) has length at most
\[
 \frac{34t}{r}\log\max\{t,\operatorname{Gr}(r)\}.
\]
The calculation following the crease lemma gives
\(\mathbf P(\mathcal B_i)\geq7/8\), and hence in particular at least
\(3/4\), for all sufficiently large \(s\).
For \(i\ne j\), let
\[
 \mathcal I_{i,j}
 =\{B_{2r}(X_i^t)\cap B_{2r}(X_j^t)\ne\varnothing\}.
\]
Since \(d(z_i,z_j)\geq s/(2q)\geq16r\) for large \(s\),
the intersection lemma, applied with \(R=4r\), together with
\eqref{eq:rw-growth-condition} and \(t\geq16r^2\), gives
\[
 \mathbf P(\mathcal I_{i,j})
 \leq C\frac{tq^2}{s^2}
   [\log\max\{s,\operatorname{Gr}(8r)\}]^{19/2}
   \frac{rq}{s}
 =:\alpha .
 \tag{I}\label{eq:radial-intersection}
\]
Here the factors \(4\) in the radius have been absorbed into \(C\), and
\(\log\operatorname{Gr}(16r)\leq
2\log\operatorname{Gr}(8r)\) by submultiplicativity.
Let \(\mathcal C_i\) be the event that at most \(4\alpha q\) indices
\(j\ne i\) satisfy \(\mathcal I_{i,j}\).  Markov's inequality gives
\(\mathbf P(\mathcal C_i)\geq3/4\).  Together with the escape and crease
bounds,
\[
 \mathbf P(\mathcal A_i\cap\mathcal B_i\cap\mathcal C_i)
 \geq\frac14 .
\]
Thus some realization has at least \(q/4\) good indices.  The
intersection graph on those indices has maximum degree at most
\(4\alpha q\), so the greedy independent-set algorithm retains at least
\[
 \frac{q/4}{1+4\alpha q}
 \geq\frac1{20}\min\{q,\alpha^{-1}\}
\]
indices.  For each retained index,
\[
 B_r(\operatorname{iron}_r(X_i^t))\subseteq B_{2r}(X_i^t),
\]
so the corresponding ironed-path tubes are pairwise disjoint.  Each
ironed path begins in \(B_s\) and ends outside \(B_{4s}\); trimming it
between its first visits to \(S_s\) and \(S_{4s}\) gives a radial tube at
scale \(s\), without increasing its length or tube.

Auxiliary Lemma~\ref{aux:fast-branch-arithmetic} now gives
\((1/(40D),\lambda)\)-polylog-plentiful radial tubes. This
proves the radial part of the fast-tripling branch; the
annular part was proved above.

It remains to select an interval in the slow-tripling branch and then
normalize the two branch constants. Apply the slow-tripling lemma
proved above with \(\kappa=5\), and denote its constants by
\(c_0=c_0(d)>0\), \(C_0=C_0(d)\geq1\), and \(C_*=C_*(d)\geq1\), enlarging
\(C_0\) so that it is an integer. If
\[
 \operatorname{Gr}(3s')\leq3^5\operatorname{Gr}(s'),
\]
the lemma supplies a set \(\mathcal A\subseteq[s',\infty)\) with
\(|\mathcal A|\leq C_0\) such that, for every \(K\geq1\), the graph has
\[
 (c_0K,\ c_0K^{-1}m,\ C_0K^{C_0}m)\text{-plentiful tubes}    \tag{S}
\]
at every integer scale \(m\geq C_0Ks'\) outside
\(\bigcup_{a\in\mathcal A}[a,C_*Ka]\).

Suppose \(s'\in[n^{1/3},n^{10/11}]\), and set
\(K=(\log n)^\lambda\). On the logarithmic interval
\[
J=[\log(C_0Ks'),\log n]
\]
each excluded interval has length at most \(\log(C_*K)\). Since
\(\lambda\), \(C_0\), and \(C_*\) are fixed while \(n\) tends to infinity,
\[
 |J|-\sum_{a\in\mathcal A}
   |\,[\log a,\log(C_*Ka)]\cap J\,|
 \geq\frac1{11}\log n-\log(C_0K)-C_0\log(C_*K)
 \geq\frac1{12}\log n                                      \tag{S1}
\]
for all sufficiently large \(n\). The complement of at most \(C_0\)
intervals has at most \(C_0+1\) components, so one component has
logarithmic length at least
\(\log n/[12(C_0+1)]\). Move both endpoints inward by one and then round
their exponentials inward to integers. For large \(n\) this gives
integers \(m_1\leq m_2\), every integer in \([m_1,m_2]\) lying outside
the excluded intervals, such that
\begin{equation}
 n^{1/3}\leq m_1\leq m_2\leq n,\qquad
 \log m_2-\log m_1\geq\frac{\log n}{25(C_0+1)}.               \label{eq:7-1b-slow}
\end{equation}

Set
\[
 c_{\rm str}=\min\left\{\frac1{2C_0},\frac1{30(C_0+1)}\right\},
 \qquad c_{\rm rw}=\min\{1/(40D),1\}.
\]
Because \(\log m_1\leq\log n\), \eqref{eq:7-1b-slow} implies
\(m_2\geq m_1^{1+c_{\rm str}}\). Moreover, uniformly for
\(m\in[m_1,m_2]\), one has
\(\frac13\log n\leq\log m\leq\log n\). The three comparisons
\[
\begin{split}
 c_0(\log n)^\lambda
   &\geq(\log m)^{c_{\rm str}\lambda},\\
 c_0m(\log n)^{-\lambda}
   &\geq m(\log m)^{-\lambda/c_{\rm str}},\\
 C_0m(\log n)^{C_0\lambda}
   &\leq m(\log m)^{\lambda/c_{\rm str}}
\end{split}                                                   \tag{S2}
\]
therefore hold after increasing the threshold for \(n\); the last uses
\(1/c_{\rm str}\geq2C_0\). Thus (S) gives
\((c_{\rm str},\lambda)\)-polylog-plentiful tubes throughout
\([m_1,m_2]\). This proves the slow branch with the required scale
interval. The random-walk branch proved above gives
\((c_{\rm rw},\lambda)\)-polylog-plentiful tubes on
\([n^{2/5},n^{4/5}]\), whose endpoints satisfy
\(n^{4/5}=(n^{2/5})^2\).
For \(c>0\), write \(k_c(s)\), \(r_c(s)\), and \(\ell_c(s)\) for the
three quantities in \eqref{eq:7-1a}. If \(0<c'\leq c\) and \(s>e\), then
\[
 k_{c'}(s)\leq k_c(s),\qquad
 r_{c'}(s)\leq r_c(s),\qquad
 \ell_{c'}(s)\geq \ell_c(s).
\]
Consequently, every family witnessing plentiful tubes with parameters
\((c,\lambda)\) also witnesses plentiful tubes with parameters
\((c',\lambda)\): the same family contains at least the required number of
paths, its smaller-radius tube neighbourhoods remain disjoint, and its paths
obey the larger permitted length bound.  The accompanying scale inequality
also becomes weaker, since \(m_1^{1+c'}\leq m_1^{1+c}\).  Thus the single
choice \(c_1=\min\{c_{\rm str},c_{\rm rw}\}\) proves both conclusions,
independently of \(\lambda\). Thus the displayed minimum is a common
constant for both branches.

Finally, suppose that \(n\in\mathscr L(G,D)\). If some
\(s'\in[n^{1/3},n^{10/11}]\) satisfies
\(\operatorname{Gr}(3s')\leq3^5\operatorname{Gr}(s')\), the structure
branch and the logarithmic selection \eqref{eq:7-1b-slow} give the
required interval. Otherwise, for every
\(s\in[n^{2/5},n^{4/5}]\),
\[
 [s^{9/10},s^{11/10}]
   \subseteq[n^{1/3},n^{10/11}]\subseteq[n^{1/3},n].
\]
The failure of the slow inequality supplies the second condition in
\eqref{eq:corrected-fast-growth} throughout the first interval, while
\(n\in\mathscr L(G,D)\) supplies the first. The random-walk construction
proved above therefore applies at every such \(s\). These two cases
are exhaustive, and the preceding common-constant choice gives exactly
\eqref{eq:7-1b}; in particular, it implies
\(m_1\leq n^{1/(1+c_1)}\). \(\square\)

\begin{auxiliary}[Rounded scale selection]
\label{aux:rounded-scale-selection}
Assume \eqref{eq:7-1b}, set
\[
 \mathscr S(n)=[m_1^{1+c_1/4},m_1^{1+3c_1/4}]\cap\mathbb N,
\]
and let \(f\) be nondecreasing on the integers with
\(1\leq f(s)\leq s\). For each fixed \(R>0\), if \(n\) is sufficiently
large in terms of \(c_1,R\), there are \(m_-,m\in\mathscr S(n)\) with
\begin{equation}
 \tfrac12m(\log n)^{-R}\leq m_-
      \leq2m(\log n)^{-R},\qquad
 \frac{f(m)}{f(m_-)}\leq(\log n)^{8R/c_1}.                   \label{eq:7-1c}
\end{equation}
\end{auxiliary}

\begin{proof}
Starting with the least integer in \(\mathscr S(n)\), take
successive ceilings after multiplication by \((\log n)^R\), stopping
before the right endpoint. Since \(m_1\geq n^{1/3}\), the logarithmic
width of \(\mathscr S(n)\) is at least \((c_1/6)\log n+O(1)\).
Hence there are at least
\(c_1\log n/(7R\log\log n)\) consecutive pairs for large \(n\).
If the second inequality in \eqref{eq:7-1c} failed for all pairs, telescoping
would give
\[
 f(s_{\rm last})>(\log n)^{(8R/c_1)
   c_1\log n/(7R\log\log n)}>n,
\]
contrary to \(f(s_{\rm last})\leq s_{\rm last}\leq n\). Ceilings change
the reciprocal scale relation by a factor at most two.
\end{proof}

\begin{proposition}[Site multiscale criterion]
\label{prop:site-multiscale-criterion}
Fix
\(d<\infty\) and \(a,b>0\). There are \(K,N<\infty\) such that the
following holds. Let \(G\) be an infinite connected unimodular
transitive degree-\(d\) graph that is not one-dimensional. Let
\(p\in[a,1-b]\), let \(n\geq N\) be a real scale, and set
\[
 \delta_0=(\log\log n)^{-1/2}
      +K\operatorname{Burn}^{\rm s}(G,n,p).
\]
Suppose \(\delta_0\leq1\),
\(\operatorname{Spr}(p;2(\log\log n)^{-1/2}
+K\operatorname{Burn}^{\rm s}(G,n,p))\leq1-b\), and
\[
 \mathbb P_p(u\leftrightarrow v)
 \geq \exp\!\left[-\sqrt{\log\log((\log n)^{1/2})}\right]
 \quad\text{for every }u,v\in B_n.
\]
Then
\begin{equation}
 p_c^{\rm s}(G)\leq
 \operatorname{Spr}\left(p;2(\log\log n)^{-1/2}
      +K\operatorname{Burn}^{\rm s}(G,n,p)\right).           \label{eq:7-2}
\end{equation}
\end{proposition}

\textbf{Proof.} The premise forces \(p\geq1/d\) once \(N=N(d)\) is large.
Indeed, if \(p<1/d\), put \(N_n=\lfloor n\rfloor\) and choose
\(y\in S_{N_n}(o)\).  (The graph is infinite and connected, so this
sphere is nonempty.)  The simple-path union bound gives
\[
 \mathbb P_p(o\leftrightarrow y)
 \leq \frac{d}{d-1}\frac{p[p(d-1)]^{N_n}}{1-p(d-1)}
 \leq C_d\left(\frac{d-1}{d}\right)^{\lfloor n\rfloor}
 \leq C'_d\left(\frac{d-1}{d}\right)^n.
\]
Since \(y\in B_n(o)\), this contradicts the displayed two-point premise
at radius \(n\) for all sufficiently large \(n\). Thus we
may use the lower parameter bound required by
Propositions~\ref{prop:site-snowballing} and
\ref{prop:site-low-growth-step}.

We first isolate the two product-space facts used below.

\begin{auxiliary}[Site Hamming]\label{aux:site-hamming}
Let \(D\) be a finite graph, let \(A,B\subseteq V(D)\) be nonempty,
and let \(0<p<1\), \(0<\theta\leq1\), and \(t\geq0\). Adopt the
convention
\[
 \max_{\substack{x,y\in A\\x\ne y}}
 \mathbb P_p(x\leftrightarrow y)=0
 \quad\text{when }|A|=1.
\]
If
\[
 \min_{x\in A}\mathbb P_p(x\leftrightarrow B)\geq\theta
 \geq2|A|\max_{\substack{x,y\in A\\x\ne y}}
                  \mathbb P_p(x\leftrightarrow y),
\]
then
\begin{equation}
 \mathbb P_{\operatorname{Spr}(p;t)}(A\leftrightarrow B)
 \geq1-e^{-t\theta|A|/2}.                                     \label{eq:7-2c}
\end{equation}
\end{auxiliary}
\begin{proof}
We first prove the finite-product differential Hamming inequality needed
here, so Proposition~6.7 of \cite{EasoHutchcroft2023} is not an input.
Let \(I=V(D)\), let \(F\subseteq\{0,1\}^I\) be decreasing, and write
\(H_F(\omega)\) for the minimum number of open coordinates that must be
closed to obtain a configuration in \(F\), with \(H_F=0\) on \(F\).
Harris--FKG implies that \(\mathbb P_q(\,\cdot\mid F)\) is stochastically
dominated by \(\mathbb P_q\). Since the product space is finite, the
monotone coupling theorem gives a coupling \((\Omega,\Xi)\) with
\(\Omega\sim\mathbb P_q\),
\(\Xi\sim\mathbb P_q(\,\cdot\mid F)\), and \(\Xi\leq\Omega\)
coordinatewise. Consequently,
\begin{equation}
 H_F(\Omega)\leq |\Omega|-|\Xi|,
 \qquad
 \mathbb E_q H_F
 \leq |I|q-\mathbb E_q[|\omega|\mid F].                    \label{eq:7-2c-hamming-coupling}
\end{equation}
Direct differentiation of the finite Bernoulli polynomial gives
\[
 \frac d{dq}\log\mathbb P_q(F)
 =\frac{\mathbb E_q[|\omega|\mid F]-|I|q}{q(1-q)}.
\]
Together with \eqref{eq:7-2c-hamming-coupling}, this proves
\begin{equation}
 \frac d{dq}[-\log\mathbb P_q(F)]
 \geq\frac{\mathbb E_q H_F}{q(1-q)}                         \label{eq:7-2c-hamming-differential}
\end{equation}
for \(0<q<1\). This derivation also proves the endpoint cases needed by
continuity whenever the probabilities are nonzero.

Apply \eqref{eq:7-2c-hamming-differential} to
\(F=\{A\not\leftrightarrow B\}\). For a site configuration, the
distance to \(F\) is at least the
number of distinct open clusters that meet both \(A\) and \(B\): at least
one open vertex must be closed in each such cluster, and these vertex sets
are disjoint.  That number is at least
\[
 \sum_{x\in A}\mathbf1_{\{x\leftrightarrow B\}}
 -\sum_{\substack{x,y\in A\\x\ne y}}
       \mathbf1_{\{x\leftrightarrow y\}}.
\]
Taking expectations and using the two hypotheses gives expected Hamming
distance at least \(\theta|A|/2\) at density \(p\). The Hamming distance
to the decreasing event \(F\) is an increasing function of the
configuration, so stochastic domination gives the same lower bound at
every density \(q\geq p\). If
\(q(t)=\operatorname{Spr}(p;t)\), the chain rule applied to
\eqref{eq:7-2c-hamming-differential}, together with
\(q'(t)=-(1-q(t))\log(1-q(t))\), gives
\[
 \frac d{dt}[-\log\mathbb P_{q(t)}(A\not\leftrightarrow B)]
 \geq\frac{-\log(1-q(t))}{q(t)}\frac{\theta|A|}{2}
 \geq\frac{\theta|A|}{2}.
\]
Integration proves \eqref{eq:7-2c}.
\end{proof}

\begin{auxiliary}[Fixed-cluster tube gluing]
\label{aux:fixed-cluster-tube}
Couple site percolation by uniform labels,
let \(0<p<q<1\), let \(H\subseteq V(G)\) be finite, and condition on two
distinct \(p\)-open clusters \(C,D\) of the configuration induced on
\(H\). More precisely, condition on the exact event that they are two
distinct induced clusters with these vertex sets. If \(F=C\cup D\), the
conditional label law on \(H\) is the product law
\begin{equation}
 \begin{array}{c|c}
 z\in F&U_z\sim\operatorname{Unif}[0,p],\\
 z\in\partial_V^H F&U_z\sim\operatorname{Unif}(p,1],\\
 z\in H\setminus(F\cup\partial_V^H F)&U_z\sim\operatorname{Unif}[0,1],
 \end{array}                                                   \label{eq:7-2c-prime}
\end{equation}
with all displayed coordinates independent; no condition outside \(H\)
is asserted. Let \(T_1,\ldots,T_k\subseteq H\) be tubes whose sets of
vertices outside \(F\) are pairwise disjoint. If ordinary
\(p\)-site percolation on \(T_j\setminus F\), with every site of \(F\)
declared open but with no vertex identifications, connects \(C\) to
\(D\) in the induced graph \(G[T_j\cup F]\) with probability at least
\(\xi\) for every \(1\leq j\leq k\), then, conditionally on \(C,D\),
the events
\[
 \{C\leftrightarrow D\text{ in }G[T_j\cup F]
       \text{ at parameter }q\},\qquad 1\leq j\leq k,
\]
are independent and each has probability at least
\begin{equation}
 \xi\min\left\{1,\frac{q-p}{p(1-p)}\right\}^{2}.              \label{eq:7-2d}
\end{equation}
\end{auxiliary}
\begin{proof}
The conditional product law in \eqref{eq:7-2c-prime} is
Lemma~\ref{lem:stopped-product-kernel} for the breadth-first cluster
exploration restricted to \(H\), with external boundary taken relative
to \(H\), and the partition \([0,p]\cup(p,1]\). Coordinates outside
\(H\) are neither tested nor conditioned. To prove the remaining
claim, \eqref{eq:7-2c-prime} makes the bulk sites independent with
\(q\)-density \(q\) and the boundary sites independent with conditional
\(q\)-density \(r=(q-p)/(1-p)\). Fix \(j\). In an auxiliary ordinary
\(p\)-configuration on \(T_j\setminus F\), with \(F\) declared open,
choose a shortest open path from the set \(C\) to the set \(D\) in
\(G[T_j\cup F]\). Choose its endpoints to be its first visit to \(C\)
and last visit to \(D\); by shortestness, all its internal sites lie
outside \(F\). Only the first internal site can be adjacent to \(C\), and
only the last can be adjacent to \(D\): any other such adjacency would
give a strictly shorter path in the same induced graph. Thus the path
uses at most two sites of \(\partial_V^H F\). If \(r<p\), independently
thin the open boundary sites with retention \(r/p\), while coupling every
other trial site monotonically from density \(p\) to density \(q\); the
chosen witness survives with probability at least \((r/p)^2\). If
\(r\geq p\), couple every trial coordinate monotonically. This proves
\eqref{eq:7-2d}. Each resulting event is measurable with respect to the
coordinates in \(T_j\setminus F\), since the sites of \(F\) are fixed
open. The assumed disjointness of these coordinate sets therefore proves
conditional independence, even though the deterministic set \(F\) is
shared by all trial domains.

No extension to arbitrary seed-generated cluster unions is asserted.
Indeed, conditioning unused seeds to be closed may impose additional
one-coordinate restrictions away from the cluster boundary. In the
peeling application below, Lemma~\ref{lem:site-peeling-filtration}
instead proves the stronger fact that unused seed sites and extinct
exploration data lie outside every later trial-coordinate set. The
residual product law is then stated in \eqref{eq:7-5}, and the
shortest-path comparison is repeated directly in
\(G[Q_u\cup F_i]\).
\end{proof}

\begin{lemma}[Adjacent-center annular comparison]
\label{lem:adjacent-annular-comparison}
Let \(x,y\) be adjacent vertices, and let \(a,r,R,b\) be nonnegative
integers such that
\begin{equation}
 a+1\leq r<R<b-1.                                      \label{eq:adjacent-annular-radii}
\end{equation}
If a path \(P\), oriented from its first specified endpoint to its
second, starts on \(S_a(y)\) and ends on \(S_b(y)\), then \(P\) contains
a subpath in \(B_R(x)\) joining \(B_r(x)\) to \(S_R(x)\). In particular,
if \(P\) is open, the cluster of the configuration induced on
\(B_R(x)\) that contains this subpath is an \([r,R]\)-crossing cluster
centered at \(x\).
\end{lemma}

\begin{proof}
Write \(z\) and \(w\) for the specified initial and terminal vertices of
\(P\). The triangle inequality gives
\[
 d(x,z)\leq a+1\leq r,
 \qquad
 d(x,w)\geq b-1>R.
\]
Graph distance from \(x\) changes by at most one across an edge. Hence
the portion of \(P\) from \(z\) through its first subsequent visit to
\(S_R(x)\) is contained in \(B_R(x)\) and joins \(B_r(x)\) to
\(S_R(x)\), as required.
\end{proof}

The low-growth step below uses only
Lemma~\ref{lem:plentiful-tubes} and Auxiliary
Lemma~\ref{aux:quotient-cayley-model} for geometry,
Proposition~\ref{prop:site-snowballing} for site snowballing, the site
Hamming lemma \eqref{eq:7-2c}, the fixed-cluster law
\eqref{eq:7-2c-prime}--\eqref{eq:7-2d}, and the peeling law in
Lemma~\ref{lem:site-peeling-filtration}. All
five inputs have their complete hypotheses stated here. The remaining
operations are site FKG, monotonicity, union bounds, and deterministic
annular chaining. Writing
\[
 \delta_K(b,n)=
 \left(\frac{K\log\log n}
 {\min\{\log n,\log\operatorname{Gr}_G(b)\}}\right)^{1/4},
\]
this is the sprinkling amount used when
Proposition~\ref{prop:site-low-growth-step} is invoked in
Proposition~\ref{prop:site-multiscale-criterion}.

\begin{proposition}[Parameterized site low-growth step]
\label{prop:site-low-growth-step}
For every integer \(d\geq2\), \(D\geq1\), and \(b_0\in(0,1)\), there
are \(\lambda_0,c_{\rm out}>0\) such that, for every
\(\lambda\geq\lambda_0\), there are \(K_1,n_0<\infty\) with the
following property. Let \(G\) be an infinite connected unimodular
transitive degree-\(d\) graph that is not one-dimensional. Suppose
\(K\geq K_1\), \(n\geq n_0\) is a real scale, and \(b\) is an integer
satisfying \(1\leq b\leq n^{1/3}/8\), while
\(1/d\leq p_1\leq p_2\leq1-b_0\),
\[
T_{\min}\leq T\leq3\sqrt{\log\log n},
\]
where \(T_{\min}=T_{\min}(d,D,b_0)\) is a fixed constant, and
\begin{equation}
\begin{gathered}
n\in\mathscr L(G,D),\qquad
\kappa_{p_1}^{\rm s}(n,\infty)
      \geq e^{-T},\\
\mathbb P_{p_1}(\operatorname{Piv}^{\rm s}[4b,n^{1/3}])
      \leq(\log n)^{-1},\qquad
\delta_K(b,n)\leq1,\\
p_2\geq\operatorname{Spr}(p_1;\delta_K(b,n)),
\end{gathered}                                                 \label{eq:7-2e}
\end{equation}
then either \(p_2\geq p_c^{\rm s}(G)\), or
\begin{equation}
 \kappa_{p_2}^{\rm s}
      \bigl(e^{(\log n)^{c_{\rm out}\lambda}},n\bigr)
      \geq e^{-3T}.                                           \label{eq:7-2f}
\end{equation}
\end{proposition}

\textbf{Proof of Proposition~\ref{prop:site-low-growth-step}.}

Work throughout with i.i.d. uniform labels and
\(\omega_q=\{z:U_z\leq q\}\). Thus all parameter values below live on
one monotone probability space. Let \(s_1,s_2,s_3,h_0\) be the constants
of Proposition 6.2 with sprinkling bound \(D=1\). Fix
\[
 0<c_3<\frac{s_1}{2^{14}},\qquad c_4=\frac{s_2}{4}.
\]
The later induction uses the special case
\(T=\sqrt{\log\log n}\), but its proof is uniform in the endpoint lower
bound. Indeed, throughout the argument that lower bound occurs only as
one endpoint factor in \eqref{eq:7-2j}, two endpoint factors in \eqref{eq:7-2l}, \eqref{eq:7-6}, and
\eqref{eq:7-7g}, and as \(\theta\) in the Hamming estimate \eqref{eq:7-2r}. Uniformly for
\(T\leq3\sqrt{\log\log n}\), after increasing \(n_0\),
\begin{equation}
 (\log n)^{-1}\leq\tfrac12e^{-T},\qquad
 e^{-2T}\geq(\log n)^{-1/2},\qquad
 e^{-6T}\geq(\log n)^{-1},\qquad
 c_*e^{-2T}\geq e^{-3T}                                      \label{eq:7-2e-prime}
\end{equation}
for every fixed \(c_*>0\), provided also that \(T\geq T_{\min}(c_*)\).
These four inequalities are exactly the numerical comparisons needed
in those occurrences. Thus the proof below establishes the displayed
parameterized form. Put
\(q_2=\operatorname{Spr}(p_1;\delta_K(b,n))\). If
\(q_2\geq p_c^{\rm s}\) the alternative holds; otherwise prove \eqref{eq:7-2f} at
\(q_2\) and use monotonicity to pass to the given \(p_2\geq q_2\). We may
therefore assume \(p_2=q_2<p_c^{\rm s}\). The whole sprinkling interval
is then subcritical, so the uniqueness hypothesis of Proposition 6.2
holds automatically.

Use in both snowballing applications the intensity
\[
 h=\max\{\operatorname{Gr}_G(b)^{-1},n^{-1/15}\}.
\]
It is at least \(|B_b|^{-1}\). Since \(\delta\leq1\) and
\(\delta^4=K\log\log n/
\min\{\log n,\log\operatorname{Gr}_G(b)\}\), its logarithm satisfies
\[
 \delta^4\log(1/h)\geq\tfrac1{15}K\log\log n.
\]
Thus \(h\leq h_0\) for large \(n\), and the fourth-power chain-length
condition of Proposition 6.2 is satisfied in the common-corridor
application because
\begin{equation}
 h^{s_1(\delta/4)^4}\leq(\log n)^{-2c_3K},
 \qquad k\leq5(\log n)^{c_3K},                                 \label{eq:7-2h}
\end{equation}
and in the separated-set application it gives the same inequalities with
\(\delta/2\) and \(k\leq(\log n)^{c_3K}+1\). After increasing \(K_1\)
and \(n_0\), \eqref{eq:7-2h} is stronger than the chain-length requirement in
both cases. These are the two uses of snowballing inside
Proposition~\ref{prop:site-low-growth-step}; the separate Full-Space use
is checked after its proof.

We now prove the two branches. Put
\(\delta=\delta(p_1,p_2)=\delta_K(b,n)\) and
\[
 p_t=\operatorname{Spr}\bigl(p_1;(t-1)\delta\bigr)
 \quad(1\leq t\leq2).
\]
Let \(c_1=c_1(d,D)>0\) be the constant in
Lemma~\ref{lem:plentiful-tubes}. That lemma supplies
\(m_1,m_2\) with
\[
 n^{1/3}\leq m_1,\qquad m_1^{1+c_1}\leq m_2\leq n,
\]
such that all scales in \([m_1,m_2]\) have
\((c_1,\lambda)\)-polylog-plentiful radial and annular tubes. Define
\[
 c_{\rm out}=\frac{c_1}{4},
\]
decreasing the constant in \eqref{eq:7-2f} if necessary, and choose
\(\lambda_0\) large enough that
\begin{equation}
 \lambda/c_1>D+2,\qquad c_1\lambda/2>2D+4.                    \label{eq:7-2h-prime}
\end{equation}
Define
\[
 \mathscr S(n)=
 [m_1^{1+c_1/4},m_1^{1+3c_1/4}]\cap\mathbb N.
\]
Whenever a displayed ball or sphere has a nonintegral radius, we take its
integer part. Thus, for example, \(\operatorname{tz}(m/2)\) means
\(\operatorname{tz}(\lfloor m/2\rfloor)\). Additive constants below have
been chosen to absorb this convention.
Finally put
\begin{equation}
 \operatorname{tz}(m)=\max\{r\in\{0,\ldots,m\}:
       \tau_{p_{3/2}}^{B_m}(B_r)\geq(\log n)^{-1}\}.             \label{eq:7-2i}
\end{equation}
For large \(n\), the set contains \(0\). Domain monotonicity shows that
\(m\mapsto\operatorname{tz}(m)\) is nondecreasing, and by definition
\(0\leq\operatorname{tz}(m)\leq m\). To meet the global hypothesis of
Auxiliary Lemma~\ref{aux:rounded-scale-selection}, define, for every
integer \(s\geq1\),
\begin{equation}
 \widetilde f(s)=\max\{1,\operatorname{tz}(s)\}.              \label{eq:7-2i-prime}
\end{equation}
Then \(\widetilde f\) is nondecreasing and
\(1\leq\widetilde f(s)\leq s\) for every \(s\geq1\). We shall apply
Auxiliary Lemma~\ref{aux:rounded-scale-selection} to \(\widetilde f\),
and use \eqref{eq:7-2j} below to identify
it with \(\operatorname{tz}\) on the entire sampled scale interval.

\textbf{Common corridor estimate.} For \(2n^{1/3}\leq m\leq n\), the pivotal
hypothesis in \eqref{eq:7-2e} gives
\begin{equation}
 \tau_{p_1}^{B_{m/2}}(B_{4b})
 \geq\kappa_{p_1}^{\rm s}(n,\infty)
      -\mathbb P_{p_1}(\operatorname{Piv}^{\rm s}[4b,n^{1/3}])
 \geq\tfrac12e^{-T}.                                           \label{eq:7-2j}
\end{equation}
Indeed, for \(a,c\in B_{4b}\), the full-space event
\(\{a\leftrightarrow c\}\) has probability at least
\(\kappa_{p_1}^{\rm s}(n,\infty)\). If it does not occur inside
\(B_{m/2}\), the two induced clusters of \(a,c\) in \(B_{m/2}\) are
distinct and both meet \(S_{m/2}\). Since \(m/2\geq n^{1/3}\), their
restrictions give two distinct crossings of the annulus
\([4b,n^{1/3}]\). This is the pivotal event subtracted in \eqref{eq:7-2j}. Thus
\(\operatorname{tz}(m/2)\geq4b\) for large \(n\).

Given a path \(\gamma\) of length at most
\(\operatorname{tz}(m/2)(\log n)^{c_3K}\), select at most
\(5(\log n)^{c_3K}\) successive centers on it at spacing at most
\(\operatorname{tz}(m/2)/4\), and use the balls \(B_b\) around the
centers. Put \(\Lambda=B_{m/2}(\gamma)\). Since
\(\operatorname{tz}(m/2)\geq4b\), every pair of sites in two consecutive
balls lies in the translate of \(B_{\operatorname{tz}(m/2)}\) about the
first center. The translated definition \eqref{eq:7-2i}, with its witnessing
domain contained in \(\Lambda\), therefore gives
\[
 \tau_{p_{3/2}}^\Lambda
   (B_b(u_i)\cup B_b(u_{i+1}))\geq(\log n)^{-1}
       \geq4h^{s_1(\delta/4)^4}.
\]
Equation \eqref{eq:7-2j} gives the endpoint within-set bounds
\(\tau_{p_{3/2}}^\Lambda(B_b(u_1)),
\tau_{p_{3/2}}^\Lambda(B_b(u_k))\geq e^{-T}/2\).
Together with \eqref{eq:7-2h}, these are exactly the adjacent-set and chain-length
hypotheses of Proposition 6.2 on
\([p_{3/2},p_{7/4}]\).

For completeness, its thickening radius can be chosen at most \(m/4\).
Indeed, Lemma 6.1, \(m\geq n^{1/3}\), \(h\geq n^{-1/15}\), and
\(\log\operatorname{Gr}_G(m)\leq(\log m)^D\) give, uniformly on the
whole interval \([p_1,p_{7/4}]\) (Lemma 6.1 is uniform for
\(q\geq1/d\)),
\begin{equation}
 \mathbb P_q\!\left(\operatorname{Piv}^{\rm s}
       [1,(m/4)h]\right)
 \leq C_d\left(\frac{4\log\operatorname{Gr}_G(m)}{mh}\right)^{2/5}
 \leq C_{d,D}(\log n)^{2D/5}n^{-8/75}
 \leq \tfrac12n^{-1/15}<h.                                    \label{eq:7-2k}
\end{equation}
The penultimate inequality holds after increasing the fixed threshold
for \(n\), since \(8/75>1/15\); the last one uses
\(h\geq n^{-1/15}\). Thus the strict pivotal hypothesis of
Proposition~6.2 is satisfied.
Taking its radius parameter to be \(m/4\), Proposition 6.2 therefore
yields, since \(B_{m/2}(\Lambda)\subseteq B_m(\gamma)\),
\begin{equation}
 \kappa_{p_{7/4}}^{\rm s}
 \left(\operatorname{tz}(m/2)(\log n)^{c_3K},m\right)
 \geq c_4e^{-2T}                                               \label{eq:7-2l}
\end{equation}
for every \(2n^{1/3}\leq m\leq n\).
Every scale from \(\mathscr S(n)\), its half, and each scale \(2m_-\)
selected below lies in this range for sufficiently large \(n\): this
follows from \(m_1\geq n^{1/3}\) and the positive power gaps in the
definition of \(\mathscr S(n)\). These are the only applications of
\eqref{eq:7-2l}.

\textbf{Positive two-point-zone branch.} Suppose there is \(m\in\mathscr S(n)\)
such that
\begin{equation}
 \operatorname{tz}(m/2)(\log n)^{c_3K}
       \geq m(\log n)^{3\lambda/c_1}.                           \label{eq:7-2m}
\end{equation}
Then \eqref{eq:7-2l} is at least \((\log n)^{-1}\) on the right-hand scale in
\eqref{eq:7-2m}. Put
\[
 s^\circ=m(\log n)^{3\lambda/(2c_1)},\qquad
 s_0=\lfloor s^\circ\rfloor,\qquad
 s_- =\left\lfloor\frac{9s_0}{10}\right\rfloor .
\]
The power gaps defining \(\mathscr S(n)\), together with
\(m_1^{1+c_1}\leq m_2\), imply
\begin{equation}
 m_1\leq s_-<s_0\leq m_2
\label{eq:7-2m-integer-scales}
\end{equation}
for sufficiently large \(n\). Indeed, at the lower endpoint of
\(\mathscr S(n)\), the quotient \(s^\circ/m_1\) is a positive power of
\(m_1\), up to a polylogarithmic factor, while at the upper endpoint
\(m_2/s^\circ\) has the same property; hence the additive losses of two
from the floors are absorbed uniformly. For either integer
\(\sigma\in\{s_-,s_0\}\), we also have, after increasing \(n_0\),
\begin{equation}
 \sigma(\log\sigma)^{-\lambda/c_1}\geq m,
 \qquad
 \sigma(\log\sigma)^{\lambda/c_1}
       \leq m(\log n)^{3\lambda/c_1},
                                                                    \label{eq:7-2m-rounded-tubes}
\end{equation}
and the number of tubes is at least
\((\log(n^{1/3}))^{c_1\lambda}\), since
\(\sigma\geq m_1\geq n^{1/3}\). Thus \(B_m(\gamma)\) is contained in
the corresponding tube for each tube path \(\gamma\), and \eqref{eq:7-2l} and
\eqref{eq:7-2m} give its ordinary \(p_{7/4}\)-crossing probability at least
\((\log n)^{-1}\). Fix one deterministic plentiful family
\(T_1^{\rm rad},\ldots,T_N^{\rm rad}\) of radial tubes at scale
\(s_-\), and let \(E_i\) be the event that the endpoints of its central
path are joined by a \(p_{7/4}\)-open path in \(T_i^{\rm rad}\). Define
the localized radial event
\[
 \mathcal R_o=\bigcup_{i=1}^N E_i.
\]
The tubes are vertex-disjoint, so the \(E_i\) depend on disjoint label
sets and are independent. Write \(a=c_1\lambda\). Since
\(\log(n^{1/3})=(\log n)/3\), the number \(N\) of available tubes and
a common lower bound \(q\) for their crossing probabilities satisfy
\[
 N\geq 3^{-a}(\log n)^a,
 \qquad q\geq(\log n)^{-1}.
\]
Independence therefore gives
\[
 \mathbb P_{p_{7/4}}\!\left(\bigcap_{i=1}^N E_i^{\mathrm c}\right)
 \leq(1-q)^N\leq e^{-Nq}
 \leq \exp\!\left[-3^{-a}(\log n)^{a-1}\right].
\]
The second inequality in \eqref{eq:7-2h-prime} gives
\(a>4D+8\geq12\). Hence \(a/4-1>0\), and, after increasing \(n_0\),
\[
 3^{-a}(\log n)^{a-1}
 \geq(\log n)^{3a/4}.
\]
Consequently
\begin{equation}
 \mathbb P_{p_{7/4}}(\mathcal R_o)
 \geq1-\exp[-(\log n)^{3c_1\lambda/4}]                         \label{eq:7-2n}
\end{equation}
after increasing \(n_0\). The event \(\mathcal R_o\) supplies a
specified open path from \(S_{s_-}\) to \(S_{4s_-}\) inside one of the
selected radial tubes. Since its central path starts in \(B_{s_0}\), has
length at most \(m(\log n)^{3\lambda/c_1}\), and its tube radius is at
most \(s_0\), every such witness lies in
\(B_{2s_0+m(\log n)^{3\lambda/c_1}}\).

Put
\[
 R_{\rm loc}=4s_0+m(\log n)^{3\lambda/c_1},
 \qquad H_o=B_{R_{\rm loc}}(o).
\]
We work from this point with clusters of induced configurations. Fix
\(u,v\in S_{s_0}\), let \(C_u^{H_o},C_v^{H_o}\) be their
\(p_{7/4}\)-clusters in \(H_o\), and suppose both induced clusters meet
\(S_{3s_0}\). If they coincide, then they are already connected inside
\(H_o\). Otherwise condition on their exact vertex sets. Auxiliary
Lemma~\ref{aux:fixed-cluster-tube}, with ambient set \(H_o\), gives the
product law \eqref{eq:7-2c-prime}; in particular the forced-closed
boundary is \(\partial_V^{H_o}F_o\), where
\[
 F_o=C_u^{H_o}\cup C_v^{H_o},
\]
not the boundary of a full-space cluster.

Choose canonically a simple path \(A_u\subseteq C_u^{H_o}\) from
\(u\in S_{s_0}\) to its first visit to \(S_{3s_0}\), and define
\(A_v\subseteq C_v^{H_o}\) similarly. These are
\((s_0,3s_0)\)-crossings contained in \(B_{3s_0}\). Choose the
plentiful annular tubes at the integer scale \(s_0\) as a
deterministic function of \((A_u,A_v)\). If \(T_j\) denotes one of the
resulting tubes, then
\[
 T_j\subseteq
 B_{3s_0+m(\log n)^{3\lambda/c_1}+s_0}(o)\subseteq H_o;
\]
the connector starts in \(B_{3s_0}\), its length is at most
\(m(\log n)^{3\lambda/c_1}\), and its tube radius is at most \(s_0\).
The sets \(T_j\setminus F_o\) are pairwise disjoint. Every \(T_j\)
contains \(B_m(\gamma_j)\) for its connector
\(\gamma_j\), whose length is within the first argument of
\eqref{eq:7-2l}. Hence ordinary \(p_{7/4}\)-site percolation in \(T_j\)
connects the two selected crossing paths with probability at least
\((\log n)^{-1}\). A fortiori, after declaring every site of \(F_o\)
open, the same lower bound holds for a connection between
\(C_u^{H_o}\) and \(C_v^{H_o}\) in \(G[T_j\cup F_o]\).

Apply Auxiliary Lemma~\ref{aux:fixed-cluster-tube} with
\(H=H_o\), \(p=p_{7/4}\), \(q=p_2\), and
\(p_2-p_{7/4}\geq c(d,b_0)\delta\). It gives conditionally independent
successes in the respective induced graphs \(G[T_j\cup F_o]\), each of
probability at least
\(c\delta^2(\log n)^{-1}\). Since
\(\delta^2\geq(\log n)^{-1/2}\), each success probability is at least
\(c(\log n)^{-3/2}\). There are at least
\((\log n)^{3c_1\lambda/4}\) conditionally independent trials after
increasing \(n_0\). Thus \(1-x\leq e^{-x}\) gives
\begin{equation}
 \mathbb P(C_u^{H_o}\not\leftrightarrow C_v^{H_o}
       \text{ in the }p_2\text{-configuration induced on }H_o
       \mid C_u^{H_o},C_v^{H_o})
 \leq \exp\!\left[-c(\log n)^{3c_1\lambda/4-3/2}\right].     \label{eq:7-2o}
\end{equation}
Indeed, \(T_j\cup F_o\subseteq H_o\), so every successful trial is a
connection inside the ambient induced configuration on \(H_o\); its
random coordinates are exactly \(T_j\setminus F_o\).

Define \(\mathcal U_o\) to be the event that all
\(p_{7/4}\)-clusters of the configuration induced on \(H_o\) that meet
both \(S_{s_0}\) and \(S_{3s_0}\) are contained in one \(p_2\)-cluster of the
configuration induced on \(H_o\). We make the union bound and its
exponent explicit. Using the global ordering fixed in Section~1, assign
to each induced crossing cluster its least vertex in \(S_{s_0}\). There are
at most \(\operatorname{Gr}_G(s_0)^2\leq\exp[2(\log n)^D]\) ordered pairs
of possible representatives. For each fixed pair, condition on its two
induced clusters; if they are distinct and both cross, \eqref{eq:7-2o}
applies, whereas otherwise that pair contributes zero. Taking
expectations and summing over the deterministic representative pairs
shows that
\begin{equation}
 \begin{split}
 \mathbb P(\mathcal U_o^c)
 &\leq \exp\!\left[2(\log n)^D
          -c(\log n)^{3c_1\lambda/4-3/2}\right]\\
 &\leq \exp[-(\log n)^{c_1\lambda/2}].
 \end{split}                                                    \label{eq:7-2o-prime}
\end{equation}
Indeed, writing \(a=c_1\lambda\), condition
\eqref{eq:7-2h-prime} gives
\[
 \left(\frac{3a}{4}-\frac32\right)-\frac a2
   =\frac a4-\frac32>D+\frac12,
\]
so the negative term in the first line of
\eqref{eq:7-2o-prime} eventually dominates both
\(2(\log n)^D\) and \((\log n)^{a/2}\), including the fixed constant
\(c>0\).

By definition, \(\mathcal U_o\) is a local statement inside \(H_o\);
no connectivity of full clusters outside \(H_o\) is used. The scale
margins above give \(R_{\rm loc}\leq n/2\) for large \(n\). For every
\(x\), let \(H_x=B_{R_{\rm loc}}(x)\) and let \(\mathcal U_x\) be the
translated induced-cluster event. It obeys the same bound
\eqref{eq:7-2o-prime} by transitivity.

For a path \(\zeta\) from \(u\) to \(v\), let \(\mathcal R_x\) be the
translated localized radial-tube event in \eqref{eq:7-2n}. The deterministic annular chaining
containment is
\begin{equation}
 \{u\leftrightarrow S_n(u)\}_{p_{7/4}}
 \cap\{v\leftrightarrow S_n(v)\}_{p_{7/4}}
 \cap\bigcap_{x\in\zeta}(\mathcal R_x\cap\mathcal U_x)
 \subseteq\{u\leftrightarrow v\text{ in }B_n(\zeta)\}_{p_2}.   \label{eq:7-2p}
\end{equation}
We verify the containment rather than relying on overlap of annuli. For
adjacent \(x,y\in\zeta\), apply
Lemma~\ref{lem:adjacent-annular-comparison} with
\[
 a=s_-,\qquad r=s_0,
 \qquad R=3s_0,\qquad b=4s_-.
\]
For sufficiently large \(s_0\), the definition
\(s_-=\lfloor9s_0/10\rfloor\) gives
\(s_-+1\leq s_0<3s_0<4s_--1\), so these integers satisfy
\eqref{eq:adjacent-annular-radii}. Lemma
\ref{lem:adjacent-annular-comparison} therefore supplies from the path
in \(\mathcal R_y\) an \(x\)-centered \([s_0,3s_0]\)-crossing subpath in
\(B_{3s_0}(x)\subseteq H_x\). The corresponding subpath of the path in
\(\mathcal R_x\) has the same property. Their components in the
\(p_{7/4}\)-configuration induced on \(H_x\) are consequently among the
components quantified by \(\mathcal U_x\), which joins them inside
\(H_x\) at density \(p_2\). Iterating along \(\zeta\) joins all the
selected radial paths through these local connections. At \(u\), the
path supplied by
\(\{u\leftrightarrow S_n(u)\}_{p_{7/4}}\) starts in \(B_{s_0}(u)\) and,
since \(3s_0<n\), has an initial subpath in \(B_{3s_0}(u)\) crossing
\([s_0,3s_0]\). Hence \(\mathcal U_u\) joins its induced component in
\(H_u\) to the induced component of the radial path at \(u\). The same
argument applies at \(v\).

Finally, truncate the endpoint connections at their first visits to
\(S_n(u)\) and \(S_n(v)\), so they lie in the corresponding radius-\(n\)
balls. Every annular connection used by \(\mathcal U_x\) lies in
\(B_{R_{\rm loc}}(x)\subseteq B_n(x)\). Thus all pieces just described
lie in \(B_n(\zeta)\), proving \eqref{eq:7-2p}.
The endpoint and radial events are increasing, so site FKG applies to
them; the failures of \(\mathcal U_x\) are subtracted by a union bound.
For \(\operatorname{len}(\zeta)\leq
\exp[(\log n)^{c_1\lambda/4}]\), \eqref{eq:7-2n}--\eqref{eq:7-2p} give
\[
 \mathbb P_{p_2}(u\leftrightarrow v\text{ in }B_n(\zeta))
 \geq e^{-2T}(1-e^{-(\log n)^{c_1\lambda/2}})^{|\zeta|}
      -|\zeta|e^{-(\log n)^{c_1\lambda/2}}
 \geq e^{-3T},
\]
which is \eqref{eq:7-2f}.

\textbf{Negative two-point-zone branch.} Suppose \eqref{eq:7-2m} fails for every
\(s\in\mathscr S(n)\). Fix such an \(s\), put
\(\mu=\lfloor s/2\rfloor\), and write \(t_\mu=\operatorname{tz}(\mu)\).
The failure of \eqref{eq:7-2m}, with \(K_1\) increased in terms of \(\lambda\),
implies
\begin{equation}
 t_\mu+1\leq\mu/4.                                             \label{eq:7-2q0}
\end{equation}
We first prove that there is a set
\(U_s\subseteq B_{t_\mu+1}\) of cardinality at least
\((\log n)^{c_3K}\) such that
\begin{equation}
 \mathbb P_{p_1}(x\leftrightarrow y\text{ in }B_{\mu/2})
       <(\log n)^{-c_3K}\quad(x\ne y\in U_s).                  \label{eq:7-2q}
\end{equation}

Suppose otherwise. Since \(t_\mu\) is the maximum in \eqref{eq:7-2i}, \eqref{eq:7-2q0}
and the definition of \(\tau\) give \(x,y\in B_{t_\mu+1}\) with
\begin{equation}
 \mathbb P_{p_{3/2}}(x\leftrightarrow y\text{ in }B_\mu)
       <(\log n)^{-1}.                                         \label{eq:7-2q1}
\end{equation}
Join \(x\) to \(o\) and \(o\) to \(y\) by geodesics and loop-erase their
concatenation. This gives a simple path
\(P=(z_0=x,\ldots,z_M=y)\) contained in \(B_{t_\mu+1}\). Put
\(\eta=(\log n)^{-c_3K}\). Starting with \(v_0=z_0\), if
\(v_j=z_{a_j}\), let \(\ell_j\) be the largest \(\ell\geq a_j\) for
which
\[
 \mathbb P_{p_1}(v_j\leftrightarrow z_\ell
                         \text{ in }B_{\mu/2})\geq\eta.
\]
If \(\ell_j<M\), set \(v_{j+1}=z_{\ell_j+1}\); otherwise append \(y\)
and stop. Every newly selected \(v_{j+1}\) has connection probability
less than \(\eta\) to every earlier selected vertex, by the maximality of
the corresponding \(\ell_i\). The assumed failure of \eqref{eq:7-2q} therefore
forces the procedure to stop after fewer than
\((\log n)^{c_3K}+1\) vertices. Moreover, if \(w_j=z_{\ell_j}\), then
\(w_j\) is adjacent to \(v_{j+1}\); FKG and the event that
\(v_{j+1}\) is open show that every successive pair in the resulting
chain from \(x\) to \(y\) has \(p_1\)-connection probability in
\(B_{\mu/2}\) at least \(p_1\eta\).

For a chain vertex \(v_j\), \eqref{eq:7-2q0} gives
\(B_{\mu/4}(v_j)\subseteq B_{\mu/2}\). Apply the translated form of
\eqref{eq:7-2j} at scale \(\mu/2\) to connect every point of \(B_b(v_j)\) to
\(v_j\) inside this translated ball. For arbitrary
\(a\in B_b(v_j)\) and \(c\in B_b(v_{j+1})\), these two endpoint events,
the successive-pair event, and site FKG give
\begin{equation}
 \mathbb P_{p_1}(a\leftrightarrow c\text{ in }B_{\mu/2})
 \geq \tfrac14p_1e^{-2T}(\log n)^{-c_3K}.                    \label{eq:7-2q-prime}
\end{equation}
Since \(p_1\geq1/d\), the constant furnished by this estimate is
\(1/(4d)\), not \(1/4\); only a positive constant depending on \(d\) is
used in the subsequent application.
Here \(\mu/2\geq2n^{1/3}\) for all \(s\in\mathscr S(n)\), once \(n_0\)
is increased, so this use of \eqref{eq:7-2j} is within its stated range.

Apply Proposition 6.2 to the sets \(B_b(v_j)\), with base domain
\(\Lambda=B_{\mu/2}\), sprinkling interval \([p_1,p_{3/2}]\), and
thickening radius \(\mu/4\). The within-set bounds follow from the same
translated application of \eqref{eq:7-2j}; \eqref{eq:7-2q-prime} gives every cross-set bound;
\eqref{eq:7-2h} and \eqref{eq:7-2e-prime} give the chain-length inequalities; and \eqref{eq:7-2k}, at
scale \(\mu\), gives the pivotal hypothesis. Thus \eqref{eq:6-13} yields
\[
 \mathbb P_{p_{3/2}}(x\leftrightarrow y\text{ in }B_\mu)
       \geq c e^{-2T}\geq(\log n)^{-1},
\]
contradicting \eqref{eq:7-2q1}. This proves \eqref{eq:7-2q}.

Choose \(A_s\subseteq U_s\) with
\(|A_s|=\lceil(\log n)^{c_3K-1/2}\rceil\). Apply Auxiliary
Lemma~\ref{aux:site-hamming}
in the finite graph induced by \(B_{\mu/2}\), with target
\(S_{\mu/2}\). For \(z\in A_s\), choose \(y_z\in S_{3\mu/4}\). Since
\(d(z,y_z)\leq\mu\leq n\), \eqref{eq:7-2e} gives
\(\mathbb P_{p_1}(z\leftrightarrow y_z)\geq e^{-T}\); the first part of
any witnessing path up to its first visit to \(S_{\mu/2}\) lies in
\(B_{\mu/2}\). Thus the first Hamming hypothesis holds with
\(\theta=e^{-T}\). The second follows from \eqref{eq:7-2q}, because \eqref{eq:7-2e-prime} gives
\(e^{-T}\geq2(\log n)^{-1/2}\) and
\(2|A_s|(\log n)^{-c_3K}\leq2(\log n)^{-1/2}+o(1)\). Finally the
sprinkling time from \(p_1\) to \(p_{3/2}\) is \(\delta/2\), so the Site
Hamming estimate \eqref{eq:7-2c} bounds the failure probability by
\(\exp[-\delta e^{-T}|A_s|/4]\). Since
\(\delta\geq(\log n)^{-1/4}\),
\(|A_s|\geq(\log n)^{c_3K-1/2}\), and
\(T\leq3\sqrt{\log\log n}\), we have
\[
 \frac{\delta e^{-T}|A_s|}{4}
 \geq(\log n)^{c_3K-1}
       \left[\frac14(\log n)^{1/4}
       e^{-3\sqrt{\log\log n}}\right].
\]
The bracketed quantity tends to infinity: with
\(u=\log\log n\), its logarithm is
\(u/4-3\sqrt u-\log4\), which tends to infinity. After increasing
\(n_0\), it is therefore at least one uniformly in all admissible
parameters. Consequently \eqref{eq:7-2c} proves, for every
\(s\in\mathscr S(n)\),
\begin{equation}
 \mathbb P_{p_{3/2}}\!\left(
   B_{\operatorname{tz}(\lfloor s/2\rfloor)+1}
       \leftrightarrow S_{\lfloor s/2\rfloor/2}
       \text{ in }B_{\lfloor s/2\rfloor/2}\right)
 \geq1-\exp[-(\log n)^{c_3K-1}].                              \label{eq:7-2r}
\end{equation}

Put \(R_0=5\lambda/c_1\). We first verify the promised identification on
\(\mathscr S(n)\). Since \(m_1\geq n^{1/3}\) and
\(m_1^{1+c_1}\leq n\), every \(s\in\mathscr S(n)\) satisfies
\[
 s\geq n^{1/3},\qquad
 s\leq n^{(1+3c_1/4)/(1+c_1)}\leq n/2
\]
when \(n\) is sufficiently large in terms of \(c_1\). Applying
\eqref{eq:7-2j} with its scale variable equal to \(2s\) therefore gives
\begin{equation}
 \operatorname{tz}(s)\geq4b\geq4,\qquad
 \widetilde f(s)=\operatorname{tz}(s)
 \quad(s\in\mathscr S(n)).                                  \label{eq:7-2r-prime}
\end{equation}
Auxiliary Lemma~\ref{aux:rounded-scale-selection}, applied to the globally
admissible function
\(\widetilde f\), now supplies integers \(m_-,m\in\mathscr S(n)\) for
which
\[
 \tfrac12m(\log n)^{-R_0}\leq m_-
       \leq2m(\log n)^{-R_0},\qquad
 \frac{\widetilde f(m)}{\widetilde f(m_-)}
 \leq(\log n)^{40\lambda/c_1^2}.
\]
By \eqref{eq:7-2r-prime}, the last ratio is exactly
\(\operatorname{tz}(m)/\operatorname{tz}(m_-)\), so all subsequent uses
of the selected scales are unchanged.

Apply \eqref{eq:7-2l} at scale \(2m_-\). Since
\(2m_-\leq\frac m3(\log m)^{-4\lambda/c_1}\) and
\(2\operatorname{tz}(m/2)+2\leq
\operatorname{tz}(m_-)(\log n)^{c_3K}\) after increasing \(K_1\),
monotonicity of the corridor function gives
\begin{equation}
 \kappa_{p_{7/4}}^{\rm s}
 \left(2\operatorname{tz}(m/2)+2,
       \frac m3(\log m)^{-4\lambda/c_1}\right)
 \geq c_4e^{-2T}.                                             \label{eq:7-2b}
\end{equation}
Failure of \eqref{eq:7-2m} also gives
\begin{equation}
 t:=\operatorname{tz}(m/2)
 \leq \frac m{17}(\log m)^{-c_3K/2}.                          \label{eq:7-2s}
\end{equation}
Splitting a path of length at most \(6t+6\) into at most three pieces of
length at most \(2t+2\), applying \eqref{eq:7-2b} to the pieces, and using site FKG
gives the enlarged corridor estimate
\[
 \kappa_{p_{7/4}}^{\rm s}
 \left(6t+6,\frac m3(\log m)^{-4\lambda/c_1}\right)
 \geq c_4^3e^{-6T}.
\]

We make the conditional calculation in the cluster-merging lemma formal.
Let \(D_m=B_{m/8}\),
\[
 k=2\lfloor(\log n)^D\rfloor,
 \quad \varepsilon=(\log n)^{-(D+1)},
 \quad r_i=\frac m8-\frac{im}{40}(\log n)^{-D},
 \quad q_i=\operatorname{Spr}(p_{7/4};i\varepsilon).
\]
Let \(\mathscr C_0\) be the family of \(q_0\)-site clusters in \(D_m\)
meeting \(S_{r_0}\). Recursively, let \(\mathscr C_{i+1}\) be the family
of \(q_{i+1}\)-clusters in \(D_m\) containing some
\(C\in\mathscr C_i\) with \(C\cap S_{r_{i+1}}\ne\varnothing\). Write
\(\mathcal H_i=(\mathscr C_0,\ldots,\mathscr C_i)\) for the complete
exploration history.

The geometric trial construction needs the following explicit covering
event. For each \(u\in D_m\), choose on a fixed geodesic from \(u\)
to \(o\) a vertex \(w(u)\) such that
\[
 d(u,w(u))\leq t+2,\qquad B_{t+1}(w(u))\subseteq D_m,
 \qquad d(o,w(u))\leq m/8-t-2.
\]
Take \(w(u)=u\) when the last inequality already holds. Let \(\Omega\)
be the intersection, over \(u\in D_m\), of the translated
\(p_{3/2}\)-events
\[
 B_{t+1}(w(u))\leftrightarrow S_{m/4}(w(u))
       \text{ in }B_{m/4}(w(u)).
\]
This is precisely the translated event in \eqref{eq:7-2r} with target scale
\(s=m\), up to the standing integer-rounding convention. Consequently a
union bound gives
\begin{equation}
 \mathbb P(\Omega^c)
 \leq\operatorname{Gr}_G(m)
       e^{-(\log n)^{c_3K-1}}
 \leq e^{-(\log n)^{c_3K/2}}.                                \label{eq:7-2t}
\end{equation}
On \(\Omega\), the witnessing path starts in \(D_m\) and ends outside
\(D_m\), since \(B_{t+1}(w(u))\subseteq D_m\), whereas every
\(y\in S_{m/4}(w(u))\) satisfies \(d(o,y)>m/8\).
Restricting at its first hit on \(S_{m/8}\) therefore gives a
\(p_{3/2}\)-cluster in \(D_m\) meeting \(S_{m/8}\) and
\(B_{2t+3}(u)\). In particular,
\begin{equation}
 d\left(u,\bigcup\mathscr C_0\right)\leq2t+3
 \quad(u\in D_m).                                             \label{eq:7-2u}
\end{equation}
If a possible value \(\mathscr F_i\) satisfies
\(\mathbb P(\mathscr C_i=\mathscr F_i,\Omega)>0\), then \eqref{eq:7-2u} and the
radial separation \eqref{eq:7-4} below imply deterministically that
\begin{equation}
 d(u,F_i)\leq2t+3
 \quad\text{for every }u\in
 B_{r_{i+1}+m(\log m)^{-\lambda/c_1}},
 \qquad F_i=\bigcup\mathscr F_i.                              \label{eq:7-2v}
\end{equation}
Indeed, let \(C_0\in\mathscr C_0\) be the cluster supplied by \eqref{eq:7-2u},
and choose \(x\in C_0\) with \(d(u,x)\leq2t+3\). For the displayed
range of \(u\), \eqref{eq:7-4} gives
\[
 d(o,x)\leq r_{i+1}+m(\log m)^{-\lambda/c_1}+2t+3<r_i.
\]
Since \(C_0\) also meets \(S_{r_0}\), a path in \(C_0\) from \(x\) to
\(S_{r_0}\) meets each \(S_{r_j}\), \(1\leq j\leq i\). Induction in
the definition of the peeling families therefore places \(C_0\) inside
a member of \(\mathscr F_i\). This proves \eqref{eq:7-2v}.

\begin{lemma}[Site peeling filtration]\label{lem:site-peeling-filtration}
Fix a possible history
\(\mathcal H_i=(\mathscr F_0,\ldots,\mathscr F_i)\), and put
\(F_i=\bigcup\mathscr F_i\). Let \(Q_1,\ldots,Q_s\) be pairwise
disjoint subsets of
\[
  B_{r_{i+1}+m(\log m)^{-\lambda/c_1}
         +(m/2)(\log m)^{-4\lambda/c_1}}.
\]
If \(\lambda/c_1>D+2\) and \(n\) is sufficiently large, then,
conditionally on \(\mathcal H_i\), the coordinates in
\((\bigcup_jQ_j)\setminus F_i\) are mutually independent and have
laws
\begin{equation}
 \begin{array}{c|c}
 z\in\partial_V^{D_m}F_i & U_z\text{ is uniform on }(q_i,1],\\
 d(z,F_i)\geq2 & U_z\text{ is uniform on }[0,1].
 \end{array}                                                \label{eq:7-3}
\end{equation}
Here \(\partial_V^{D_m}\) is the external vertex boundary relative
to \(D_m\). The sites of \(F_i\) may have additional interval
restrictions recording when they entered the exploration, but they
are deterministically \(q_i\)-open and are not random inputs to a
trial.
\end{lemma}

\textbf{Proof.} Explore a site cluster by revealing each reached open site and
each tested site in its external vertex boundary. For \(i=0\), specifying
\(\mathscr C_0\) reveals \(F_0\), \(\partial_V^{D_m}F_0\), and the closed sites
of \(S_{r_0}\setminus F_0\), and nothing else. The last set is outside all
the \(Q_j\)'s by the displayed radial bound below.

Proceed inductively. A member of \(\mathscr F_j\) that is not used as a
seed at stage \(j+1\) does not meet \(S_{r_{j+1}}\). Since it is connected
and met the preceding outer sphere, it cannot enter \(B_{r_{j+1}-1}\),
and its revealed boundary cannot enter \(B_{r_{j+1}-2}\). Thus all
coordinates belonging only to an extinct part of the history lie outside
the later trial region. Every member that is used as a seed is contained
in a member of \(\mathscr F_i\). A site revealed earlier as closed and
adjacent to such a surviving cluster is either subsequently opened and
belongs to \(F_i\), or is still closed and belongs to
\(\partial_V^{D_m}F_i\). Therefore, on sites in the trial region outside
\(F_i\), the entire history imposes exactly the two coordinatewise
conditions in \eqref{eq:7-3}. To state the factorization explicitly, let \(W_i\)
be the ball displayed in the lemma. There are intervals \(I_z\subseteq
[0,q_i]\), one for each \(z\in F_i\cap W_i\), and an event \(R_i\)
measurable with respect to coordinates outside \(W_i\), such that
\begin{equation}
 \{\mathcal H_i=(\mathscr F_0,\ldots,\mathscr F_i)\}
 =R_i\cap\!\bigcap_{z\in F_i\cap W_i}\!\{U_z\in I_z\}
     \cap\!\bigcap_{z\in\partial_V^{D_m}F_i\cap W_i}\!\{U_z>q_i\}.
                                                                    \label{eq:7-3a}
\end{equation}
Here is the induction proving that no further condition is hidden in
\eqref{eq:7-3a}; set \(F_{-1}=\varnothing\). At stage \(j\), once the seed set and all earlier explored sets
are fixed, the atom that the newly explored cluster is \(K\) is exactly
\begin{equation}
 \bigcap_{z\in K\setminus F_{j-1}}\{U_z\leq q_j\}
 \cap
 \bigcap_{z\in\partial_V^{D_m}K\setminus F_{j-1}}
             \{U_z>q_j\}.                                    \label{eq:7-3b}
\end{equation}
The exploration order adds no condition: the next site is selected from
the already determined frontier before its label is read. Intersecting
\eqref{eq:7-3b} over \(j\leq i\) gives one interval condition for each tested
coordinate because \(q_0<\cdots<q_i\). A coordinate can change from a
boundary site to an open-cluster site only once. The radial induction
above places every tested coordinate in \(W_i\) either in \(F_i\) or in
\(\partial_V^{D_m}F_i\); all tests belonging only to an extinct cluster
or an earlier seed sphere are outside \(W_i\) and are absorbed into
\(R_i\). This proves \eqref{eq:7-3a} by induction on \(j\).
Indeed, the upper endpoint of \(I_z\) is \(q_a\), where \(a\) is the
first stage at which \(z\) belongs to an explored cluster. Its lower
endpoint is \(q_b\) if \(b<a\) is the last earlier stage at which \(z\)
was exposed as a boundary site, and is zero if there was no such stage.
The preceding induction proves \eqref{eq:7-3a}: all other earlier
conditions belong to extinct clusters or seed spheres and hence to
\(R_i\). Once the vertex sets are fixed, their connectedness and their
partition into clusters are deterministic. Since \eqref{eq:7-3a} is a product of
one-coordinate interval events, conditioning preserves independence
between all coordinates in \(W_i\setminus F_i\); this is precisely the
partition-cell conclusion of Lemma~\ref{lem:stopped-product-kernel}.
Thus the terminal history, even when generated adaptively, imposes no
additional condition on a trial coordinate, proving the claim.
\(\square\)

It remains to verify the radial bound used in the lemma. Increase the
auxiliary tube parameter \(\lambda\), if necessary, so that
\(\lambda/c_1>D+2\); this only increases the constant in the geometric
input. We have
\[
 r_i-r_{i+1}=\frac{m}{40}(\log n)^{-D},
\]
whereas a trial center lies in
\(B_{r_{i+1}+m(\log m)^{-\lambda/c_1}}\) and its trial radius is
\(\frac m2(\log m)^{-4\lambda/c_1}\). The prescribed choice of
\(\lambda\), followed by increasing \(K_1(d,D,\lambda)\) in \eqref{eq:7-2s},
gives, for all sufficiently large \(n\),
\begin{equation}
 m(\log m)^{-\lambda/c_1}
 +\frac m2(\log m)^{-4\lambda/c_1}
 +2t+4< r_i-r_{i+1}.                                          \label{eq:7-4}
\end{equation}
All earlier seed spheres lie still farther out. Applying
Lemma~\ref{lem:site-peeling-filtration} with
\(\mathscr F=\mathscr F_i\), under the monotone coupling and conditional
on the full history, will therefore expose no extinct-cluster coordinate
inside a trial tube.

We also spell out where the trial centers come from. Fix a value
\(\mathscr F_i\) compatible with \(\Omega\), an eligible
\(C\in\mathscr F_i\) other than a deterministically chosen cluster
minimizing \(d(o,C)\), where eligible means
\(d(o,C)\leq r_{i+1}\), and
put
\[
 r_* =\left\lfloor m(\log m)^{-2\lambda/c_1}\right\rfloor,
 \qquad a_* =\left\lfloor\frac{r_*}{3}\right\rfloor .
\]
The power gaps defining \(\mathscr S(n)\) imply, uniformly for the
selected \(m\) and all sufficiently large \(n\),
\begin{equation}
 m_1\leq a_*\leq r_*\leq m_2,
 \qquad a_*\geq\tfrac14m(\log m)^{-2\lambda/c_1}.              \label{eq:7-4-integer-scales}
\end{equation}
Thus both applications of Lemma~\ref{lem:plentiful-tubes} below are at
integer scales belonging to its guaranteed interval. If
\(d(o,C)\leq r_*\), connectivity, the fact that every member descends
from a cluster meeting \(S_{r_0}\), and the choice of the distinguished
cluster show that both \(C\) and the union of the other clusters contain
an \((r_*,3r_*)\)-crossing. Apply the plentiful-annular-tubes property at
scale \(r_*\).

If \(d(o,C)>r_*\), let \(v\in C\) minimize \(d(o,v)\). A path in \(C\)
from \(v\) toward \(S_{r_0}\), truncated on first reaching
\(S_{r_*}(v)\), gives the first crossing. For the second, fix a geodesic
\(v=v_0,v_1,\ldots,o\) from \(v\) to \(o\) and use its initial segment
through distance \(r_*\); this segment exists because
\(d(o,v)>r_*\). Minimality of \(v\) implies
\(v_j\notin C\) for every \(j\geq1\). More precisely, use the portions
of these two paths between
\(S_{a_*}(v)\) and \(S_{3a_*}(v)\) when applying plentiful annular
tubes at the integer scale \(a_*\). These portions exist because
\(3a_*\leq r_*\), and they are exactly \((a_*,3a_*)\)-crossings. If
\(x\) is on the selected inward portion,
then
\[
 d(x,C)\geq d(o,v)-d(o,x)=d(v,x)\geq a_*>2t+3,
\]
where the final inequality follows from \eqref{eq:7-2s} after increasing
\(K_1\). Consequently every connector from the crossing in \(C\) to the
inward crossing exits the \((2t+3)\)-neighborhood of \(C\). All these
uses are therefore legitimate at the precise integer scales displayed
in \eqref{eq:7-4-integer-scales}.

Orient every resulting connecting path from its crossing in \(C\) toward
the other crossing. There are two cases. If the path never leaves the
\((2t+3)\)-neighborhood of \(C\), take its terminal vertex as the trial
center. This case can occur only in the near-root construction, and that
terminal vertex belongs to the crossing contained in another old cluster.
Thus the terminal center itself belongs to another old cluster and is at
distance at most \(2t+3\) from \(C\).
Otherwise take the first path vertex outside the
\((2t+3)\)-neighborhood of \(C\), equivalently the successor of the last
vertex in the maximal initial path segment that remains in that
neighborhood. This vertex exists, is at distance at most \(2t+4\) from
\(C\), and lies outside its \((2t+3)\)-neighborhood.
Equation \eqref{eq:7-2v} places it within distance \(2t+3\) of \(F_i\);
the old cluster supplied there cannot be \(C\), so it is another member of
\(\mathscr F_i\). In the exit case the center is therefore within
\(2t+4\) of \(C\) and within \(2t+3\) of another old cluster. These two
case-by-case bounds are exactly the two ball-intersection assertions below.
Thus we obtain a set \(U\) of at least
\[
 \tfrac12(\log m)^{c_1\lambda}
\]
centers, mutually \(R=m(\log m)^{-4\lambda/c_1}\)-separated and contained
in \(B_{r_{i+1}+m(\log m)^{-\lambda/c_1}}\), such that
\[
 B_{2t+4}(u)\cap C\ne\varnothing,
 \qquad
 B_{2t+4}(u)\cap\bigcup(\mathscr F_i\setminus\{C\})
       \ne\varnothing
\]
for every \(u\in U\). The cardinality, separation, and containment follow
from the definitions of polylog-plentiful tubes as follows. For
\(\sigma\in\{a_*,r_*\}\),
\[
 \log\sigma\sim\log m,\qquad
 \sigma(\log\sigma)^{-\lambda/c_1}
       \geq m(\log m)^{-4\lambda/c_1}=R,
 \qquad
 \sigma(\log\sigma)^{\lambda/c_1}
       \leq m(\log m)^{-\lambda/c_1}
\]
after increasing the scale threshold (fixed numerical factors are
absorbed by the strict logarithmic gaps). Hence the disjoint tubes give
mutually \(R\)-separated choices, every connector has the asserted
containment, and
\(\lceil(\log\sigma)^{c_1\lambda}\rceil
 \geq\tfrac12(\log m)^{c_1\lambda}\). No percolation variable is used
in this step.

Lemma~\ref{lem:site-peeling-filtration} now gives within each trial tube
\begin{equation}
 \begin{array}{c|c}
 z\in F_i & z\text{ is deterministically open},\\
 z\in\partial_V^{D_m}F_i&
 \mathbb P(z\in\omega_{q_{i+1}}\mid z\notin\omega_{q_i})
   =\dfrac{q_{i+1}-q_i}{1-q_i}\geq c\varepsilon,\\
 d(z,F_i)\geq2&
 \mathbb P(z\in\omega_{q_{i+1}})=q_{i+1},
 \end{array}                                                  \label{eq:7-5}
\end{equation}
and all coordinates in the last two rows remain independent.
Since the \(q_i\)'s remain in a fixed compact subset of \((0,1)\), this
boundary density is also at most \(q_{i+1}\) for all sufficiently large
\(n\), as required for the thinning below.

We now specify disjoint trial coordinate sets. Put
\(R=m(\log m)^{-4\lambda/c_1}\). For each trial center \(u\), choose,
deterministically from \(\mathscr F_i\), a path \(\gamma_u\) of length at
most \(4t+8\) from \(C\) to
\(\bigcup(\mathscr F_i\setminus\{C\})\), contained in
\(B_{2t+4}(u)\), and set
\[
 Q_u=B(\gamma_u,R/3).
\]
The geometric construction gives \(d(u,v)\geq R\) for distinct centers,
while
\(Q_u\subset B(u,2t+4+R/3)\). After increasing \(K_1\), \eqref{eq:7-2s} gives
\(2t+4<R/6\), so the \(Q_u\)'s are pairwise disjoint and lie inside both
\(D_m\) and the radial region of
Lemma~\ref{lem:site-peeling-filtration}. Since \(F_i\subseteq D_m\),
we also have \(Q_u\cup F_i\subseteq D_m\).

In the induced graph \(G[Q_u\cup F_i]\), declare every site of \(F_i\)
open, without identifying distinct vertices, and first sample ordinary
\(q_{i+1}\)-site percolation on \(Q_u\setminus F_i\). Let
\[
 D_i=F_i\setminus C
   =\bigcup(\mathscr F_i\setminus\{C\}),
\]
and define the auxiliary success event to be
\(C\leftrightarrow D_i\) in \(G[Q_u\cup F_i]\). On this event choose a
canonical shortest open path between the two sets. Its internal vertices
lie outside \(F_i\). Minimality implies that only its first internal site
can touch \(C\), only its last can touch \(D_i\), and no strictly
intermediate site touches any member of \(\mathscr F_i\); any such
contact would give a shorter path in the same induced graph. Obtain the
conditional law in \eqref{eq:7-5} by independently thinning open boundary sites
from density \(q_{i+1}\) to
\(r=(q_{i+1}-q_i)/(1-q_i)\). Conditional on the ordinary configuration,
the canonical witness uses at most two boundary sites and hence survives
with probability at least
\((r/q_{i+1})^2\geq(c\varepsilon)^2\). (With zero or one such site the
true retention probability is larger.) Declaring \(F_i\) open can only
increase the ordinary corridor probability between the selected
endpoints in \(Q_u\). Thus the corridor estimate gives, uniformly over
the complete history,
\begin{equation}
 \mathbb P\left(C\leftrightarrow
       \bigcup(\mathscr F_i\setminus\{C\})
       \text{ in }G[Q_u\cup F_i]\mid\mathcal H_i\right)
 \geq(c\varepsilon)^2 c_4^3e^{-6T}.                            \label{eq:7-6}
\end{equation}
This event is measurable with respect to the random coordinates
\(Q_u\setminus F_i\), with \(F_i\) a deterministic open set under the
conditioned history. A success joins two old clusters inside \(D_m\)
because \(Q_u\cup F_i\subseteq D_m\).
We record explicitly how the fixed prefactor in \eqref{eq:7-6} is
absorbed. Put \(a=2D+3\) and \(c_{\rm tr}=c^2c_4^3>0\). Since
\(T\leq3\sqrt{\log\log n}\),
\begin{equation}
 \log\bigl(e^{-6T}\log n\bigr)
 =\log\log n-6T
 \geq\log\log n-18\sqrt{\log\log n}\longrightarrow\infty.
                                                               \label{eq:7-6a}
\end{equation}
Every \(m\in\mathscr S(n)\) satisfies
\(\log m\geq c_m\log n\) for a fixed \(c_m=c_m(d,D)>0\).
Consequently, after increasing the threshold for \(n\),
\(c_{\rm tr}e^{-6T}\log n\geq c_m^{-a}\). Using
\(\varepsilon=(\log n)^{-(D+1)}\), the right side of
\eqref{eq:7-6} is therefore at least
\[
 c_{\rm tr}e^{-6T}(\log n)^{-2D-2}
 =c_{\rm tr}e^{-6T}\log n\,(\log n)^{-a}
 \geq c_m^{-a}(\log n)^{-a}
 \geq(\log m)^{-a}.
\]
This is the asserted lower bound \((\log m)^{-2D-3}\).
The sets of random coordinates \(Q_u\setminus F_i\) are disjoint, so
Lemma~\ref{lem:site-peeling-filtration} makes the trial events
conditionally independent. With
\(|U|\geq\frac12(\log m)^{c_1\lambda}\) and
\(\varepsilon=(\log n)^{-(D+1)}\), \eqref{eq:7-6} gives
\begin{equation}
 \mathbb P(\text{no trial merges }C\mid\mathcal H_i)
 \leq\left(1-(\log m)^{-2D-3}\right)^{|U|}
 \leq e^{-2(\log n)^{c_1\lambda/2}}.                         \label{eq:7-7}
\end{equation}
This bound is uniform over all histories ending in an
\(\Omega\)-compatible \(\mathscr F_i\), so averaging over those histories
gives the same estimate conditional only on
\(\mathscr C_i=\mathscr F_i\). By \eqref{eq:7-2h-prime}, a union bound over at most
\(e^{(\log n)^D}\) old clusters
proves the required merging estimate for every compatible family. Notice
that we have not conditioned the trial law on \(\Omega\); compatibility is
used only for the deterministic covering property \eqref{eq:7-2v}. This
preserves the product factorization supplied by
Lemma~\ref{lem:site-peeling-filtration}.

We finish the negative branch explicitly. Let \(\mathcal E_i\) be the
event that every member of \(\mathscr C_i\), other than a fixed
deterministically chosen closest member and satisfying
\(d(o,C)\leq r_{i+1}\),
merges at level \(q_{i+1}\) with another member. Equation \eqref{eq:7-7} and the
last union bound give, for every \(\Omega\)-compatible value
\(\mathscr F_i\),
\begin{equation}
 \mathbb P(\mathcal E_i^c\mid\mathscr C_i=\mathscr F_i)
 \leq\tfrac12\exp[-(\log n)^{c_1\lambda/2}].                  \label{eq:7-7a}
\end{equation}
On \(\mathcal E_i\), every new cluster except possibly the distinguished
one contains at least two old clusters, and hence
\begin{equation}
 |\mathscr C_{i+1}|
 \leq\left\lfloor\frac{|\mathscr C_i|-1}{2}\right\rfloor+1.    \label{eq:7-7b}
\end{equation}

Define \(\operatorname{Piv}^{\rm s}_{p,q}[r,R]\) using induced clusters:
it is the event that the \(p\)-configuration induced on \(B_R\) has two
distinct clusters that both meet \(B_r\) and \(S_R\), and that these two
clusters remain in distinct clusters of the \(q\)-configuration induced
on \(B_R\). Thus, on its complement, every pair of induced
\(p\)-crossing clusters is joined at parameter \(q\) inside \(B_R\).
Since
\[
 k\varepsilon\leq\frac2{\log n}\leq\frac\delta4
\]
for large \(n\), \(q_k\leq p_2\). We check directly that the new
induced-cluster definition matches the peeling process. Two clusters
witnessing
\(\operatorname{Piv}^{\rm s}_{p_{7/4},p_2}[m/16,m/8]\) are distinct
members of \(\mathscr C_0\), since \(D_m=B_{m/8}\). Moreover,
\[
 r_k\geq \frac m8-\frac m{20}=\frac{3m}{40}>\frac m{16}.
\]
Each witnessing cluster contains a path from \(B_{m/16}\) to
\(S_{m/8}\), so, with the standing floor convention and after increasing
the fixed scale threshold, it meets every seed sphere
\(S_{r_i}\), \(0\leq i\leq k\). Consequently both clusters have
descendants in every \(\mathscr C_i\). Since they remain disconnected in
the \(p_2\)-configuration induced on \(D_m\), and \(q_k\leq p_2\), their
two descendants in \(\mathscr C_k\) are distinct. Therefore
\[
 \operatorname{Piv}^{\rm s}_{p_{7/4},p_2}[m/16,m/8]
       \subseteq\{|\mathscr C_k|\geq2\}.
\]
This is the direction required by the argument. Since
\(|\mathscr C_0|\leq\operatorname{Gr}_G(m)
\leq e^{(\log n)^D}\), iterating \eqref{eq:7-7b} for
\(k=2\lfloor(\log n)^D\rfloor\) gives
\(|\mathscr C_k|\leq1\) on \(\bigcap_{i<k}\mathcal E_i\). Since on
\(\Omega\) every realized \(\mathscr C_i\) is compatible, \eqref{eq:7-2t},
\eqref{eq:7-7a}, and a union bound yield, after taking \(K_1\) large relative to
\(\lambda\),
\begin{equation}
 \mathbb P\bigl(
   \operatorname{Piv}^{\rm s}_{p_{7/4},p_2}[m/16,m/8]\bigr)
 \leq\mathbb P(\Omega^c)
      +\sum_{i<k}\mathbb P(\Omega\cap\mathcal E_i^c)
 \leq e^{-(\log n)^{c_1\lambda/3}}.                            \label{eq:7-7d}
\end{equation}

By \eqref{eq:7-2r} with target scale \(s=m\) and by \eqref{eq:7-2s}, a path from
\(B_{\operatorname{tz}(m/2)+1}\) to \(S_{m/4}\) crosses \(S_{m/17}\);
hence
\begin{equation}
 \mathbb P_{p_{7/4}}(S_{m/17}\leftrightarrow S_{m/4}
                   \text{ in }B_{m/4})
 \geq1-e^{-(\log n)^{c_1\lambda/3}}.                           \label{eq:7-7e}
\end{equation}
Indeed, stop the path supplied by \eqref{eq:7-2r} at its first visit to
\(S_{m/4}\) and retain the segment after its last preceding visit to
\(S_{m/17}\); this entire segment lies in \(B_{m/4}\).
Let \(\gamma\) have length at most
\(\exp[(\log n)^{c_1\lambda/4}]\), with endpoints \(u,v\). Translate
\eqref{eq:7-7e} and the complement of \eqref{eq:7-7d} to every vertex of \(\gamma\). The
deterministic chaining implication is
\[
 \begin{split}
 &\{u\leftrightarrow S_n(u)\}_{p_{7/4}}
 \cap\{v\leftrightarrow S_n(v)\}_{p_{7/4}}\\
 &\quad{}\cap\bigcap_{x\in\gamma}
   \left(\{S_{m/17}(x)\leftrightarrow S_{m/4}(x)
              \text{ in }B_{m/4}(x)\}_{p_{7/4}}
   \cap
   (\operatorname{Piv}^{\rm s}_{p_{7/4},p_2}
      [m/16,m/8](x))^c\right)\\
 &\hspace{5cm}\subseteq
   \{u\leftrightarrow v\text{ in }B_n(\gamma)\}_{p_2}.
 \end{split}
\]
For completeness, apply the adjacent-center lemma explicitly. For each
\(x\in\gamma\), choose canonically a \(p_{7/4}\)-open path \(P_x\)
realizing the annular crossing, oriented from \(S_{m/17}(x)\) toward
\(S_{m/4}(x)\). For adjacent \(x,y\in\gamma\), take
\[
 a=\lfloor m/17\rfloor,\qquad r=\lfloor m/16\rfloor,
 \qquad R=\lfloor m/8\rfloor,\qquad b=\lfloor m/4\rfloor.
\]
When \(m\geq272\),
\[
 a+1\leq r,
 \qquad b-1>R,
\]
so Lemma~\ref{lem:adjacent-annular-comparison} shows that \(P_y\), as
well as \(P_x\), contains an \(x\)-centered crossing subpath in
\(B_{m/8}(x)\). Let \(C_{y\to x}\) and \(C_{x\to x}\) be the respective
clusters of these subpaths in the \(p_{7/4}\)-configuration induced on
\(B_{m/8}(x)\). They are exactly clusters of the kind quantified by the
induced pivotal event. Its complement therefore places them in one
\(p_2\)-cluster of \(B_{m/8}(x)\), whether they were distinct or already
equal at density \(p_{7/4}\). Notice that this conclusion connects the
actual crossing subpaths inside the indicated ball; it does not infer
local connectivity from equality of full-space clusters. For three
successive centers \(x,y,z\), the two local connections are linked by
the common path \(P_y\subseteq B_{m/4}(y)\). Iteration therefore joins
all the selected paths \(P_x\) locally along \(\gamma\).

The endpoint clusters are treated in the same way, but we record the
localization. On \(\{u\leftrightarrow S_n(u)\}_{p_{7/4}}\), take a
canonical open path stopped on its first visit to \(S_n(u)\), and retain
its initial segment through its first visit to \(S_{m/8}(u)\). This
segment and the corresponding initial segment of \(P_u\) lie in
\(B_{m/8}(u)\) and cross from \(B_{m/16}(u)\) to \(S_{m/8}(u)\). Their
induced \(p_{7/4}\)-clusters are therefore joined inside that ball by the
complement of the pivotal event at \(u\). The identical statement holds
at \(v\). Chaining these endpoint connections with the preceding local
connections proves the inclusion. The selected annular
paths and all joining paths lie in \(B_n(\gamma)\), since
\(B_{m/4}(x)\subseteq B_n(\gamma)\) for \(x\in\gamma\) and \(m\leq n\);
the two stopped endpoint paths lie in \(B_n(u)\) and \(B_n(v)\),
respectively. The finitely many scales with \(m<272\) are already
absorbed by the initial threshold in Proposition~\ref{prop:site-low-growth-step}.

Apply site FKG to the endpoint and annular crossing events, then subtract
the two-parameter pivotal failures by a union bound. Equations
\eqref{eq:7-2e}, \eqref{eq:7-7d}, and \eqref{eq:7-7e} give
\begin{equation}
 \begin{split}
 \mathbb P_{p_2}(u\leftrightarrow v\text{ in }B_n(\gamma))
 &\geq e^{-2T}
   \left(1-e^{-(\log n)^{c_1\lambda/3}}\right)^{|\gamma|}
   -|\gamma|e^{-(\log n)^{c_1\lambda/3}}\\
 &\geq e^{-3T}.
 \end{split}                                                    \label{eq:7-7g}
\end{equation}
This is \eqref{eq:7-2f} in the negative branch. Together with
\eqref{eq:7-2p}, it proves
Proposition~\ref{prop:site-low-growth-step}.

The only boundary-count comparison used in absorbing the constants above is
\(|\partial_VA|\leq|\partial_EA|\leq d|\partial_VA|\). We emphasize the
correct finite-energy bookkeeping for path surgery. A path of \(r\) edges
contains at most \(r+1\) sites. If the percolation parameter ranges in a
fixed compact interval \(I=[p_-,p_+]\subset(0,1)\), forcing those sites
open has cost at most
\[
 p_-^{-(r+1)}\leq \exp\{C_I(r+1)\};
\]
equivalently, an injective opening surgery that records the modified subset
has a Radon--Nikodym and fibre loss of the form \(\exp\{O_I(r+1)\}\), not a
constant uniform in \(r\). No such additional surgery is used in the
derivation of \eqref{eq:7-7d}--\eqref{eq:7-7g}: that derivation uses FKG,
the displayed pivotal estimates, and deterministic local chaining. In the
places where path-opening surgery is used, its length-dependent loss is
retained explicitly, as in \eqref{eq:6-10a-prime} and
Lemma~\ref{lem:finite-energy-path-opening}. \(\square\)

For later use in the Full-Space implication, we record separately the
numerical check for Proposition 6.2. Let \(s_1,s_3\) be its constants with
\(D=1\), put
\[
 c_0=(8\cdot3^{20})^{-1},\qquad
 h_j=\exp[-c_0(\log n_j)^{20}],\qquad
 \Delta_j=T_j^{-1}.
\]
If \(r\in[n_j^{1/3},n_j]\) satisfies
\(\log\operatorname{Gr}(r)>(\log r)^{20}\) and
\(\rho=\lfloor r/3\rfloor\), submultiplicativity gives
\(\operatorname{Gr}(\rho)\geq h_j^{-1}\) for large \(n_j\). Moreover,
since \(n_{j+1}=\exp[(\log n_j)^9]\),
\begin{equation}
 h_j^{s_1\Delta_j^4}
 \leq\frac{s_3}{2n_{j+1}+1},
 \qquad
 4h_j^{s_1\Delta_j^4}\leq e^{-T_j},
 \qquad h_j\leq h_0.                                           \label{eq:7-2g}
\end{equation}
Indeed, the logarithms on the left are of order
\(-(\log n_j)^{20}/(\log\log n_j)^2\). The logarithms on the right have
orders \(-(\log n_j)^9\) and \(-\sqrt{\log\log n_j}\), respectively.

We record the assembly into the induction rather than leave the three
implications implicit. Given \(n_0,p_0,\delta_0\), first define
\[
 n_{-1}:=(\log n_0)^{1/2}.
\]
For each integer \(i\geq1\), recursively define
\[
 n_i:=e^{(\log n_{i-1})^9},\qquad
 \delta_i:=(\log\log n_i)^{-1/2}.
\]
Finally set \(p_1:=\operatorname{Spr}(p_0;\delta_0)\), and, for each
integer \(i\geq1\), recursively set
\(p_{i+1}:=\operatorname{Spr}(p_i;\delta_i)\). Define
\begin{equation}
 \begin{aligned}
 T_i&=\sqrt{\log\log n_i}\quad(i\geq-1),\\
 \mathsf F_0&=\left\{
   \mathbb P_{p_0}(u\leftrightarrow v)
      \geq e^{-T_{-1}}
      \text{ for all }u,v\in B_{n_0}\right\},\\
 \mathsf F_i&=\left\{
   \mathbb P_{p_i}(u\leftrightarrow v)
      \geq e^{-T_i}
      \text{ for all }u,v\in B_{n_i}\right\}\quad(i\geq1),\\
 \mathsf C_i&=\left\{
   \kappa_{p_i}^{\rm s}(e^{(\log m)^{10}},m)
      \geq e^{-T_i}
      \text{ for every }m\in\mathscr L(G,20)
          \cap[n_{i-2},n_{i-1}]\right\}\quad(i\geq1).
 \end{aligned}                                                  \label{eq:7-8}
\end{equation}
Notice the exact identities
\begin{equation}
 T_{i+1}=3T_i\quad(i\geq0),\qquad
 \log\log m\in[T_i^2/9,T_i^2]
       \quad(m\in[n_{i-1},n_i],\ i\geq1).                    \label{eq:7-8a}
\end{equation}
Proposition~\ref{prop:site-low-growth-step} gives the two Corridor
implications, and Proposition~\ref{prop:site-snowballing} gives the
Full-Space implication. The numerical checks
\eqref{eq:7-2g}--\eqref{eq:7-2h} yield constants \(K(d,a,b),N(d,a,b)\) such that
\begin{equation}
 \begin{aligned}
 \mathsf F_0&\Longrightarrow
      [\mathsf C_1\ \text{or}\ p_1\geq p_c^{\rm s}],\\
 \mathsf F_i\cap\bigcap_{k=1}^i\mathsf C_k&\Longrightarrow
      [\mathsf C_{i+1}\ \text{or}\ p_{i+1}\geq p_c^{\rm s}]
       &&(i\geq1),\\
 \mathsf F_j\cap\mathsf C_{j+1}&\Longrightarrow
      [\mathsf F_{j+1}\ \text{or}\ p_{j+1}\geq p_c^{\rm s}]
       &&(j\geq0).
 \end{aligned}                                                  \label{eq:7-9}
\end{equation}

We give the details for the two possible scale regimes.

\textbf{The base Corridor implication.} Fix
\(m\in\mathscr L(G,20)\cap[n_{-1},n_0]\), and take
\(b=b_{\rm s}(m,p_0)\). Finiteness of the burn-in gives \(b\geq1\), its
definition gives the pivotal bound in \eqref{eq:7-2e}, and
\[
 \delta_K(b,m)
 \leq K^{1/4}\operatorname{Burn}^{\rm s}(G,n_0,p_0)
 \leq\delta_0.
\]
By transitivity, \(\mathsf F_0\) gives
\[
 \kappa_{p_0}^{\rm s}(m,\infty)
 \geq e^{-T_{-1}}
 \geq e^{-\sqrt{\log\log m}}.
\]
Apply Proposition~\ref{prop:site-low-growth-step} with \(n=m\) and
\(T=\sqrt{\log\log m}\). Its corridor alternative is at least
\[
 e^{-3\sqrt{\log\log m}}
 \geq e^{-3T_0}=e^{-T_1},
\]
and choosing \(\lambda\geq10/c_{\rm out}\) makes its allowed path length at least
\(e^{(\log m)^{10}}\). This proves the first line of \eqref{eq:7-9}.
The premise of Proposition~\ref{prop:site-multiscale-criterion} uses
\(T_{-1}\), rather than \(T_0\), so that the hypothesis of
Proposition~\ref{prop:site-low-growth-step} remains available near
\(m=n_{-1}\).
Here \(n_0\) is chosen large enough that
\(\sqrt{\log\log n_{-1}}\geq T_{\min}\), so the lower restriction on
\(T\) in Proposition~\ref{prop:site-low-growth-step} is satisfied
throughout the base interval.

\textbf{The later Corridor implications.} Fix \(i\geq1\) and
\(m\in\mathscr L(G,20)\cap[n_{i-1},n_i]\). We construct
\(b\leq m^{1/3}/8\) satisfying
\begin{equation}
 \delta_K(b,m)\leq\delta_i,
 \qquad
 \mathbb P_{p_i}(\operatorname{Piv}^{\rm s}[4b,m^{1/3}])
       \leq(\log m)^{-1}.                                    \label{eq:7-12}
\end{equation}
Put \(s=(\log m)^{2/3}\).

Because \(\mathscr L(G,20)\) was defined on all real outer scales, this
is a genuine dichotomy even though \(s\) need not be an integer. If
\(s\notin\mathscr L(G,20)\), its defining universal condition fails, so
there is \(t\in[s^{1/3},s]\) for which the strict reverse inequality
below holds. Take the integer test radius \(b=\lceil s\rceil\). For all
sufficiently large \(m\), it satisfies
\(1\leq b\leq m^{1/3}/8\), and monotonicity gives
\[
 \log\operatorname{Gr}_G(b)\geq\log\operatorname{Gr}_G(t)
 >(\log t)^{20}\geq c(\log\log m)^{20},
\]
so the first inequality in \eqref{eq:7-12} follows from \eqref{eq:7-8a}. With
\(N=m^{1/3}\), the coarse surgery \eqref{eq:6-10a-prime} and Lemma 6.1 give
\[
 \begin{split}
 \mathbb P_{p_i}(\operatorname{Piv}^{\rm s}[4b,N])
 &\leq \exp(Cb)
   C\left(\frac{(\log m)^{20}}{N-4b}\right)^{1/4}\\
 &\leq(\log m)^{-1}
 \end{split}
\]
for large \(n_0\), since \(b=(\log m)^{2/3}+O(1)=o(\log m)\).

Suppose instead that \(s\in\mathscr L(G,20)\). For large \(n_0\),
\(s\in[n_{-1},n_{i-1}]\), so some \(k\in\{1,\ldots,i\}\) has
\(s\in[n_{k-2},n_{k-1}]\). Hence \(\mathsf C_k\) gives
\begin{equation}
 \kappa_{p_k}^{\rm s}(e^{(\log s)^{10}},s)\geq e^{-T_k}.      \label{eq:7-14}
\end{equation}
Define \(\operatorname{Gr}^{-1}(x)=
\min\{r:\operatorname{Gr}_G(r)\geq x\}\) and take
\begin{equation}
 b=\left\lceil\frac15\min\left\{
 e^{(\log\log m)^9},
 \operatorname{Gr}^{-1}(e^{(\log m)^{1/10}})
 \right\}\right\rceil.                                      \label{eq:7-15}
\end{equation}
Submultiplicativity of ball volume and the bounded one-step volume ratio
give
\begin{equation}
 \log\operatorname{Gr}_G(b)\geq c(\log\log m)^9,
 \qquad
 \log\operatorname{Gr}_G(10b+1)
       \leq3(\log m)^{1/10}+C_d.                              \label{eq:7-16}
\end{equation}
For the lower bound, if the first term realizes the minimum, use
\(\operatorname{Gr}(5b)\geq5b+1\); if the inverse-growth term realizes
it, use its defining inequality. In either case
\(\log\operatorname{Gr}(5b)\geq(\log\log m)^9\) for large \(m\), and
then use
\(\operatorname{Gr}(5b)\leq\operatorname{Gr}(b)^5\). For the upper
bound, write
\(R=\operatorname{Gr}^{-1}(e^{(\log m)^{1/10}})\). Ceilings give
\(10b+1\leq3R\) for large \(m\). Minimality of \(R\), the bounded
one-step volume ratio, and submultiplicativity give
\(\operatorname{Gr}(3R)\leq\operatorname{Gr}(R)^3
\leq d^3e^{3(\log m)^{1/10}}\).

Since \(s\) is a low-growth scale,
\(\log\operatorname{Gr}(s)\leq(\log s)^{20}
< (\log m)^{1/10}\). Thus \eqref{eq:7-15} also gives, for large \(m\),
\begin{equation}
 5b\geq s,\qquad 8b\leq e^{(\log s)^{10}},\qquad
 10b\leq\tfrac12m^{1/3}.                                     \label{eq:7-17}
\end{equation}
For \(a,c\in S_{4b}\), a path through the root has length at most \(8b\)
and lies in \(B_{4b}\). Its \(s\)-neighborhood lies in
\(B_{4b+s}\subseteq B_{10b}\). Monotonicity in \(p\), \eqref{eq:7-14}, and
\eqref{eq:7-17} therefore imply
\begin{equation}
 \min_{a,c\in S_{4b}}
 \mathbb P_{p_i}(a\leftrightarrow c\text{ in }B_{10b})
 \geq e^{-T_k}.                                               \label{eq:7-18}
\end{equation}
Apply Auxiliary Lemma~\ref{aux:site-nearby-clusters} with radii
\(4b,10b,N=m^{1/3}\), then Lemma~\ref{lem:site-akn} with
exponent \(1/4\). Equations \eqref{eq:7-16}--\eqref{eq:7-18} yield
\[
 \begin{split}
 \mathbb P_{p_i}(\operatorname{Piv}^{\rm s}[4b,N])
 &\leq C\left(\frac{(\log m)^{20}}{m^{1/3}}\right)^{1/4}
       \exp\{C(\log m)^{1/10}+T_k\}\\
 &\leq(\log m)^{-1}.
 \end{split}
\]
Also \eqref{eq:7-16} and \eqref{eq:7-8a} give
\[
 \delta_K(b,m)\leq C K^{1/4}(\log\log m)^{-2}\leq\delta_i.
\]
This proves \eqref{eq:7-12} in both cases.

By \(\mathsf F_i\) and transitivity,
\(\kappa_{p_i}^{\rm s}(m,\infty)\geq e^{-T_i}\). Moreover
\(T_i\leq3\sqrt{\log\log m}\) by \eqref{eq:7-8a}.
Proposition~\ref{prop:site-low-growth-step}, \eqref{eq:7-12}, and
\(\lambda\geq10/c_{\rm out}\) now give either
\(p_{i+1}\geq p_c^{\rm s}\)
or the bound \(e^{-3T_i}=e^{-T_{i+1}}\) required by
\(\mathsf C_{i+1}\). This proves the second line of \eqref{eq:7-9}.

\textbf{The Full-Space implication.} The numbers \(n_j\) are real, but
membership in \(\mathscr L(G,20)\) is defined for real outer scales. If
\(n_j\in\mathscr L(G,20)\), then \(\mathsf C_{j+1}\) applies at
\(m=n_j\), and
\(e^{(\log n_j)^{10}}\geq2n_{j+1}\), so it directly gives
\(\mathsf F_{j+1}\). If \(n_j\notin\mathscr L(G,20)\), failure of its
defining universal condition supplies a real scale
\(r\in[n_j^{1/3},n_j]\) with
\(\log\operatorname{Gr}(r)>(\log r)^{20}\). Apply Proposition 6.2
along a geodesic of length at most \(2n_{j+1}\), using balls whose
radius is the explicit integer \(\rho=\lfloor r/3\rfloor\), the endpoint lower
bound \(e^{-T_j}\),
the sub-sprinkle
\[
 p_j'=\operatorname{Spr}(p_j;T_j^{-1})\leq p_{j+1},
\]
and ghost intensity \(h=h_j\) from \eqref{eq:7-2g}. Submultiplicativity supplies
\(|B_\rho|\geq h^{-1}\). Moreover, the union of two balls at consecutive
geodesic vertices has diameter at most \(2\rho+1\leq n_j\), so
\(\mathsf F_j\) gives the adjacent-set hypothesis of Proposition 6.2.
The chain contains at most \(2n_{j+1}+1\) sets. If
\(p_{j+1}<p_c^{\rm s}\), then throughout the relevant sprinkling
interval there is no infinite cluster, and
\(\mathbb P_q(\operatorname{Piv}^{\rm s}[1,R])\downarrow0\) as
\(R\to\infty\), uniformly in \(q\): the pivotal event is bounded by the
event that one of the at most \(d\) neighboring clusters reaches distance
\(R-1\), which is dominated by its probability at the subcritical upper
endpoint. Hence one may choose a finite thickening radius \(R\) with
\(hR\geq1\) that satisfies the pivotal hypothesis of Proposition 6.2.
Since the base domain here is all of \(V\), its value does not enlarge the
output domain. If instead
\(p_{j+1}\geq p_c^{\rm s}\), the alternative in \eqref{eq:7-9} already holds.
The baseline part of \(\delta_0\) is \(T_0^{-1}\), so this sub-sprinkle
is available also when \(j=0\). Equation \eqref{eq:7-2g} verifies the chain-length
and error hypotheses. The conclusion is at least
\(c'e^{-2T_j}\geq e^{-3T_j}=e^{-T_{j+1}}\), first at \(p_j'\) and then
at \(p_{j+1}\) by monotonicity. For \(j=0\),
\(\mathsf F_0\) is stronger than the endpoint bound \(e^{-T_0}\) used
here. This proves the third line. In each line the alternative is
introduced before snowballing, so all uses of Proposition 6.2 occur
below \(p_c^{\rm s}\).

We now complete the proof of
Proposition~\ref{prop:site-multiscale-criterion}. Apply those implications
with \(n_0=n\) and \(p_0=p\), using the sequences just defined in \eqref{eq:7-8}.
Either some
\(p_i\geq p_c^{\rm s}\), or the
Full-Space lower bound propagates at every \(n_i\). In the latter case,
if \(p_\infty<p_c^{\rm s}\), site sharpness
\cite[Theorems~2--3 and Section~6]{AntunovicVeselic2008} would give
\(\mathbb P_{p_\infty}(u\leftrightarrow v)\leq e^{-c d(u,v)}\),
contradicting the propagated lower bound
\(e^{-(\log\log n_i)^{1/2}}\). Thus \(p_c^{\rm s}\leq p_\infty\).
The sprinkling semigroup and the geometric series for \(\delta_i\) give
\[
 p_\infty\leq\operatorname{Spr}\left(
 p;2(\log\log n)^{-1/2}
 +K\operatorname{Burn}^{\rm s}(G,n,p)\right),
\]
which is \eqref{eq:7-2}. \(\square\)

\section{Proof of the main theorem}

\begin{proof}[Proof of Theorem~\ref{thm:site-locality}]

Let \(G_n\to G\), and pass to a tail on which the degrees equal \(d\) and
\(p_c^{\rm s}(G_n)\leq1-b\).

\par\noindent\textbf{1.} Lemma~\ref{lem:lower-semicontinuity} gives
   \(p_c^{\rm s}(G)\leq\liminf_n p_c^{\rm s}(G_n)\).
\par\noindent\textbf{2.} If \(G\) has polynomial growth, use the site case of the
   Contreras--Martineau--Tassion locality theorem
   \cite[Theorem~1.1]{ContrerasMartineauTassion2023}.
\par\noindent\textbf{3.} If the canonical action
   \(\operatorname{Aut}(G)\curvearrowright G\) is nonunimodular, use
   Theorem~\ref{prop:nonunimodular-height} for \(G\) and the eventual
   nonunimodular tail
   \(G_n\), always with the canonical actions
   \(\operatorname{Aut}(G)\) and \(\operatorname{Aut}(G_n)\), in the
   height-locality argument \eqref{eq:5-12d}--\eqref{eq:5-13}.
\par\noindent\textbf{4.} Otherwise the canonical action of \(G\) is
   unimodular and \(G\) has superpolynomial growth. This branch includes
   every canonically unimodular graph, even if it admits some other closed
   transitive nonunimodular subgroup. The canonical actions of \(G_n\) are
   eventually unimodular. The approximating graphs are also eventually
   not one-dimensional. Indeed, if an
   infinite locally finite graph satisfies \(|B_R(o)|\leq CR\) for all
   \(R\geq1\), then
   \(|B_R(o)|=1+\sum_{j=1}^R|S_j(o)|\) shows that infinitely many of its
   pairwise disjoint spheres have cardinality at most
   \(M:=\lceil2C\rceil\); otherwise the displayed sum eventually exceeds
   \(CR\). Select such spheres \((\Sigma_i)_{i\geq1}\). Every
   \(\Sigma_i\) separates \(o\) from infinity. For every \(p<1\), let
   \(E_i\) be the event that every site of \(\Sigma_i\) is closed. Their
   vertex-disjointness makes the events
   \(E_i\) independent, while
   \(\mathbb P_p(E_i)=(1-p)^{|\Sigma_i|}\geq(1-p)^M>0\).
   Therefore \(\sum_i\mathbb P_p(E_i)=\infty\), and the second
   Borel--Cantelli lemma says that infinitely many \(E_i\)'s occur almost
   surely. Every infinite path from the root must meet every separating
   cutset, so the occurrence of even one \(E_i\) blocks an infinite open
   root path. Hence \(p_c^{\rm s}=1\) for every transitive graph of
   linear growth. This contradicts
   \(p_c^{\rm s}(G_n)\leq1-b\). If upper semicontinuity failed, pass to a
   subsequence and set
   \[
      p_*:=\inf_n p_c^{\rm s}(G_n)>p_c^{\rm s}(G),\qquad
      p_0=\tfrac12\bigl(p_*+p_c^{\rm s}(G)\bigr).
   \]
   Let \(m(R)\) be such that \(B_R(G_j)\cong B_R(G)\) for
   \(j\geq m(R)\). We claim that
   \begin{equation}
    \lim_{R\to\infty}\sup_{q\in[p_0,p_*]}
       \sup_{j\geq m(R)}\operatorname{Burn}^{\rm s}(G_j,R,q)=0.
                                                               \label{eq:8-1}
   \end{equation}
   To prove this, fix \(m\in\mathscr L(G_j,20)\cap
   [\sqrt{\log R},R]\) and put \(N_m=\lfloor m^{1/3}\rfloor\).  The
   low-growth definition and monotonicity give, for large \(m\),
   \(\log|B_{N_m}(G_j)|\leq(\log N_m)^{21}\); the harmless extra power
   accounts for the floor in \(N_m\). Apply the low-growth
   version \eqref{eq:6-10} with \(K=21\), \(\varepsilon=1/8\), and outer radius
   \(N_m\), uniformly for \(q\in[p_0,p_*]\). If \(c_0=c_0(d,p_0)>0\)
   is its inner-radius constant, choose \(c<c_0/16\). Since
   \(4\lfloor c\log m\rfloor\leq c_0\log N_m\) for all sufficiently
   large \(m\),
   \[
    \mathbb P_q(\operatorname{Piv}^{\rm s}
       [4\lfloor c\log m\rfloor,N_m])
    \leq C\left(\frac{(\log N_m)^{21}}{N_m}\right)^{3/8}
    \leq(\log m)^{-1}.
   \]
   Consequently
   \begin{equation}
      b_{\rm s}(m,q)\geq c\log m                              \label{eq:8-1a}
   \end{equation}
   after decreasing \(c\) to meet the upper restriction in its
   definition.  Since \(c\log m\leq R\), local agreement gives
   \[
    |B_{\lfloor c\log m\rfloor}(G_j)|
      =|B_{\lfloor c\log m\rfloor}(G)|.
   \]
   Trofimov's near-polynomial growth theorem
   \cite[Theorem~1]{Trofimov2003NearPolynomial} says that a connected
   vertex-transitive graph whose ball volumes are bounded by one fixed
   polynomial along an unbounded sequence of radii has polynomial growth.
   Its contrapositive, applied to the non-polynomial-growth graph \(G\),
   implies
   \begin{equation}
    \inf_{t\geq c\log\sqrt{\log R}}
       \frac{\log|B_t(G)|}{\log t}\longrightarrow\infty.       \label{eq:8-1b}
   \end{equation}
   Equations \eqref{eq:8-1a}--\eqref{eq:8-1b}, monotonicity of ball volume, and
   \(\log m/\log\log m\to\infty\), uniformly for
   \(m\geq\sqrt{\log R}\), show that every term in the supremum \eqref{eq:7-1}
   tends to zero uniformly in \(q,j\).  This proves \eqref{eq:8-1}.

   By its definition, \(p_0>p_c^{\rm s}(G)\). Separately, the elementary
   simple-path bound gives
   \(p_c^{\rm s}(G)\geq1/(d-1)>1/d\). Apply
   Proposition~\ref{prop:site-multiscale-criterion} with \(a=1/d\) and
   with the tail gap \(b\), and let \(K,N\) be its constants. For \(R\geq N\) and
   \(j\geq m(R)\), set
   \[
    \varepsilon_{j,R}
      =2(\log\log R)^{-1/2}
       +K\operatorname{Burn}^{\rm s}(G_j,R,p_0).
   \]
   By \eqref{eq:8-1}, \(\sup_{j\geq m(R)}\varepsilon_{j,R}\to0\).  Hence there
   is \(R_0\geq N\) such that, for every \(R\geq R_0\) and
   \(j\geq m(R)\),
   \[
      \varepsilon_{j,R}\leq1,\qquad
      \operatorname{Spr}(p_0;\varepsilon_{j,R})<p_*.
   \]
   The first inequality also bounds the smaller \(\delta_0\) appearing
   in Proposition~\ref{prop:site-multiscale-criterion}, and the second
   implies its required upper bound
   \(1-b\), since \(p_*\leq p_c^{\rm s}(G_j)\leq1-b\).

   Fix \(c_{\rm tail}>0\) so small that
   \(2c_{\rm tail}\log(2/p_0)<1/4\), and for each \(k\geq1\) set
   \[
      R_k=\lceil c_{\rm tail}\log k\rceil.
   \]
   For every sufficiently large \(k\), \(R_k\geq R_0\).  Since
   \(p_c^{\rm s}(G_j)\geq p_*\), the conclusion of
   Proposition~\ref{prop:site-multiscale-criterion}
   at scale \(R_k\) is false for every \(j\geq m(R_k)\).  Its
   contrapositive therefore gives
   vertices \(u,v\in B_{R_k}\)
   such that
   \[
    \mathbb P_{p_0}^{G_j}(u\leftrightarrow v)
       <\eta(R_k):=
       \exp\!\left[-\sqrt{\log\log((\log R_k)^{1/2})}\right].
   \]
   Since \(p_0<p_*\leq p_c^{\rm s}(G_j)\), site percolation at \(p_0\)
   on \(G_j\) has no infinite cluster. Hence
   Lemma~\ref{lem:finite-energy-path-opening}, applied to \(u,v\) and a
   geodesic of length at most \(2R_k\), yields
   \begin{equation}
    \mathbb P_{p_0}^{G_j}(|K_o|\geq k)^2
    \leq C_{d,p_0}R_k(2/p_0)^{2R_k+1}k^{-1/2}
       +\eta(R_k).                                              \label{eq:8-3}
   \end{equation}
   The choice of \(c_{\rm tail}\) makes the first term
   \(O(k^{-1/4})\), after enlarging the implicit constant to absorb the
   logarithmic factor, while \(\eta(R_k)\to0\). Hence the
   left side tends to zero uniformly over \(j\geq m(R_k)\). But
   \(p_0>p_c^{\rm s}(G)\), so
   \(\theta_{p_0}^{\rm s}(G)>0\). Choose \(k\) so large that the bound
   in \eqref{eq:8-3} is smaller than
   \(\theta_{p_0}^{\rm s}(G)^2/4\), and then choose
   \(j\geq m(R_k)\) so large that \(B_k(G_j)\cong B_k(G)\). The event
   \(\{|K_o|\geq k\}\) is local: a breadth-first exploration discovers
   \(k\) cluster vertices or stops after inspecting only
   \(B_{k-1}(o)\). Hence the rooted \(k\)-ball isomorphism gives
   \(\mathbb P_{p_0}^{G_j}(|K_o|\geq k)
     =\mathbb P_{p_0}^{G}(|K_o|\geq k)\). An open path from \(o\) to
   \(B_k(o)^c\) contains at least \(k+1\) distinct vertices. Repeating
   the local-event equality in the comparison itself gives
   \[
   \begin{aligned}
    \mathbb P_{p_0}^{G_j}(|K_o|\geq k)
       &=\mathbb P_{p_0}^{G}(|K_o|\geq k)\\
       &\geq\mathbb P_{p_0}^{G}(o\leftrightarrow B_k^c)
        \geq\theta_{p_0}^{\rm s}(G),
   \end{aligned}
   \]
   contradicting \eqref{eq:8-3}.

Thus
\[
 \limsup_n p_c^{\rm s}(G_n)\leq p_c^{\rm s}(G).
\]
Together with Step 1, this completes the proof.
\end{proof}

\section*{Acknowledgements}
We thank Khoi M. N. Nguyen and Viet-Hoang Tran for their technical
guidance and support.

\raggedbottom
\begingroup
\raggedright
\bibliographystyle{plain}
\bibliography{references}

@misc{EasoHutchcroft2023,
  author        = {Philip Easo and Tom Hutchcroft},
  title         = {The Critical Percolation Probability Is Local},
  year          = {2023},
  eprint        = {2310.10983v1},
  archiveprefix = {arXiv},
  primaryclass  = {math.PR},
  doi           = {10.48550/arXiv.2310.10983},
  url           = {https://arxiv.org/abs/2310.10983v1},
  note          = {arXiv:2310.10983v1, preprint, 17 October 2023}
}

@article{EasoHutchcroft2025Presentation,
  author  = {Philip Easo and Tom Hutchcroft},
  title   = {Uniform Finite Presentation for Groups of Polynomial Growth},
  journal = {Discrete Analysis},
  year    = {2025},
  note    = {Paper No.~1, 29 pp.},
  doi     = {10.19086/da.127778},
  url     = {https://doi.org/10.19086/da.127778}
}

@article{TesseraTointon2021,
  author  = {Romain Tessera and Matthew C. H. Tointon},
  title   = {A Finitary Structure Theorem for Vertex-Transitive Graphs of
             Polynomial Growth},
  journal = {Combinatorica},
  year    = {2021},
  volume  = {41},
  number  = {2},
  pages   = {263--298},
  doi     = {10.1007/s00493-020-4295-6},
  url     = {https://doi.org/10.1007/s00493-020-4295-6}
}

@article{BreuillardGreenTao2012,
  author  = {Emmanuel Breuillard and Ben Green and Terence Tao},
  title   = {The Structure of Approximate Groups},
  journal = {Publications Math{\'e}matiques de l'IH{\'E}S},
  year    = {2012},
  volume  = {116},
  pages   = {115--221},
  doi     = {10.1007/s10240-012-0043-9},
  url     = {https://doi.org/10.1007/s10240-012-0043-9}
}

@article{BreuillardTointon2016,
  author  = {Emmanuel Breuillard and Matthew C. H. Tointon},
  title   = {Nilprogressions and Groups with Moderate Growth},
  journal = {Advances in Mathematics},
  year    = {2016},
  volume  = {289},
  pages   = {1008--1055},
  doi     = {10.1016/j.aim.2015.11.025},
  url     = {https://doi.org/10.1016/j.aim.2015.11.025}
}

@article{Timar2007Cutsets,
  author  = {{\'A}d{\'a}m Tim{\'a}r},
  title   = {Cutsets in Infinite Graphs},
  journal = {Combinatorics, Probability and Computing},
  year    = {2007},
  volume  = {16},
  number  = {1},
  pages   = {159--166},
  doi     = {10.1017/S0963548306007838},
  url     = {https://doi.org/10.1017/S0963548306007838}
}

@article{Hutchcroft2020Locality,
  author  = {Tom Hutchcroft},
  title   = {Locality of the Critical Probability for Transitive Graphs of
             Exponential Growth},
  journal = {The Annals of Probability},
  year    = {2020},
  volume  = {48},
  number  = {3},
  pages   = {1352--1371},
  doi     = {10.1214/19-AOP1395},
  url     = {https://doi.org/10.1214/19-AOP1395}
}

@article{Hutchcroft2020Nonunimodular,
  author  = {Tom Hutchcroft},
  title   = {Nonuniqueness and Mean-Field Criticality for Percolation on
             Nonunimodular Transitive Graphs},
  journal = {Journal of the American Mathematical Society},
  year    = {2020},
  volume  = {33},
  number  = {4},
  pages   = {1101--1165},
  doi     = {10.1090/jams/953},
  url     = {https://doi.org/10.1090/jams/953}
}

@book{LyonsPeres2016,
  author    = {Russell Lyons and Yuval Peres},
  title     = {Probability on Trees and Networks},
  series    = {Cambridge Series in Statistical and Probabilistic Mathematics},
  volume    = {42},
  publisher = {Cambridge University Press},
  address   = {New York},
  year      = {2016},
  pages     = {xv+699},
  doi       = {10.1017/9781316672815},
  url       = {https://doi.org/10.1017/9781316672815}
}

@article{MorrisPeres2005,
  author  = {Ben Morris and Yuval Peres},
  title   = {Evolving Sets, Mixing and Heat Kernel Bounds},
  journal = {Probability Theory and Related Fields},
  year    = {2005},
  volume  = {133},
  number  = {2},
  pages   = {245--266},
  doi     = {10.1007/s00440-005-0434-7},
  url     = {https://doi.org/10.1007/s00440-005-0434-7}
}

@article{ContrerasMartineauTassion2023,
  author  = {Daniel Contreras and S{\'e}bastien Martineau and Vincent Tassion},
  title   = {Locality of Percolation for Graphs with Polynomial Growth},
  journal = {Electronic Communications in Probability},
  year    = {2023},
  volume  = {28},
  pages   = {1--9},
  doi     = {10.1214/22-ECP508},
  url     = {https://doi.org/10.1214/22-ECP508},
  note    = {Paper No.~1}
}

@article{Li2025,
  author  = {Zhongyang Li},
  title   = {Critical Site Percolation and Cutsets},
  journal = {Electronic Communications in Probability},
  year    = {2025},
  volume  = {30},
  pages   = {1--9},
  doi     = {10.1214/25-ECP679},
  url     = {https://doi.org/10.1214/25-ECP679}
}

@article{Timar2006,
  author  = {{\'A}d{\'a}m Tim{\'a}r},
  title   = {Percolation on Nonunimodular Transitive Graphs},
  journal = {The Annals of Probability},
  year    = {2006},
  volume  = {34},
  number  = {6},
  pages   = {2344--2364},
  doi     = {10.1214/009117906000000494},
  url     = {https://doi.org/10.1214/009117906000000494}
}

@article{BenjaminiLyonsPeresSchramm1999,
  author  = {Itai Benjamini and Russell Lyons and Yuval Peres and Oded Schramm},
  title   = {Group-Invariant Percolation on Graphs},
  journal = {Geometric and Functional Analysis},
  year    = {1999},
  volume  = {9},
  number  = {1},
  pages   = {29--66},
  doi     = {10.1007/s000390050080},
  url     = {https://doi.org/10.1007/s000390050080}
}

@article{Hutchcroft2016,
  author  = {Tom Hutchcroft},
  title   = {Critical Percolation on Any Quasi-Transitive Graph of Exponential
             Growth Has No Infinite Clusters},
  journal = {Comptes Rendus Math{\'e}matique},
  year    = {2016},
  volume  = {354},
  number  = {9},
  pages   = {944--947},
  doi     = {10.1016/j.crma.2016.07.013},
  url     = {https://doi.org/10.1016/j.crma.2016.07.013}
}

@book{Grimmett1999,
  author    = {Geoffrey Grimmett},
  title     = {Percolation},
  series    = {Grundlehren der Mathematischen Wissenschaften},
  volume    = {321},
  edition   = {Second},
  publisher = {Springer},
  address   = {Berlin},
  year      = {1999},
  doi       = {10.1007/978-3-662-03981-6},
  url       = {https://doi.org/10.1007/978-3-662-03981-6}
}

@article{AntunovicVeselic2008,
  author  = {Ton{\'c}i Antunovi{\'c} and Ivan Veseli{\'c}},
  title   = {Sharpness of the Phase Transition and Exponential Decay of the
             Subcritical Cluster Size for Percolation on Quasi-Transitive
             Graphs},
  journal = {Journal of Statistical Physics},
  year    = {2008},
  volume  = {130},
  number  = {5},
  pages   = {983--1009},
  doi     = {10.1007/s10955-007-9459-x},
  url     = {https://arxiv.org/abs/0707.1089v3},
  eprint  = {0707.1089v3},
  archivePrefix = {arXiv},
  primaryClass = {math.PR},
  note    = {The theorem numbering cited in the text refers to arXiv v3}
}

@article{Trofimov2003NearPolynomial,
  author  = {Vladimir I. Trofimov},
  title   = {Undirected and Directed Graphs with Near Polynomial Growth},
  journal = {Discussiones Mathematicae Graph Theory},
  year    = {2003},
  volume  = {23},
  number  = {2},
  pages   = {383--391},
  doi     = {10.7151/dmgt.1208},
  url     = {https://doi.org/10.7151/dmgt.1208}
}

@article{vanDenBergKesten1985,
  author  = {Jacob van den Berg and Harry Kesten},
  title   = {Inequalities with Applications to Percolation and Reliability},
  journal = {Journal of Applied Probability},
  year    = {1985},
  volume  = {22},
  number  = {3},
  pages   = {556--569},
  doi     = {10.2307/3213860},
  url     = {https://doi.org/10.2307/3213860}
}

@article{Reimer2000,
  author  = {David Reimer},
  title   = {Proof of the {Van den Berg--Kesten} Conjecture},
  journal = {Combinatorics, Probability and Computing},
  year    = {2000},
  volume  = {9},
  number  = {1},
  pages   = {27--32},
  doi     = {10.1017/S0963548399004113},
  url     = {https://doi.org/10.1017/S0963548399004113}
}

@article{Russo1981,
  author  = {Lucio Russo},
  title   = {On the Critical Percolation Probabilities},
  journal = {Zeitschrift f{\"u}r Wahrscheinlichkeitstheorie und Verwandte
             Gebiete},
  year    = {1981},
  volume  = {56},
  number  = {2},
  pages   = {229--237},
  doi     = {10.1007/BF00535742},
  url     = {https://doi.org/10.1007/BF00535742}
}

@article{Talagrand1994,
  author  = {Michel Talagrand},
  title   = {On {Russo}'s Approximate Zero-One Law},
  journal = {The Annals of Probability},
  year    = {1994},
  volume  = {22},
  number  = {3},
  pages   = {1576--1587},
  doi     = {10.1214/aop/1176988612},
  url     = {https://doi.org/10.1214/aop/1176988612}
}

@article{AizenmanKestenNewman1987,
  author  = {Michael Aizenman and Harry Kesten and Charles M. Newman},
  title   = {Uniqueness of the Infinite Cluster and Continuity of Connectivity
             Functions for Short and Long Range Percolation},
  journal = {Communications in Mathematical Physics},
  year    = {1987},
  volume  = {111},
  number  = {4},
  pages   = {505--531},
  doi     = {10.1007/BF01219071},
  url     = {https://doi.org/10.1007/BF01219071}
}

@book{Diestel2017,
  author    = {Reinhard Diestel},
  title     = {Graph Theory},
  series    = {Graduate Texts in Mathematics},
  volume    = {173},
  edition   = {Fifth},
  publisher = {Springer},
  address   = {Berlin},
  year      = {2017},
  doi       = {10.1007/978-3-662-53622-3},
  url       = {https://doi.org/10.1007/978-3-662-53622-3}
}

@book{BoucheronLugosiMassart2013,
  author    = {St{\'e}phane Boucheron and G{\'a}bor Lugosi and Pascal Massart},
  title     = {Concentration Inequalities: A Nonasymptotic Theory of
               Independence},
  publisher = {Oxford University Press},
  address   = {Oxford},
  year      = {2013},
  doi       = {10.1093/acprof:oso/9780199535255.001.0001},
  url       = {https://doi.org/10.1093/acprof:oso/9780199535255.001.0001}
}

@article{Cornulier2019Ends,
  author  = {Yves Cornulier},
  title   = {On the Space of Ends of Infinitely Generated Groups},
  journal = {Topology and its Applications},
  year    = {2019},
  volume  = {263},
  pages   = {279--298},
  doi     = {10.1016/j.topol.2019.05.013},
  eprint  = {1901.11073},
  archivePrefix = {arXiv},
  primaryClass = {math.GR},
  url     = {https://arxiv.org/abs/1901.11073}
}
\endgroup

\end{document}